\documentclass[11pt,onecoulme]{article}
\usepackage{amsfonts}
\usepackage{mathrsfs}
\usepackage{amssymb}
\usepackage{color, amsmath,amssymb, amsfonts, amstext,amsthm, latexsym}

\usepackage{amssymb, epsfig, amssymb, latexsym}
\usepackage{amsmath}
\usepackage{graphicx}
\usepackage{longtable}
\usepackage{appendix}
\DeclareMathOperator{\esssup}{esssup}

\numberwithin{equation}{section}

\allowdisplaybreaks

\newtheorem{definition}{Definition}[section]
\newtheorem{theorem}[definition]{Theorem}
\newtheorem{lemma}[definition]{Lemma}
\newtheorem{remark}[definition]{Remark}

\newtheorem{example}[definition]{Example}

  \renewcommand\appendix{\par
    \setcounter{section}{0}
    \setcounter{subsection}{0}
    \gdef\thesection{ Appendix \Alph{section}}}

\newtheorem{assumption}{Assumption}

\title{Optimal Control  of Unbounded Functional Stochastic Evolution Systems in Hilbert Spaces: Second-Order Path-dependent  HJB Equation}
\author{Shanjian Tang\footnote{Institute of Mathematical Finance and Department of Finance and Control Sciences,
	School of Mathematical Sciences, Fudan University,
	Shanghai 200433, P. R. China, sjtang@fudan.edu.cn. This author is partially supported by National Science Foundation of China (Grant
	No. 11631004) and National Key R\&D Program of China (Grant No. 2018YFA0703903)} \quad Jianjun Zhou\footnote{College of Science,
	Northwest A\&F University, Yangling 712100, Shaanxi, P. R.
	China, zhoujianjun@nwsuaf.edu.cn. This author is partially supported by  the National Natural Science Foundation of China  (Grant No. 11401474), Shaanxi Natural Science Foundation
	(Grant No. 2021JM-083)}  }

\date{}

\begin{document}

\maketitle

\pagestyle{plain}

\begin{abstract}
	Optimal control and the associated second-order path-dependent Hamilton-Jacobi-Bellman (PHJB) equation are studied for unbounded functional stochastic evolution systems in Hilbert spaces.  The notion of viscosity solution without $B$-continuity is introduced in the sense of Crandall and Lions, and is shown to coincide  with the classical solutions and to satisfy a stability property. The value functional  is proved to be the unique continuous viscosity solution to the associated PHJB equation, without assuming any $B$-continuity  on the coefficients. In particular, in the Markovian case,  our result provides a new theory of viscosity solutions to the Hamilton-Jacobi-Bellman equation for optimal  control of stochastic evolutionary equations---driven by a linear unbounded operator---in a Hilbert space, and removes  the  $B$-continuity  assumption on the coefficients, which was initially introduced for first-order equations by Crandall and  Lions (see \emph{J. Func. Anal.}  \textbf{90} (1990), 237-283; \textbf{97} (1991), 417-465), and was subsequently used by \'Swiech (\emph{Comm. Partial Differential Equations} \textbf{19} (1994), 1999-2036) and  Fabbri, Gozzi, and \'Swiech (\emph{Probability Theory and Stochastic Modelling} \textbf{82}, 2017, Springer, Berlin).

\medskip


\end{abstract}

%
%
%
%
\newpage
\section{Introduction}

\par
Let $\{W(t),t\geq0\}$ be a cylindrical Wiener process in
Hilbert space $\Xi$ on a  complete probability  space  $(\Omega,{\mathcal {F}}, P)$, with $\{{\mathcal {F}}_{s}\}_{0\leq s\leq T}$ being the  natural filtration, augmented with the totality ${\mathcal{N}}$ of all $P$-null sets in $\mathcal{F}$. The process  $u(\cdot)=\{u(s), s\in [t,T]\}$ is  ${\mathcal {F}}_{s}$-progressively measurable,  taking values in a Polish space $(U,d)$; and the totality of such processes is denoted by  $ {\mathcal{U}}[t,T]$.
Consider a strongly continuous semi-group   $\{e^{tA}, t\geq0\}$  of bounded linear operators  in Hilbert space
$H$ with $A$ being the generator.

Denote by $\Lambda_t(H)$ (or $\Lambda_t$  for simplicity)  the totality of all
continuous $H$-valued functions defined on $[0,t]$,  and set ${\Lambda}:={\Lambda}(H):=\bigcup_{t\in[0,T]}{\Lambda}_{t}(H)$. For a path  $\gamma_t\in \Lambda_t$,  $\gamma_t(s)$ is the value of $\gamma_t$  at
time $s\in [0,t]$, and  the path $\gamma_{t, r}^A\in \Lambda_r$ with $r> t$ is defined as follows:
$$
\gamma_{t,r}^A(\sigma):=e^{(\sigma-t)^+A}\gamma_t(\sigma\wedge t),  \quad \sigma\in [0,r]. 
$$
For $\gamma_t,\eta_s\in {\Lambda}$, 
\begin{eqnarray*}
	||\gamma_t||_0:=\sup_{0\leq r\leq t}|\gamma_t(r)|;
\end{eqnarray*}
\begin{eqnarray*}
	d_{\infty}(\gamma_t,\eta_s)
	:=|s-t|+||\gamma_{t,t\vee s}^A-\eta_{s,s\vee t}^A||_0.
\end{eqnarray*}
Then, $||\cdot ||_0$ is a   norm on ${\Lambda}_t$  and $d_\infty(\cdot,\cdot)$ is a metric on ${\Lambda}$.

Assume that the pair of functionals $(F, G):\Lambda\times U\rightarrow H\times  L_2(\Xi, H)$   are jointly continuous with $(F,G)(\cdot, u)$ being uniformly  Lipschitz continuous  in  the path space {$(\Lambda, d_{\infty})$} for each $u\in U$.

For a fixed finite time $T>0$, consider the  following controlled functional stochastic evolution
equation (FSEE) in Hilbert space $H$:
\begin{eqnarray}\label{state1}
\begin{cases}
dX^{\gamma_t,u}(s)= AX^{\gamma_t,u}(s)ds+
F\left(X_s^{\gamma_t,u},u(s)\right)ds\\
\  \  \ \ \ \ \ \ \ \ \ \ \ \  \ \ \ \ +G\left(X_s^{\gamma_t,u},u(s)\right)dW(s),  \quad s\in (t,T];\\
\ \ \ \ \ X_t^{\gamma_t,u}=\gamma_t\in {\Lambda}_t,
\end{cases}
\end{eqnarray}
with $X^{\gamma_t,u}$ being  the $H$-valued unknown process. Here, $X^{\gamma_t,u}(s)$ is the value of $X^{\gamma_t,u}$  at
time $s$, and $X^{\gamma_t,u}_s$ is the whole history of the path  $X^{\gamma_t,u}$ from time 0 to $s$.

For given  $(t,\gamma_t)\in [0,T]\times {\Lambda}$,   we aim  at  maximizing the following utility functional
\begin{eqnarray}\label{cost1}
J(\gamma_t,u(\cdot)):=Y^{\gamma_t,u}(t),\quad u\in {{\mathcal 	{U}}}[t,T]
\end{eqnarray}
where the process $Y^{\gamma_t,u}$ is defined via solution of the following backward stochastic differential equation (BSDE):
\begin{eqnarray}\label{fbsde1}
\begin{aligned}
Y^{\gamma_t,u}(s)=&\phi(X_T^{\gamma_t,u})+\int^{T}_{s}q(X_\sigma^{\gamma_t,u},Y^{\gamma_t,u}(\sigma),Z^{\gamma_t,u}(\sigma),u(\sigma))d\sigma\\
&-\int^{T}_{s}Z^{\gamma_t,u}(\sigma)dW(\sigma),\quad \ a.s.\ \ \mbox{for all}\ \ s\in [t,T].
\end{aligned}
\end{eqnarray}
Here  $q: {\Lambda}\times \mathbb{R}\times \Xi\times U\to \mathbb{R}$ and $\phi: {\Lambda}_T\to \mathbb{R}$ are given real functionals and are  uniformly  Lipschitz continuous on  the path space $(\Lambda, ||\cdot||_0)$. The optimal stochastic control problem is to look for $\bar u\in \mathcal{U}[0,T]$ such that
$$J(\gamma_0,\bar u)=\sup_{u\in \mathcal{U}[0,T]} J(\gamma_0, u). $$
The value functional of the  optimal control problem is given by
\begin{eqnarray}\label{value1}
V(\gamma_t):=\mathop{\esssup}\limits_{u(\cdot)\in{\mathcal{U}}[t,T]} Y^{\gamma_t,u}(t),\ \  (t,\gamma_t)\in [0,T]\times {\Lambda}.
\end{eqnarray}

It is well-known that a powerful basic tool to  the optimal stochastic control problem is the so-called dynamic programming method---initially due to R. Bellman---which, in particular, indicates that
the value functional  $V$ should be ``the solution" of  the following second-order path-dependent Hamilton-Jacobi-Bellman (PHJB) equation:
\begin{eqnarray}\label{hjb1}
\quad&\begin{cases}
\partial_tV(\gamma_t)+\langle A^*\partial_xV(\gamma_t), \gamma_t(t)\rangle _H
+{\mathbf{H}}(\gamma_t,V(\gamma_t),\partial_xV(\gamma_t),\partial_{xx}V(\gamma_t))\\[2mm]
\qquad\qquad = 0,\quad   (t,\gamma_t)\in
[0,T)\times {\Lambda};\\[3mm]
V(\gamma_T)=\phi(\gamma_T),\quad \gamma_T\in {\Lambda}_T,
\end{cases}
\end{eqnarray}
with
\begin{eqnarray*}
	&&{\mathbf{H}}(\gamma_t,r,p,l)\\
	&:=&\sup_{u\in{
			{U}}}\left\{
	\langle p, F(\gamma_t,u)\rangle_H+\frac{1}{2}\mbox{Tr}[ l G(\gamma_t,u)G^*(\gamma_t,u)] +q(\gamma_t,r,pG(\gamma_t,u),u)\right\}, \\
	&&\qquad\qquad\qquad \ \  (t,\gamma_t,r,p,l)\in [0,T]\times{\Lambda}\times \mathbb{R}\times H\times {\mathcal{S}}(H).
\end{eqnarray*}
Here,   $A^*$ and $G^*$ are the adjoint operators of $A$ and $G$,  respectively, ${\mathcal{S}}(H)$ the space of bounded, self-adjoint operators on $H$ equipped with the operator norm,  ${\langle\cdot,\cdot\rangle_{H}}$ the scalar product of $H$,  and
$\partial_t,\partial_x$ and $\partial_{xx}$  the so-called pathwise  (or functional or Dupire; see \cite{dupire1, cotn0, cotn1}) derivatives, where $ \partial_t$  is  the horizontal derivative, and $\partial_x$ and $\partial_{xx}$  are the first- and  second-order vertical derivatives, respectively.

\par
A notion of viscosity solutions for second-order Hamilton-Jacobi-Bellman (HJB) equations in Hilbert spaces has been introduced  by Lions~\cite{lio1, lio3} for a bounded operator $A$, and by \'{S}wi\c{e}ch~\cite{swi} for an unbounded operator $A$. In the latter case, the notion of the so-called $B$-continuous viscosity solutions plays a crucial role, which was initially introduced for first-order equations by
Crandall and  Lions~\cite{cra6, cra7}. In  the earlier paper,  Lions~\cite{lio2}  considered
a specific second-order HJB equation for optimal control of the Zakai
equation using  some ideas of the $B$-continuous viscosity solutions. Based on the {Crandall-Ishii lemma} of Lions~\cite{lio3}, the first comparison theorem for $B$-continuous viscosity sub- and super-solutions was proved by \'{S}wi\c{e}ch~\cite{swi}. The reader is referred to   the monograph   of  Fabbri et al.~\cite{fab1} for a detailed account
of  the theory of viscosity solutions in a separable Hilbert space.   Their  structural assumption
that the state space is a separable Hilbert space,   excludes  our metric space $\Lambda$. The preceding PHJB equation~\eqref{hjb1}  involves the path-wise derivatives (rather than Fr\'{e}chet derivatives) in $\Lambda$, and the above existing results in separable Hilbert spaces do not apply  to $\Lambda$.

\par
The viscosity solution of fully nonlinear  first-order path-dependent  Hamilton-Jacobi equations is fairly complete now, and the reader is referred  to  Lukoyanov  \cite{luk} and  Zhou~\cite{zhou}  for path variables  in an Euclidean space, and  to Zhou~\cite{zhou6}  for path variables in a separable  Hilbert space. 
The  Crandall-Lions viscosity solution of  the second-order PHJB equation, even  in  finite dimensional spaces,  turns out to be quite elusive, and had been open for a long time.   Early  attempts are referred to Peng~\cite{peng2.5}, Tang and Zhang~\cite{tang1},   Zhou~\cite{zhou3}, and in particular, to the striking  series  of intensive studies initiated by by Ekren et al. \cite{ekren1}  and subsequently  continued by Ekren et al.~\cite{ekren3, ekren4},  Ren \cite{ren}, Ekren \cite{ekren0},  Ren et al.~\cite{ren1},  Ren and Rosestolato~\cite{ren2}, and Cosso et al.~\cite{cosso},   whose notion of viscosity solutions  is featured by appealing to a nonlinear $g$-expectation (relevant to an appropriate class of equivalent probability measures), and  therefore lacks the salient path-wise feature in  Crandall-Lions' notion of viscosity solution  whose  tangency condition is  typically point-wise.

Recently, using the Borwein-Preiss  variational principle  and constructing suitable smooth gauge-type functionals  in the path space,  the second author~\cite{zhou5} resolves the elusive uniqueness of  viscosity solutions of PHJB equations in the finite-dimensional space in  the sense of Crandall-Lions.  Note that Cosso and  Russo~\cite{cosso1} also gave
a  comparison theorem (and thus  the uniqueness ) of viscosity solutions to the path-dependent heat equation with the Borwein-Preiss variational principle in a quite different but  interesting manner.  None of the above results  directly applies to our current situation, as a Crandall-Ishii lemma is involved  in the infinite-dimensional context.



\par
{In this monograph},  we shall
develop a notion  of  Crandall-Lions viscosity solutions  to the PHJB equation~\eqref{hjb1} on the space of  $H$-valued continuous paths,  
and show that the value functional
$V$  defined in  (\ref{value1}) is the  unique viscosity solution to the PHJB equation (\ref{hjb1}).

\par
As mentioned earlier, in the Markovian case with unbounded linear term, the core of  Crandall-Lions viscosity solutions is $B$-continuity. Specifically, the $B$-continuity of the coefficients (which ensures the $B$-continuity of the value functional  and   the comparison theorem)  is assumed  to overcome the difficulty caused by the unbounded operator $A$. {In this monograph}, we  remove the $B$-continuity assumption on the coefficients,  used in the existing  theory of Crandall-Lions viscosity solutions, and in the meantime we  extend it to the infinite-dimensional path-dependent context.
\par

Our core results are as follows.   We define the functional $\Upsilon:\Lambda\rightarrow \mathbb{R}$ as follows:
$$
\Upsilon(\gamma_t):=S(\gamma_t)+3|\gamma_t(t)|^6, \quad \gamma_t\in \Lambda
$$
with
\begin{eqnarray*}
	S(\gamma_t):=\begin{cases}
		{||\gamma_{t}||^{-12}_{0}}(||\gamma_{t}||_{0}^6-|\gamma_{t}(t)|^6)^3,  &||\gamma_{t}||_{0}\neq0; \\
		0,  &||\gamma_{t}||_{0}=0.
	\end{cases}
\end{eqnarray*}
It is equivalent to $||\cdot||_0^6$ and is differentiable in the sense of horizontal/vertical derivatives. It
plays a crucial  role in the proof of our stability and uniqueness results. The uniqueness property is derived from the comparison theorem. Here are three key  ingredients of  the proof of  our comparison theorem.
\par
(a)
For every fixed $(t,\gamma_t)\in [0,T)\times\Lambda_t$,   the functional   $f(\eta_s):=\Upsilon(\eta_s-\gamma_{t,s}^A)$, $(s,\eta_s)\in[t,T]\times \Lambda$ turns out to be very powerful in the construction of  an effective  test function in our definition of viscosity solutions as we show that $f$ satisfies  a functional It\^o inequality. This is important as  functional $\Upsilon$ is equivalent to $||\cdot||_0^6$. Then we  define an auxiliary function $\Psi$ which includes the  functional $f$.  By this, we only need the continuity  under $||\cdot||_0$ rather than  the $B$-continuity of the value functional. In particular, the comparison theorem is  established
when the coefficients is uniformly  Lipschitz continuous  with respect to the norm $||\cdot||_0$.
We emphasize that, in our viscosity solution theory in infinite dimensional spaces, the
$B$-continuity assumption on the coefficients is  removed in our framework. As a  crucial  consequence,  also as a salient feature,  our notion of viscosity solutions is still powerful  even in the
Markovian  case. More precisely, in the
Markovian  case, the functional $f$ above  reduces to    $f(s,y):=|y-e^{(s-t)A}x|_H^2$, $(s,y)\in[t,T]\times H$,  for every fixed $(t,x)\in [0,T)\times H$.   With help of this test functional, the $B$-continuity assumption on the coefficients, conventionally used in the literature (see Fabbri et al.~\cite[Chapter 3, p. 171-365 ]{fab1}),  is dispensed with in this monograph.
\par
(b)
We  define in terms of $\Upsilon$ a  smooth  gauge-type function $\overline{\Upsilon}:\Lambda\times \Lambda\rightarrow \mathbb{R}$ by
$$
\overline{\Upsilon}(\gamma_t,\eta_s)
:=\Upsilon(\eta_{s,s\vee t}^A-\gamma_{t,t\vee s}^A)+|s-t|^2, \quad
(\gamma_t,\eta_s)\in \hat{\Lambda}^2.
$$
Then we apply
a modification of Borwein-Preiss variational principle (see   Borwein \& Zhu~\cite[Theorem 2.5.2]{bor1}) to ensure the existence of  a maximum point of a perturbed auxiliary function $\Psi$.
\par
(c)   Unfortunately, the second-order vertical derivative $\partial_{xx}\Upsilon$ does not vanish. 
To apply the Crandall-Ishii lemma (see Theorem \ref{theorem0513}
), a stronger convergence property of auxiliary functional is required.
With help of a delicate calculation, we fortunately obtain the required convergence property of the auxiliary functional  and succeed at proving the comparison theorem.

\par
Regarding existence, we  show that the value functional  $V$  defined in  (\ref{value1}) is  a viscosity solution to the PHJB
equation given in  (\ref{hjb1}) by functional It\^o  formula, functional It\^o inequality  and dynamic programming principle.
Such a formula was firstly provided  in
Dupire  \cite{dupire1}
(see Cont and  Fourni\'e  \cite{cotn0}, \cite{cotn1} for a more general and systematic research). {In this book} we extend the functional It\^o
formula to  infinite dimensional spaces. We also provide a functional It\^o inequality for $f$ defined in (a) from functional It\^o
formula.
\par
Finally, following Dupire  \cite{dupire1}, we study PHJB equation (\ref{hjb1}) in the metric space $(\Lambda,d_\infty)$ {in the book}, while some literatures (for example,  \cite{cosso,cosso1}) study PHJB equations in a complete pseudometric space. The reason why we do so is the convenience of defining $\gamma_{t,\cdot}^A$ for every $(t,\gamma_t)\in [0,T)\times \Lambda$ in our framework, which is useful in defining metric $d_\infty$ and test functional $f$.
\par
The contents of {the monograph} are  outlined as follows. In Chapter 2,  we give preliminaries on notations and  BSDEs, and a  functional It\^o formula which  will be used to prove the existence of
viscosity solutions to PHJB equation (\ref{hjb1}). We also present a modification of
Borwein-Preiss variational principle.
In Chapter 3, we  construct  the smooth  gauge-type functionals $\Upsilon^{m,M}$ which are the key to proving the stability and uniqueness results of viscosity solutions.
Chapter 4 is devoted to studying  FSEE (\ref{state1}).  A functional It\^o inequality for $f$ defined in (a) is also provided by  functional It\^o
formula.
In Chapter 5, we introduce  preliminary results on path-dependent stochastic  optimal control problems, and  the dynamic programming principle,  which will be used in the subsequent chapters.
In Chapter 6, we define classical and viscosity solutions to
PHJB equations (\ref{hjb1}) and  prove  that the value functional $V$ defined by (\ref{value1}) is a viscosity solution.   We also show
the consistency with the notion of classical solutions and the stability result.  A Crandall-Ishii lemma  for functionals defined on the space of  $H$-valued continuous paths 
is given in Chapter 7.
The uniqueness of viscosity solutions to  (\ref{hjb1}) is proved in Chapter 8, and is specialized at the Markovian case in Chapter 9.   Chapter 10 is devoted to PHJB equations with a quadratic growth in the gradient.
\newpage
\section{Preliminaries
}  \label{RDS}

\subsection{Spaces of paths in Hilbert spaces}
Let $\Xi$, $K$ and $H$ denote real separable Hilbert spaces of  inner products
$\langle\cdot,\cdot\rangle_\Xi$, $\langle\cdot,\cdot\rangle_K$ and $\langle\cdot,\cdot\rangle_H$, respectively.  The notation $|\cdot|$ stands for the norm in
various spaces, and a subscript is attached to indicate the underlying Hilbert space whenever  necessary.
$L(\Xi,H)$ is the space of all
bounded linear operators from $\Xi$ into $H$, and  the subspace of
all Hilbert-Schmidt operators, equipped with the Hilbert-Schmidt norm, is
denoted by $L_2(\Xi,H)$. Write $L(H):=L(H,H)$. Denote by ${\mathcal{S}}(H)$ the Banach space of all bounded and self-adjoint
operators in Hilbert space $H$,  endowed with the operator norm. The operator $A$,  with the domain being denoted by ${\mathcal {D}}(A)$, generates a strongly continuous semi-group   $\{e^{tA}, t\geq0\}$ of bounded linear operators in Hilbert space $H$. $A^*$ is  the adjoint operator of $A$, whose domain is denoted by ${\mathcal {D}}(A^*)$.
Let
$T>0$ be a fixed number.  For each  $t\in[0,T]$,
define
$\hat{\Lambda}_t:=D([0,t];H)$ as  the set of c$\grave{a}$dl$\grave{a}$g  $H$-valued
functions on $[0,t]$.
Write  $\hat{\Lambda}^t:=\bigcup_{s\in[t,T]}\hat{\Lambda}_{s}$  and  $\hat{\Lambda}:=\hat{\Lambda}^0$.
\par
As in Dupire \cite{dupire1}, an element of $\hat{\Lambda}$ is denoted by a lower
case letter with a subscript indicating the final time of the time domain whenever necessary, e.g. $\gamma_t\in \hat{\Lambda}_t\subset \hat{\Lambda}$. Note that, for any $\gamma\in  \hat{\Lambda}$, there exists only one $t$ such that $\gamma\in  \hat{\Lambda}_t$. For any $0\leq s\leq t$, the value of
$\gamma_t$ at time $s$ will be denoted by $\gamma_t(s)$. Moreover, if
a path $\gamma_t$ is fixed, the path $\gamma_t|_{[0,s]}$, for $0\leq s\leq t$, will denote the restriction of the path  $\gamma_t$ to the interval
$[0,s]$.
\par
Following Dupire \cite{dupire1},  we define for $( x,\bar{t}, \gamma_t)\in H\times [t,T]\times  \hat{\Lambda}_t$ with $t\in [0, T] $, the following three paths  $\gamma^x_{t}\in\hat{\Lambda}_t$ and $\gamma_{t,\bar{t}}, \gamma_{t,\bar{t}}^A\in \hat{\Lambda}_{\bar{t}}$ as follows:
\begin{eqnarray*}
	\gamma^x_{t}(s)&:=&\gamma_t(s) 1_{[0,t)}(s)+[\gamma_t(t)+x] 1_{\{t\}}(s), \quad s\in [0,t], \\
	\gamma_{t,\bar{t}}(s)&:=& \gamma_t(s\wedge t),  \quad s\in [0,\bar{t}], 
                              \quad \hbox{\rm and }\\
	\gamma_{t,\bar{t}}^A(s)&:=&e^{(s-t)^+A}\gamma_t(s\wedge t), 
                             \quad s\in [0,\bar{t}].
\end{eqnarray*}
We define a norm on $\hat{\Lambda}_t$  and a metric on $\hat{\Lambda}$ as follows: for any 
$\gamma_t,\eta_s\in \hat{\Lambda}$,
\begin{eqnarray}\label{2.1}
||\gamma_t||_0:=\sup_{0\leq s\leq t}|\gamma_t(s)|,\  d_{\infty}(\gamma_t,\eta_s):=|t-s|+||\gamma_{t,t\vee s}^A-\eta_{s,s\vee t}^A||_0.
\end{eqnarray}
Then $(\hat{\Lambda}_t, ||\cdot||_0)$ is a Banach space, and $(\hat{\Lambda}^t, d_{\infty})$ is a complete metric space by Lemma 5.1 in \cite{zhou6}.

{ In the book,}  whenever convenient, we also use the notation $f|_a$ or $f(\cdot)|_a$ or $f(x)|_{x=a}$ to denote the image of a map $f$ at the point $a$, and use the notation  $f|^b_a$, $f(\cdot)|^b_a$ or $f(x)|^{x=b}_{x=a}$ to denote the difference $f(b)-f(a)$ of the images of a map $f$ at two points $b$ and $a$.

\par Now we define the path  derivatives of Dupire \cite{dupire1}.
\begin{definition}\label{definitionc0} (Path derivatives)
	Let $t\in [0,T)$ and  $f:\hat{\Lambda}^t\rightarrow \mathbb{R}$.
	\begin{description}
		\item[(i)]  Given $(s,\gamma_s)\in [t,T)\times \hat{\Lambda}^t$, the horizontal derivative of $f$ at $\gamma_s$ (if the corresponding limit exists and is finite) is defined as
		\begin{eqnarray}\label{2.3}
		\partial_tf(\gamma_s):=\lim_{h\to 0^+}\frac{1}{h}\left[f(\gamma_{s,s+h})-f(\gamma_s)\right].
		\end{eqnarray}
		For the final time $T$, the horizontal derivative of $f$ at $\gamma_T\in\hat{\Lambda}^t$ (if the corresponding limit exists and is finite) is defined as
		$$
		\partial_tf(\gamma_T):=\lim_{s\uparrow T^-}\partial_tf(\gamma_T|_{[0,s]}).
		$$
		If the above limit exists and is finite for every $(s,\gamma_s)\in [t,T]\times \hat{\Lambda}^t$, the functional $\partial_tf:\hat{\Lambda}^t\rightarrow \mathbb{R}$ is called the horizontal derivative of $f$ with domain $\hat{\Lambda}^t$.
		\par
		\item[(ii)]  Given $(s,\gamma_s)\in [t,T]\times \hat{\Lambda}^t$,
		if there is $B\in H$ such that
		$$
		\lim_{|h|\rightarrow0}\frac{1}{|h|}{\left|f(\gamma_s^{h})-f(\gamma_s)-\langle B, h\rangle_H\right|}=0,
		$$
		we say $\partial_{x}f(\gamma_s):=B$ to be the first-order vertical derivative  of $f$ at $\gamma_s$. If $\partial_xf$ exists  for every $(s,\gamma_s)\in [t,T]\times \hat{\Lambda}^t$, the map $\partial_xf:\hat{\Lambda}^t\rightarrow H$ is called the first-order vertical  derivative of $f$ with domain $\hat{\Lambda}^t$.
		\par
		If   there is $B_1\in {\mathcal{S}}(H)$ such that
		$$
		\lim_{|h|\rightarrow0}\frac{1}{|h|}{\left|\partial_{x}f(\gamma_s^{h})-\partial_{x}f(\gamma_s)-B_1h\right|}=0,
		$$
		we say $\partial_{xx}f(\gamma_s):=B_1$ to be the second-order vertical derivative  of $f$ at $\gamma_s$.
		If $\partial_{xx}f$ exists  for every $(s,\gamma_s)\in [t,T]\times \hat{\Lambda}^t$, the map $\partial_{xx}f:\hat{\Lambda}^t\rightarrow {\mathcal{S}}(H)$ is called the second-order vertical  derivative of $f$ with domain $\hat{\Lambda}^t$.
	\end{description}
\end{definition}

\begin{definition}\label{definitionc0409}
	Let $f:\hat{\Lambda}^t\rightarrow K$ be given for  some $t\in[0,T)$.
	\begin{description}
		\item[(i)]
		We say $f\in C^0(\hat{\Lambda}^t,K)$ if $f$ is continuous  on the metric space $(\hat{\Lambda}^t, d_{\infty})$.
		\item[(ii)]
		We say $f\in C_p^0(\hat{\Lambda}^t,K)$ if $f\in C^0(\hat{\Lambda}^t,K)$ and $f$ grows  in a polynomial way, i.e.,  there exist constants $L>0$ and $q\geq 0$ such that, for all $\gamma_s\in \hat{\Lambda}^t$,
		$$
		|f(\gamma_s)|\leq L(1+||\gamma_s||_0^q).
		$$
	\end{description}
\end{definition}
For simplicity, we write $C^0(\hat{\Lambda}^t):=C^0(\hat{\Lambda}^t,\mathbb{R})$ and $C^0_p(\hat{\Lambda}^t):=C_p^0(\hat{\Lambda}^t,\mathbb{R})$.
\begin{definition}\label{definitionc04091}
	Let $t\in[0,T)$ and $f:\hat{\Lambda}^t\rightarrow \mathbb{R}$ be given.
	\begin{description}
		\item[(i)]  We say $f\in C^{1,2}(\hat{\Lambda}^t)\subset C^0(\hat{\Lambda}^t)$ if   $\partial_tf$, $\partial_{x}f$ and $\partial_{xx}f$ exist and are continuous in $\gamma_s$  on the metric space $(\hat{\Lambda}^t, d_{\infty})$.
		\par
		\item[(ii)] We say
		$f\in C^{1,2}_p(\hat{\Lambda}^t)$ if $f\in C^{1,2}(\hat{\Lambda}^t)$ and $f$ together with  all its derivatives grow  in a polynomial way.
	\end{description}
\end{definition}
\par
Recall that $\Lambda_t:= C([0,t],H)$ is the set of all continuous $H$-valued functions defined over $[0,t]$. Set ${\Lambda}^t:=\bigcup_{s\in[t,T]}{\Lambda}_{s}$  and ${\Lambda}:={\Lambda}^0$.
Clearly, $\Lambda=\bigcup_{t\in[0,T]}{\Lambda}_{t}\subset\hat{\Lambda}$, and each $\gamma\in \Lambda$ is also an element of $\hat{\Lambda}$. $(\Lambda_t, ||\cdot||_0)$ is a
Banach space, and $(\Lambda^t,d_{\infty})$ is a complete metric space (see~\cite[Lemma 5.1]{zhou6}).

\begin{definition}\label{definitionc2}
	Let  $t\in [0,T)$ and  $f:{\Lambda}^t\rightarrow \mathbb{R}$  be given. We say
	\begin{description}
		\item[(i)]
		$f\in C^0({\Lambda}^t)$ if $f$ is continuous on the metric space $(\Lambda^t, d_{\infty})$,
		\item[(ii)]
		$f\in C_p^0({\Lambda}^t)$ if $f\in  C^0({\Lambda}^t)$ and grows  in a polynomial way,  and
		\item[(iii)] $f\in C^{1,2}_p({\Lambda}^t)$ if
		it has an extension $\hat{f}\in C^{1,2}_p(\hat{{\Lambda}}^t)$.
	\end{description}
\end{definition}
\par

By a cylindrical Wiener process defined on a  complete probability  space
$(\Omega,{\mathcal {F}},P)$, and with values in a Hilbert space $\Xi$,
we mean a family $\{W(t),t\geq0\}$ of linear mappings $\Xi\rightarrow
L^2(\Omega)$ such that for every $\xi, \eta \in \Xi$,
$\{W(t)\xi,t\geq0\}$ is a real Wiener process and
${\mathbb{E}}[W(t)\xi\cdot W(t)\eta]=t \langle\xi,\eta\rangle_\Xi $.  The filtration $\{{\mathcal {F}}_{t}\}_{0\leq t\leq T}$  is the natural one of $W$, augmented with the totality  ${\mathcal{N}}$ of all the $P$-null sets of $\mathcal{F}$:
$$
{\mathcal{F}}_{t}=\sigma(W(s):s\in[0,t])\vee \mathcal
{N}.
$$
The filtration
$\{{\mathcal {F}}_{t}\}_{0\leq t\leq T}$  satisfies the usual conditions.
For every $[t_1,t_2]\subset[0,T]$, we also use the
notation:
$$
{\mathcal{F}}_{t_1}^{t_2}=\sigma(W(s)-W(t_1):s\in[t_1,t_2])\vee \mathcal
{N}.
$$
We also write ${\mathcal{F}}^t$ for $\{{\mathcal{F}}^s_t, t\leq s\leq T\}$.
\par
By $\mathcal{P}$,   we denote the predictable $\sigma$-algebra generated by predictable processes,  and by $\mathcal{B}(\Theta)$,  we denote the Borel $\sigma$-algebra of a topological space $\Theta$.

Next, we introduce spaces of random variables or stochastic processes, taking values in a Hilbert space $K$.

For  $p\in[1,\infty)$ and $t\in [0,T]$, $L^p(\Omega,{\mathcal{F}}_{t},P;K)$ is the space of $K$-valued ${\mathcal{F}}_{t}$-measurable random variables $\xi$, equipped with the norm
$$|\xi|=\mathbb{E}[|\xi|^p]^{1/p}<\infty;$$
$L^{p}_{\mathcal{P}}(\Omega\times [0,T];K)$ is  the space of all predictable  processes $y\in L^p(\Omega\times [0,T];K)$, equipped  with the norm
$$|{y}|=\mathbb{E}\left[\int_{0}^{T}|{y}(t)|^{p}dt\right]^{1/p}<\infty;$$
$L^{p}_{\mathcal{P}}(\Omega;L^{2}([0,T];K))$, is  the space of all predictable  processes $\{y(t),t\in [0,T]\}$, equipped with  the  norm
$$|{y}|=\mathbb{E}\left[\left(\int_{0}^{T}|{y}(t)|^{2}dt\right)^{p/2}\right]^{1/p}<\infty;$$
for  $t\in (0,T]$, {$ L^{p}_{\mathcal{P}}(\Omega,\Lambda_t(K))$} is  the space of all $K$-valued continuous predictable processes
$\{y(s),s\in[0,t]\}$, equipped with the norm
$$|{y}|=\mathbb{E}\left[\|y_t\|_0^{p}\right]^{1/p}<\infty.$$
Elements of $L^{p}_{\mathcal{P}}(\Omega,\Lambda_t(K))$ are identified if they are  indistinguishable.

\subsection{Functional  stochastic calculus}
Assume that $\vartheta\in L^{p}_{\mathcal{P}}(\Omega\times [0,T];H)$, $\varpi\in L^{p}_{\mathcal{P}}(\Omega\times [0,T];L_2(\Xi,H))$ for some $p>2$, and $(\theta,\xi_{\theta})\in [0,T)\times L_{\mathcal{P}}^p(\Omega,\Lambda_\theta(H))$, 
then the following process
\begin{eqnarray}\label{formular1}
\begin{aligned}
X(s)=&e^{(s-\theta)A}\xi_{\theta}(\theta)+\int^{s}_{\theta}e^{(s-\sigma)A}\vartheta(\sigma)d\sigma\\
&+\int^{s}_{\theta}e^{(s-\sigma)A}\varpi(\sigma)dW(\sigma),\quad s\in [\theta,T],
\end{aligned}
\end{eqnarray}
and  $X(s)=\xi_\theta(s)$ for $s\in [0,\theta)$,
is well defined and $\mathbb{E}[\|X\|_0^{p}]<\infty$ (see~\cite[Proposition  7.3]{da}).

\begin{lemma}\label{theoremito}
Suppose  that  $f\in C_p^{1,2}(\hat{\Lambda}^{{t}})$ with $A^*\partial_xf\in C_p^0(\hat{\Lambda}^{{t}},H)$
	for some $t\in[\theta,T)$. Then, under the above conditions, $P$-a.s.,  for all $s\in [t,T]$:
	\begin{eqnarray}\label{statesop0}
	\begin{aligned}
	f(X_s)=&f(X_{t})+\int_{t}^{s}\biggl[\partial_tf(X_\sigma)+\langle A^*\partial_xf(X_\sigma), \, X(\sigma)\rangle_H\\
	&
	+\langle \partial_xf(X_\sigma), \, \vartheta(\sigma)\rangle_H+\frac{1}{2}\mbox{\rm Tr}(\partial_{xx}f(X_\sigma)\varpi(\sigma)\varpi^*(\sigma))\biggr]d\sigma\\
	&+\int^{s}_{t}\langle \partial_xf(X_\sigma), \, \varpi(\sigma)dW(\sigma)\rangle_H.
	\end{aligned}
	\end{eqnarray}
	Here and in the following, for $s\in [0,T]$, $X(s)$ is  the value of $X$  at
	time $s$, $X_s$ is the restriction of the process $X$ to the time interval $[0, s]$, and $D^*$ is the adjoint of $D\in L_2(\Xi,H)$.
\end{lemma}

\begin{proof} The proof is  similar to   that of Cont \& Fournie \cite[Theorem 4.1]{cotn1}.  See also Dupire \cite{dupire1}.

We assume that the process $X(s),\ s\in[0,T]$ is
bounded. This can be shown by localization. Namely
for arbitrary constant $C> 0$ define a stopping time
$\tau_C$:
$$
{\tau_C=\inf\{s\in[0,T]: |X(s)|\geq C\}}
$$
with the convention that $\tau_C=T$
if this set is empty. 
Define
$$
\vartheta_C(s):={\mathbf{1}}_{[0,\tau_C]}\vartheta(s),\quad
\varpi_C(s):={\mathbf{1}}_{[0,\tau_C]}\varpi(s),\quad
s\in[0,T],
$$
and $X^{C}(s):={e^{(s-\tau_C)^+A}X (s\wedge \tau_C)}$, i.e.
\begin{eqnarray*}
	\begin{cases}
		X^{C}(s)=X(s),\ \ \ \ \ \ \ \ \ \ \ \ \ \ \ \ \ \ \ \ \mbox{if}\ s\leq \tau_C;\\
		X^{C}(s)=e^{(s-\tau_C)A}X(\tau_C), \ \ \ \ \ \ \  \mbox{if}\ s>\tau_C.
	\end{cases}
\end{eqnarray*}
Therefore,
\begin{eqnarray*}
	X^{C}(s)&=&e^{(s-\theta)A}{\xi_{\theta}(\theta\wedge\tau_C)}
+\int_{\theta}^{s}{e^{(s-\sigma)A}}\vartheta_C(\sigma)d\sigma\\
	&&+\int_{\theta}^{s}{e^{(s-\sigma)A}}\varpi_C(\sigma)dW(\sigma), \ \
	s\in [\theta,T],
\end{eqnarray*}
and $X^{C}(s)={e^{(s-\tau_C)^+A}\xi (s\wedge \tau_C)},
  \ s\in [0,\theta)$.
If the formula (\ref{statesop0}) is true for $\vartheta_C$, $\varpi_C$ and
$X^{C}$ for arbitrary $C>0$, then, 
it is true in the general
case.
\par
For  $s\in[t,T]$, define $$X^n(\sigma):=X{\mathbf{1}}_{[0,t)}(\sigma)+\sum^{2^n-1}_{i=0}X(t_{i+1}){\mathbf{1}}_{[t_i,t_{i+1})}(\sigma)+X(s){\mathbf{1}}_{\{s\}}(\sigma),  \quad \sigma\in [0,s]. $$
Here $t_i:=t+\frac{i(s-t)}{2^n}$. For $(\sigma,\gamma_\sigma)\in [0,T]\times\hat{\Lambda}$, define $\gamma_{\sigma-}$  by
$$
\gamma_{\sigma-}(l):=\gamma_{\sigma}(l),\quad l\in [0,\sigma);\quad  \mbox{and } \gamma_{\sigma-}(\sigma):=\lim_{l\uparrow \sigma}\gamma_{\sigma}(l).
$$ We start with the decomposition
\begin{eqnarray}\label{decom}
&&f({X^n_{t_{i+1}-}})-f({X^n_{t_{i}-}})=f({X^n_{t_{i+1}-}})-f({X^n_{t_{i}}})
+f({X^n_{t_{i}}})-f({X^n_{t_{i}-}}),\ i\geq0.
\end{eqnarray}
Let $\psi(\sigma):=f({X^n_{t_{i},t_i+\sigma}}), \sigma\in [0,h]$, we have
$$f({X^n_{t_{i+1}-}})-f({X^n_{t_{i}}})=\psi(h)-\psi(0),$$
with $h:=\frac{s-t}{2^n}$.
Denoting by $\psi_{t^+}$ the right derivative of $\psi$, we have
\begin{eqnarray*}
	&&\psi_{t^+}(l)=\lim_{\delta\to 0^+}\frac{1}{\delta}{[\psi(l+\delta)-\psi(l)]}\\
	&=&\lim_{\delta\rightarrow0^+}\frac{1}{\delta}{[f({X^n_{t_{i},t_i+l+\delta}})-f({X^n_{t_{i},t_i+l}})]}
	=\partial_tf(X_{t_i,t_i+l}^n),\quad  l\in[0,h).
\end{eqnarray*}
Since
\begin{eqnarray*}
	&&d_\infty({X^n_{t_{i},t_i+l_1}},{X^n_{t_{i},t_i+l_2}})\\
	&\leq&|l_1-l_2|+
                    \sup_{0\leq\sigma\leq |l_1-l_2|}|X^n(t_{i+1})-e^{\sigma A}X^n(t_{i+1})|, \ \ \ l_1, l_2\in [0,h],
\end{eqnarray*}
and  $f\in C^{1,2}_p(\hat{\Lambda}^{t})$, we have $\psi$ is continuous on $[0,h]$ and $\psi_{t^+}$ is continuous on $[0,h)$.  
Then, it follows  from~\cite[Corollary 1.2, Chapter 2]{pazy} that $\psi\in C^{1}([0,h);\mathbb{R})$. Therefore,
$$
f({X^n_{t_{i+1}}}_{-})-f({X^n_{t_{i}}})=\psi(h)-\psi(0)=\int^{h}_{0}\psi_{t^+}(l)dl=\int^{t_{i+1}}_{t_i}\partial_tf(X_{t_i,l}^n)dl, \ i\geq 0.
$$
The term $f({X^n_{t_{i}}})-f({X^n_{t_{i}-}})$ in (\ref{decom}) can be written as $\pi(X(t_{i+1}))-\pi(X(t_i))$, where
$\pi(x):=f({X^n_{t_{i}-}}+(x-X(t_i)){\mathbf{1}}_{\{t_i\}})$. Since $f\in C^{1,2}_p(\hat{\Lambda}^{t})$ and $A^*\partial_xf\in C_p^0(\hat{\Lambda}^{{t}},H)$, $\pi, \nabla_x\pi, \nabla_x^2\pi$ and $A^*\nabla_x\pi$
are continuous and grow in a polynomial way on $H$, and
$$\nabla_x\pi(x)=\partial_xf({X^n_{t_{i}-}}+(x-X(t_i)){\mathbf{1}}_{\{t_i\}}),\ \ \ \nabla_x^2\pi(x)=\partial_{xx}f({X^n_{t_{i}-}}+(x-X(t_i)){\mathbf{1}}_{\{t_i\}}).$$
Here  and in the rest of this paper,
$\nabla_x$ and $\nabla_{x}^2$ denote respectively   the   first- and second-order Fr\'echet  derivatives
with respect to $x\in H$.
Applying the It\^o formula (see \cite[Proposition 1.164]{fab1}) to $\pi$  and the continuous process  $(X(t_{i}+l))_{l\geq0}$, we have
\begin{eqnarray}\label{ito}
\qquad&\begin{aligned}
&\quad f({X^n_{t_{i}}})-f({X^n_{t_{i}-}})=\pi({X(t_{i+1})})-\pi(X(t_i))\\
&=\int^{t_{i+1}}_{t_i}\biggl[\langle A^*\partial_xf({X^n_{t_{i}}}_{-}+(X(\sigma)-X(t_i)){\mathbf{1}}_{\{t_{i}\}}),X(\sigma)\rangle_H\\
&\quad +
\langle \partial_xf({X^n_{t_{i}}}_{-}+(X(\sigma)-X(t_i)){\mathbf{1}}_{\{t_{i}\}}),\vartheta(\sigma)\rangle_H\\
&\quad
+\frac{1}{2}\mbox{Tr}[\partial_{xx}f({X^n_{t_{i}-}}+(X(\sigma)-X(t_i)){\mathbf{1}}_{\{t_i\}})(\varpi\varpi^*)(\sigma)]\biggr]d\sigma\\
&\quad
+ \int^{t_{i+1}}_{t_i}\langle \partial_xf({X^n_{t_{i}-}}+(X(\sigma)-X(t_i)){\mathbf{1}}_{\{t_i\}}), \varpi(\sigma)dW(\sigma)\rangle _H, \quad  i\geq 0.
\end{aligned}
\end{eqnarray}
Summing over $i\geq 0$ and denoting by $i(\sigma)$ the index such that $\sigma\in [t_{i(\sigma)},t_{i(\sigma)+1})$, we have
\begin{eqnarray}\label{app}
\qquad&\begin{aligned}
& f(X^n_s)-f(X_{t})=f(X^n_s)-f(X^n_{t-})\\
=&\int^{s}_{t}\partial_tf(X_{t_{i(\sigma)},\sigma}^n)d\sigma\\
&+\int^{s}_{t}[\langle A^*\partial_xf({X^n_{t_{i(\sigma)}}}_{-}+(X(\sigma)-X(t_{i(\sigma)})){\mathbf{1}}_{\{t_{i(\sigma)}\}}), X(\sigma)\rangle_H \\
&+\langle\partial_xf({X^n_{t_{i(\sigma)}}}_{-}+(X(\sigma)-X(t_{i(\sigma)})){\mathbf{1}}_{\{t_{i(\sigma)}\}}),\vartheta(\sigma)\rangle_H \\
&
+\frac{1}{2}\mbox{Tr}[\partial_{xx}f({X^n_{t_{i(\sigma)}-}}+(X(\sigma)-X(t_{i(\sigma)})){\mathbf{1}}_{\{t_{i(\sigma)}\}})(\varpi\varpi^*)(\sigma)]]d\sigma\\
&
+ \int^{s}_{t}\langle \partial_xf({X^n_{t_{i(\sigma)}-}}+(X(\sigma)-X(t_{i(\sigma)})){\mathbf{1}}_{\{t_{i(\sigma)}\}}),\varpi(\sigma)dW(\sigma)\rangle_H.
\end{aligned}
\end{eqnarray}
 Since, for all $\sigma\in [t_{i(\sigma)},t_{i(\sigma)+1})$,
$$
d_\infty(X_{t_{i(\sigma)},\sigma}^n,X_\sigma)= ||X_{t_{i(\sigma)}, \sigma}^n-X_\sigma||_0\leq ||X_{s}^n-X_s||_0, 
$$
\begin{eqnarray*}
	&&d_\infty(X_\sigma, {X^n_{t_{i(\sigma)}}}_{-}+(X(\sigma)-X(t_{i(\sigma)})){\mathbf{1}}_{\{t_{i(\sigma)}\}})\\
	&\leq& h+ ||X_{s}^n-X_s||_0+\sup_{l\in [0,\sigma-t_{i(\sigma)}]}|e^{lA}X(\sigma)-X(t_{i(\sigma)+1})|\\
	&\leq& h+ ||X_{s}^n-X_s||_0+\sup_{l_1,l_2\in [0,s], |l_1-l_2|\leq h}|X(l_1)-X(l_2)|\\
	&&  +\sup_{l\in [0,h]}|e^{lA}X(\sigma)-X(\sigma)|,
\end{eqnarray*}
and
$$\lim_{n\to \infty}||X^n_s-X_s||_0=0,$$
both facts  $f\in C^{1,2}_p(\hat{\Lambda}^{t})$ and $A^*\partial_xf\in C_p^0(\hat{\Lambda}^{{t}},H)$
yield the following   three limits
$$\lim_{n\to \infty}  f(X^n_s)=f(X_s),  \quad \lim_{n\to \infty} \partial_tf(X_{t_{i(\sigma)},\sigma}^n)=\partial_t f(X_\sigma),$$
and
\begin{eqnarray*}
	&&\lim_{n\to \infty} (A^*\partial_xf, \partial_xf, \partial_{xx}f)({X^n_{t_{i(\sigma)}}}_{-}+(X(\sigma)-X(t_{i(\sigma)})){\mathbf{1}}_{\{t_{i(\sigma)}\}}) \\
	&=&(A^*\partial_xf, \partial_xf, \partial_{xx}f)(X_\sigma).
\end{eqnarray*}
Since  $f\in C^{1,2}_p(\hat{\Lambda}^{t})$, $A^*\partial_xf\in C_p^0(\hat{\Lambda}^{{t}},H)$ and $X$ is bounded, the integrands of all the  integrals of equality~\eqref{app} are bounded.  The dominated convergence and Burkholder-Davis-Gundy inequalities  for  the stochastic integrals 
then ensure that the Lebesgue integrals converge almost surely,
and the stochastic integral in probability, to the terms appearing in (\ref{statesop0}) as $n\rightarrow\infty$.
\end{proof}

From  the last Lemma, we have the following important results.
\begin{lemma}\label{0815lemma}
	If $f\in C_p^{1,2}(\Lambda^t)$ is the restriction of  $\hat{f}\in C_p^{1,2}(\hat{\Lambda}^t)$, then the following derivatives
	$$
	\partial_tf:=\partial_t\hat{f}, \ \ \ \partial_xf:=\partial_x\hat{f}, \ \ \ \partial_{xx}f:=\partial_{xx}\hat{f} \ \ \mbox{on} \ \Lambda^t
	$$
	do not depend on the choice of $\hat{f}$. That is, if $f$ has another extension $\hat{f}'\in C_p^{1,2}(\hat{\Lambda}^t)$, both the derivatives of $\hat{f}'$ and $\hat{f}$ have the identical restriction to $\Lambda^t$.
\end{lemma}

\begin{proof}
By the definition of the horizontal derivative, it is clear that $\partial_t\hat{f}(\gamma_l)=\partial_t\hat{f}'(\gamma_l)$ for every $(l,\gamma_l)\in [t,T]\times\Lambda^t$.
Next,  setting $A=\mathbf{0}$,  $\varpi=\mathbf{0}$,  $\theta=l$, $\xi_{l}=\gamma_l\in \Lambda_l$
and $\vartheta(\cdot)\equiv h\in H$ in (\ref{formular1}),  in view of Lemma \ref{theoremito}, we have
$$
\int^{s}_{l}\langle \partial_x\hat{f}(X_\sigma), h\rangle_H d\sigma=\int^{s}_{l}\langle \partial_x\hat{f}'(X_\sigma), h\rangle_H d\sigma, \quad s\in [l,T],
$$
where $X_\sigma(r):=\gamma_l(r\wedge l)+(r-l)h{\mathbf{1}}_{(l,\sigma]}(r)$ for $r\in [0,\sigma]$.
Here and
in the sequel, for notational simplicity, we use $\mathbf{0}$ to denote elements, operators, or paths  which are
identically equal to zero.
By the continuity of $\partial_x\hat{f}$ and $\partial_x\hat{f}'$  and the arbitrariness of $h\in H$, we  have  $\partial_x\hat{f}(\gamma_l)=\partial_x\hat{f}'(\gamma_l)$ for every $(l,\gamma_l)\in [t,T]\times\Lambda^t$.
Finally,  setting $A=\mathbf{0}$,  $\vartheta=\mathbf{0}$,  $\theta=l$, $\xi_{l}=\gamma_l\in \Lambda_l$ and $\varpi(\cdot)\equiv a\in L_2(\Xi, H)$ in (\ref{formular1}),   in view of  Lemma \ref{theoremito}, we have
$$
\int^{s}_{l}\mbox{Tr}(\partial_{xx}\hat{f}(X_\sigma)aa^*)d\sigma =\int^{s}_{l}\mbox{Tr}(\partial_{xx}\hat{f}'(X_\sigma)aa^*)d\sigma, \ \ s\in [l,T],
$$
where $X_\sigma(r):=\gamma_l(r\wedge l)+a(W(r)-W(l)){\mathbf{1}}_{(l,\sigma]}(r)$ for $r\in [0,\sigma]$.
Since $\partial_{xx}\hat{f}\in {\mathcal{S}}(H)$ and $\partial_{xx}\hat{f}'\in {\mathcal{S}}(H)$ are continuous and $a\in L_2(\Xi, H)$ is arbitrary, we also have  $\partial_{xx}\hat{f}(\gamma_l)=\partial_{xx}\hat{f}'(\gamma_l)$ for every $(l,\gamma_l)\in [t,T]\times\Lambda^t$.
\end{proof}

\subsection{Borwein-Preiss variational principle}

In this section,  we introduce  {a modification of Borwein-Preiss variational principle (see Borwein \& Zhu  \cite[Theorem 2.5.2]{bor1}), which is} crucial to proving  the uniqueness and  stability  of viscosity solutions. We
firstly recall the definition of gauge-type function for compete  metric  space {$(E,d)$}.
\begin{definition}\label{gaupe}
	Let $(E,d)$ be a compete metric space.
	We say that a  continuous  functional $\rho:H\times H\rightarrow [0,\infty)$ is a {gauge-type function} on the compete metric space $(E,d)$
	provided that:
	\begin{description}
		\item[(i)] $\rho(x,x)=0$ for all $x\in E$,
		\item[(ii)] for any $\varepsilon>0$, there exists $\delta>0$ such that, for all $x,y\in E$, we have $\rho(x,y)\leq \delta$ implies that
		$d(x,y)<\varepsilon$.
	\end{description}
\end{definition}

\begin{lemma}\label{theoremleft}
	Let $t\in [0,T]$ be fixed and let $f:\Lambda^t\rightarrow \mathbb{R}$ be an upper semicontinuous functional  bounded from above. Suppose that $\rho$ is a gauge-type function
	and $\{\delta_i\}_{i\geq0}$ is a sequence of positive number, and suppose that $\varepsilon>0$ and $(t_0,\gamma^0_{t_0})\in [t,T]\times \Lambda^t$ satisfy
	$$
	f\left(\gamma^0_{t_0}\right)\geq \sup_{(s,\gamma_s)\in [t,T]\times \Lambda^t}f(\gamma_s)-\varepsilon.
	$$
	Then there exist $(\hat{t},\hat{\gamma}_{\hat{t}})\in [t,T]\times \Lambda^t$ and a sequence $\{(t_i,\gamma^i_{t_i})\}_{i\geq1}\subset [t,T]\times \Lambda^t$ such that
	\begin{description}
		\item[(i)] $\rho(\gamma^0_{t_0},\hat{\gamma}_{\hat{t}})\leq \frac{\varepsilon}{\delta_0}$,  $\rho(\gamma^i_{t_i},\hat{\gamma}_{\hat{t}})\leq \frac{\varepsilon}{2^i\delta_0}$ and $t_i\uparrow \hat{t}$ as $i\rightarrow\infty$,
		\item[(ii)]  $f(\hat{\gamma}_{\hat{t}})-\sum_{i=0}^{\infty}\delta_i\rho(\gamma^i_{t_i},\hat{\gamma}_{\hat{t}})\geq f(\gamma^0_{t_0})$, and
		\item[(iii)]  $f(\gamma_s)-\sum_{i=0}^{\infty}\delta_i\rho(\gamma^i_{t_i},\gamma_s)
		<f(\hat{\gamma}_{\hat{t}})-\sum_{i=0}^{\infty}\delta_i\rho(\gamma^i_{t_i},\hat{\gamma}_{\hat{t}})$ for all $(s,\gamma_s)\in [\hat{t},T]\times \Lambda^{\hat{t}}\setminus \{(\hat{t},\hat{\gamma}_{\hat{t}})\}$.
		
	\end{description}
\end{lemma}

\begin{proof} The proof is similar to that of Borwein \& Zhu  \cite[Theorem 2.5.2 ]{bor1}, and is given below for  convenience of the reader.

Define sequences $\{(t_i,\gamma^i_{t_i})\}_{i\geq1}$ and $\{B_i\}_{i\geq1}$ inductively as follows. Set
\begin{eqnarray}\label{left1}
B_0:=\{(s,\gamma_s)\in [t_0,T]\times \Lambda^{t_0}| \ f(\gamma_s)-\delta_0\rho(\gamma_s,\gamma^0_{t_0})\geq f(\gamma^0_{t_0})\}.
\end{eqnarray}
Since $(t_0,\gamma^0_{t_0})\in B_0$, $B_0$ is nonempty. Moreover it is closed because both $f$ and $-\rho(\cdot,\gamma^0_{t_0})$ are upper semicontinuous functions. We also have that, for
all $(s,\gamma_s)\in B_0$,
\begin{eqnarray}\label{left2}
\delta_0\rho(\gamma_s,\gamma^0_{t_0})\leq f(\gamma_s)-f(\gamma^0_{t_0})\leq \sup_{(l,\eta_l)\in [t,T]\times \Lambda^t}f(\eta_l)-f(\gamma^0_{t_0})\leq \varepsilon.
\end{eqnarray}
Take $(t_1,\gamma^1_{t_1})\in B_0$ such that
\begin{eqnarray}\label{left21111}
f(\gamma^1_{t_1})-\delta_0\rho(\gamma^1_{t_1},\gamma^0_{t_0})\geq \sup_{(s,\gamma_s)\in B_0}[f(\gamma_s)-\delta_0\rho(\gamma_s,\gamma^0_{t_0})]-\frac{\delta_1\varepsilon}{2\delta_0},
\end{eqnarray}
and define in a similar way
\begin{eqnarray}\label{left3}
B_1:=\left\{\begin{matrix} (s,\gamma_s)\\
\in B_0\cap [t_1,T]\times \Lambda^{t_1}\end{matrix} \,\, \bigg{|} \,\,  \begin{matrix} f(\gamma_s)-\sum_{k=0}^{1}\delta_k\rho(\gamma_s,\gamma^k_{t_k})\\[3mm]
\quad \geq f(\gamma^1_{t_1})-\delta_0\rho(\gamma^1_{t_1},\gamma^0_{t_0})\end{matrix}\right\}.
\end{eqnarray}
In general, suppose that we have defined $(t_j,\gamma^j_{t_j})$ and $B_j$ for $j=1,2,\ldots, i-1$ such that
\begin{eqnarray}\label{left4}
\begin{aligned}
&\quad f(\gamma^j_{t_j})-\sum_{k=0}^{j-1}\delta_k\rho(\gamma^j_{t_j},\gamma^k_{t_k})\\
&\geq \sup_{(s,\gamma_s)\in B_{j-1}}\bigg{[}f(\gamma_s)-\sum_{k=0}^{j-1}\delta_k\rho(\gamma_{s},\gamma^k_{t_k})\bigg{]}-\frac{\delta_j\varepsilon}{2^j\delta_0},
\end{aligned}
\end{eqnarray}
and
\begin{eqnarray}\label{left5}
\quad&B_j:=\left\{ \begin{matrix} (s,\gamma_s)\\
\in B_{j-1}\cap [t_j,T]\times \Lambda^{t_j}\end{matrix} \,\, \bigg{|}
\begin{matrix} f(\gamma_s)-\sum_{k=0}^{j}\delta_k\rho(\gamma_s,\gamma^k_{t_k})\\[3mm]
\quad \geq f(\gamma^j_{t_j})-\sum_{k=0}^{j-1}\delta_k\rho(\gamma^j_{t_j},\gamma^k_{t_k})\end{matrix}\right\}.
\end{eqnarray}
We choose $(t_i,\gamma^i_{t_i})\in B_{i-1}$ such that
\begin{eqnarray}\label{left6}
\begin{aligned}
&\quad f(\gamma^i_{t_i})-\sum_{k=0}^{i-1}\delta_k\rho(\gamma^i_{t_i},\gamma^k_{t_k})\\
&\geq \sup_{(s,\gamma_s)\in B_{i-1}}\bigg{[}f(\gamma_s)-\sum_{k=0}^{i-1}\delta_k\rho(\gamma_{s},\gamma^k_{t_k})\bigg{]}-\frac{\delta_i\varepsilon}{2^i\delta_0},
\end{aligned}
\end{eqnarray}
and we define
\begin{eqnarray}\label{left7}
\quad &B_i:=\left\{\begin{matrix} (s,\gamma_s)\\
\in B_{i-1}\cap [t_i,T]\times \Lambda^{t_i}\end{matrix}\,  \bigg{|}  \begin{matrix} f(\gamma_s)-\sum_{k=0}^{i}\delta_k\rho(\gamma_s,\gamma^k_{t_k})\\
\quad \geq f(\gamma^i_{t_i})-\sum_{k=0}^{i-1}\delta_k\rho(\gamma^i_{t_i},\gamma^k_{t_k})\end {matrix}\right\}.
\end{eqnarray}
We can see that  for every $i=1,2,\ldots$,  $B_i$ is a closed and nonempty set. It follows from (\ref{left6}) and (\ref{left7}) that, for all $(s,\gamma_s)\in B_i$,
\begin{eqnarray*}
	\delta_i\rho(\gamma_s,\gamma^i_{t_i})&\leq& \bigg{[}f(\gamma_s)-\sum_{k=0}^{i-1}\delta_k\rho(\gamma_s,\gamma^k_{t_k})\bigg{]}-\bigg{[}f(\gamma^i_{t_i})-\sum_{k=0}^{i-1}\delta_k\rho(\gamma^i_{t_i},\gamma^k_{t_k})\bigg{]}\\
	&\leq&\sup_{(s,\gamma_s)\in B_{i-1}}\bigg{[}f(\gamma_s)-\sum_{k=0}^{i-1}\delta_k\rho(\gamma_s,\gamma^k_{t_k})\bigg{]}-\bigg{[}f(\gamma^i_{t_i})-\sum_{k=0}^{i-1}\delta_k\rho(\gamma^i_{t_i},\gamma^k_{t_k})\bigg{]}\\
	&\leq& \frac{\delta_i\varepsilon}{2^i\delta_0},
\end{eqnarray*}
which implies that
\begin{eqnarray}\label{left8}
\rho(\gamma_s,\gamma^i_{t_i})\leq \frac{\varepsilon}{2^i\delta_0},\ \ \mbox{for all}\  (s,\gamma_s)\in B_i.
\end{eqnarray}
Since $\rho$ is a gauge-type function, inequality (\ref{left8}) implies that $\sup_{(s,\gamma_s)\in B_i}d_\infty(\gamma_s, \gamma^i_{t_i})$ $\rightarrow0$ as $i\rightarrow \infty$, and therefore,
$\sup_{(s,\gamma_s),(l,\eta_l)\in B_i}d_\infty(\gamma_s, \eta_l)\rightarrow0$ as $i\rightarrow \infty$. Since $\Lambda^t$ is complete, by Cantor's intersection theorem there exists a unique
$(\hat{t},\hat{\gamma}_{\hat{t}})\in \bigcap_{i=0}^{\infty}B_i$. Obviously, we have $d_\infty(\gamma^i_{t_i}, \hat{\gamma}_{\hat{t}})\rightarrow0$ and $t_i\uparrow \hat{t}$  as $i\rightarrow \infty$. Then $(\hat{t},\hat{\gamma}_{\hat{t}})$
satisfies (i) by (\ref{left2}) and  (\ref{left8}). For any $(s,\gamma_s)\in [\hat{t},T]\times \Lambda^{\hat{t}}$ and $(s,\gamma_s)\neq (\hat{t},\hat{\gamma}_{\hat{t}})$, we have $(s,\gamma_s)\notin\bigcap_{i=0}^{\infty}B_i$, and therefore, for some $j$,
\begin{eqnarray}\label{left9}
\begin{aligned}
f(\gamma_s)-\sum_{k=0}^{\infty}\delta_k\rho(\gamma_s,\gamma^k_{t_k})&\leq f(\gamma_s)-\sum_{k=0}^{j}\delta_k\rho(\gamma_s,\gamma^k_{t_k})\\
&<
f(\gamma^j_{t_j})-\sum_{k=0}^{j-1}\delta_k\rho(\gamma^j_{t_j},\gamma^k_{t_k}).
\end{aligned}
\end{eqnarray}
On the other hand, it follows from (\ref{left1}), (\ref{left7}) and $(\hat{t},\hat{\gamma}_{\hat{t}})\in \bigcap_{i=0}^{\infty}B_i$ that, for any $q\geq j$,
\begin{eqnarray}\label{left10}
\begin{aligned}
f(\gamma^0_{t_0})&\leq f(\gamma^j_{t_j})-\sum_{k=0}^{j-1}\delta_k\rho(\gamma^j_{t_j},\gamma^k_{t_k})\leq f(\gamma^q_{t_q})-\sum_{k=0}^{q-1}\delta_k\rho(\gamma^q_{t_q},\gamma^k_{t_k})\\
&\leq f(\hat{\gamma}_{\hat{t}})-\sum_{k=0}^{q}\delta_k\rho(\hat{\gamma}_{\hat{t}},\gamma^k_{t_k}).
\end{aligned}
\end{eqnarray}
Letting $q\rightarrow\infty$ in (\ref{left10}), we obtain
\begin{eqnarray}\label{left11}
f(\gamma^0_{t_0})\leq  f(\gamma^j_{t_j})-\sum_{k=0}^{j-1}\delta_k\rho(\gamma^j_{t_j},\gamma^k_{t_k})\leq f(\hat{\gamma}_{\hat{t}})-\sum_{k=0}^{\infty}\delta_k\rho(\hat{\gamma}_{\hat{t}},\gamma^k_{t_k}),
\end{eqnarray}
which verifies (ii). Combining  (\ref{left9}) and (\ref{left11}) yields (iii).
\end{proof}

For every $t\in [0,T]$, define $\Lambda^t\otimes\Lambda^t:=\{(\gamma_s,\eta_s)|\gamma_s,\eta_s\in \Lambda^t\}$.
It is clear that $(\Lambda^t\otimes\Lambda^t,d_{1,\infty})$ is a compete metric space,
where $d_{1,\infty}((\gamma_t^1,\gamma_t^2),(\eta_s^1,\eta_s^2))=d_\infty(\gamma_t^1,\eta_s^1)+d_\infty(\gamma_t^2,\eta_s^2)$
for all $(t,(\gamma_t^1,\gamma_t^2))$, $(s,(\eta_s^1,\eta_s^2))\in [0,T]\times(\Lambda\otimes\Lambda)$. Similar to Lemma \ref{theoremleft}, we have the following.

\begin{lemma}\label{theoremleft1} 
	Let $t\in [0,T]$ be fixed and let $f:\Lambda^t\otimes\Lambda^t\rightarrow \mathbb{R}$ be an upper semicontinuous functional  bounded from above. Suppose that $\rho$ is a gauge-type function on $(\Lambda^t\otimes\Lambda^t, d_{1,\infty})$
	and $\{\delta_i\}_{i\geq0}$ is a sequence of positive number, and suppose that $\varepsilon>0$ and $(t_0,(\gamma^0_{t_0},\eta^0_{t_0}))\in [t,T]\times (\Lambda^t\otimes\Lambda^t)$ satisfy
	$$
	f(\gamma^0_{t_0},\eta^0_{t_0})\geq \sup_{\scriptsize\begin{matrix}(s,(\gamma_s,\eta_s))\\
		\in [t,T]\times (\Lambda^t\otimes\Lambda^t)\end{matrix}}\!\!\!\!\! f(\gamma_s,\eta_s)-\varepsilon.
	$$
	Then there exist $(\hat{t},(\hat{\gamma}_{\hat{t}},\hat{\eta}_{\hat{t}}))\in [t,T]\times (\Lambda^t\otimes\Lambda^t)$ and a sequence $\{(t_i,(\gamma^i_{t_i},\eta^i_{t_i}))\}_{i\geq1}\subset [t_0,T]\times (\Lambda^t\otimes\Lambda^t)$ such that
	\begin{description}
		\item[(i)] $\rho((\gamma^0_{t_0},\eta^0_{t_0}),(\hat{\gamma}_{\hat{t}},\hat{\eta}_{\hat{t}}))\leq \frac{\varepsilon}{\delta_0}$,
		$\rho((\gamma^i_{t_i},\eta^i_{t_i}),(\hat{\gamma}_{\hat{t}},\hat{\eta}_{\hat{t}}))\leq \frac{\varepsilon}{2^i\delta_0}$ and $t_i\uparrow \hat{t}$ as $i\rightarrow\infty$,
		\item[(ii)]  $\displaystyle f(\hat{\gamma}_{\hat{t}},\hat{\eta}_{\hat{t}})-\sum_{i=0}^{\infty}\delta_i\rho((\gamma^i_{t_i},\eta^i_{t_i}),(\hat{\gamma}_{\hat{t}},\hat{\eta}_{\hat{t}}))\geq f(\gamma^0_{t_0},\eta^0_{t_0})$, and
		\item[(iii)]  $\displaystyle f(\gamma_s,\eta_s)-\sum_{i=0}^{\infty}\delta_i\rho((\gamma^i_{t_i},\eta^i_{t_i}),(\gamma_s,\eta_s))
		<f(\hat{\gamma}_{\hat{t}},\hat{\eta}_{\hat{t}})-\sum_{i=0}^{\infty}\delta_i\rho((\gamma^i_{t_i},\eta^i_{t_i}),$ $(\hat{\gamma}_{\hat{t}},\hat{\eta}_{\hat{t}}))$ for all $(s,(\gamma_s,\eta_s)) $ $\in [\hat{t},T]\times (\Lambda^{\hat{t}}\otimes\Lambda^{\hat{t}})\setminus \{(\hat{t},(\hat{\gamma}_{\hat{t}},\hat{\eta}_{\hat{t}}))\}$.
		\end{description}
\end{lemma}

\subsection{Backward stochastic differential equations}

We consider the backward stochastic differential
equations (BSDEs) in a Hilbert space $K$:
\begin{eqnarray}\label{bsde}
\qquad Y(t)=\eta+\int^{T}_{t}f(\sigma,Y(\sigma),Z(\sigma))d\sigma-\int^{T}_{t}Z(\sigma) dW_\sigma, \quad  0\leq t\leq T.
\end{eqnarray}
Assume that the mapping $f: \Omega \times [0,T]\times
K\times L_2(\Xi,K)\rightarrow K$ is  measurable with
respect to ${\mathcal {P}}\times {\mathcal {B}} (K\times L_2(\Xi,K))$, and
$\eta: \Omega \rightarrow K$ is ${\mathcal {F}}_T$-measurable.
The pair $(f,\eta)$ is called the parameters of BSDE~\eqref{bsde}.
\par
We have  the existence and uniqueness result on  BSDEs.

\begin{lemma}\label{lemma2.5} (\cite[Proposition 4.3]{fuh0}) Assume that (i) there
	exists $L>0$ such that
	\begin{eqnarray*}
		|f(\sigma,y_1,z_1)-f(\sigma,y_2,z_2)|\leq
		L(|y_1-y_2|+|z_1-z_2|),
	\end{eqnarray*}
	$P$-a.s. for every $\sigma \in [0,T], y_1,y_2\in K, z_1,z_2\in
	L_2(\Xi,K)$;
	\par
	(ii) there exists $p\in[2,\infty)$ such that
	\begin{eqnarray*}
		\mathbb{E}\left[\left(\int^{T}_{0}|f(\sigma,0,0)|^2d\sigma\right)^{\frac{p}{2}}\right]<\infty,\quad \mathbb{E}|\eta|^p<\infty.
	\end{eqnarray*}
	Then there exists a unique pair of processes
	$$(Y,Z)\in  L^p_{\mathcal {P}}(\Omega, \Lambda_T(K))\times
	L^p_{\mathcal {P}} (\Omega;L^2([0,T];L_2(\Xi,K)))$$
	such that  (\ref{bsde})
	holds for $t\in[0,T]$ and
	\begin{eqnarray}\label{lemma2.510}
	&& \mathbb{E}\left[\|Y_T\|_{0}^p\right]+\mathbb{E}\left[\left(\int^{T}_{0}|Z(\sigma)|^2d\sigma\right)^{\frac{p}{2}}\right]
	\nonumber \\
	&\leq& C_p\, \mathbb{E}\left[|\eta|^p\right]+
	C_p\, \mathbb{E}\left[\left(\int^{T}_{0}|f(\sigma,0,0)|^2d\sigma\right)^{\frac{p}{2}}\right],
	\end{eqnarray}
	for some constant $C_p>0$ depending only on $p$, $L$, $T$.
	Moreover, let two BSDEs of  parameters    $(\eta^1,f^1)$ and $(\eta^2,f^2)$ satisfy Assumptions (i) and (ii).
	Then the difference of both solutions
	$(Y^1,Z^1)$ and $(Y^2,Z^2)$ of BSDE (\ref{bsde}) with
	the data $(\eta^1,f^1)$ and
	$(\eta^2,f^2)$, respectively, has the
	following  estimate
	\begin{eqnarray}\label{lemma2.51}
	\qquad&\begin{aligned} &\mathbb{E}\left[\|Y^1_T-Y^2_T\|_{0}^p\right]+\mathbb{E}\left[\bigg{(}\int^{T}_{0}|Z^1(\sigma)-Z^2(\sigma)|^2d\sigma\bigg{)}^{\frac{p}{2}}\right]\\
	\leq&\,  C_p\, \mathbb{E}\left[\left|\eta^1-\eta^2\right|^p\right]+
	C_p\, \mathbb{E}\left[\bigg{(}\int^{T}_{0}\left|(f^1-f^2)(\sigma,Y^2(\sigma),Z^2(\sigma))\right|^2d\sigma\bigg{)}^{\frac{p}{2}}\right].
	\end{aligned}
	\end{eqnarray}
\end{lemma}
We also have  the following comparison theorem on BSDEs in infinite dimensional spaces.
\begin{lemma}\label{lemma2.70904}
	(\cite[Theorem 2.7]{zhou1})   Let two BSDEs of  parameters    $(\eta^1,f^1)$ and $(\eta^2,f^2)$ satisfy all the assumptions of Lemma \ref{lemma2.5} with $K=\mathbb{R}$.
	Denote by $(Y^1,Z^1)$ and  $(Y^2,Z^2)$  their respective adapted solutions.
	%
	%
	If
	$$
	\eta^1\geq \eta^2, P\mbox{-a.s. and}
	\ (\ f^1- f^2)(t,Y^2(t),Z^2(t))\geq 0,\ \
	dP\otimes dt\mbox{-a.s.}
	$$
	Then we have that $Y^1(t)\geq Y^2(t)$, a.s., for
	all $t\in[0,T]$.
	\par
	Moreover, the comparison is strict: that is
	$$
	Y^1(0)=Y^2(0)\quad \Leftrightarrow \quad \eta^1=\eta^2,\quad  (f^1- f^2)(t,Y^2(t),Z^2(t))=0,\ \
	dP\otimes dt \mbox{-a.s.}
	$$
\end{lemma}
\newpage
\section{Construction of Smooth  gauge-type functionals in path spaces}
In this chapter, we construct smooth  functionals $\Upsilon^{m,M}$, which are crucial  to proving  the uniqueness and  stability  of viscosity solutions.
\par
For every $m\in \mathbb{N}^+$, define  the path functional $S_m:\hat{\Lambda}\rightarrow \mathbb{R}$ by
\begin{eqnarray*}
	S_m(\gamma_t)=\begin{cases}
		\frac{(||\gamma_{t}||_{0}^{2m}-|\gamma_{t}(t)|^{2m})^3}{||\gamma_{t}||^{4m}_{0}}, \ \
		~~~~ ||\gamma_{t}||_{0}\neq0; \\
		0, \qquad \qquad \qquad \qquad  ||\gamma_{t}||_{0}=0;
	\end{cases}\ (t,\gamma_t)\in [0,T]\times{\hat{\Lambda}}.
\end{eqnarray*}
\par
Define, for every $m\in \mathbb{N}^+$ and $M\in \mathbb{R}$,
\begin{eqnarray*}
	\begin{aligned}
		\Upsilon^{m,M}(\gamma_t,\eta_s):= 
		S_m(\eta_{s,s\vee t}^A-\gamma_{t,t\vee s}^A)&+M|e^{{(t-s)^+A}}\eta_s(s)-e^{(s-t)^+A}\gamma_t(t)|^{2m}, \\
		& (t,\gamma_t; s,\eta_s)\in ([0,T]\times\hat{\Lambda})^2
	\end{aligned}
\end{eqnarray*}
and
$$
\overline{\Upsilon}^{m,M}(\gamma_t,\eta_s):=
\Upsilon^{m,M}(\gamma_t,\eta_s)+|s-t|^2, \quad (t,\gamma_t; s,\eta_s)\in ([0,T]\times\hat{\Lambda})^2.
$$
For simplicity of notation,  we write  $\Upsilon^{m,M}(\gamma_t):=\Upsilon^{m,M}(\gamma_t,\eta_t)$ when $\eta_t(r)={\mathbf{0}}$ for all $r\in [0,t]$. It is clear that $\Upsilon^{m,M}(\gamma_t,\eta_t)=\Upsilon^{m,M}(\gamma_t-\eta_t)$ for all $\gamma_t,\eta_t\in \hat{\Lambda}$. We write $S:=S_3, \Upsilon:=\Upsilon^{3,3}$, and $\overline{\Upsilon}:=\overline{\Upsilon}^{3,3}$.

\par
Now we study the   regularity  of $\Upsilon^{m,M}$ in the sense of  horizontal/vertical derivatives.
\begin{lemma}\label{theoremS}
	For  $m\in \mathbb{N}^+$ and $M\in \mathbb{R}$,  the derivatives $\partial_{t}\Upsilon^{m,M}(\cdot),\,  \partial_{x}\Upsilon^{m,M}(\cdot),$ and $\partial_{xx}\Upsilon^{m,M}(\cdot)$ exist,  and for all $(t,\gamma_t)\in [0, T]\times{\hat{\Lambda}}$,
	\begin{eqnarray}\label{220817a0}
	\partial_{t}\Upsilon^{m,M}(\gamma_t)=0,
	\end{eqnarray}
	\begin{eqnarray}\label{220817a}
	|\partial_{x}\Upsilon^{m,M}(\gamma_t)|\leq 2m(3+|M-3|){|\gamma_t(t)|^{2m-1}},
	\end{eqnarray}
	and
	\begin{eqnarray}\label{220817a1}
	\qquad\quad  |\partial_{xx}\Upsilon^{m,M}(\gamma_t)|\leq 2m[3(6m-1)+(2m-1)|M-3|]{|\gamma_t(t)|^{2m-2}}.
	\end{eqnarray}
	For $m\geq2$, we have $\Upsilon^{m,M}(\cdot)\in C^{1,2}_p(\hat{\Lambda})$.
\end{lemma}

\begin{proof}
First, we prove $\Upsilon^{m,M}(\cdot)\in C^0(\hat{\Lambda})$. 
For any $(t,\gamma_t; s, \eta_s)\in ([0,T]\times \hat{\Lambda})^2$ such that $s\geq t$, we have
\begin{eqnarray*}
	|\gamma_t(t)-\eta_s(s)|&\leq& \left|\gamma_t(t)-e^{(s-t)A}\gamma_t(t)\right|+\left|e^{(s-t)A}\gamma_t(t)-\eta_s(s)\right|\\
	&\leq& \left|\gamma_t(t)-e^{(s-t)A}\gamma_t(t)\right|+d_\infty(\gamma_t,\eta_s)
\end{eqnarray*}
and
\begin{eqnarray*}
	\left|\, ||\gamma_t||_0-||\eta_s||_0\,\right|&\leq& ||\gamma_{t,s}^A||_0-||\gamma_t||_0+||\gamma_{t,s}^A-\eta_s||_0\\
	&\leq& \sup_{t\leq l\leq s}\left|(e^{(l-t)A}-I)\gamma_t(t)\right|+d_\infty(\gamma_t,\eta_s).
\end{eqnarray*}
Thus $\Upsilon^{m,M}(\eta_s)\rightarrow \Upsilon^{m,M}(\gamma_t)$ as $\eta_s\rightarrow \gamma_t$ in $(\hat{\Lambda},d_\infty)$ and $s\downarrow t$ (equivalently, as $\eta_s\rightarrow \gamma_t$ in $(\hat{\Lambda}^t,d_\infty)$).
For any $(t,\gamma_t), (s, \eta_s)\in [0,T]\times \hat{\Lambda}$ and $s<t$, with
$$\gamma_t(t-):=\lim_{l\uparrow t}\gamma_t(l) \quad \text{and}\quad M_1=:\sup_{s\in [0,T]}|e^{sA}|,$$
 we have
\begin{eqnarray*}
	&&|\gamma_t(t)-\eta_s(s)|\\&
	\leq& \left|\gamma_t(t)-e^{(t-s)A}\eta_s(s)\right|+\left|e^{(t-s)A}(\eta_s(s)-\gamma_t(s))\right|+\left|e^{(t-s)A}(\gamma_t(s)-\gamma_t(t-))\right|\\
	&&+\left|(e^{(t-s)A}-I)\gamma_t(t-)\right|+|\gamma_t(t-)-\gamma_t(s)|+|\gamma_t(s)-\eta_s(s)|\\
	&\leq& (M_1+2)d_\infty(\gamma_t,\eta_s)+(M_1+1)|\gamma_t(t-)-\gamma_t(s)|+|(e^{(t-s)A}-I)\gamma_t(t-)|
\end{eqnarray*}
and
\begin{eqnarray*}
	&&\left|\, ||\gamma_t||_0-||\eta_s||_0\, \right|\leq ||\eta_{s,t}^A||_0-||\eta_s||_0+||\eta_{s,t}^A-\gamma_t||_0\\
	&\leq& \sup_{s\leq l\leq t}\left|(e^{(l-s)A}-I)\gamma_t(t)\right|+\sup_{s\leq l\leq t}\left|(e^{(l-s)A}-I)\right||\gamma_t(t)-\eta_s(s)|+d_\infty(\eta_s,\gamma_t).
\end{eqnarray*}
Thus $\Upsilon^{m,M}(\eta_s)\rightarrow \Upsilon^{m,M}(\gamma_t)$ as $\eta_s\rightarrow \gamma_t$ in $(\hat{\Lambda},d_\infty)$  and $s\uparrow t$. Therefore,   $\Upsilon^{m,M}(\cdot)\in C^0(\hat{\Lambda})$.

Second, by the definition of $\Upsilon^{m,M}(\cdot)$, we have
$$\partial_t\Upsilon^{m,M}(\gamma_{t})=0, \quad \forall (t,\gamma_t)\in [0,T]\times \hat{\Lambda}.$$

Third, we calculate $ \partial_{x}\Upsilon^{m,M}(\cdot)$. Define $g:{\hat{\Lambda}}\rightarrow \mathbb{R}$ by
$$g(\gamma_t):=|\gamma_t(t)|^{2m},\quad (t,\gamma_t)\in [0, T]\times{\hat{\Lambda}}.$$
We have
$$
\Upsilon^{m,M}(\gamma_t)=S_m(\gamma_t)+Mg(\gamma_t),\quad (t,\gamma_t)\in [0, T]\times{\hat{\Lambda}}.
$$
Clearly,
\begin{eqnarray}\label{s10902jia}
\partial_xS_m(\gamma_{0})=0,\quad \gamma_0\in \hat{\Lambda}_0.
\end{eqnarray}
For every $(t,\gamma_t)\in (0,T]\times \hat{\Lambda}$, define
$$||\gamma_t||_{0^-}:=\sup_{0\leq s<t}|\gamma_t(s)|.$$
Since
\begin{eqnarray}\label{04050}
\qquad\quad  |\gamma_t(t)+h|^{2m}=|\gamma_t(t)|^{2m}+2m|\gamma_t(t)|^{2m-2}\langle\gamma_t(t), h\rangle_H+o(|h|),  \quad h\in H,
\end{eqnarray}
we have
\begin{eqnarray*}
	&&\left(||\gamma_t||^{2m}_{0}-|\gamma_t(t)+{h}|^{2m}\right)^3
	-\left(||\gamma_t||^{2m}_{0}-|\gamma_t(t)|^{2m}\right)^3\\
	&=&-6m(||\gamma_t||^{2m}_{0}-|\gamma_t(t)|^{2m})^2|\gamma_t(t)|^{2m-2}\langle\gamma_t(t), h\rangle_H+o(|h|), \ \ h\in H.
\end{eqnarray*}
Then,
if $|\gamma_t(t)|<||\gamma_t||_{0^-}$,
\begin{eqnarray*}
	&&\lim_{|h|\rightarrow0}\frac{1}{|h|}\bigg{|}S_m(\gamma_t^{h})-S_m(\gamma_{t})\\
	&&  \ \ \ \ \ \ \ \ \ \ \ \ \ \ +\frac{6m} {||\gamma_t||^{4m}_{0}}(||\gamma_t||^{2m}_{0}-|\gamma_t(t)|^{2m})^2|\gamma_t(t)|^{2m-2}\langle\gamma_t(t), h\rangle_H
	\bigg{|}\\
	&=&\lim_{|h|\rightarrow0}\frac{1}{|h|\times||\gamma_t||^{4m}_{0}}\bigg{|}{\left(||\gamma_t||^{2m}_{0}-|\gamma_t(t)+{h}|^{2m}\right)^3
		-{\left(||\gamma_t||^{2m}_{0}-|\gamma_t(t)|^{2m}\right)^3}}\\
	&&\quad\quad\quad\quad\quad\quad+{6m\left(||\gamma_t||^{2m}_{0}-|\gamma_t(t)|^{2m}\right)^2|\gamma_t(t)|^{2m-2}\langle \gamma_t(t), h\rangle_H }\bigg{|} =0.
\end{eqnarray*}
Thus,
\begin{eqnarray}\label{s1}
\partial_xS_m(\gamma_{t})=-\frac{6m}{||\gamma_t||^{4m}_{0}}{\left(||\gamma_t||^{2m}_{0}-|\gamma_t(t)|^{2m}\right)^2
	|\gamma_t(t)|^{2m-2}\gamma_t(t)}.
\end{eqnarray}
If  $|\gamma_t(t)|>||\gamma_t||_{0^-}$,
\begin{eqnarray}\label{s2}
\partial_{x}S_m(\gamma_{t})=0;
\end{eqnarray}
if  $|\gamma_t(t)|=||\gamma_t||_{0^-}\neq0$,
since
\begin{eqnarray}\label{jiaxis}
\begin{aligned}
&\quad ||\gamma^{h}_t||_{0}^{2m}-|\gamma_t(t)+{h}|^{2m}\\
&=
\begin{cases}
0,\ \ \ \ \ \ \ \ \ \  \ \ \ \ \ \ \ \ \ \ \ \ \ \ \  \ \ \ \ \ \ \
|\gamma_t(t)+h|\geq |\gamma_t(t)|;\\
|\gamma_t(t)|^{2m}-|\gamma_t(t)+{h}|^{2m}, \ \  |\gamma_t(t)+h|<|\gamma_t(t)|,\end{cases}
\end{aligned}
\end{eqnarray}
we have from~(\ref{04050})
\begin{eqnarray}\label{s3}
\begin{aligned}
0&\leq\lim_{|h|\rightarrow0}\frac{1}{|h|}{\left|S_m(\gamma_t^{h})-S_m(\gamma_t)\right|}\\
&\leq\lim_{|h|\rightarrow0}\frac{1}{|h|}||\gamma^{h}_t||_{0}^{-4m}{
	\left| |\gamma_t(t)|^{2m}-|\gamma_t(t)+{h}|^{2m}\right|^3}=0;
\end{aligned}
\end{eqnarray}
if $|\gamma_t(t)|=||\gamma_t||_{0^-}=0$,
\begin{eqnarray}\label{ss4}
\partial_{x}S_m(\gamma_{t})=0.
\end{eqnarray}
Notice that
\begin{eqnarray}\label{2208171}
\partial_xg(\gamma_t)=2m|\gamma_t(t)|^{2m-2}\gamma_t(t),\ \ (t,\gamma_t)\in [0, T]\times{\hat{\Lambda}}.
\end{eqnarray}
In view of~(\ref{s10902jia}), (\ref{s1}), (\ref{s2}), (\ref{s3}), (\ref{ss4}) and (\ref{2208171}),  we obtain that, for all $(t,\gamma_t)\in [0, T]\times{\hat{\Lambda}}$,
\begin{eqnarray}\label{03108}
\partial_{x}\Upsilon^{m,M}(\gamma_t)=\begin{cases}6m\frac{||\gamma_t||^{4m}_{0}-\left(||\gamma_t||^{2m}_{0}-|\gamma_t(t)|^{2m}\right)^2}{||\gamma_t||^{4m}_{0}}{|\gamma_t(t)|^{2m-2}\gamma_t(t)}\\
+2m(M-3)|\gamma_t(t)|^{2m-2}\gamma_t(t),  \ \ \ \ \ \ ||\gamma_t||_{0}\neq0;\\
0, \ \ \ \ \ \ \ \ \ \ \ \ \ \ \ \ \ \ \ \ \ \ \ \ \ \ \ \ \ \ \ \ \ \ \ \ \ \ \ \ \ \ ||\gamma_t||_{0}=0.
\end{cases}
\end{eqnarray}
From (\ref{03108}) it follows  that
\begin{eqnarray*}
	|\partial_{x}\Upsilon^{m,M}(\gamma_t)|&\leq& 6m|\gamma_t(t)|^{2m-1}+2m|M-3||\gamma_t(t)|^{2m-1}\\
	&=&2m(3+|M-3|)|\gamma_t(t)|^{2m-1}.
\end{eqnarray*}
That is (\ref{220817a}).

We now consider $\partial_{xx}\Upsilon^{m,M}(\cdot)$.   Clearly,
\begin{eqnarray}\label{s10902jia1}
\partial_{xx}S_m(\gamma_{0})=0,\  \ \gamma_0\in \Lambda_0.
\end{eqnarray}
For every $(t,\gamma_t)\in (0,T]\times \hat{\Lambda}$, since
\begin{eqnarray*}
	&&\left(||\gamma_t||^{2m}_{0}-|\gamma_t(t)+h|^{2m}\right)^2|\gamma_t(t)\\
	&&+h|^{2m-2}(\gamma_t(t)+h)-\left(||\gamma_t||^{2m}_{0}-|\gamma_t(t)|^{2m}\right)^2|\gamma_t(t)|^{2m-2}\gamma_t(t)\\
	&=&\left(||\gamma_t||^{2m}_{0}-|\gamma_t(t)+h|^{2m}\right)^2|\gamma_t(t)+h|^{2m-2}h\\
	&&+[2(m-1)\left(||\gamma_t||^{2m}_{0}-|\gamma_t(t)+h|^{2m}\right)^2|\gamma_t(t)|^{2m-4}\langle\gamma_t(t), h\rangle_H{\mathbf{1}_{\{m>1\}}}\\
	&&~~~-4m\left(||\gamma_t||^{2m}_{0}-|\gamma_t(t)|^{2m}\right)|\gamma_t(t)|^{4m-4}\langle\gamma_t(t), h\rangle_H+o(h)]\gamma_t(t),
\end{eqnarray*}
we have
if $|\gamma_t(t)|<||\gamma_t||_{0^-}$,
\begin{eqnarray}\label{s5}
\begin{aligned}
&\quad\partial_{xx}S_m(\gamma_t)\\
&=\frac{24m^2}{{||\gamma_t||_{0}^{4m}}}\left(||\gamma_t||^{2m}_{0}-|\gamma_t(t)|^{2m}\right)|\gamma_t(t)|^{4m-4}\langle\gamma_t(t), \cdot\rangle_H\gamma_t(t)\\
&\quad-\frac{12m(m-1)}{{||\gamma_t||_{0}^{4m}}}\left(||\gamma_t||^{2m}_{0}-|\gamma_t(t)|^{2m}\right)^2|\gamma_t(t)|^{2m-4}\langle\gamma_t(t), \cdot\rangle_H\gamma_t(t){\mathbf{1}_{\{m>1\}}}\\
&\quad-\frac{6m}{{||\gamma_t||_{0}^{4m}}}\left(||\gamma_t||^{2m}_{0}-|\gamma_t(t)|^{2m}\right)^2|\gamma_t(t)|^{2m-2}I;
\end{aligned}
\end{eqnarray}
if $|\gamma_t(t)|>||\gamma_t||_{0^-}$,
\begin{eqnarray}\label{s4}
\partial_{xx}S_m(\gamma_t)=0;
\end{eqnarray}
if $|\gamma_t(t)|=||\gamma_t||_{0^-}\neq0$, by (\ref{04050}) and (\ref{jiaxis}),
we have
\begin{eqnarray}\label{s6666}
\begin{aligned}
0&\leq\lim_{|h|\rightarrow0}\frac{1}{|h|}\left|\partial_{x}S_m(\gamma_t^{h})-\partial_{x}S_m(\gamma_t)\right|\\
&\leq\lim_{|h|\rightarrow0}\frac{6m} {|h|}||\gamma^{h}_t||_{0}^{-4m} \bigg{|}\left(|\gamma_t(t)|^{2m}
-|\gamma_t(t)+h|^{2m}\right)^2\\
&\quad \ \ \ \ \ \ \ \ \times|\gamma_t(t)+h|^{2m-2}(\gamma_t(t)+h)\bigg{|} =0;
\end{aligned}
\end{eqnarray}
if $|\gamma_t(t)|=||\gamma_t||_{0^-}=0$,
\begin{eqnarray}\label{ss42}
\partial_{xx}S_m(\gamma_t)=0.
\end{eqnarray}
Notice that, for $(t,\gamma_t)\in [0, T]\times{\hat{\Lambda}}$,
\begin{eqnarray}\label{2208172}
\begin{aligned}
&\quad\partial_{xx}g(\gamma_t)\\
&=4m(m-1)|\gamma_t(t)|^{2m-4}\langle\gamma_t(t), \cdot\rangle_H\gamma_t(t)\mathbf{1}_{\{m>1\}}+2m|\gamma_t(t)|^{2m-2}I.
\end{aligned}
\end{eqnarray}
From  (\ref{s10902jia1})-(\ref{ss42}) and (\ref{2208172}),  we have, for all $(t,\gamma_t)\in [0, T]\times{\hat{\Lambda}}$,
\begin{eqnarray}\label{03109}
\begin{aligned}
\ \ \ \ \ &\quad
\partial_{xx}\Upsilon^{m,M}(\gamma_t)\\
&=\begin{cases}{24m^2}{{||\gamma_t||_{0}^{-4m}}}\left(||\gamma_t||^{2m}_{0}-|\gamma_t(t)|^{2m}\right)|\gamma_t(t)|^{4m-4}\langle\gamma_t(t), \cdot\rangle_H\gamma_t(t)\\
+{12m(m-1)}{{||\gamma_t||_{0}^{-4m}}}\left(||\gamma_t||_{0}^{4m}-\left(||\gamma_t||^{2m}_{0}-|\gamma_t(t)|^{2m}\right)^2\right)\\
\times|\gamma_t(t)|^{2m-4}\langle\gamma_t(t), \cdot\rangle_H\gamma_t(t)\mathbf{1}_{\{m>1\}}\\
+{6m}{{||\gamma_t||_{0}^{-4m}}}\left(||\gamma_t||_{0}^{4m}-\left(||\gamma_t||^{2m}_{0}-|\gamma_t(t)|^{2m}\right)^2\right)|\gamma_t(t)|^{2m-2}I\\
+4m(m-1)(M-3)|\gamma_t(t)|^{2m-4}\langle\gamma_t(t), \cdot\rangle_H\gamma_t(t)\mathbf{1}_{\{m>1\}}\\
+2m(M-3)|\gamma_t(t)|^{2m-2}I,  \ \ \ \  \ \ \ \ \ \ \ \ \ \ \ \ \ \ \ \ \ \ \ \  \ \ \  \ \ \ \ \ ||\gamma_t||_{0}\neq0;\\
0, \ \ \ \  \ \ \ \ \ \ \ \ \ \ \ \ \ \ \ \ \ \ \ \  \ \ \ \ \ \ \ \  \ \ \ \ \ \ \ \ \ \ \  \ \ \  \ \  \ \ \ \ \ \ \  \ \ \ \ \ \ \ \ \ ||\gamma_t||_{0}=0.
\end{cases}
\end{aligned}
\end{eqnarray}
From (\ref{03109}), we have
\begin{eqnarray*}
	|\partial_{xx}\Upsilon^{m,M}(\gamma_t)|&\leq&  (24m^2+12m(m-1)+6m)|\gamma_t(t)|^{2m-2}\\
	&&+2m|M-3|(1+2(m-1))|\gamma_t(t)|^{2m-2}\\
	&=&2m[3(6m-1)+(2m-1)|M-3|]|\gamma_t(t)|^{2m-2}.
\end{eqnarray*}
That is (\ref{220817a1}).

By  (\ref{03108}), we see that  $\partial_{x}\Upsilon^{m,M}\in C^0(\hat{\Lambda})$.
By  (\ref{03109}), we see that $\partial_{xx}\Upsilon^{m,M}\in C^0(\hat{\Lambda})$ for all  $m\geq 2$.
Notice that
\begin{eqnarray}\label{220817b}
|\Upsilon^{m,M}(\gamma_t)|\leq (|M|+1)||\gamma_t||_0^{2m},\ \ (t,\gamma_t)\in [0, T]\times{\hat{\Lambda}}.
\end{eqnarray}
Then, from (\ref{220817a0}), (\ref{220817a}) and (\ref{220817a1}) we have  $\Upsilon^{m,M}(\cdot)\in C^{1,2}_{p}(\hat{\Lambda})$ for $m\geq 2$.  The proof is complete.
\end{proof}

For every $m\in \mathbb{N}^+$ and $M\in \mathbb{R}$, define $\overline{\Upsilon}^{m,M,2}:(\Lambda\otimes \Lambda)^2\rightarrow \mathbb{R}$ as follows: for  every $(\gamma_t,\gamma_t'), (\eta_s,\eta_s')  \in \Lambda\otimes \Lambda$,
$$
\overline{\Upsilon}^{m,M,2}((\gamma_t,\gamma_t'), (\eta_s,\eta_s')) :=   {\Upsilon}^{m,M}(\gamma_t,\eta_s) +{\Upsilon}^{m,M}(\gamma_t',\eta_s')+|t-s|^2.
$$
We now state below  some properties of $\Upsilon^{m,M}$, $\overline{\Upsilon}^{m,M}$ and $\overline{\Upsilon}^{m,M,2}$.
\begin{lemma}\label{theoremS0719}
	For every $m\in \mathbb{N}^+$ and  $M\in \mathbb{R}$,
	\begin{eqnarray}\label{s0}
	\begin{aligned}
	&\quad||\gamma_t||_{0}^{2m}+(M-3)|\gamma_t(t)|^{2m}\\
	&\leq   \Upsilon^{m,M}(\gamma_t)
	\leq 3||\gamma_t||_{0}^{2m}+(M-3)|\gamma_t(t)|^{2m}, \quad (t,\gamma_t)\in [0, T]\times{\hat{\Lambda}}.
	\end{aligned}
	\end{eqnarray}
	For $M\geq 3$, the functional  $\overline{\Upsilon}^{m,M}$ (resp., $\overline{\Upsilon}^{m,M,2}$) is  a gauge-type one on compete metric space $(\Lambda^t,d_{\infty})$ (resp., $(\Lambda^t\otimes \Lambda^t,d_{1,\infty})$) for each $t\in [0,T]$.
\end{lemma}

\begin{proof}  If $||\gamma_t||_0=0$, it is clear that (\ref{s0}) holds. Then we may assume that
$||\gamma_t||_0\neq0$.
Letting $\alpha:=|\gamma_t(t)|^{2m}$ and $M=3$, we have
\begin{eqnarray*}
	 \Upsilon^{m,3}(\gamma_t)=\frac{(||\gamma_t||_0^{2m}-|\gamma_t(t)|^{2m})^3}{||\gamma_t||_0^{4m}}+3|\gamma_t(t)|^{2m}:=f(\alpha)=\frac{(||\gamma_t||_0^{2m}-\alpha)^3}{||\gamma_t||_0^{4m}}+3\alpha.
\end{eqnarray*}
By
\begin{eqnarray}\label{03130}
f'(\alpha)=-3\frac{(||\gamma_t||_0^{2m}-\alpha)^2}{||\gamma_t||_0^{4m}}+3\geq0, \ \ \mbox{for all}\   0\leq \alpha\leq ||\gamma_t||_0^{2m},
\end{eqnarray}
we get that
$$
||\gamma_t||_0^{2m} =f(0)\leq  \Upsilon^{m,3}(\gamma_t)=f(\alpha)\leq f(||\gamma_t||_0^{2m})=3||\gamma_t||_0^{2m},\ \ (t,\gamma_t)\in [0,T]\times\hat{\Lambda}.
$$
Since
$${\Upsilon}^{m,M}(\gamma_t)={\Upsilon}^{m,3}(\gamma_t)+(M-3)|\gamma_t(t)|^{2m},$$
we get (\ref{s0}).
\par
Let $M\geq3$. It follows from (\ref{s0}) that, for every    $(t,(\gamma_t,\gamma_t')), (s,(\eta_s,\eta_s'))\in  [0,T]\times (\Lambda\otimes \Lambda)$,
\begin{eqnarray}\label{0612d}
\begin{aligned}
\overline{\Upsilon}^{m,M}(\gamma_t,\eta_s)&={\Upsilon}^{m,M}(\gamma_{t,t\vee s}^A-\eta_{s, t\vee s}^A)+|s-t|^2\\
&\geq ||\gamma_{t,t\vee s}^A-\eta_{s, t\vee s}^A||_0^{2m}+|s-t|^2,
\end{aligned}
\end{eqnarray}
\begin{eqnarray}\label{0612d0719}
\begin{aligned}
&\quad\overline{\Upsilon}^{m,M,2}((\gamma_t,\gamma_t'),(\eta_s,\eta_s'))\\
&={\Upsilon}^{m,M}(\gamma_t,\eta_s) +{\Upsilon}^{m,M}(\gamma_t',\eta_s')+|t-s|^2\\
&\geq ||\gamma_{t,t\vee s}^A-\eta_{s, t\vee s}^A||_0^{2m}+||{\gamma'}_{t,t\vee s}^A-{\eta'}_{s, t\vee s}^A||_0^{2m}+|s-t|^2.
\end{aligned}
\end{eqnarray}
Recalling that $d_{\infty}(\gamma_t,\eta_s)=|t-s|
+||\gamma_{t,t\vee s}^A-\eta_{s, t\vee s}^A||_{0}$ and $d_{1,\infty}((\gamma_t,\gamma_t'),(\eta_s,\eta_s'))=d_{\infty}(\gamma_t,\eta_s)+d_{\infty}(\gamma_t',\eta_s')$,
we see that  $\overline{\Upsilon}^{m,M}$ (resp., $\overline{\Upsilon}^{m,M,2}$) is  a gauge-type function on compete metric space $(\Lambda^t,d_{\infty})$ (resp., $(\Lambda^t\otimes \Lambda^t,d_{1,\infty})$).
\end{proof}

To prove the uniqueness of viscosity solutions with  Crandall-Ishii lemma (see Theorem ~\ref{theorem0513}),
we also need the following lemma.

\begin{lemma}\label{theoremS00044} For $m\in \mathbb{N}^+$ and $M\geq3$, we have, for all $(t,\gamma_t,\eta_t)\in [0,T]\times \hat{\Lambda}\times \hat{\Lambda}$,
	\begin{eqnarray}\label{jias5}
	\left(\Upsilon^{m,M}(\gamma_t+\eta_t)\right)^{\frac{1}{2m}}\leq \left(\Upsilon^{m,M}(\gamma_t)\right)^{\frac{1}{2m}}+ \left(\Upsilon^{m,M}(\eta_t)\right)^{\frac{1}{2m}}.
	\end{eqnarray}
\end{lemma}

\begin{proof} The proof is  quite similar to the finite dimensional case (see \cite[Lemma 3.3]{zhou5}), and is given below for the reader's convenience.

If one of $||\gamma_t||_{0}$, $||\eta_t||_{0}$ and $||\gamma_t+\eta_t||_{0}$ is equal to $0$, it is clear that (\ref{jias5}) holds. Then we may assume that
all of $||\gamma_t||_{0}$, $||\eta_t||_{0}$ and $||\gamma_t+\eta_t||_{0}$ are not equal to $0$.
By the definition of $\Upsilon^{m,M}$, we get, for every $(t, \gamma_t, \eta_t)\in [0,T]\times {\hat{\Lambda}}\times {\hat{\Lambda}}$,
\begin{eqnarray*}
	&&\Upsilon^{m,M}(\gamma_t+\eta_t)\\
	&=&\frac{(||\gamma_t+\eta_t||_{0}^{2m}-|\gamma_t(t)+\eta_t(t)|^{2m})^3}
	{||\gamma_t+\eta_t||_{0}^{4m}}+M|\gamma_t(t)+\eta_t(t)|^{2m}\\
	 &=&||\gamma_t+\eta_t||_{0}^{2m}-\frac{|\gamma_t(t)+\eta_t(t)|^{6m}}{||\gamma_t+\eta_t||_{0}^{4m}}+3\frac{|\gamma_t(t)+\eta_t(t)|^{4m}}{||\gamma_t+\eta_t||_{0}^{2m}}\\
	&&+(M-3)|\gamma_t(t)+\eta_t(t)|^{2m}.
\end{eqnarray*}
Define
\begin{eqnarray*}
	f(x,y):=x-\frac{y^3}{x^2}+3\frac{y^2}{x}+(M-3)y, \ \ 0\leq y\leq x, \ x>0.
\end{eqnarray*}
Then, we have
\begin{eqnarray*}
	\Upsilon^{m,M}(\gamma_t+\eta_t)=f(||\gamma_t+\eta_t||_{0}^{2m},|\gamma_t(t)+\eta_t(t)|^{2m}).
\end{eqnarray*}
Since
$$
f_x(x,y)=1+2\bigg{(}\frac{y}{x}\bigg{)}^3-3\bigg{(}\frac{y}{x}\bigg{)}^2=\bigg{(}\frac{2y}{x}+1\bigg{)}\bigg{(}\frac{y}{x}-1\bigg{)}^2\geq0
$$
and
$$
f_y(x,y)=-3\frac{y^2}{x^2}+6\frac{y}{x}+(M-3)\geq0
$$
for $0\leq y\leq x$ and $x>0$,  together with
$$
||\gamma_t+\eta_t||_{0}\leq ||\gamma_t||_{0}+||\eta_t||_{0}
$$
and $$|\gamma_t(t)+\eta_t(t)|\leq |\gamma_t(t)|+|\eta_t(t)|,
$$
we have
\begin{eqnarray*}
	&&\Upsilon^{m,M}(\gamma_t+\eta_t)
\leq f\left((||\gamma_t||_0+||\eta_t||_{0})^{2m},\, (|\gamma_t(t)|+|\eta_t(t)|)^{2m}\right)\\
	&=&(||\gamma_t||_{0}+||\eta_t||_{0})^{2m}-\frac{(|\gamma_t(t)|+|\eta_t(t)|)^{6m}}{(||\gamma_t||_{0}+||\eta_t||_{0})^{4m}}
	+3\frac{(|\gamma_t(t)|+|\eta(t)|)^{4m}}{(||\gamma_t||_{0}+||\eta_t||_{0})^{2m}}\\
	&&+(M-3)(|\gamma_t(t)|+|\eta_t(t)|)^{2m}.
\end{eqnarray*}
Setting $a=\frac{|\gamma_t(t)|}{||\gamma_t||_{0}}$, $b=\frac{|\eta_t(t)|}{||\eta_t||_{0}}$, $\alpha=||\gamma_t||_{0}$ and $\beta=||\eta_t||_{0}$,
we get that
\begin{eqnarray*}
	&&\Upsilon^{m,M}(\gamma_t+\eta_t)\\
	&\leq& (\alpha+\beta)^{2m}-\frac{(a\alpha+b\beta)^{6m}}{(\alpha+\beta)^{4m}}+3\frac{(a\alpha+b\beta)^{4m}}{(\alpha+\beta)^{2m}}+(M-3)(a\alpha+b\beta)^{2m}\\
	&=& (\alpha+\beta)^{2m}g^{2m}\left(\frac{a\alpha+b\beta}{\alpha+\beta}\right),    \\
	\Upsilon^{m,M}(\gamma_t)&=& \alpha^{2m}-{\alpha^{2m}}a^{6m}+3{\alpha^{2m}}a^{4m}+(M-3)a^{2m}\alpha^{2m}
	=\alpha^{2m}g^{2m}(a),\\
	\Upsilon^{m,M}(\eta_t)&=& \beta^{2m}-{\beta^{2m}}b^{6m}+3{\beta^{2m}}b^{4m}+(M-3)b^{2m}\beta^{2m}
	=\beta^{2m}g^{2m}(b),
\end{eqnarray*}
where
\begin{eqnarray}\label{0404}
g(x)=\left(1-{x^{6m}}+3{x^{4m}}+(M-3)x^{2m}
\right)^{\frac{1}{2m}},\ \  x\in [0,1].
\end{eqnarray}
From Lemma \ref{theoremS0404} below, we see that the function $g$ is  convex  on $[0,1]$ and thus
\begin{eqnarray*}
	&& \left(\Upsilon^{m,M}(\gamma_t+\eta_t)\right)^{\frac{1}{2m}}-\left(\Upsilon^{m,M}(\gamma_t)\right)^{\frac{1}{2m}}-\left(\Upsilon^{m,M}(\eta_t)\right)^{\frac{1}{2m}} \\
	&\leq&(\alpha+\beta)g\left(\frac{a\alpha+b\beta}{\alpha+\beta}\right)-\alpha g(a)-\beta g(b)\\
	& =&(\alpha+\beta)\left(g\left(\frac{\alpha}{\alpha+\beta}a+\frac{\beta}{\alpha+\beta}b\right)-\frac{\alpha}{\alpha+\beta} g(a)-\frac{\beta}{\alpha+\beta} g(b)\right)\leq0.
\end{eqnarray*}
Then,  we have the inequality~(\ref{jias5}). The proof is  complete.
\end{proof}

To complete the previous proof, it remains to state and prove the following lemma.
\begin{lemma}\label{theoremS0404} For $m\in {\mathbb{N}}^+$ and $M\geq3$, the function $g$ defined by (\ref{0404}) is  convex  on $[0,1]$.
\end{lemma}

\begin{proof}
By the definition of $g$, for all $x\in [0,1]$,
\begin{eqnarray*}
	g'(x)=\frac{1}{2m}g^{1-2m}(x)(-6mx^{6m-1}+12mx^{4m-1}+2m(M-3)x^{2m-1}
	),
\end{eqnarray*}
and
\begin{eqnarray*}
	&& g''(x)\\&=&\frac{1}{2m}g^{1-2m}(x)(-6m(6m-1)x^{6m-2}+12m(4m-1)x^{4m-2}\\
	&&+2m(2m-1)(M-3)x^{2m-2}
	)\\
	&&-\frac{1}{2m}(1-\frac{1}{2m})g^{1-4m}(x)(-6mx^{6m-1}+12mx^{4m-1}+2m(M-3)x^{2m-1}
	)^2\\
	&=& 3g^{1-4m}(x)x^{4m-2}(2x^{8m}-(2m+7)x^{6m}+6x^{4m}-(6m-1)x^{2m}+8m-2)\\
	&&+g^{1-4m}(x)x^{2m-2}(M-3)(-(8m+2)x^{6m}+(6m+3)x^{4m}+2m-1).
\end{eqnarray*}
Since
\begin{eqnarray*}
	&&2x^{8m}-(2m+7)x^{6m}+6x^{4m}-(6m-1)x^{2m}+8m-2\\
	&\geq& 2x^{8m}-(2m+7)x^{6m}+6x^{4m}-(6m-1)x^{2m}+(2m-1)x^{4m}+6m-1\\
	&=&x^{4m}(1-x^{2m})(2m+5-2x^{2m})+(6m-1)(1-x^{2m})\geq0, \quad  \ x\in [0,1]
\end{eqnarray*}
and
\begin{eqnarray*}
	&&-(8m+2)x^{6m}+(6m+3)x^{4m}+2m-1\\
	&\geq& -(8m+2)x^{6m}+(6m+3)x^{4m}+(2m-1)x^{4m}\\
	&=&(8m+2)x^{4m}(1-x^{2m})\geq0, \quad  \ x\in [0,1],
\end{eqnarray*}
we get that   $g''(x)\geq0$ for all $x\in [0,1]$ and  the function $g$  is a convex function on $[0,1]$.
The proof is now complete.
\end{proof}
\par
Note that $\Upsilon^{1,M}(\cdot)\notin C^{1,2}_p(\hat{\Lambda})$ since $\partial_{xx}\Upsilon^{1,M}(\cdot)$ is not continuous at $\gamma_t$ when $||\gamma_t||_0=0$.
Therefore, we  define, for every $\varepsilon>0$,
$$
\Upsilon^{\varepsilon}(\gamma_t)=\frac{(||\gamma_t||_0^2-|\gamma_t(t)|^2)^3}{\varepsilon^2+||\gamma_t||_0^4}+3|\gamma_t(t)|^2, \ \ \ (t,\gamma_t)\in[0,T]\times {\hat{\Lambda}}.
$$
The following lemma  will be used to study the PHJB equations under quadratic growth assumptions.
\begin{lemma}\label{theoremS1}  For every $\varepsilon>0$, $\Upsilon^{\varepsilon}$ satisfies the following inequality.
	\begin{eqnarray}\label{s088}
	\left(\left( ||\gamma_t||_0^2-\frac{\varepsilon}{2}\right)\vee 0\right)\leq \Upsilon^{\varepsilon}(\gamma_t)\leq 3||\gamma_t||_0^2, \ \ (t,\gamma_t)\in[0,T]\times {\hat{\Lambda}}.
	\end{eqnarray}
	Moreover,
	$\Upsilon^{\varepsilon}\in C^{1,2}_{p}(\hat{\Lambda})$,  and for $(t,\gamma_t)\in [0, T]\times{\hat{\Lambda}}$,
	\begin{eqnarray}\label{220817a01028}
	\partial_{t}\Upsilon^{\varepsilon}(\gamma_t)=0,
	\end{eqnarray}
	\begin{eqnarray}\label{220817a1028}
	|\partial_{x}\Upsilon^{\varepsilon}(\gamma_t)|\leq 6|\gamma_t(t)|,
	\end{eqnarray}
	and
	\begin{eqnarray}\label{220817a11028}
	|\partial_{xx}\Upsilon^{\varepsilon}(\gamma_t)|\leq 30.
	\end{eqnarray}
\end{lemma}

\begin{proof}
First, we prove (\ref{s088}).  Letting $\alpha:=|\gamma_t(t)|^{2}$, we have
\begin{eqnarray*}
	 \Upsilon^{\varepsilon}(\gamma_t)=\frac{(||\gamma_t||_0^{2}-|\gamma_t(t)|^{2})^3}{\varepsilon^2+||\gamma_t||_0^{4}}+3|\gamma_t(t)|^{2}:=f(\alpha)=\frac{(||\gamma_t||_0^{2}-\alpha)^3}{\varepsilon^2+||\gamma_t||_0^{4}}+3\alpha.
\end{eqnarray*}
Since
\begin{eqnarray*}
	f'(\alpha)=-3\frac{(||\gamma_t||_0^{2}-\alpha)^2}{\varepsilon^2+||\gamma_t||_0^{4}}+3\geq0, \quad  0\leq \alpha\leq ||\gamma_t||_0^{2},
\end{eqnarray*}
we have for all $(t,\gamma_t)\in [0,T]\times\hat{{\Lambda}}$,
$$
||\gamma_t||_0^{2}- \frac{\varepsilon}{2} \leq\frac{||\gamma_t||_0^{6}}{\varepsilon^2+||\gamma_t||_0^{4}} =f(0)\leq  \Upsilon^{\varepsilon}(\gamma_t)=f(\alpha)\leq f(||\gamma_t||_0^{2})=3||\gamma_t||_0^{2}.
$$
Noting $ \Upsilon^{\varepsilon}(\gamma_t)\geq0$ for every $(t,\gamma_t)\in [0,T]\times\hat{\Lambda}$, we have (\ref{s088}) holds true.
Second, it is clear that $\Upsilon^\varepsilon(\cdot)\in C^0(\hat{\Lambda})$ and $\partial_t\Upsilon^\varepsilon(\gamma_t)=0$ for all $(t,\gamma_t)\in[0,T]\times {\hat{\Lambda}}$. Similar to the proof of Lemma \ref{theoremS}, we show that for all $(t,\gamma_t)\in [0,T]\times\hat{{\Lambda}}$,
\begin{eqnarray}\label{2210281}
\partial_{x}\Upsilon^{\varepsilon}(\gamma_t)=\frac{-6(||\gamma_t||_0^2-|\gamma_t(t)|^2)^2\gamma_t(t)}{\varepsilon^2+||\gamma_t||_0^4}+6\gamma_t(t)
\end{eqnarray}
and
\begin{eqnarray}\label{2210282}
\begin{aligned}
&\quad
\partial_{xx}\Upsilon^\varepsilon(\gamma_t)\\
&=\frac{24(||\gamma_t||_0^2-|\gamma_t(t)|^2)\langle\gamma_t(t), \cdot\rangle_H\gamma_t(t)-6(||\gamma_t||_0^2-|\gamma_t(t)|^2)^2I}{\varepsilon^2+||\gamma_t||_0^4}+6I.
\end{aligned}
\end{eqnarray}
By  (\ref{2210281}) and (\ref{2210282}), it is clear that (\ref{220817a1028}) and (\ref{220817a11028}) hold, and   $\partial_{x}\Upsilon^{\varepsilon}(\cdot), \partial_{xx}\Upsilon^{\varepsilon}(\cdot)\in C^0(\hat{\Lambda})$.
Then, from (\ref{s088}), (\ref{220817a01028}), (\ref{220817a1028}) and (\ref{220817a11028}) we have  $\Upsilon^{\varepsilon}(\cdot)\in C^{1,2}_{p}(\hat{\Lambda})$.
The proof is  complete.\end{proof}

\newpage
\section{ Unbounded stochastic evolution equations
	 in path spaces}
\par
In this chapter, we consider the controlled state
equation (\ref{state1}).
Let us introduce the admissible control. Let $t$ and $s$ be two deterministic times such that $0\leq t\leq s\leq T$.
\begin{definition}
	An admissible control process $u(\cdot)=\{u(r),  r\in [t,s]\}$ on $[t,s]$  is an $\{{\mathcal{F}}_r\}_{t\leq r\leq s}$-progressively measurable process taking values in a Polish  space $(U,d)$. The set of all admissible controls on $[t,s]$ is denoted by ${\mathcal{U}}[t,s]$. We identify two processes $u(\cdot)$ and $\tilde{u}(\cdot)$ in ${\mathcal{U}}[t,s]$
	and write $u(\cdot)\equiv\tilde{u}(\cdot)$ on $[t,s]$, if $P(u(\cdot)=\tilde{u}(\cdot) \ a.e. \ \mbox{in}\ [t,s])=1$.
\end{definition}
The precise notion of solution to equation (\ref{state1}) will be given next. We make the following assumption.
\begin{assumption}\label{hypstate}
	\begin{description}
		\item[(i)]
		The operator $A$ is the generator of
		a $C_0$  {semi-group}   of bounded linear operator $\{e^{tA}, t\geq0\}$ in Hilbert space
		$H$.
		\item[(i')]
		The operator $A$ is the generator of a $C_0$ contraction
		{semi-group}  of bounded linear operators  $\{e^{tA}, t\geq0\}$ in
		Hilbert space $H$.
		\par
		\item[(ii)] $F:{\Lambda}\times U\rightarrow H$ and $G:{\Lambda}\times U\rightarrow L_2(\Xi,H)$ are continuous, {and
        $F,G$ are continuous in $\gamma_t\in \Lambda$, uniformly in $u\in U$.} Moreover,
		there exists a constant $L>0$ such that we have  for all $(t, \gamma_t, \eta_t, u) \in [0,T]\times {\Lambda}\times {\Lambda}\times U$,
		\begin{eqnarray}\label{assume1111}
		\begin{aligned}
		&|F(\gamma_t,u)|^2\vee|G(\gamma_t,u)|^2_{L_2(\Xi,H)}\leq
		L^2(1+||\gamma_t||^2_0),\\
		&|F(\gamma_t,u)-F(\eta_t,u)|\vee|G(\gamma_t,u)-G(\eta_t,u)|_{L_2(\Xi,H)}\leq
		L||\gamma_t-\eta_t||_0.
		\end{aligned}
		\end{eqnarray}
	\end{description}
\end{assumption}
\par
We say that $X$ is a mild solution of equation $(\ref{state1})$ with initial data $\xi_t\in L_{\mathcal{P}}^p(\Omega, \Lambda_t(H))$ for some $p>2$ if it is a continuous, $\{{\mathcal{F}}_s\}_{s\geq0}$-predictable $H$-valued process such that  $P$-a.s.,
\begin{eqnarray*}
	X(s)&=&e^{(s-t)A}\xi_t(t)+\int_{t}^{s}{e^{(s-\sigma)A}}F(X_\sigma,u(\sigma))d\sigma\\
	&&+\int_{t}^{s}{e^{(s-\sigma)A}}G(X_\sigma,u(\sigma))dW(\sigma),  \ s\in [t,T],
\end{eqnarray*}
and  $X(s)=\xi_t(s)$ for $s\in[0,t)$. To emphasize the dependence on initial data and control, we denote the solution by $X^{\xi_t,u}(\cdot)$.
\par
The following lemma is standard, and is available, for example,  in \cite[Theorem 3.6]{ro}.

\begin{lemma}\label{lemmaexist0409}
Assume that Assumption \ref{hypstate} (i) and (ii)  hold. Then for every $p>2$, $u(\cdot)\in {\mathcal{U}}[0,T]$,
	$\xi_t\in L^{p}_{\mathcal{P}}(\Omega, \Lambda_t(H))$, (\ref{state1}) admits a
	unique mild solution $X^{\xi_t,u}\in L_{\mathcal{P}}^p(\Omega, \Lambda_T(H))$ and   
	the following estimate hold:
	\begin{eqnarray}\label{state1est}
	\mathbb{E}||X_T^{\xi_t,u}||^p_0\leq C_p(1+\mathbb{E}||\xi_t||^p_0).
	\end{eqnarray}
	The constant $C_p$ depending only on  $p$, $T$, $L$ and
	$M_1=:\sup_{s\in [0,T]}|e^{sA}|$.
	\par
	Moreover,  let $A_\mu:=\mu A(\mu I-A)^{-1}$ be the Yosida approximation of $A$ and
	let $X^\mu$ be the solution of the following:
	\begin{eqnarray}\label{07162}
	\begin{aligned}
	X^\mu(s)&=e^{(s-t)A_\mu}\xi_t(t)+\int^{s}_{t}e^{(s-\sigma)A_\mu}F(X^{\mu}_\sigma,u(\sigma))d\sigma\\
	&\quad+\int^{s}_{t}e^{(s-\sigma)A_\mu}G(X^{\mu}_\sigma,u(\sigma))dW(\sigma),\quad  s\in [t,T];
	\end{aligned}
	\end{eqnarray}
	and $  X^\mu(s)=\xi_t(s), \ s\in[0,t)$.
	Then
	\begin{eqnarray}\label{0717}
	\lim_{\mu\rightarrow\infty}\mathbb{E}\left[\|X^{\xi_t,u}_T-X^\mu_T\|_0^p\right]=0.
	\end{eqnarray}
\end{lemma}

The proof  is  given below for a completeness and subsequent references.

\begin{proof}  Define  $\Phi$ in the Banach space $L^p_{\mathcal{P}}(\Omega, \Lambda_T(H))$  the transformation by the formula
\begin{eqnarray*}
	\Phi(X)(s)&=&e^{(s-t)A}\xi_t(t)+\int^{s}_{t}e^{(s-\sigma)A}F(X_\sigma,u(\sigma))d\sigma\\
	&&+\int^{s}_{t}e^{(s-\sigma)A}G(X_\sigma,u(\sigma))dW(\sigma),\quad  s\in [t,T];\\
	\Phi(X)(s)&=&\xi_t(s),\ \ s\in [0,t).
\end{eqnarray*}
We show that it is a contraction, under an equivalent norm
$$||X||=\bigg{(}\mathbb{E}\sup_{s\in[0,T]} e^{-\beta sp}|X(s)|^p\bigg{)}^{\frac{1}{p}},$$
with $\beta>0$ being waiting to be determined.

We shall use the so called factorization method (see~\cite[Theorem 5.2.5 ]{da}). Let us take $\alpha\in (0,1)$ such that
$$
\frac{1}{p}<\alpha<\frac{1}{2}\ \ \mbox{and let}\ \ c_{\alpha}^{-1}=\int^{s}_{\sigma}(s-l)^{\alpha-1}(l-\sigma)^{-\alpha}dl.
$$
Then by the Fubini theorem and stochastic Fubini theorem, for $s\in [t,T]$,
\begin{eqnarray*}
	&&\Phi(X)(s)\\
	&=&e^{(s-t)A}\xi_t(t)+c_{\alpha}\int^{s}_{t}\int^{s}_{\sigma}(s-l)^{\alpha-1}(l-\sigma)^{-\alpha}e^{(s-l)A}e^{(l-\sigma)A}dlF(X_\sigma,u(\sigma))d\sigma\\
	&&+c_{\alpha}\int^{s}_{t}\int^{s}_{\sigma}(s-l)^{\alpha-1}(l-\sigma)^{-\alpha}e^{(s-l)A}e^{(l-\sigma)A}dlG(X_\sigma,u(\sigma))dW(\sigma)\\
	&=&e^{(s-t)A}\xi_t(t)+c_{\alpha}\int^{s}_{t}(s-l)^{\alpha-1}e^{(s-l)A}Y(l)dl,
\end{eqnarray*}
where
$$
Y(l)=\int^{l}_{t}(l-\sigma)^{-\alpha}e^{(l-\sigma)A} F(X_\sigma,u(\sigma))d\sigma+\int^{l}_{t}(l-\sigma)^{-\alpha}e^{(l-\sigma)A} G(X_\sigma,u(\sigma))dW(\sigma).
$$
By the H\"{o}lder inequality, setting 
$q=\frac{p}{p-1}$,
\begin{eqnarray*}
	&&e^{-\beta s}\bigg{|}\int^{s}_{t}(s-l)^{\alpha-1}e^{(s-l)A}Y(l)dl\bigg{|}\\
	&\leq& \bigg{(}\int^{s}_{t}e^{-q\beta (s-l)}(s-l)^{(\alpha-1)q}dl\bigg{)}^{\frac{1}{q}}
	\bigg{(}\int^{s}_{t}e^{-p\beta l}\left|e^{(s-l)A}Y(l)\right|^pdl\bigg{)}^{\frac{1}{p}}\\
	&\leq& M_1\bigg{(}\int^{T}_{0}e^{-q\beta l}l^{(\alpha-1)q}dl\bigg{)}^{\frac{1}{q}}
	\bigg{(}\int^{T}_{t}e^{-p\beta l}|Y(l)|^pdl\bigg{)}^{\frac{1}{p}}.
\end{eqnarray*}
Then we get
\begin{eqnarray*}
	||\Phi(X)||\leq M_1(\mathbb{E}||\xi_t||^p_0)^{\frac{1}{p}}+M_1c_{\alpha}\bigg{(}\int^{T}_{0}e^{-q\beta l}l^{(\alpha-1)q}dl\bigg{)}^{\frac{1}{q}}
	\bigg{(}\mathbb{E}\int^{T}_{t}e^{-p\beta l}|Y(l)|^pdl\bigg{)}^{\frac{1}{p}}.
\end{eqnarray*}
By the Burkholder-Davis-Gundy
inequality and  Assumption \ref{hypstate} (i) and (ii), there is a constant $c_p$ depending only on $p$  (it may vary from line to line)  such that
\begin{eqnarray*}
	\mathbb{E}|Y(l)|^p
	&\leq& c_p\mathbb{E}\bigg{(}\int^{l}_{t}(l-\sigma)^{-\alpha}\left|e^{(l-\sigma)A} F(X_\sigma,u(\sigma))\right|d\sigma\bigg{)}^p\\
	&&+c_p\mathbb{E}\bigg{(}\int^{l}_{t}(l-\sigma)^{-2\alpha}\left|e^{(l-\sigma)A} G(X_\sigma,u(\sigma))\right|^2_{L_2(\Xi,H)}d\sigma\bigg{)}^{\frac{p}{2}}\\
	&\leq&c_pM_1^pL^p\mathbb{E}\bigg{(}\int^{l}_{t}(l-\sigma)^{-\alpha}(1+||X_\sigma||_0)d\sigma\bigg{)}^p\\
	&&+c_pM_1^pL^p\mathbb{E}\bigg{(}\int^{l}_{t}(l-\sigma)^{-2\alpha}(1+||X_\sigma||^2_0)d\sigma\bigg{)}^{\frac{p}{2}}\\
	&\leq&c_pM_1^pL^p\mathbb{E}\sup_{\sigma\in [0,l]}[(1+||X_\sigma||_0)^pe^{-p\beta\sigma}]\\
	&&\times\bigg{[}\bigg{(}\int^{l}_{t}(l-\sigma)^{-\alpha}e^{\beta\sigma}d\sigma\bigg{)}^p
	+\bigg{(}\int^{l}_{t}(l-\sigma)^{-2\alpha}e^{2\beta\sigma}d\sigma\bigg{)}^{\frac{p}{2}}\bigg{]},
\end{eqnarray*}
which implies
\begin{eqnarray*}
	&&e^{-p\beta l}\mathbb{E}|Y(l)|^p\\
	&\leq& c_pM_1^pL^p(1+||X||^p)\bigg{[}\bigg{(}\int^{T}_{0}\sigma^{-\alpha}e^{-\beta\sigma}d\sigma\bigg{)}^p
	+\bigg{(}\int^{T}_{0}\sigma^{-2\alpha}e^{-2\beta\sigma}d\sigma\bigg{)}^{\frac{p}{2}}\bigg{]}.
\end{eqnarray*}
We conclude that
\begin{eqnarray*}
	||\Phi(X)||&\leq& M_1(\mathbb{E}||\xi_t||^p_0)^{\frac{1}{p}}+M_1^2Lc_{\alpha}(Tc_p(1+||X||^p))^{\frac{1}{p}}\bigg{(}\int^{T}_{0}e^{-q\beta l}l^{(\alpha-1)q}dl\bigg{)}^{\frac{1}{q}}\\
	&&\times
	\bigg{[}\bigg{(}\int^{T}_{0}\sigma^{-\alpha}e^{-\beta\sigma}d\sigma\bigg{)}^p
	+\bigg{(}\int^{T}_{0}\sigma^{-2\alpha}e^{-2\beta\sigma}d\sigma\bigg{)}^{\frac{p}{2}}\bigg{]}^{\frac{1}{p}}.
\end{eqnarray*}
Hence,  $\Phi$ is a well defined transform on $L^p_{\mathcal{P}}(\Omega, \Lambda_T(H))$. For $X,X'\in L^p_{\mathcal{P}}(\Omega,  \Lambda_T(H)) $,  we can prove in a similar manner the following estimate
\begin{eqnarray*}
	||\Phi(X)-\Phi(X')||&\leq& M_1^2Lc_{\alpha}(Tc_p)^{\frac{1}{p}}||X-X'||\bigg{(}\int^{T}_{0}e^{-q\beta l}l^{(\alpha-1)q}dl\bigg{)}^{\frac{1}{q}}\\
	&&\times
	\bigg{[}\bigg{(}\int^{T}_{0}\sigma^{-\alpha}e^{-\beta\sigma}d\sigma\bigg{)}^p
	+\bigg{(}\int^{T}_{0}\sigma^{-2\alpha}e^{-2\beta\sigma}d\sigma\bigg{)}^{\frac{p}{2}}\bigg{]}^{\frac{1}{p}}.
\end{eqnarray*}
Therefore,  $\Phi$  is a contraction for sufficiently large $\beta$. In particular, we obtain  the
estimate (\ref{state1est}).
\par
Finally, for some positive constant $c_p$ depending only on $p$ that might vary from line to line,
\begin{eqnarray}\label{0716}
\begin{aligned}
 &\quad\mathbb{E}\sup_{r\in[0,s]}\left|X^{\xi_t,u}(r)-X^\mu(r)\right|^p\\
&\leq c_p\bigg{(}\mathbb{E}\int^{s}_{t}\sup_{\theta\in[0,T]}\left|(e^{\theta A}-e^{\theta A_\mu})F(X^{\xi_t,u}_\sigma,u(\sigma))\right|^pd\sigma\\
&\quad+L^pM^p_1\mathbb{E}\int^{s}_{t}\sup_{\theta\in[0,\sigma]}\left|X^{\xi_t,u}(\theta)-X^\mu(\theta)\right|^pd\sigma\\
&\quad+\mathbb{E}\sup_{r\in[t,s]}\bigg{|}\int^{r}_{t}(e^{(r-\sigma)A}G(X^{\xi_t,u}_\sigma,u(\sigma))-e^{(r-\sigma)A_\mu}G(X^\mu_\sigma,u(\sigma)))dW(\sigma)\bigg{|}^p\\
&\quad+\mathbb{E}\sup_{r\in[0,s]}\left|(e^{rA}-e^{rA_\mu})\xi_t(t)\right|^p\bigg{)}.
\end{aligned}
\end{eqnarray}
On the other hand, by the stochastic Fubini
theorem,
\begin{eqnarray*}
	&&\int^{r}_{t}\left(e^{(r-\sigma)A}G(X^{\xi_t,u}_\sigma,u(\sigma))-e^{(r-\sigma)A_\mu}G(X^\mu_\sigma,u(\sigma))\right)dW(\sigma)\\
	&=&c_\alpha\int^{r}_{t}(r-\theta)^{\alpha-1}\left(e^{(r-\theta)A}\hat{Y}(\theta)-e^{(r-\theta)A_\mu}Y^\mu(\theta)\right)d\theta,
\end{eqnarray*}
with
$$
\hat{Y}(\theta)=\int^{\theta}_{t}(\theta-\sigma)^{-\alpha}e^{(\theta-\sigma)A}G(X^{\xi_t,u}_\sigma,u(\sigma))dW(\sigma)
$$
and
$$
Y^\mu(\theta)=\int^{\theta}_{t}(\theta-\sigma)^{-\alpha}e^{(\theta-\sigma)A_\mu}G(X^\mu_\sigma,u(\sigma))dW(\sigma).
$$
From the H\"{o}lder inequality and the Burkholder-Davis-Gundy
inequality, we have
\begin{eqnarray}\label{07161}
\qquad&\begin{aligned}
 &\quad
\mathbb{E}\sup_{r\in[t,s]}\bigg{|}\int^{r}_{t}\left(e^{(r-\sigma)A}G(X^{\xi_t,u}_\sigma,u(\sigma))-e^{(r-\sigma)A_\mu}
G(X^\mu_\sigma,u(\sigma))\right)dW(\sigma)\bigg{|}^p\\
&\leq c_pc^p_\alpha\bigg{(}\int^{T}_{0}r^{q(\alpha-1)}dr\bigg{)}^{\frac{p}{q}}\\
&\quad
\times \bigg{(}\mathbb{E}\int^{s}_{t}\!\! \sup_{r\in[0,T]}\left|\left(e^{rA}-e^{rA_\mu}\right)\hat{Y}(\theta)\right|^p\! d\theta+M^p_1\mathbb{E}\int^{s}_{t}|\hat{Y}(\theta)-Y^\mu(\theta)|^pd\theta\bigg{)}\\
&\leq c_pc^p_\alpha\bigg{(}\int^{T}_{0}r^{q(\alpha-1)}dr\bigg{)}^{\frac{p}{q}}\bigg{(}\mathbb{E}\int^{s}_{t}\sup_{r\in[0,T]}\left|\left(e^{rA}-e^{rA_\mu}\right)\hat{Y}(\theta)\right|^pd\theta
\\
&\quad+
M_1^{2p}L^p\bigg{(}\int^{T}_{0}r^{-2\alpha}dr\bigg{)}^{\frac{p}{2}}\mathbb{E}\int^{s}_{t}\sup_{\sigma\in[0,\theta]}|X^{\xi_t,u}(\sigma)-X^\mu(\sigma)|^pd\theta\\
&\quad+\mathbb{E}\int^{s}_{t}\!\!\! \bigg{(}\int^{\theta}_{t}(\theta-\sigma)^{-2\alpha}\sup_{r\in[0,T]}\left|\left(e^{rA}-e^{rA_\mu}\right)G(X^{\xi_t,u}_\sigma,u(\sigma))\right|^2
d\sigma\bigg{)}^{\frac{p}{2}}d\theta\bigg{)}.
\end{aligned}
\end{eqnarray}
From (\ref{0716}) and (\ref{07161}),  applying  Gronwall's inequality  and the dominated
convergence theorem, we have
$$\lim_{ \mu\to \infty}\mathbb{E}\|X^{\xi_t,u}_T-X^\mu_T\|_0^p=0. $$
The proof is complete.
\end{proof}

\par
The next lemma concerns the local boundedness and the continuity of the trajectory $X^{\xi_t,u}$.

\begin{lemma}\label{lemmaexist111}
	Let Assumption \ref{hypstate}  (i) and (ii)  be satisfied. Then, for any $p>2$, there is a  positive constant $C_{1,p}$,  depending only on  $p$, $T$, $L$ and $M_1$, such that for $0\leq t\leq \bar{t}\leq T$, $\xi_t, \xi'_t\in L_{\mathcal{P}}^p(\Omega, \Lambda_t(H))$ and $u(\cdot), v(\cdot)\in {\mathcal{U}}[0,T]$,
	\begin{eqnarray}\label{2.6}
	\begin{aligned}
	&\quad \sup_{u(\cdot)\in {\mathcal{U}}[0,T]}\mathbb{E}\left[\left|X^{\xi_t,u}(s)-e^{(s-t)A}\xi_t(t)\right|^p\right]\\
	&\leq
	C_{1,p}\left(1+\mathbb{E}\left[||\xi_t||_0^p\right]\right)|s-t|^{\frac{p}{2}}, \quad s\in[t,T]
	\end{aligned}
	\end{eqnarray}
	and
	\begin{eqnarray}\label{0604}
&\quad	\begin{aligned}
	&\qquad\mathbb{E}\left[\left\|X^{\xi'_t,u}_T-X^{\xi_{t,\bar{t}}^A,v}_T\right\|^p_0\right]
\leq\, C_{1,p}\left(1+\mathbb{E}\left[||\xi'_t||^p_0\right]\right)\, (\bar{t}-t)^{\frac{p}{2}}\\
&\qquad+C_{1,p}\, \mathbb{E}\left[||\xi'_t-\xi_t||^p_0\right]
	+C_{1,p}\int^{T}_{\bar{t}}\!\! \mathbb{E}\left[\left|(F, G)(X^{\xi_{t,\bar{t}}^A,v}_r,\cdot)\bigm|^{u(r)}_{v(r)}\right|^p_{H\times L_2(\Xi,H)}\right]dr.
	\end{aligned}
	\end{eqnarray}
	
\end{lemma}

\begin{proof}
For any  $\xi_t\in L_{\mathcal{P}}^p(\Omega, \Lambda_t(H))$, by (\ref{assume1111}) and (\ref{state1est}), we have
\begin{eqnarray*}
	&&\mathbb{E}\left[\left|X^{\xi_t,u}(s)-e^{(s-t)A}\xi_t(t)\right|^p\right]\\
	&\leq&  c_p\, L^pM^p_1\left[1+C_p\left(1+\mathbb{E}\left[||\xi_t||^p_0\right]]\right)\right]|s-t|^{\frac{p}{2}}\left(1+|s-t|^{\frac{p}{2}}\right).
\end{eqnarray*}
Here and in the rest of this proof, $c_p$ is a positive constant, which depends only on $p$ and may vary from line to line.
Taking the supremum in ${\mathcal{U}}[0,T]$, we obtain (\ref{2.6}).
For any $0\leq t\leq \bar{t}\leq T$, $\xi_t, \xi'_t\in L_{\mathcal{P}}^p(\Omega, \Lambda_t(H))$ and $u(\cdot),v(\cdot)\in {\mathcal{U}}[0,T]$, by (\ref{assume1111}) and (\ref{state1est}), we have
\begin{eqnarray*}
	&&\mathbb{E}\sup_{\bar{t}\leq s\leq \sigma}\left|X^{\xi'_t,u}(s)-X^{\xi_{t,\bar{t}}^A,v}(s)\right|^p\\
	&\leq& c_pM^p_1\mathbb{E}|\xi'_t(t)-\xi_t(t)|^p+
	c_p\mathbb{E}\bigg{(}\int^{\bar{t}}_{t}\left|e^{(s-r)A}F(X^{\xi'_t,u}_r,u(r))\right|dr\bigg{)}^p\\
	&&+c_p\mathbb{E}\bigg{|}\int^{\bar{t}}_{t}e^{(s-r)A}G(X^{\xi'_t,u}_r,u(r))dW(r)\bigg{|}^p\\
	&&
	+c_p\mathbb{E}\bigg{(}\sup_{\bar{t}\leq s\leq \sigma}\int^{s}_{\bar{t}}\left|e^{(s-r)A}\left(F(X^{\xi'_t,u}_r,u(r))-F(X^{\xi_{t,\bar{t}}^A,v}_r,v(r))\right)\right|dr\bigg{)}^p\\
	&&
	+c_p\mathbb{E}\sup_{\bar{t}\leq s\leq \sigma}\bigg{|}\int^{s}_{\bar{t}}e^{(s-r)A}\left(G(X^{\xi'_t,u}_r,u(r))-G(X^{\xi_{t,\bar{t}}^A,v}_r,v(r))\right)dW(r)\bigg{|}^p\\
	 &\leq&c_pM^p_1\mathbb{E}|\xi'_t(t)-\xi_t(t)|^p+c_pL^pM^p_1[1+C_p(1+\mathbb{E}||\xi'_t||^p_0)](\bar{t}-t)^{\frac{p}{2}}\left(1+(\bar{t}-t)^{\frac{p}{2}}\right)\\
	&&
	+c_pL^pM^p_1\int^{\sigma}_{\bar{t}}\!\!\!\!~\! \mathbb{E}\!\!\left[\Bigm\|X^{\xi'_t,u}_r-X^{\xi_{t,\bar{t}}^A,v}_r\Bigm\|^p_0\right]\! dr
	+c_pM^p_1\int^{T}_{\bar{t}}\!\!\!\!\mathbb{E}\!\!\left[\left|F(X^{\xi_{t,\bar{t}}^A,v}_r,\cdot)\Bigm|^{u(r)}_{v(r)}\right|^p\right]\! dr\\
	&&+c_p\mathbb{E}\sup_{\bar{t}\leq s\leq \sigma}\bigg{|}\int^{s}_{\bar{t}}e^{(s-r)A}\left(G(X^{\xi'_t,u}_r,u(r))-G(X^{\xi_{t,\bar{t}}^A,v}_r,v(r))\right)dW(r)\bigg{|}^p.
\end{eqnarray*}
Proceeding as in the proof of Lemma~\ref{lemmaexist0409},  we get  for $\frac{1}{p}<\alpha<\frac{1}{2}$ and for a suitable constant $c_p$
\begin{eqnarray*}
	&&\mathbb{E}\sup_{\bar{t}\leq s\leq \sigma}\bigg{|}\int^{s}_{\bar{t}}e^{(s-r)A}\left(G(X^{\xi'_t,u}_r,u(r))-G(X^{\xi_{t,\bar{t}}^A,v}_r,v(r))\right)dW(r)\bigg{|}^p \\
	&\leq&c_pc^p_\alpha M_1^{2p}L^p\bigg{(}\int^{T}_{0}r^{q(\alpha-1)}dr\bigg{)}^{\frac{p}{q}}
	 \bigg{(}\int^{T}_{0}r^{2\alpha}dr\bigg{)}^{\frac{p}{2}}\int^{\sigma}_{\bar{t}}\mathbb{E}\left[\Bigm\|X^{\xi'_t,u}_r-X^{\xi_{t,\bar{t}}^A,u}_r\Bigm\|^p_0\right]dr\\
	&&\!\!\!\!\!\!+c_pc^p_\alpha M_1^{2p}\!\! \left(\!\int^{T}_{0}r^{q(\alpha-1)}dr\right)^{\frac{p}{q}} \left(\int^{T}_{0}\!\! r^{2\alpha}dr\right)^{\frac{p}{2}}
	\int^{T}_{\bar{t}}\!\!\! \mathbb{E}\left[
	\left|G(X^{\xi_{t,\bar{t}}^A,v}_r,\cdot)\Bigm|^{u(r)}_{v(r)}\right|^p\right] dr.
\end{eqnarray*}
Thus,
\begin{eqnarray*}
	&&\mathbb{E}\left[\Bigm\|X^{\xi'_t,u}_\sigma-X^{\xi_{t,\bar{t}}^A,u}_\sigma\Bigm\|^p_0\right]\\
	&\leq& c_p\, M_1^p\, \mathbb{E}\left[||\xi'_t-\xi_t||^p_0\right]+c_p\, L^pM^p_1\, (1+T^{\frac{p}{2}})\left[1+C_p(1+\mathbb{E}\left[||\xi'_t||^p_0\right])\right](\bar{t}-t)^{\frac{p}{2}}\\
	&&+c_p\, L^pM_1^p\, (1+c^p_\alpha M_1^p)\int^{\sigma}_{\bar{t}}\mathbb{E}\left[\Bigm\|X^{\xi'_t,u}_r-X^{\xi_{t,\bar{t}}^A,u}_r\Bigm\|^p_0\right]dr\\
	&&+c_pM_1^p(1+c^p_\alpha M_1^p)\int^{T}_{\bar{t}}\!\!\! \mathbb{E}\left[\left|F(X^{\xi_{t,\bar{t}}^A,v}_r,\cdot)\Bigm|^{u(r)}_{v(r)}\right|^p+\left|G(X^{\xi_{t,\bar{t}}^A,v}_r,\cdot)\Bigm|^{u(r)}_{v(r)}\right|^p\right]dr.
\end{eqnarray*}
Then, by Gronwall's inequality, we obtain (\ref{0604}). The proof is complete.
\end{proof}

For the particular case of a deterministic $\xi_t$, i.e. $\xi_t=\gamma_t\in \Lambda$, denote by $X^{\gamma_t,u}(\cdot)$ the solution of equation (\ref{state1}) corresponding to $(\gamma_t,u(\cdot))\in  \Lambda\times {\mathcal{U}}[t,T]$.
From  Lemmas \ref{theoremito}, we have the following
lemma which is used to prove the existence  of viscosity solutions.

\begin{lemma}\label{theoremito3}
	Let $X^{\gamma_t,u}$ be a solution of equation  (\ref{state1}) with initial data $\gamma_t\in \Lambda_t$ and $f\in  C_p^{1,2}({\Lambda}^{{\hat{t}}})$  such that  $A^*\partial_xf\in C_p^0(\hat{\Lambda}^{{\hat{t}}},H)$
	for some $\hat{t}\in[t,T)$. Then, $P$-a.s., for any $s\in [\hat{t},T]$,
	\begin{eqnarray}\label{statesop03}
	\begin{aligned}
	&f\left(X^{\gamma_t,u}_s\right)=f\left(X^{\gamma_t,u}_{\hat{t}}\right)+\int_{\hat{t}}^{s}\partial_tf\left(X^{\gamma_t,u}_\sigma\right)d\sigma\\
	&\qquad\quad+\int^{s}_{\hat{t}}\biggl[\left\langle A^*\partial_xf\left(X^{\gamma_t,u}_\sigma\right), \,  X^{\gamma_t,u}\left(\sigma\right)\right\rangle_H+\left\langle \partial_xf\left(X^{\gamma_t,u}_\sigma\right),\, F\left(X^{\gamma_t,u}_\sigma,u(\sigma)\right)\right\rangle_H\\
	&\qquad\qquad\quad+\frac{1}{2}\mbox{\rm Tr}\left(\partial_{xx}f\left(X^{\gamma_t,u}_\sigma\right)
	\, {(GG^*)}\left(X^{\gamma_t,u}_\sigma,u\left(\sigma\right)\right)\right)\biggr]d\sigma\\
	 &\quad\quad\quad+\int^{s}_{\hat{t}}\left\langle\partial_xf\left(X^{\gamma_t,u}_\sigma\right),\, G\left(X^{\gamma_t,u}_\sigma,u(\sigma)\right)dW(\sigma)\right\rangle_H.
	\end{aligned}
	\end{eqnarray}
\end{lemma}

The following    lemma is also needed to prove the existence and uniqueness of viscosity solutions.

\begin{lemma}\label{theoremito2}
	Let Assumption \ref{hypstate} be satisfied. For every $(t,\gamma_t, \eta_t)\in [0,T)\times \Lambda_t\times \Lambda_t$, $m\geq 2$, $M\geq3$  and $\varepsilon>0$,  $P$-a.s.,  for all $s\in [t,T]$:
	\begin{eqnarray}\label{jias510815jia}
	\begin{aligned}
	\ \ \ \ &\quad
	\Upsilon^{m,M}(X^{\gamma_t,u}_s-\eta_{t,s}^A)\\
	&\leq \Upsilon^{m,M}(X^{\gamma_t,u}_t-\eta_t)+\int^{s}_{{t}}\biggl[\langle \partial_x\Upsilon^{m,M}(X^{\gamma_t,u}_\sigma-\eta_{t,\sigma}^A),\, F(X^{\gamma_t,u}_\sigma,u(\sigma))\rangle_H\\
	&\quad+\frac{1}{2}\mbox{\rm Tr}\left(\partial_{xx}\Upsilon^{m,M}(X^{\gamma_t,u}_\sigma-\eta_{t,\sigma}^A)\, {(GG^*)}(X^{\gamma_t,u}_\sigma,u(\sigma))\right)\biggr]d\sigma\\
	&\quad
	+\int^{s}_{t}\langle \partial_x\Upsilon^{m,M}(X^{\gamma_t,u}_\sigma-\eta_{t,\sigma}^A),\,  G(X^{\gamma_t,u}_\sigma,u(\sigma))dW(\sigma)\rangle_H;
	\end{aligned}
	\end{eqnarray}
	\begin{eqnarray}\label{jias510815jia1102}
	\begin{aligned}
	\ \ \ \ &\quad
	\Upsilon^{\varepsilon}(X^{\gamma_t,u}_s-\eta_{t,s}^A)\\
	&\leq  \Upsilon^{\varepsilon}(X^{\gamma_t,u}_t-\eta_t)+\int^{s}_{{t}}\biggl[\left\langle \partial_x \Upsilon^{\varepsilon}(X^{\gamma_t,u}_\sigma-\eta_{t,\sigma}^A),\,  F(X^{\gamma_t,u}_\sigma,u(\sigma))\right\rangle_H\\
	&\quad+\frac{1}{2}\mbox{\rm Tr}\left(\partial_{xx} \Upsilon^{\varepsilon}(X^{\gamma_t,u}_\sigma-\eta_{t,\sigma}^A)\, {(GG^*)}(X^{\gamma_t,u}_\sigma,u(\sigma))\right)\biggr]d\sigma\\
	&\quad
	+\int^{s}_{t}\left\langle \partial_x \Upsilon^{\varepsilon}(X^{\gamma_t,u}_\sigma-\eta_{t,\sigma}^A),\,  G(X^{\gamma_t,u}_\sigma,u(\sigma))dW(\sigma)\right\rangle_H;
	\end{aligned}
	\end{eqnarray}
	and, for every   $(t,\gamma_t, \eta_t)\in [0,T)\times \Lambda_t\times \Lambda_t$ and $m\in \mathbb{N}^+$,  $P$-a.s.,  for all $s\in [t,T]$:
	\begin{eqnarray}\label{jias510815jia11}
	\begin{aligned}
	\ \ \ \ &\quad
	|X^{\gamma_t,u}(s)-e^{(s-t)A}\eta_{t}(t)|^{2m}=|y(s)|^{2m}\\
	&\leq |y(t)|^{2m}+2m\int^{s}_{{t}}\biggl[|y(s)|^{2m-2}\left\langle y(s),\,  F(X^{\gamma_t,u}_\sigma,u(\sigma))\right\rangle_H\\
	&\quad+\frac{1}{2}\mbox{\rm Tr}\biggl((|y(s)|^{2m-2}I+2(m-1)|y(s)|^{2m-4}y(s)\langle y(s), \cdot\rangle_H{\mathbf{1}_{\{m>1\}}})\\
	&\quad\quad\quad\times (GG^*)(X^{\gamma_t,u}_\sigma,u(\sigma))\biggr)\biggr]d\sigma\\
	&\quad
	+2m\int^{s}_{t}|y(s)|^{2m-2}\langle y(s), \, G(X^{\gamma_t,u}_\sigma,u(\sigma))dW(\sigma)\rangle_H,
	\end{aligned}
	\end{eqnarray}
	where
	$y(s):=X^{\gamma_t,u}(s)-e^{A(s-t)}\eta_t(t)$ as $s\in [t, T]$,  and $y(s):=\gamma_t(s)-\eta_t(s)$ as $s\in [0, t)$.
\end{lemma}

\begin{proof} Let $X^\mu$ be the solution of equation (\ref{07162}) with initial data $\gamma_t\in \Lambda_t$ and
define $y^\mu$ by $y^\mu(s):=X^\mu(s)-e^{(s-t)A_\mu}\eta_t(t)$ for $s\in [t,T]$ and $y^\mu(s):=\gamma_t(s)-\eta_t(s)$ for $s\in [0,t)$.  Then, $y^\mu$ solves the following FSEE
\begin{eqnarray*}
	y^\mu(s)&=&e^{(s-t)A_\mu}y^\mu_t(t)+\int^{s}_{t}e^{(s-\sigma)A_\mu}F(X^{\mu}_\sigma, u(\sigma))d\sigma\\
	&&+\int^{s}_{t}e^{(s-\sigma)A_\mu}G(X^{\mu}_\sigma,u(\sigma))dW(\sigma),\ s\in [t,T],
\end{eqnarray*}
and $y^\mu(s)=\gamma_t(s)-\eta_t(s), \ s\in[0,t)$.
From Lemmas  \ref{theoremito} and \ref{theoremS}, we have, $P$-a.s.,  for all $s\in [t,T]$,
\begin{eqnarray*}
	\Upsilon^{m,M}(y^\mu_s)&=&\Upsilon^{m,M}(y^\mu_{{t}})+\int^{s}_{{t}}\biggl[\langle \partial_x\Upsilon^{m,M}(y^\mu_\sigma), \, A_\mu y^\mu(\sigma)+F(X^{\mu}_\sigma,u(\sigma))\rangle_H\\
	&&+\frac{1}{2}\mbox{Tr}\left(\partial_{xx}\Upsilon^{m,M}(y^\mu_\sigma)(GG^*)(X^{\mu}_\sigma,u(\sigma))\right)\biggr]d\sigma\\
	&&+\int^{s}_{t}\langle \partial_x\Upsilon^{m,M}(y^\mu_\sigma),\,  G(X^{\mu}_\sigma,u(\sigma))dW(\sigma)\rangle_H.
\end{eqnarray*}
Since $A$ is the infinitesimal generator of a $C_0$ contraction {semi-group}, we have
$$\langle x, A_\mu x\rangle_H\leq0,\quad x\in H.$$
Then, for $M\geq 3$, if $||y^\mu_\sigma||^2_{0}\neq0$,
\begin{eqnarray*}
	&&\langle \partial_x\Upsilon^{m,M}(y^\mu_\sigma), A_\mu y^\mu(\sigma)\rangle_H\\
	&=&2mM|y^\mu(\sigma)|^{2m-2}\left\langle y^\mu(\sigma), \, A_\mu y^\mu(\sigma)\right\rangle_H\\
	&& \quad -\frac{6m}{||y^\mu_\sigma||^{4m}_{0}}\left(||y^\mu_\sigma||^{2m}_{0}-|y^\mu(\sigma)|^{2m}\right)^2|y^\mu(\sigma)|^{2m-2} \left\langle y^\mu(\sigma),\,  A_\mu y^\mu(\sigma)\right\rangle_H\\
	&=&2m|y^\mu(\sigma)|^{2m-2}\bigg{(}M-\frac{3}{||y^\mu_\sigma||^{4m}_{0}}\left(||y^\mu_\sigma||^{2m}_{0}-|y^\mu(\sigma)|^{2m}\right)^2\bigg{)}\langle y^\mu(\sigma), A_\mu y^\mu(\sigma)\rangle_H\\
	&\leq& 0.
\end{eqnarray*}
Thus, $P$-a.s.,  for all $s\in [t,T]$:
\begin{eqnarray*}
	\Upsilon^{m,M}(y^\mu_s)&\leq& \Upsilon^{m,M}(y^\mu_{{t}})+\int^{s}_{{t}}[\langle \partial_x\Upsilon^{m,M}(y^\mu_\sigma), F(X^{\mu}_\sigma,u(\sigma))\rangle_H \\
	&&+\frac{1}{2}\mbox{Tr}(\partial_{xx}\Upsilon^{m,M}(y^\mu_\sigma)(GG^*)(X^{\mu}_\sigma,u(\sigma)))]d\sigma\\
	&&
	+\int^{s}_{t}\langle \partial_x\Upsilon^{m,M}(y^\mu_\sigma), G(X^{\mu}_\sigma,u(\sigma))dW(\sigma)\rangle_H.
\end{eqnarray*}
Letting $\mu\rightarrow\infty$, we have  from (\ref{0717}), $P$-a.s.,  for all $s\in [t,T]$,
\begin{eqnarray*}
	\Upsilon^{m,M}(y_s)&\leq & \Upsilon^{m,M}(y_{t})+\int^{s}_{{t}}[\langle \partial_x\Upsilon^{m,M}(y_\sigma), F(X^{\gamma_t,u}_\sigma,u(\sigma))\rangle_Hd\sigma\\
	&&+\frac{1}{2}\mbox{Tr}(\partial_{xx}\Upsilon^{m,M}(y_\sigma)(GG^*)(X^{\gamma_t,u}_\sigma,u(\sigma)))]d\sigma\\
	&&
	+\int^{s}_{t}\langle \partial_x\Upsilon^{m,M}(y_\sigma), G(X^{\gamma_t,u}_\sigma,u(\sigma))dW(\sigma)\rangle_H,
\end{eqnarray*}
where
$y(s)=X^{\gamma_t,u}(s)-e^{A(s-t)}\eta_t(t),\  t\leq s\leq T$ and $y(s)=\gamma_t(s)-\eta_t(s),\  0\leq s<t$.
That is (\ref{jias510815jia}).
\par
Applying  Lemma  \ref{theoremito}  to $\Upsilon^{\varepsilon}$ and $ y^\mu$, by Lemma \ref{theoremS1} and the proof procedure above, we have (\ref{jias510815jia1102}).
In a similar (even easier) way, we have (\ref{jias510815jia11}).
The proof is now complete.
\end{proof}

\begin{remark}\label{remarks}
	\begin{description}
		\item[(i)]      Since in general the path functional $||\cdot||_{0}^6 \notin C^{1,2}_p({\Lambda})$, then, for every $a_{\hat{t}}\in \Lambda$, $||\gamma_t-a_{\hat{t},t}^A||_0^6$ fails to be used directly  to construct  any auxiliary functional to prove the
		uniqueness and stability of viscosity solutions. However, by the preceding lemma,
		the  equivalent functional  $\Upsilon(\gamma_t-a_{\hat{t},t}^A)$ behaves very much like $||\gamma_t-a_{\hat{t},t}^A||_{0}^6$, and even gets rid of the effect of the unbounded operator $A$.  It shall be used to prove the uniqueness  of viscosity solutions when  the coefficients are merely strongly  continuous  under $||\cdot||_0$.
		\item[(ii)] {Lemma \ref{theoremS0719}} states
   that $\overline{\Upsilon}(\cdot,\cdot)$ is a gauge-type function.  We can apply it to Lemma \ref{theoremleft} to
		ensure the existence of  a maximizing point of a perturbed auxiliary functional in the proof of uniqueness.
	\end{description}
\end{remark}

The following lemma will be used to prove Lemma \ref{lemmaexist1022}  in the next chapter.
\begin{lemma}\label{lemmaexist10220}
	Assume that Assumption \ref{hypstate}    holds. Then, for any $m\in \mathbb{N}^+$, there is a positive constant $C_{2,m}$,  depending only on  $m$, $T$ and
 {$L$}, such that for any $(t,\gamma_t,\eta_t)\in [0,T)\times\Lambda_t\times\Lambda_t$, we have
	\begin{eqnarray}\label{2.61022}
	\sup_{u(\cdot)\in {\mathcal{U}}[0,T]}\mathbb{E}[||X^{\gamma_t,u}_T||_0^{2m}|{\mathcal{F}}_t]\leq
	C_{2,m}(1+||\gamma_t||_0^{2m})
	\end{eqnarray}
	and
	\begin{eqnarray}\label{06041022}
	\sup_{u(\cdot)\in {\mathcal{U}}[0,T]}\mathbb{E}[||X^{\gamma_t,u}_T-X^{\eta_t,u}_T||^{2m}_0|{\mathcal{F}}_t]
	\leq C_{2,m}||\gamma_t-\eta_t||^{2m}_0.
	\end{eqnarray}
	
\end{lemma}

\begin{proof} Similar to the proof of Lemma \ref{theoremito2}, we have, for $m\geq2$ and every $(t,\gamma_t, \eta_t)\in [0,T)\times\Lambda_t\times\Lambda_t$ and $u(\cdot)\in {\mathcal{U}}[0,T]$,  $P$-a.s.,  for all $s\in [t,T]$:
\begin{eqnarray}\label{jias510815jia00l}
&\begin{aligned}
&\qquad\qquad \Upsilon^{m,3}(X^{\gamma_t,u}_s-X^{\eta_t,u}_s)
\leq \Upsilon^{m,3}(\gamma_t-\eta_t)\\
&\qquad+\int^{s}_{{t}}\biggl\{\langle \partial_x\Upsilon^{m,3}(X^{\gamma_t,u}_\sigma-X^{\eta_t,u}_\sigma), \,
 F(X^{\gamma_t,u}_\sigma,u(\sigma))-F(X^{\eta_t,u}_\sigma,u(\sigma))\rangle_H\\[8mm]
&\!\!\!\!\!\!\!\!\!\! +\frac{1}{2}\mbox{Tr}\Bigl(\partial_{xx}\Upsilon^{m,3}(X^{\gamma_t,u}_\sigma-X^{\eta_t,u}_\sigma)\left[G(\cdot,u(\sigma))\Bigm|^{X^{\gamma_t,u}_\sigma}_{X^{\eta_t,u}_\sigma}\right]\left[G^*(\cdot,u(\sigma))\Bigm|^{X^{\gamma_t,u}_\sigma}_{X^{\eta_t,u}_\sigma}\right]\Bigr)\biggr\}d\sigma\\
&
+\int^{s}_{t}\langle \partial_x\Upsilon^{m,3}(X^{\gamma_t,u}_\sigma-X^{\eta_t,u}_\sigma), \,  [G(X^{\gamma_t,u}_\sigma,u(\sigma))-G(X^{\eta_t,u}_\sigma,u(\sigma))]dW(\sigma)\rangle_H.
\end{aligned}
\end{eqnarray}
Taking  expectation conditioned  on ${\mathcal{F}}_t$ on both sides of the last inequality, we have from  Lemma~\ref{theoremS}
\begin{eqnarray*}
	 &&\mathbb{E}[||X^{\gamma_t,u}_s-X^{\eta_t,u}_s||^{2m}_0|{\mathcal{F}}_t]\leq\mathbb{E}[\Upsilon^{m,3}(X^{\gamma_t,u}_s-X^{\eta_t,u}_s)|{\mathcal{F}}_t]\nonumber\\
	&\leq& 3||\gamma_t-\eta_t||^{2m}_0+
	3mL(6mL-L+2)\int^{s}_{{t}}\mathbb{E}[||X^{\gamma_t,u}_\sigma-X^{\eta_t,u}_\sigma||^{2m}_0|{\mathcal{F}}_t]d\sigma.
\end{eqnarray*}
Then, using  Gronwall's inequality, we have (\ref{06041022}) for $m\geq2$.  For every $\varepsilon>0$, note that (\ref{jias510815jia00l}) also holds true for $\Upsilon^\varepsilon$, by the similar proof procedure above, we have
\begin{eqnarray*}
	 &&\mathbb{E}[||X^{\gamma_t,u}_s-X^{\eta_t,u}_s||^{2}_0|{\mathcal{F}}_t]\leq\mathbb{E}[\Upsilon^{\varepsilon}(X^{\gamma_t,u}_s-X^{\eta_t,u}_s)|{\mathcal{F}}_t]+\frac{\varepsilon}{2}\\
	&\leq& 3||\gamma_t-\eta_t||^{2}_0+\frac{\varepsilon}{2}
	+3L(5L+2)\int^{s}_{{t}}\mathbb{E}[||X^{\gamma_t,u}_\sigma-X^{\eta_t,u}_\sigma||^{2}_0|{\mathcal{F}}_t]d\sigma.
\end{eqnarray*}
Then, using  Gronwall's inequality and letting $\varepsilon\rightarrow0$, we have (\ref{06041022}) for $m=1$.
In a similar way, we have (\ref{2.61022}). The proof is complete.
\end{proof}
\newpage
\section{Dynamic programming of optimal control in path spaces}
\par
In this chapter, we consider optimal control problem  (\ref{state1}), (\ref{fbsde1}) and (\ref{value1}). We make the following assumption.

\begin{assumption}\label{hypcost} The two functionals
	$
	q: {\Lambda}\times \mathbb{R}\times \Xi\times U\rightarrow \mathbb{R}$ and $\phi: {\Lambda}_T\rightarrow \mathbb{R}$ are continuous, {and
        $q$ is continuous in $\xi_t\in \Lambda$, uniformly in $u\in U$. Moreover,}     there is a  constant $L>0$
	such that, for all $(t, \gamma_t, \eta_T, \gamma'_t, \eta'_T, y, z, y',z', u)
	\in [0,T]\times(\Lambda\times {\Lambda}_T)^2\times (\mathbb{R}\times \Xi)^2\times U$,
	\begin{eqnarray*}
		&&|q(\gamma_t,y,z,u)|\leq L(1+||\gamma_t||_0+|y|+|z|),
		\\
		&&|q(\gamma_t,y,z,u)-q(\gamma'_{t},y',z',u)|\leq L(||\gamma_t-\gamma'_{t}||_0+|y-y'|+|z-z'|),
		\\
		&&|\phi(\eta_T)-\phi(\eta'_T)|\leq L||\eta_T-\eta'_{T}||_0.
	\end{eqnarray*}
\end{assumption}
Combining Lemmas~\ref{lemma2.5} and \ref{lemmaexist111}, we have
\begin{lemma}\label{lemmaexist}
	Assume that Assumption \ref{hypstate} (i) and (ii)  and Assumption \ref{hypcost}  hold. Then for every
	$(t,\xi_t,u(\cdot))\in [0,T)\times L_{\mathcal{P}}^p(\Omega, \Lambda_t(H))\times {\mathcal{U}}[0,T]$ with $p>2$,  BSDE (\ref{fbsde1}) has unique
	pair of solutions $(Y^{\xi_t,u}, Z^{\xi_t,u})$.  Furthermore, let  $X^{\xi'_t,v}$ and $(Y^{\xi'_t,v}, Z^{\xi'_t,v})$ be the solutions of FSEE (\ref{state1}) and BSDE (\ref{fbsde1})
	corresponding to $(t,\xi'_t,v(\cdot))\in [0,T]\times L_{\mathcal{P}}^p(\Omega, \Lambda_t(H))\times {\mathcal{U}}[0,T]$. Then, we have  the following estimates
	\begin{eqnarray}\label{fbjia4}
	&\displaystyle \qquad\mathbb{E}\left[\sup_{t\leq s\leq T}\left|Y^{\xi_t,u}(s)-Y^{\xi'_t,v}(s)\right|^p\right]+ {{\mathbb{E}}\bigg{(}\int^{T}_{t}\left|Z^{\xi_t,u}(s)-Z^{\xi'_t,v}(s)\right|^2ds\bigg{)}^{\frac{p}{2}}}\nonumber\\[3mm]
	&\leq C_p\, \mathbb{E}||\xi_t-\xi'_t||_0^p\\[3mm]
	&\displaystyle\!\!\!+C_p\!\! \int^{T}_{t}\!\!\! \mathbb{E}\!\left[\left|(F, G)(X^{\xi'_{t},v}_r,\cdot)\Bigm|^{u(r)}_{v(r)}\right|^p_{H\times L_2(\Xi, H)}
	\!\!\!\!\!\!+\left|q(X^{\xi'_{t},v}_r,Y^{\xi'_t,v}(r),Z^{\xi'_t,v}(r),\cdot)\Bigm|^{u(r)}_{v(r)}\right|^p\right]\! dr \nonumber
	\end{eqnarray}
	and
	\begin{eqnarray}\label{fbjia5}
\qquad\qquad	\mathbb{E}\bigg{[}\sup_{t\leq s\leq T}\left|Y^{\xi_t,u}(s)\right|^p\bigg{]}
	+ {\mathbb{E}}\bigg{(}\int^{T}_{t}\left|Z^{\xi_t,u}(s)\right|^2ds\bigg{)}^{\frac{p}{2}}\leq C_p(1+\mathbb{E}||\xi_t||_0^p)
	\end{eqnarray}
	for a positive constant $C_p$ depending only on  $(p, T, L, M_1)$.
\end{lemma}

\begin{proof} We have the
existence and uniqueness of the solution of the backward equation (\ref{fbsde1}) from  Lemma ~\ref{lemma2.5}.
Using  inequalities (\ref{lemma2.51}) and (\ref{0604}),  we  get inequality  (\ref{fbjia4}).  Combining inequalities (\ref{lemma2.510}) and (\ref{state1est}), we obtain inequality (\ref{fbjia5}).
\end{proof}

By Lemma \ref{lemmaexist10220}, we also have the following lemma which is used to prove dynamic programming principle.
\begin{lemma}\label{lemmaexist1022}
	Assume that Assumptions \ref{hypstate}   and \ref{hypcost}  hold. Then for every
	$(t,\gamma_t, \eta_t,$ $u(\cdot))\in [0,T]\times \Lambda_t\times \Lambda_t\times {\mathcal{U}}[0,T]$,  we have
	\begin{eqnarray}\label{fbjia41022}
	\begin{aligned}
	&\quad
	\mathbb{E}[|Y^{\gamma_t,u}(s)-Y^{\eta_t,u}(s)|^2|{\mathcal{F}}_t]\\
	 &\quad+\int^{T}_{s}\!\!\!\mathbb{E}[|Y^{\gamma_t,u}(\sigma)-Y^{\eta_t,u}(\sigma)|^2+|Z^{\gamma_t,u}(\sigma)-Z^{\eta_t,u}(\sigma)|^2|{\mathcal{F}}_t]d\sigma\\
	&\leq C_{3,1}||\gamma_t-\eta_t||_0^2, \ {s\in[t,T]}
	\end{aligned}
	\end{eqnarray}
	and
	\begin{eqnarray}\label{fbjia51022}
	\quad& \mathbb{E}[|Y^{\gamma_t,u}(s)|^2|{\mathcal{F}}_t]
	+\int^{T}_{t}\!\mathbb{E}[|Z^{\gamma_t,u}(s)|^2|{\mathcal{F}}_t]ds\leq C_{3,1}(1+||\gamma_t||_0^2), \ {s\in[t,T]}
	\end{eqnarray}
	for a positive constant $C_{3,1}$ depending only on   $T$, $L$  and $M_1$.
\end{lemma}

\begin{proof}
We only prove (\ref{fbjia41022}), and (\ref{fbjia51022}) is proved in a similar way.  For simplicity, we write
\begin{eqnarray*}
	\begin{aligned}
		\hat{X}_s:=&X^{\gamma_t,u}_s-X^{\eta_t,u}_s,\ \hat{Y}(s):=Y^{\gamma_t,u}(s)-Y^{\eta_t,u}(s),\\
		\hat{Z}(s):=&Z^{\gamma_t,u}(s)-Z^{\eta_t,u}(s),\\
		\hat{q}(s):=&\, q(\Xi,u(s))\Bigm|^{\Xi=(X^{\gamma_t,u}_s,Y^{\gamma_t,u}(s),Z^{\gamma_t,u}(s))}_{\Xi=(X^{\eta_t,u}_s,Y^{\eta_t,u}(s),Z^{\eta_t,u}(s))}\, ,\ \quad s\in [t,T].
	\end{aligned}
\end{eqnarray*}
Applying  It\^o formula  to $e^{\beta s}|\hat{Y}(s)|^2$ on $[t,T]$, we obtain
\begin{eqnarray}\label{jias510815jia1022}
\begin{aligned}
&\quad e^{\beta s}|\hat{Y}(s)|^2+\int^{T}_{s}e^{\beta \sigma}[\beta|\hat{Y}(\sigma)|^2+|\hat{Z}(\sigma)|^2]d\sigma\\
&= e^{\beta T}|\phi(X_T^{\gamma_t,u})-\phi(X_T^{\eta_t,u})|^2+2\int^{T}_{{s}}e^{\beta \sigma}\langle\hat{Y}(\sigma),
\hat{q}(\sigma)\rangle_Hd\sigma\\
&\quad-2\int^{s}_{t}e^{\beta \sigma}\langle\hat{Y}(\sigma), \hat{Z}(\sigma)dW(\sigma)\rangle_H.
\end{aligned}
\end{eqnarray}
Taking conditional expectation with respect to ${\mathcal{F}}_t$, 
we obtain,  {for $\beta>0$},
\begin{eqnarray*}
	\begin{aligned}
		&e^{\beta s}\mathbb{E}[|\hat{Y}(s)|^2|{\mathcal{F}}_t]+\int^{T}_{s}e^{\beta \sigma}\mathbb{E}[\beta|\hat{Y}(\sigma)|^2+|\hat{Z}(\sigma)|^2|{\mathcal{F}}_t]d\sigma\\
		\leq& e^{\beta T}L^2\mathbb{E}[||\hat{X}_T||_0^2|{\mathcal{F}}_t]+\int^{T}_{{s}}e^{\beta \sigma}\mathbb{E}[(2L^2+3L)|\hat{Y}(\sigma)|^2+L||\hat{X}_\sigma||_0^2+\frac{1}{2}|\hat{Z}(\sigma)|^2
		|{\mathcal{F}}_t]d\sigma.
	\end{aligned}
\end{eqnarray*}
Let $\beta=2L^2+3L+1$, by (\ref{06041022}), we obtain (\ref{fbjia41022}).
\end{proof}

The functionals $J(\xi_t,u(\cdot)):=J(\gamma_t,u(\cdot))|_{\gamma_t=\xi_t}$ and $Y^{\xi_t,u}(t)$, $(t,\xi_t)\in [0,T]\times L_{\mathcal{P}}^p(\Omega, \Lambda_t(H))$, are related by the following
theorem.
\begin{theorem}\label{theoremj=y}
	\ \
	Under  Assumption \ref{hypstate} (i) and (ii)  and Assumption \ref{hypcost},    for every $p>2$, $(t, u(\cdot))\in [0,T]\times {\mathcal{U}}[t,T]$ and $\xi_t, \xi'_t\in L_{\mathcal{P}}^p(\Omega, \Lambda_t(H))$, we have
	\begin{eqnarray}\label{j=y}
	J(\xi_t,u(\cdot))=Y^{\xi_t,u}(t)
	\end{eqnarray}
	and
	\begin{eqnarray}\label{0903jia1022}
	|Y^{\xi_t,u}(t)-Y^{\xi'_t,u}(t)|\leq C_{3,1}^{\frac{1}{2}}||\xi_t-\xi'_t||_0.
	\end{eqnarray}
\end{theorem}

\begin{proof}
Let $\{h^n_t\}$, $n\in \mathbb{N}$,
be a dense subset of $\Lambda_t$, $B(h^n_t,\frac{1}{k})$
be the open sphere in $\Lambda_t$ with the radius $\frac{1}{k}$ and the center $h^n_t$.  Set $B_{n,k}:=B(h^n_t,\frac{1}{k})\setminus
\bigcup_{m<n}B(h^m_t,\frac{1}{k})$  and $A_{n,k}:=\{\omega\in \Omega|\xi_t(\omega)\in
B_{n,k}\}$. Then $\cup_{n=1}^{\infty}A_{n,k}=\Omega$ and the
sequence $f^k_t(\omega):= \Sigma^{\infty}_{n=1}h^n_t1_{
	A_{n,k}}(\omega)$ is ${\mathcal
	{F}}_{t}$-measurable and
converges to $\xi_t$ strongly and
uniformly.
\par
For every  $n$ and $u(\cdot)\in {\mathcal {U}}[t,T]$, we put $(X^{n}(s),Y^{n}(s),Z^{n}(s))=(X^{h^n_t,u}(s),$ $Y^{h^n_t,u}(s),Z^{h^n_t,u}(s))$.
Then $X^{n}(s)$ is the solution of the FSEE
\begin{eqnarray}\label{stateint333}
\begin{aligned}
X^n(s)=&e^{(s-t)A}h^n_t(t)+\int_{t}^{s}e^{(s-\sigma)A}F(X^n_\sigma,u(\sigma))d\sigma\\
&+\int_{t}^{s}e^{(s-\sigma)A}G(X^n_\sigma,u(\sigma))dW(\sigma), \ \
s\in [t,T],
\end{aligned}
\end{eqnarray}
where $X^n_t=h^n_t$; and $(Y^{n}(s),Z^{n}(s))$ is the solution of the associated  BSDE
\begin{eqnarray}\label{fbsde333}
Y^{n}(s)=\phi(X_T^{n})+\int^{T}_{s}q(X_\sigma^{n},Y^{n}(\sigma),Z^{n}(\sigma),u(\sigma))d\sigma-\int^{T}_{s}Z^{n}(\sigma)dW(\sigma),\  \ s\in [t,T].\nonumber
\end{eqnarray}
The above two equations are multiplied by $1_{A_{n,k}}$  and summed up with respect to $n$. Thus, taking into account that $\sum_{n=1}^{\infty}\varphi(h^n_t)1_{A_{n,k}}=\varphi(\sum_{n=1}^{\infty}h^n_t1_{A_{n,k}})$, we obtain
\begin{eqnarray*}
	&&\sum_{n=1}^{\infty}1_{A_{n,k}}X^n(s)\\
	&=&\sum_{n=1}^{\infty}1_{A_{n,k}}e^{(s-t)A}h^n_t(t)+\int_{t}^{s}e^{(s-\sigma)A}F\left(\sum_{n=1}^{\infty}1_{A_{n,k}}X^n_\sigma, u(\sigma)\right)d\sigma\nonumber\\
	&& \ \
	+\int_{t}^{s}e^{(s-\sigma)A}G\left(\sum_{n=1}^{\infty}1_{A_{n,k}}X^n_\sigma,u(\sigma)\right)dW(\sigma), \nonumber
\end{eqnarray*}
and
\begin{eqnarray}\label{fbsde333}
\begin{aligned}
&\quad\sum_{n=1}^{\infty}1_{A_{n,k}}Y^{n}(s)\\
&=\phi\left(\sum_{n=1}^{\infty}1_{A_{n,k}}X_T^{n}\right)-\int^{T}_{s}\sum_{n=1}^{\infty}1_{A_{n,k}}Z^{n}(\sigma)dW(\sigma)\\
&\quad+\int^{T}_{s}q\left(\sum_{n=1}^{\infty}1_{A_{n,k}}X_\sigma^{n},
\sum_{n=1}^{\infty}1_{A_{n,k}}Y^{n}(\sigma),\sum_{n=1}^{\infty}1_{A_{n,k}}Z^{n}(\sigma),u(\sigma)\right)d\sigma.
\end{aligned}
\end{eqnarray}
Then,  the strong uniqueness property of the solution to the FSEE and the BSDE yields
\begin{eqnarray}\label{stateint333444}
X^{f^k_t,u}(s)&=&\sum_{n=1}^{\infty}1_{A_{n,k}}X^n(s),\nonumber\\
 \ \ \ \ \ \ \ \left(Y^{f^k_t,u}(s),Z^{f^k_t,u}(s)\right)
&=&\left(\sum_{n=1}^{\infty}1_{A_{n,k}}Y^{n}(s),\sum_{n=1}^{\infty}1_{A_{n,k}}Z^{n}(s)\right),\ \ s\in [t,T].
\end{eqnarray}
Finally, from $J(h^n_t,u(\cdot))=Y^n(t)$ for $n\geq 1$, we deduce that
\begin{eqnarray}\label{stateint33344455}
\begin{aligned}
Y^{f^k_t,u}(t)
&=\sum_{n=1}^{\infty}1_{A_{n,k}}Y^{n}(t)=\sum_{n=1}^{\infty}1_{A_{n,k}}J(h^n_t,u(\cdot))\\
&=J\left(\sum_{n=1}^{\infty}1_{A_{n,k}}h^n_t,u(\cdot)\right)=J(f^k_t,u(\cdot)).
\end{aligned}
\end{eqnarray}
Consequently, from the estimates (\ref{fbjia4}) and (\ref{fbjia41022}),  we get
\begin{eqnarray*}
	&&\mathbb{E}|Y^{\xi_t,u}(t)-J(\xi_t,u(\cdot))|^p\\
	&\leq&  2^{p-1}\mathbb{E}|Y^{\xi_t,u}(t)-Y^{f^{k}_t,u}(t)|^p+2^{p-1} \mathbb{E}|J(f^k_t,u(\cdot))-J(\xi_t,u(\cdot))|^p\\
	&\leq& 2^{p-1} (C_p+C_{3,1}^{\frac{p}{2}})\mathbb{E}||\xi_t-f^k_t||_0^p\rightarrow0\ \mbox{as}\ k\rightarrow \infty.
\end{eqnarray*}
Now let us prove $(\ref{0903jia1022})$. By (\ref{fbjia41022}) and (\ref{j=y}),
\begin{eqnarray*}\label{0903jia}
	&&|Y^{\xi_t,u}(t)-Y^{\xi'_t,u}(t)|\\
	&=&|J(\xi_t,u(\cdot))-J(\xi'_t,u(\cdot))|=\big{|}Y^{\gamma_t,u}(t)|_{\gamma_t=\xi_t}-Y^{\eta_t,u}(t)|_{\eta_t=\xi'_t}\big{|}\\
	&=&\big{|}(Y^{\gamma_t,u}(t)-Y^{\eta_t,u}(t))|_{\gamma_t=\xi_t,\eta_t=\xi'_t}\big{|}\leq C_{3,1}^{\frac{1}{2}}||\xi_t-\xi'_t||_0.
\end{eqnarray*}
The proof is complete.
\end{proof}

From the uniqueness of the solution of (\ref{fbsde1}), it
follows that
$$
Y^{\gamma_t,u}(t+\delta)=Y^{X^{\gamma_t,u}_{t+\delta},u}(t+\delta)=J\left(X^{\gamma_t,u}_{t+\delta},u(\cdot)\right),\ \ P\mbox{-a.s.}
$$
Formally,  under the assumptions  Assumption \ref{hypstate} (i), (ii)  and Assumption \ref{hypcost}, the  value functional $V(\gamma_t)$
defined by (\ref{value1})
is  ${\mathcal{F}}_t$-measurable.
However, we have
\begin{theorem}\label{valuedet}  Let  Assumption \ref{hypstate} (i) and (ii) and Assumption \ref{hypcost} hold true.
	Then,  $V$ is a deterministic functional.
\end{theorem}

\begin{proof}
First, we show that there exists a sequence $\{u_i(\cdot)\}_{i\geq1}\subset {\mathcal{U}}[t,T]$ such that $\{Y^{\gamma_t,u_i}(t)\}_{n\geq1}$ is a nondecreasing sequence and
\begin{eqnarray}\label{jiale}
V(\gamma_t)=\mathop{\esssup}\limits_{u(\cdot)\in {\mathcal{U}}[t,T]}Y^{\gamma_t,u}(t)=\lim_{i\rightarrow\infty} Y^{\gamma_t,u_i}(t), \ \  P\mbox{-a.s.}.
\end{eqnarray}
From \cite[Theorem A.3]{kar}, it is sufficient to prove that, for any $u_1(\cdot),u_2(\cdot)\in {\mathcal{U}}[t,T]$, we have
\begin{eqnarray}\label{jiale1}
Y^{\gamma_t,u_1}(t)\vee Y^{\gamma_t,u_2}(t)=Y^{\gamma_t,u}(t),\ \ P\mbox{-a.e.,         }
\end{eqnarray}
with $u(\cdot)\in {\mathcal{U}}[t,T]$ satisfying $u(s):=u_1(s){\mathbf{1}}_{A_1}+u_2(s){\mathbf{1}}_{A_2}$, $s\in [t,T]$, where $A_1:=\{Y^{\gamma_t,u_1}(t)\geq Y^{\gamma_t,u_2}(t)\}$ and
$A_2:=\{Y^{\gamma_t,u_1}(t)<Y^{\gamma_t,u_2}(t)\}$.
Taking into account that $\sum_{i=1}^{2}\varphi(x_i)1_{A_i}=\varphi(\sum_{i=1}^{2}x_i1_{A_{i}})$, we obtain
\begin{eqnarray*}
	 \sum_{i=1}^{2}1_{A_{i}}X^{\gamma_t,u_i}(s)&=&\gamma_t(t)+\int_{t}^{s}e^{(s-\sigma)A}F\left(\sum_{i=1}^{2}1_{A_{i}}X^{\gamma_t,u_i}_\sigma,\sum_{i=1}^{2}1_{A_{i}}u_i(\sigma)\right)d\sigma\\
	&&+\int_{t}^{s}e^{(s-\sigma)A}G\left(\sum_{i=1}^{2}1_{A_{i}}X^{\gamma_t,u_i}_\sigma,\sum_{i=1}^{2}1_{A_{i}}u_i(\sigma)\right)dW(\sigma), \nonumber
\end{eqnarray*}
and
\begin{eqnarray*}
	&&\sum_{i=1}^{2}1_{A_{i}}Y^{\gamma_t,u_i}(s)\\
	&=&\phi\left(\sum_{i=1}^{2}1_{A_{i}}X_T^{\gamma_t,u_i}\right)-\int^{T}_{s}\sum_{i=1}^{2}1_{A_i}Z^{\gamma_t,u_i}(\sigma)dW(\sigma)\nonumber\\
	&&+\int^{T}_{s}q\left(\sum_{i=1}^{2}1_{A_i}X_\sigma^{\gamma_t,u_i},
	\sum_{i=1}^{2}1_{A_{i}}Y^{\gamma_t,u_i}(\sigma),\sum_{i=1}^{2}1_{A_i}Z^{\gamma_t,u_i}(\sigma),\sum_{i=1}^{2}1_{A_{i}}u_i(\sigma)\right)d\sigma.\nonumber
\end{eqnarray*}
Then the strong uniqueness property of the solution to the FSEE and the BSDE yields
\begin{eqnarray}\label{10081118}
Y^{\gamma_t,u}(t)=\sum_{i=1}^{2}1_{A_i}Y^{\gamma_t,u_i}(t)=Y^{\gamma_t,u_1}(t)\vee Y^{\gamma_t,u_2}(t),\ \ P\mbox{-a.e.}.
\end{eqnarray}

Suppose that $\{u_i(\cdot)\}_{i\geq1}\subset {\mathcal{U}}[t,T]$ satisfy (\ref{jiale}). By (\ref{fbjia4}) {and Lemma 4.12 in \cite{yong11}}, we can suppose without lost of generality that $u_i(\cdot)$ takes the following form:
$$
u_i(s)=\sum^{n}_{j=1}\mathbf{1}_{A_{ij}}u_{ij}(s),\ \ s\in [t,T].
$$
Here $u_{ij}(\cdot)$ is ${\mathcal{F}}^t$-progressively measurable and $\{A_{ij}\}_{j=1}^{n}$ is a partition of $(\Omega,{\mathcal{F}}_t)$. {Like (\ref{10081118})}, we show that
$$
J(\gamma_t, u_i(\cdot))=\sum^{n}_{j=1}\mathbf{1}_{A_{ij}}J(\gamma_t, u_{ij}(\cdot)).
$$
It is clear that $X^{\gamma_t,u_{ij}}(s)$ is ${\mathcal{F}}_t^s$-measurable for all $s\in[t,T]$, then $Y^{\gamma_t,u_{ij}}(s)$ is ${\mathcal{F}}_t^s$-measurable for all $s\in[t,T]$.
In particular, $J(\gamma_t, u_{ij}(\cdot))=Y^{\gamma_t,u_{ij}}(t)$ is deterministic. Without lost of generality, we may assume
$$
J(\gamma_t, u_{ij}(\cdot))\leq J(\gamma_t, u_{i1}(\cdot)), \ \ j\geq 1.
$$
Then we have $J(\gamma_t, u_i(\cdot))\leq J(\gamma_t, u_{i1}(\cdot))$. Combining (\ref{jiale}), we get
$$
\lim_{i\rightarrow\infty}J(\gamma_t, u_{i1}(\cdot))=V(\gamma_t).
$$
Therefore, $V(\gamma_t)$ is deterministic.
The proof is complete.
\end{proof}

The following  property of the value functional $V$ which we present is an immediate consequence of {Lemma \ref{lemmaexist1022}} and Theorem \ref{valuedet}.
\begin{lemma}\label{lemmavaluev}
	Let  Assumption \ref{hypstate} (i), (ii)  and Assumption \ref{hypcost} be satisfied. Then,  
	for all 
	$0\leq t\leq T$, $\gamma_t, \eta_t\in \Lambda_t$,
	\begin{eqnarray}\label{valuelip}
	|V(\gamma_t)-V(\eta_t)|\leq {C_{3,1}^{\frac{1}{2}}}||\gamma_t-\eta_t||_0, \quad
	|V(\gamma_t)|\leq {C_{3,1}^{\frac{1}{2}}}(1+||\gamma_t||_0).
	\end{eqnarray}
\end{lemma}

\par
We also have
\begin{lemma}\label{lemma3.6}
	For all $t\in[0,T]$, {$\xi_t\in L_{\mathcal{P}}^p(\Omega, \Lambda_t(H))$} for some $p>2$, and $u(\cdot)\in{\mathcal
		{U}}[t,T]$ we have
	\begin{eqnarray}\label{3.14}
	V(\xi_t)\geq Y^{\xi_t,u}(t),\
	\ P\mbox{-a.s.},
	\end{eqnarray}
	and for any $\varepsilon>0$ there exists an
	admissible control $u(\cdot)\in{\mathcal
		{U}}[t,T]$ such that
	\begin{eqnarray}\label{3.15}
	V(\xi_t)\leq Y^{\xi_t,u}(t)+\varepsilon\
	\ P\mbox{-a.s.}.
	\end{eqnarray}
\end{lemma}

\begin{proof}  From Theorem \ref{theoremj=y} and  the definition of $V(\gamma_t)$,  we have, for any $u(\cdot)\in{\mathcal
	{U}}[t,T]$,
$$
V(\xi_t)=V(\eta_t)|_{\eta_t=\xi_t}=\left[\mathop{\esssup}\limits_{v(\cdot)\in{\mathcal
		{U}}[t,T]}J(\eta_t,v(\cdot))\right]\bigg{|}_{\eta_t=\xi_t}\geq
J(\eta_t,u(\cdot))|_{\eta_t=\xi_t}=Y^{\xi_t,u}(t).
$$
We now prove (\ref{3.15}). Let $\{h^n_t\}$, $n\in N$,
be a dense subset of $\Lambda_t$, $B(h^n_t,\frac{1}{k})$
be the open sphere in $\Lambda_t$ with radius
$\frac{1}{k}$ and center $h^n_t$.  Set $B_{n,k}:=B(h^n_t,\frac{1}{k})\setminus
\bigcup_{m<n}B(h_m,\frac{1}{k})$. Then the
sequence $f^k_t(\omega):= \Sigma^{\infty}_{n=1}h^n_t1_{\{\xi_t\in
	B_{n,k}\}}(\omega), k\ge 1,$ is ${\mathcal
	{F}}_{t}$-measurable and
strongly converges to $\xi_t$,
uniformly. For {$k\varepsilon>3({1+C_{3,1}^{\frac{1}{2}}})$},  we have
$$
||f^k_t-\xi_t||_0\leq \frac{\varepsilon}{3(1+C_{3,1}^{\frac{1}{2}})}\leq \frac{\varepsilon}{3}.
$$
Therefore, for every $u(\cdot)\in {\mathcal {U}}[t,T]$, by (\ref{0903jia1022}) and (\ref{valuelip}),
$$
\left|Y^{f^k_t,u}(t)-Y^{\xi_t,u}(t)\right|\leq
\frac{\varepsilon}{3},
\ \
\left|V(f^k_t)-V(\xi_t)\right|\leq
\frac{\varepsilon}{3},\ P\mbox{-a.s.}.
$$
Moreover, for every $h^n_t\in \Lambda_t$,  as in the proof  of Theorem \ref{valuedet},
we  choose an admissible control
$v^n(\cdot)$ such that
$$
V(h^n_t)\leq
Y^{h^n_t, v^n}(t)+\frac{\varepsilon}{3}, \quad  P\mbox{-a.s.}.
$$
Set
$u(\cdot):=\sum^{\infty}_{n=1}v^n(\cdot)1_{\{\xi_t\in
	B_{n,k}\}}$.  Then,
\begin{eqnarray*}
	Y^{\xi_t,u}(t)&\geq&
	-|Y^{f^k_t,u}(t)-Y^{\xi_t,u}(t)|+Y^{f^k_t,u}(t)\geq-\frac{\varepsilon}{3}+\sum^{\infty}_{n=1}Y^{h^n_t,v^{n}}(t)1_{\{\xi_t\in
		B_{n,k}\}}\\
	&\geq&-\frac{2\varepsilon}{3}+\sum^{\infty}_{n=1}V(h^n_t)1_{\{\xi_t\in
		B_{n,k}\}}=-\frac{2\varepsilon}{3}+V(f^k_t)
	\geq-\varepsilon+V(\xi_t),  \ \  P\mbox{-a.s.}.
\end{eqnarray*}
Thus, we have (\ref{3.15}).
\end{proof}

\par
We now discuss the dynamic programming principle  (DPP) for the optimal control problem (\ref{state1}), (\ref{fbsde1}) and (\ref{value1}).
For this purpose, we define the family of backward semigroups associated with BSDE (\ref{fbsde1}), following the
idea of Peng \cite{peng11}.
\par
Given the initial path $(t,\gamma_t)\in [0,T)\times{\Lambda}$, a positive number $\delta\leq T-t$, an admissible control $u(\cdot)\in {\mathcal{U}}[t,t+\delta]$ and
a real-valued random variable $\zeta\in L^2(\Omega,{\mathcal{F}}_{t+\delta},P;\mathbb{R})$, we define
\begin{eqnarray}\label{gdpp}
G^{\gamma_t,u}_{s,t+\delta}[\zeta]:=\tilde{Y}^{\gamma_t,u}(s),\ \
\ \ \ \ s\in[t,t+\delta],
\end{eqnarray}
where $(\tilde{Y}^{\gamma_t,u}(s),\tilde{Z}^{\gamma_t,u}(s))_{t\leq s\leq
	t+\delta}$ is the solution of the following
BSDE:
\begin{eqnarray}\label{bsdegdpp}
\begin{cases}
d\tilde{Y}^{\gamma_t,u}(s) =-q(X^{\gamma_t,u}_s,\tilde{Y}^{\gamma_t,u}(s),\tilde{Z}^{\gamma_t,u}(s),u(s))ds\\
\ \ \ \ \ \ \ \ \ \  \ \ \ \ \ \ \ \ \ \ +\tilde{Z}^{\gamma_t,u}(s)dW(s), \quad s\in [t, t+\delta]; \\
\ \tilde{Y}^{\gamma_t,u}(t+\delta)=\zeta
\end{cases}
\end{eqnarray}
with $X^{\gamma_t,u}(\cdot)$ being the solution of FSEE (\ref{state1}).
%
%
%

\begin{theorem}\label{theoremddp} 
	Let  Assumption \ref{hypstate} (i), (ii) and Assumption \ref{hypcost} be satisfied. Then, the value functional
	$V$ satisfies the following: for
	any $\gamma_t\in {\Lambda_t}$ and $0\leq t<t+\delta\leq T$,
	\begin{eqnarray}\label{ddpG}
	V(\gamma_t)=\mathop{\esssup}\limits_{u(\cdot)\in{\mathcal
			{U}}[t,t+\delta]}G^{\gamma_t,u}_{t,t+\delta}[V(X^{\gamma_t,u}_{t+\delta})].
	\end{eqnarray}
\end{theorem}
\par

\begin{proof}
By the definition of $V(\gamma_t)$ we have
$$
V(\gamma_t)=\mathop{\esssup}\limits_{u(\cdot)\in{\mathcal
		{U}}[t,T]}G^{\gamma_t,u}_{t,T}[\phi(X^{\gamma_t,u}_T)]=\mathop{\esssup}\limits_{u(\cdot)\in{\mathcal
		{U}}[t,T]}G^{\gamma_t,u}_{t,t+\delta}\left[Y^{X^{\gamma_t,u}_{t+\delta},u}(t+\delta)\right].
$$
From (\ref{3.14}) and the comparison theorem (see Lemma \ref{lemma2.70904}), we have
$$
V(\gamma_t)\leq \mathop{\esssup}\limits_{u(\cdot)\in{\mathcal
		{U}}[t,T]}G^{\gamma_t,u}_{t,t+\delta}[V(X^{\gamma_t,u}_{t+\delta})].
$$
On the other hand, from (\ref{3.15}), for any
$\varepsilon>0$ and $u(\cdot)\in {\mathcal {U}}[t,T]$,  there
is an admissible control $\overline{u}(\cdot)\in{\mathcal
	{U}}[t+\delta,T]$ such that
$$
V(X^{\gamma_t,u}_{t+\delta})\leq Y^{X^{
		\gamma_t,u}_{t+\delta},\overline{u}}(t+\delta)+\varepsilon, \ \  P\mbox{-a.s.}
$$
For any $u(\cdot)\in {\mathcal {U}}[t,T]$ with $\overline{u}(\cdot)\in{\mathcal {U}}[t+\delta,T]$ from above, define
$$
v(s)= \begin{cases}u(s),\ \ \ \ \ \ s\in[t,t+\delta]; \\
\overline{u}(s),\ \ \ \ \ \
s\in(t+\delta,T];\end{cases}\in {\mathcal
	{U}}[t,T].
$$
Then, in a similar way to the proof of Lemma \ref{lemmaexist1022},  we have
\begin{eqnarray*}
	&&\left|G^{\gamma_t,u}_{t,t+\delta}\left[Y^{X^{\gamma_t,u}_{t+\delta},\overline{u}}(t+\delta)\right]
	-G^{\gamma_t,u}_{t,t+\delta}[V(X^{\gamma_t,u}_{t+\delta})]\right|\leq C\varepsilon,
\end{eqnarray*}
where the constant $C$ is independent of
admissible control processes.
Therefore,
\begin{eqnarray*}
	V(\gamma_t)&\geq&
	 G^{\gamma_t,v}_{t,t+\delta}\left[Y^{X^{\gamma_t,v}_{t+\delta},v}(t+\delta)\right]=G^{\gamma_t,u}_{t,t+\delta}\left[Y^{X^{\gamma_t,u}_{t+\delta},\overline{u}}(t+\delta)\right]\\
	&\geq&G^{\gamma_t,u}_{t,t+\delta}[V(X^{\gamma_t,u}_{t+\delta})]-C\varepsilon.
\end{eqnarray*}%
For the arbitrariness of $\varepsilon$, we have (\ref{ddpG}).
\end{proof}

From  Theorem \ref{theoremddp}, we  have

\begin{theorem}\label{theorem3.9}
	Under  Assumption \ref{hypstate} (i) and (ii) and Assumption \ref{hypcost}, the value functional $V\in C^0(\Lambda)$ and
	there is a constant $C>0$ such that for every  $0\leq t\leq s\leq T, \gamma_t,\eta_t\in{\Lambda_t}$,
	\begin{eqnarray}\label{hold}
	\ \ \ \ \      |V(\gamma_t)-V(\eta_{t,s}^A)|\leq
	C(1+||\gamma_t||_0+||\eta_t||_0)(s-t)^{\frac{1}{2}}+C||\gamma_t-\eta_t||_0.
	\end{eqnarray}
\end{theorem}

\begin{proof}
Let $(t,\gamma_t,\eta_t)\in[0,T)\times {\Lambda}\times {\Lambda}$ and $s\in[t,T]$.
From Theorem \ref{theoremddp} it follows that for any
$\varepsilon>0$ there exists an admissible control
$u^\varepsilon(\cdot)\in{\mathcal {U}}[t,s]$ such that
\begin{eqnarray}\label{3.19}
G^{\gamma_t,u^\varepsilon}_{t,s}[V(X^{\gamma_t,u^\varepsilon}_{s})]
+\varepsilon\geq V(\gamma_t)\geq G^{\gamma_t,u^\varepsilon}_{t,s}[V(X^{\gamma_t,u^\varepsilon}_{s})].
\end{eqnarray}
Therefore,
\begin{eqnarray}\label{3.20}
|V(\gamma_t)-V(\eta_{t,s}^A)|\leq
|I_1|+|I_2|+\varepsilon,
\end{eqnarray}
where
$$
I_1=\mathbb{E}G^{\gamma_t,u^\varepsilon}_{t,s}[V(X^{\gamma_t,u^\varepsilon}_{s})]-
\mathbb{E}G^{\gamma_t,u^\varepsilon}_{t,s}[V(\eta_{t,s}^A)], \quad
I_2=\mathbb{E}G^{\gamma_t,u^\varepsilon}_{t,s}[V(\eta_{t,s}^A)]-
V(\eta_{t,s}^A).
$$
We have from Lemmas \ref{lemma2.5}, \ref{lemmaexist111} and \ref{lemmavaluev}  that, for some suitable constant
$C$ that is independent of the control
$u^\varepsilon(\cdot)$ and may change from line to line,
\begin{eqnarray*}
	|I_1|&\leq& C\mathbb{E}\left|V\left(X^{\gamma_t,u^\varepsilon}_{s}\right)-V(\eta_{t,s}^A)\right|\\
	&\leq&C\mathbb{E}\left|\left|X^{\gamma_t,u^\varepsilon}_{s}-\eta_{t,s}^A\right|\right|_0\leq C(1+||\gamma_t||_0)(s-t)^{\frac{1}{2}}+C||\gamma_t-\eta_t||_0.
\end{eqnarray*}
From the definition of
$G^{\gamma_t,u^\varepsilon}_{t,s}[\cdot]$, the
second term $I_2$ can be written as
\begin{eqnarray*}
	I_2&=& \mathbb{E}\bigg{[}V(\eta_{t,s}^A)+\int^{s}_{t}
	q\left(X_\sigma^{\gamma_t,u^{\varepsilon}},Y^{\gamma_t,u^{\varepsilon}}(\sigma),Z^{\gamma_t,u^{\varepsilon}}(\sigma),u^{\varepsilon}(\sigma)\right)d\sigma\\
	&&-
	\int^{s}_{t}Z^{\gamma_t,u^{\varepsilon}}(\sigma)dW(\sigma)\bigg{]}-V(\eta_{t,s}^A)\\
	&=&\mathbb{E}\int^{s}_{t}
	q\left(X_\sigma^{\gamma_t,u^{\varepsilon}},Y^{\gamma_t,u^{\varepsilon}}(\sigma),Z^{\gamma_t,u^{\varepsilon}}(\sigma),u^{\varepsilon}(\sigma)\right)d\sigma,
\end{eqnarray*}
where $(Y^{\gamma_t,u^{\varepsilon}}(\sigma),Z^{\gamma_t,u^{\varepsilon}}(\sigma))_{t\leq \sigma\leq
	s}$ is the solution of
(\ref{bsdegdpp}) with the terminal
condition $\eta = V(\eta_{t,s}^A)$
and the control $u^\varepsilon(\cdot)$.
Using Schwartz inequality, we have
\begin{eqnarray*}
	|I_2|&\leq& (s-t)^{\frac{1}{2}}\bigg{[}\int^{s}_{t}
	 \mathbb{E}\left|q\left(X_\sigma^{\gamma_t,u^{\varepsilon}},Y^{\gamma_t,u^{\varepsilon}}(\sigma),Z^{\gamma_t,u^{\varepsilon}}(\sigma),u^{\varepsilon}(\sigma)\right)\right|^2d\sigma
	\bigg{]}^{\frac{1}{2}}\\
	&\leq&C(s-t)^{\frac{1}{2}}\bigg{[}\int^{s}_{t}
	 \mathbb{E}\left(1+\left|\left|X_\sigma^{\gamma_t,u^{\varepsilon}}\right|\right|_0^{2}+\left|Y^{\gamma_t,u^{\varepsilon}}(\sigma)\right|^2+\left|Z^{\gamma_t,u^{\varepsilon}}(\sigma)\right|^2\right)d\sigma
	\bigg{]}^{\frac{1}{2}}\\
	&\leq&C(1+||\gamma_t||_0+||\eta_t||_0)(s-t)^{\frac{1}{2}}.
\end{eqnarray*}
Hence, from (\ref{3.20}),
$$
|V(\gamma_t)-V(\eta_{t,s}^A)|\leq
C(1+||\gamma_t||_0+||\eta_t||_0)(s-t)^{\frac{1}{2}}+C||\gamma_t-\eta_t||_0
+\varepsilon,
$$
and letting $\varepsilon\downarrow
0$,  we have   (\ref{hold}).
\par
Finally, let $(t,\gamma_t), (s,\gamma'_s)\in[0,T]\times {\Lambda}$ and $s\in[t,T]$, by (\ref{valuelip}) and (\ref{hold})
\begin{eqnarray*}
	&&|V(\gamma_t)-V(\gamma'_s)|\leq  |V(\gamma_t)-V(\gamma_{t,s}^A)|+ |V(\gamma_{t,s}^A)-V(\gamma'_s)|\\
	&\leq&  C(1+||\gamma_t||_0)(s-t)^{\frac{1}{2}}+Cd_\infty(\gamma_{t},\gamma'_s).
\end{eqnarray*}
Thus, $V\in C^0(\Lambda)$.
The proof is complete.
\end{proof}

\newpage
\section{Viscosity solutions to  PHJB equations: existence}

\par
In this chapter, we consider the  second  order path-dependent  Hamilton-Jacobi-Bellman
(PHJB) equation (\ref{hjb1}). As usual, we begin with classical solutions.
\par
\begin{definition}\label{definitionccc}     (Classical solution)
	A functional $v\in C_p^{1,2}({\Lambda})$  with  $A^*\partial_xv\in C_p^0(\hat{\Lambda},H)$    is called a classical solution to  PHJB equation (\ref{hjb1}) if it satisfies
	PHJB equation (\ref{hjb1})  {point-wisely}.
\end{definition}
\par
We shall prove that the value functional $V$ defined by (\ref{value1}) is a viscosity solution of PHJB equation (\ref{hjb1}).
{ First, we give the following definition.
\begin{definition}\label{definitionc202402061}
	Let $t\in[0,T)$ and $g:\hat{\Lambda}^t\rightarrow \mathbb{R}$ be given, and $A$ satisfy Assumption \ref{hypstate} (i').
		  We say $g\in C^{1,2}_{p,A-}(\hat{\Lambda}^t)\subset C^0_p(\hat{\Lambda}^t)$ if   there exist $\partial_{x}g\in  C^0_p(\hat{\Lambda}^t,H)$ and $\partial_{xx}g\in C^0_p(\hat{\Lambda}^t,\mathcal{S}(H))$  such that, for the solution $X$ of (\ref{formular1}) with initial condition $(t,\gamma_{t})\in [0,T)\times \Lambda_{{t}}$,
\begin{eqnarray}\label{statesop020240206}
\qquad&	\begin{aligned}
	g(X_s)\leq&g(X_{t})+\int_{t}^{s}\biggl[\langle \partial_xg(X_\sigma), \, \vartheta(\sigma)\rangle_H+\frac{1}{2}\mbox{\rm Tr}(\partial_{xx}g(X_\sigma)(\varpi\varpi^*)(\sigma))\biggr]d\sigma\\
	&+\int^{s}_{t}\langle \partial_xg(X_\sigma), \, \varpi(\sigma)dW(\sigma)\rangle_H.
	\end{aligned}
	\end{eqnarray}
\end{definition}}
{
Define, for all $t\in[0,T)$,
$$
\Phi_t:=\left\{\varphi\in C_p^{1,2}({\hat{\Lambda}^t}): A^*\partial_x\varphi\in C_p^0(\hat{\Lambda}^t,H)\right\}
$$
and
\begin{eqnarray*}
	\begin{aligned}
		{\mathcal{G}}_t:=&\bigg{\{}g\in C^{0}_p({\hat\Lambda^t}):  \exists \  (h_i, g_i)\in C^1([0,T];\mathbb{R})\times C^{1,2}_{p,A-}(\hat{\Lambda}^t),\, i=1,\ldots, \mathbf{N}  \\
      & \mbox{such that for all} \ \gamma_s\in \Lambda^t, \,  g(\gamma_s)=\sum^{\mathbf{N}}_{i=1}h_i(s)g_i(\gamma_s)
		\bigg{\}}.
	\end{aligned}
\end{eqnarray*}}

{ For $g\in {\mathcal{G}}_t$,
we write
\begin{eqnarray*}
	\partial_t^og(\gamma_s)&:=&\sum^{\mathbf{N}}_{i=1}h_i'(s)g_i(\gamma_s).
\end{eqnarray*}}
{
\begin{remark}\label{remarkv0129120240206}
	By Lemma \ref{theoremito2}, for   $A$ satisfying Assumption \ref{hypstate} (i') and every $(t,\gamma_{t})\in [0,T)\times \Lambda_{{t}}$, the three path functionals $\Upsilon^{m,M}({\eta}_s-\gamma_{t,s}^A)\in C^{1,2}_{p,A-}(\hat{\Lambda}^t)$ for all $m\geq 2$ and $M\geq3$, $\Upsilon^{\varepsilon}({\eta}_s-\gamma_{t,s}^A)\in C^{1,2}_{p,A-}(\hat{\Lambda}^t)$ for $\varepsilon>0$, and $|\eta(s)-e^{(s-t)A}\gamma_{t}(t)|^{2m}\in C^{1,2}_{p,A-}(\hat{\Lambda}^t)$ for all $m\geq 1$. Therefore, for $h\in C^1([0,T];\mathbb{R})$, $\delta, \delta_i,\delta'_i\geq0, N>0$, and $(\gamma_{t},\gamma^i_{t_i})\in \Lambda\times \Lambda,
		i=0,1,2,\ldots,$ such that $ t_i\leq t,   h \geq0$ and $\sum_{i=0}^{\infty}(\delta_i+\delta'_i)\leq N$, the path  functional :
$$h(s)\Upsilon(\eta_s)+\delta|\eta_s(s)-e^{(s-t)A}\gamma_{t}(t)|^{2}
+\sum_{i=0}^{\infty}[\delta_i{\Upsilon}(\eta_s-(\gamma^i)_{t_i,s}^A)
		+\delta^{'}_i|s-t_i|^2
		], \quad \eta_s\in \hat\Lambda^t, $$   which arises in the proof of the comparison theorem,  belongs to ${\mathcal{G}}_t$.
\end{remark}}
Now,  we define our notion of  viscosity solutions.

\begin{definition}\label{definition4.1} A path functional
	$w\in C^0({\Lambda})$ is called a
	viscosity sub-solution (resp., super-solution)
	to PHJB equation (\ref{hjb1}) if the terminal condition,  $w(\gamma_T)\leq \phi(\gamma_T)$(resp., $w(\gamma_T)\geq \phi(\gamma_T)$),
	$\gamma_T\in {\Lambda}_T$ is satisfied, and for any $(\varphi, g)\in \Phi_t\times {\mathcal{G}}_t$ with $t\in [0,T)$, whenever the function $w-\varphi-g$  (resp.,  $w+\varphi+g$) satisfies for $\gamma_t\in \Lambda$,
	$$
	0=({w}-\varphi-g)(\gamma_t)=\sup_{(s,\eta_s)\in [t,T]\times\Lambda^t}
	({w}- \varphi-g)(\eta_s)
	$$
	$$
	\left(\mbox{resp.,}\ \
	0=({w}+\varphi+g)(\gamma_t)=\inf_{(s,\eta_s)\in [t,T]\times\Lambda^t}
	({w}+\varphi+g)(\eta_s)\right),
	$$
	we have
	\begin{eqnarray*}
		&&\partial_t\varphi(\gamma_t)+{ \partial_t^o}g(\gamma_t)+\langle A^*\partial_x\varphi(\gamma_t), \, \gamma_t(t)\rangle_H\\[3mm]
		&&
		 +\, {\mathbf{H}}(\gamma_t,\, (\varphi+g)(\gamma_t),\, \partial_x(\varphi+g)(\gamma_t),\, \partial_{xx}(\varphi+g)(\gamma_t))\geq0
	\end{eqnarray*}
	\begin{eqnarray*}\begin{aligned}
			\biggl(\mbox{resp.,} &\, -\partial_t\varphi(\gamma_t)-{ \partial_t^o}g(\gamma_t)-\langle A^*\partial_x\varphi(\gamma_t),\,  \gamma_t(t)\rangle_H\\[3mm]
			 &+\,{\mathbf{H}}(\gamma_t,\,-(\varphi+g)(\gamma_t),\,-\partial_x(\varphi+g)(\gamma_t),\,-\partial_{xx}(\varphi+g)(\gamma_t))
			\leq0\biggr).
		\end{aligned}
	\end{eqnarray*}
	A path functional   $w\in C^0({\Lambda})$ is said to be a
	viscosity solution to PHJB equation (\ref{hjb1}) if it is
	both a viscosity sub- and
	super-solution.
\end{definition}

\begin{remark}\label{remarkv}
	Assume that all the coefficients have the following state-dependent structure: there are suitable smooth functions $\overline{F}, \overline{G},  \overline{q}$ and $\overline{\phi}$ such that $F(\gamma_t,u)=\overline{F}(t,\gamma_t(t),u), G(\gamma_t,u)=\overline{G}(t,\gamma_t(t),u)$,
	$q(\gamma_t,y,$ $z,u)=\overline{q}(t,\gamma_t(t),y,z,u)$,  and $
	\phi(\eta_T)=\overline{\phi}(\eta_T(T))$ for all
	$(t,\gamma_t,y,z,u) \in [0,T]\times{\Lambda}\times \mathbb{R}\times \Xi\times U$ and $\eta_T\in \Lambda_T$. Then
	there exists a function $\overline{V}:[0,T]\times H\rightarrow \mathbb{R}$ such that $V(\gamma_t)=\overline{V}(t,\gamma_t(t))$ for all $(t,\gamma_t)\in [0,T]\times \Lambda$, and PHJB equation (\ref{hjb1}) reduces to the following HJB equation:
	\begin{eqnarray}\label{hjb3}
	\qquad
	\begin{cases}
	\overline{V}_{t^+}(t,x)+\langle A^*\nabla_x\overline{V}(t,x), \, x\rangle_H\\
	\qquad+\overline{{\mathbf{H}}}(t, x, \overline{V}(t,x), \nabla_x\overline{V}(t,x), \nabla^2_x\overline{V}(t,x))= 0, \  (t, x)\in
	[0,T)\times H;\\[2mm]
	\overline V(T,x)=\overline{\phi}(x), \ \ \ x\in H,
	\end{cases}
	\end{eqnarray}
	where
	\begin{eqnarray*}
		&& \overline{{\mathbf{H}}}(t,x,r,p,l)\\&:=&\sup_{u\in{
				{U}}}\left\{\langle p, \overline{F}(t,x,u)\rangle_{H} +\frac{1}{2}\mbox{Tr}[ l \overline{G}(t,x,u)\overline{G}^*(t,x,u)]+\overline{q}(t,x,r,p\overline{G}(t,x,u),u)\right\}, \\[3mm]
		&&\quad \quad  (t,x,r,p,l)\in [0,T]\times H\times \mathbb{R}\times H\times {\mathcal{S}}(H).
	\end{eqnarray*}
Here, $\overline{V}_{t^+}$ is the right time-derivative of $\overline{V}$.
\end{remark}
\begin{remark}\label{remarkv12241}
Assuming the $B$-continuity  on the coefficients, Fabbri et al. \cite{fab1} studied  viscosity solution of  HJB equation (\ref{hjb3}). The additive  term $|\cdot|_H$ in a test function  will help   to ensure inequality (1.111) in~\cite[page 84]{fab1}, for  $\langle Ax, x\rangle_H\leq 0$ for all $x\in {\mathcal{D}}(A)$ when $A$ is the generator of a $C_0$ contraction semi-group.  While the inequality $\langle Ax, x-y\rangle_H\leq 0$ for all $x\in {\mathcal{D}}(A)$ fails to be true anymore for $y\ne0$,  and thus the introduction of the term $|\cdot-y|_H$ in a test function for fixed $y\in H$ will incur a trouble to ensure  inequality (1.111) in~\cite[page 84]{fab1}. This explains why the coefficients are assumed to satisfy  the $B$-continuity condition in~\cite[(3.21) and (3.22), page 184]{fab1}) to ensure
 the $B$-continuity of the value functional  and  the comparison theorem.
\end{remark}

\begin{remark}\label{remarkv12242}
	As noted in Remark \ref{remarks},  the term $\Upsilon^{m,M}(\gamma_t-a_{\hat{t},t}^A)$ with $a_{\hat{t}}\in \Lambda$ can enter into the test functions for our viscosity solutions,  for it satisfies functional It\^o inequality (\ref{jias510815jia}).
	This property  is pivotal  together with the equivalence of   functional $\Upsilon^{m,M}$  to $||\cdot||_0^{2m}$ for $M\geq3$ (see Lemma \ref{theoremS}).
	With help of these properties, the conventional strong path-continuity  of the value functional with respect to  the norm $||\cdot||_0$ is sufficient here  and the  much stronger $B$-continuity of the value functional is not appealed to  anymore. This explains why we can remove the  $B$-continuity assumption on the coefficients, introduced by Lions and having been existing up to now in the literature. In particular, the comparison theorem is  established without assuming the $B$-continuity  on the coefficients.
\end{remark}

\begin{theorem}\label{theoremvexist}
	Let Assumptions \ref{hypstate} and  \ref{hypcost}  be satisfied. Then,  the value
	functional $V$ defined by (\ref{value1}) is a
	viscosity solution to equation (\ref{hjb1}).
\end{theorem}

\begin{proof}
First, let  $(\varphi, g)\in \Phi_{\hat{t}}\times  {\mathcal{G}}_{\hat{t}}$ with $\hat{t}\in [0,T)$
such that for  some $\hat{\gamma}_{\hat{t}}\in \Lambda$,
\begin{eqnarray}\label{0714}
0=(V-\varphi-g)(\hat{\gamma}_{\hat{t}})=\sup_{(s,\eta_s)\in [\hat{t},T]\times\Lambda^{\hat{t}}}
(V- \varphi-g)(\eta_s).
\end{eqnarray}
Thus,  by the DPP of Theorem \ref{theoremddp}, we obtain that, for all ${\hat{t}}\leq {\hat{t}}+\delta\leq T$,
\begin{eqnarray}\label{4.9}
&& 0=V(\hat{\gamma}_{\hat{t}})-({{\varphi}}+g) (\hat{\gamma}_{\hat{t}})
=\mathop{\esssup}\limits_{u(\cdot)\in {\mathcal{U}}[\hat{t},\hat{t}+\delta]} G^{\hat{\gamma}_{\hat{t}},u}_{{\hat{t}},\hat{t}+\delta}\left[V\left(X^{\hat{\gamma}_{\hat{t}},u}_{{\hat{t}}+\delta}\right)\right]
-({{\varphi}}+g) (\hat{\gamma}_{\hat{t}}).
\end{eqnarray}
Then, for any $\varepsilon>0$ and $0<\delta\leq T-\hat{t}$,  we can  find a control  $
u^{{\varepsilon},\delta}(\cdot)\in {\mathcal{U}}[\hat{t},\hat{t}+\delta]$ such
that
\begin{eqnarray}\label{4.10}
-{\varepsilon}\delta
\leq G^{\hat{\gamma}_{\hat{t}},u^{{\varepsilon},\delta}}_{{\hat{t}},\hat{t}+\delta}\left[V\left(X^{\hat{\gamma}_{\hat{t}},u^{{\varepsilon},\delta}}_{{\hat{t}}+\delta}\right)\right]-({{\varphi}}+g) (\hat{\gamma}_{\hat{t}}).
\end{eqnarray}
We note that the process
$\left\{G^{\hat{\gamma}_{\hat{t}},u^{{\varepsilon},\delta}}_{s,\hat{t}+\delta}\left[V\left(X^{\hat{\gamma}_{\hat{t}},u^{{\varepsilon},\delta}}_{{\hat{t}}+\delta}\right)\right],  \, s\in[\hat{t},\hat{t}+\delta]\right\}$
is the first component of  the solution $\left(Y^{\hat{\gamma}_{\hat{t}},u^{{\varepsilon},\delta}}, \,  Z^{\hat{\gamma}_{\hat{t}},u^{{\varepsilon},\delta}}\right)$ of the following
BSDE
\begin{eqnarray}\label{bsde4.10}
&\begin{cases}
dY(s) =&
-q\left(X^{\hat{\gamma}_{\hat{t}},u^{{\varepsilon},\delta}}_s,\, Y(s),\,  Z(s), \, u^{{\varepsilon},\delta}(s)\right)\, ds\\[1mm]
&\quad +Z(s)\, dW\left(s\right), \quad   s\in[\hat{t},\hat{t}+\delta]; \\[2mm]
Y(\hat{t}+\delta)=&V\left(X^{\hat{\gamma}_{\hat{t}},
	u^{{\varepsilon},\delta}}_{{\hat{t}}+\delta}\right).
\end{cases}
\end{eqnarray}
Applying functional It\^{o} formula (\ref{statesop03}) to ${\varphi}\left(X^{\hat{\gamma}_{\hat{t}},u^{{\varepsilon},\delta}}_s\right)$ and
 {inequality (\ref{statesop020240206})} 
to
$g\left(X^{\hat{\gamma}_{\hat{t}},u^{{\varepsilon},\delta}}_s\right)$,   we get that
\begin{eqnarray}\label{bsde4.21}
\begin{aligned}
&\qquad \qquad \left({\varphi}+g\right)\left(X^{\hat{\gamma}_{\hat{t}},u^{{\varepsilon},\delta}}_s\right)\leq Y^{1,\delta}\left(s\right)\\
&\qquad=:\qquad  \left({\varphi}+g\right)\left(\hat{\gamma}_{\hat{t}}\right)+\int^{s}_{{\hat{t}}} \left({\mathcal{L}}\left({\varphi}+g\right)\right)\left(X^{\hat{\gamma}_{\hat{t}},u^{{\varepsilon},\delta}}_\sigma,
u^{{\varepsilon},\delta}\left(\sigma\right)\right)d\sigma \\
&\qquad\qquad-\int^{s}_{{\hat{t}}}q\left(\gamma,\left({\varphi}+g\right)\left(\gamma\right), \partial_x\left({\varphi}+g\right)\left(\gamma\right)
G\left(\gamma,v\right), v\right)\Biggm|_{\scriptsize \begin{matrix}
v=u^{{\varepsilon},\delta}\left(\sigma\right)\\
\gamma=X^{\hat{\gamma}_{\hat{t}},u^{{\varepsilon},\delta}}_\sigma
\end{matrix}} d\sigma\\
&\qquad\qquad
+\int^{s}_{{\hat{t}}}\left\langle \partial_x\left({\varphi}+g\right)\left(X^{\hat{\gamma}_{\hat{t}},u^{{\varepsilon},\delta}}_\sigma\right),
G\left(X^{\hat{\gamma}_{\hat{t}},u^{{\varepsilon},\delta}}_\sigma,u^{{\varepsilon},\delta}\left(\sigma\right)\right)dW\left(\sigma\right)\right\rangle_H,
\end{aligned}
\end{eqnarray}
where for  $\left(t, \gamma_t,u\right)\in [0,T]\times {\Lambda}\times U$,
\begin{eqnarray*}
	&&\left({\mathcal{L}}\left({\varphi}+g\right)\right)\left(\gamma_t,u\right)\\
	&:=&{\partial_t{\varphi}\left(\gamma_t\right)+{ \partial_t^o}g\left(\gamma_t\right)}+\left\langle A^*\partial_x {\varphi}\left(\gamma_t\right),\gamma_t\left(t\right)\right\rangle_H
	+\left\langle\partial_x \left({\varphi}+g\right)\left(\gamma_t\right),F\left(\gamma_t,u\right)\right\rangle_{H}\\
	&&+\frac{1}{2}\mbox{Tr}[\partial_{xx}\left({\varphi}+g\right)\left(\gamma_t\right)G\left(\gamma_t,u\right)G^*\left(\gamma_t,u\right)]\\
	&&+
	 q\left(\gamma_t,\left({\varphi}+g\right)\left(\gamma_t\right),{\partial_x\left({\varphi}+g\right)\left(\gamma_t\right)}G\left(\gamma_t,u\right),u\right).
\end{eqnarray*}
It is clear that
\begin{eqnarray}\label{y_1}
Y^{1,\delta}(\hat{t}\,)=({\varphi}+g)(\hat{\gamma}_{\hat{t}}).
\end{eqnarray}
Set for $s\in[\hat{t},\hat{t}+\delta]$,
\begin{eqnarray*}
	Y^{2,{\hat{\gamma}_{\hat{t}},u^{{\varepsilon},\delta}}}(s)&:=&
	Y^{1,\delta}\left(s\right)-Y^{\hat{\gamma}_{\hat{t}},u^{{\varepsilon},\delta}}\left(s\right)\geq
	\left({\varphi}+g\right)\left(X_s^{\hat{\gamma}_{\hat{t}},u^{{\varepsilon},\delta}}\right)-Y^{\hat{\gamma}_{\hat{t}},u^{{\varepsilon},\delta}}(s),  \\
	Y^{3,\delta}(s)&:=&
	Y^{1,\delta}(s)-
	\left({\varphi}+g\right)\left(X_s^{\hat{\gamma}_{\hat{t}},u^{{\varepsilon},\delta}}\right)\geq0,  \quad \hbox{\rm and}\\
	Z^{2,{\hat{\gamma}_{\hat{t}},u^{{\varepsilon},\delta}}}(s)  &:=&\partial_x\left({\varphi}+g\right)\left(X^{\hat{\gamma}_{\hat{t}},u^{{\varepsilon},\delta}}_s\right)
	G\left(X^{\hat{\gamma}_{\hat{t}},u^{{\varepsilon},\delta}}_s,u^{{\varepsilon},\delta}(s)\right)-Z^{\hat{\gamma}_{\hat{t}},u^{{\varepsilon},\delta}}(s).
\end{eqnarray*}
Comparing (\ref{bsde4.10}) and (\ref{bsde4.21}), we have, $P$-a.s.,
\begin{eqnarray*}
	&&dY^{2,{\hat{\gamma}_{\hat{t}},u^{{\varepsilon},\delta}}}\left(s\right)\\
	 &=&\bigg{[}[\left({\mathcal{L}}\left({\varphi}+g\right)\right)\left(\gamma,v\right)
	 -q\left(\gamma,\left({\varphi}+g\right)\left(\gamma\right),
	\partial_x\left({\varphi}+g\right)\left(\gamma\right)
	G\left(\gamma,v\right),v\right)]\Biggm|_{\scriptsize\begin{matrix}\gamma=X^{\hat{\gamma}_{\hat{t}},u^{{\varepsilon},\delta}}_s\\
		v=u^{{\varepsilon},\delta}\left(s\right)\end{matrix}}\\
	&&+q\left(X^{\hat{\gamma}_{\hat{t}},u^{{\varepsilon},\delta}}_s,Y^{\hat{\gamma}_{\hat{t}},u^{{\varepsilon},\delta}}\left(s\right),
	Z^{\hat{\gamma}_{\hat{t}},u^{{\varepsilon},\delta}}\left(s\right),u^{{\varepsilon},\delta}\left(s\right)\right)\bigg{]}ds
	+Z^{2,{\hat{\gamma}_{\hat{t}},u^{{\varepsilon},\delta}}}\left(s\right)dW\left(s\right)\\
	 &=&\bigg{[}\left({\mathcal{L}}\left({\varphi}+g\right)\right)\left(X^{\hat{\gamma}_{\hat{t}},u^{{\varepsilon},\delta}}_s,u^{{\varepsilon},\delta}\left(s\right)\right)
	+A_{1,\delta}\left(s\right)Y^{3,\delta}\left(s\right)
	-A_{1,\delta}\left(s\right)Y^{2,{\hat{\gamma}_{\hat{t}},u^{{\varepsilon},\delta}}}\left(s\right)\\
	&&-\left\langle A_{2,\delta}\left(s\right), Z^{2,{\hat{\gamma}_{\hat{t}},u^{{\varepsilon},\delta}}}\left(s\right)\right\rangle_{\Xi}\bigg{]}ds+Z^{2,{\hat{\gamma}_{\hat{t}},u^{{\varepsilon},\delta}}}\left(s\right)dW\left(s\right),
\end{eqnarray*}
where $|A_{1,\delta}|\vee|A_{2,\delta}|\leq L$.
Applying It\^o  formula  (see~\cite[Proposition 2.2]{el}), we have
\begin{eqnarray}\label{4.14}
\begin{aligned}
&\qquad\qquad\qquad Y^{2,{\hat{\gamma}_{\hat{t}},u^{{\varepsilon},\delta}}}(\hat{t}\,)=\mathbb{E}\bigg{[}Y^{2,{\hat{\gamma}_{\hat{t}},u^{{\varepsilon},\delta}}}(t+\delta)\Gamma^{\hat{t},\delta}(\hat{t}+\delta)\\
&\qquad\qquad-
\int^{\hat{t}+\delta}_{{\hat{t}}}\!\!\!\!\! \Gamma^{\hat{t},\delta}\left(\sigma\right)
\left[\left({\mathcal{L}}\left({\varphi}+g\right)\right)
\left(X^{\hat{\gamma}_{\hat{t}},u^{{\varepsilon},\delta}}_\sigma,u^{{\varepsilon},\delta}\left(\sigma\right)\right)+A_{1,\delta}\left(\sigma\right)Y^{3,\delta}\left(\sigma\right)\right]d\sigma\bigg{|}
{\mathcal{F}}_{\hat{t}}\bigg{]},
\end{aligned}
\end{eqnarray}
where $\Gamma^{\hat{t},\delta}(\cdot)$ solves the linear SDE
$$
d\Gamma^{\hat{t},\delta}(s)=\Gamma^{\hat{t},\delta}(s)(A_{1,\delta}(s)ds+A_{2,\delta}(s)dW(s)),\ s\in [{\hat{t}},{\hat{t}}+\delta];\ \ \ \Gamma^{\hat{t},\delta}({\hat{t}})=1.
$$
Obviously, $\Gamma^{\hat{t},\delta}\geq 0$. Combining (\ref{4.10}), (\ref{y_1}) and (\ref{4.14}), we have
\begin{eqnarray}\label{4.15}
\begin{aligned}
-\varepsilon&\leq \frac{1}{\delta}(Y^{\hat{\gamma}_{\hat{t}},u^{{\varepsilon},\delta}}(\hat{t}\,)-({\varphi}+g)(\hat{\gamma}_{\hat{t}}))\\
&=\frac{1}{\delta}(Y^{\hat{\gamma}_{\hat{t}},u^{{\varepsilon},\delta}}(\hat{t}\,)-Y^{1,\delta}({\hat{t}}))
=-\frac{1}{\delta}Y^{2,{\hat{\gamma}_{\hat{t}},u^{{\varepsilon},\delta}}}\left(\hat{t}\right)\\
&= \frac{1}{\delta}\mathbb{E}\bigg{[}-Y^{2,{\hat{\gamma}_{\hat{t}},u^{{\varepsilon},\delta}}}(\hat{t}+\delta)\Gamma^{\hat{t},\delta}(\hat{t}+\delta)\\
&\quad
+\int^{\hat{t}+\delta}_{{\hat{t}}}\!\!\!\!\! \Gamma^{\hat{t},\delta}(\sigma)\left[({\mathcal{L}}({\varphi}+g))\left(X^{\hat{\gamma}_{\hat{t}},u^{{\varepsilon},\delta}}_\sigma,u^{{\varepsilon},\delta}(\sigma)\right)
+A_{1,\delta}(\sigma)Y^{3,\delta}(\sigma)\right]d\sigma\bigg{]}\\[3mm]
&:=I_1+I_2+I_3+I_4+I_5
\end{aligned}
\end{eqnarray}
with the five terms
\begin{eqnarray*}
	 I_1&:=&-\frac{1}{\delta}\mathbb{E}\bigg{[}Y^{2,{\hat{\gamma}_{\hat{t}},u^{{\varepsilon},\delta}}}(\hat{t}+\delta)\Gamma^{\hat{t},\delta}(\hat{t}+\delta)\bigg{]}, \\
	 I_2&:=&\frac{1}{\delta}\mathbb{E}\bigg{[}\int^{\hat{t}+\delta}_{{\hat{t}}}\!\!\!({\mathcal{L}}({\varphi}+g))(\hat{\gamma}_{{\hat{t}}},u^{{\varepsilon},\delta}(\sigma))\, d\sigma\bigg{]},\nonumber\\
	 I_3&:=&\frac{1}{\delta}\mathbb{E}\bigg{[}\int^{\hat{t}+\delta}_{{\hat{t}}}\!\!\left[({\mathcal{L}}({\varphi}+g))\left(X^{\hat{\gamma}_{\hat{t}},u^{{\varepsilon},\delta}}_\sigma,u^{{\varepsilon},\delta}(\sigma)\right)-
	({\mathcal{L}}({\varphi}+g))\left(\hat{\gamma}_{{\hat{t}}},u^{{\varepsilon},\delta}(\sigma)\right)\right]d\sigma\bigg{]},\nonumber\\
	 I_4&:=&\frac{1}{\delta}\mathbb{E}\bigg{[}\int^{\hat{t}+\delta}_{{\hat{t}}}\!\!\!\!\!(\Gamma^{\hat{t},\delta}(\sigma)-1)({\mathcal{L}}({\varphi}+g))\left(X^{\hat{\gamma}_{\hat{t}},u^{{\varepsilon},\delta}}_\sigma,
	u^{{\varepsilon},\delta}(\sigma)\right)d\sigma\bigg{]}, \quad \hbox{\rm and }\nonumber\\
	 I_5&:=&\frac{1}{\delta}\mathbb{E}\bigg{[}\int^{\hat{t}+\delta}_{{\hat{t}}}\!\!\!\Gamma^{\hat{t},\delta}(\sigma)A_{1,\delta}(\sigma)Y^{3,\delta}(\sigma)\, d\sigma\bigg{]}. \nonumber
\end{eqnarray*}
Since the coefficients of the operator ${\mathcal{L}}$  have a  linear growth,
combining the regularity of $(\varphi, g)\in \Phi_{\hat{t}}\times {\mathcal{G}}_{\hat{t}}$, there exist a integer
$\bar{p}\geq1$ and a constant $C>0$ independent of $u\in U$ such that, for all $(t,\gamma_t,u)\in [0,T]\times\Lambda\times U$,
\begin{eqnarray}\label{4.4444}|({\varphi}+g)(\gamma_{t})|\vee  |
({\mathcal{L}}({\varphi}+g))(\gamma_{t},u)|
\leq  C(1+||\gamma_{t}||_0)^{\bar{p}}.
\end{eqnarray}
In view of Lemma \ref{lemmaexist111}, we also have
\begin{eqnarray*}
	\sup_{\sigma\in[\hat{t},\hat{t}+\delta]}\mathbb{E}|\Gamma^{\hat{t},\delta}(\sigma)-1|^2\leq C\delta.
\end{eqnarray*}
Thus, in view of (\ref{0714}) and (\ref{bsde4.21}), we have
\begin{eqnarray}\label{4.1611}
I_1&=& -\frac{1}{\delta}\mathbb{E}\left[\left(Y^{1,\delta}(\hat{t}+\delta)-Y^{\hat{\gamma}_{\hat{t}},u^{{\varepsilon},\delta}}(\hat{t}+\delta)\right)\Gamma^{\hat{t},\delta}(\hat{t}+\delta)\right]\nonumber\\
&\leq&-\frac{1}{\delta}\mathbb{E}\left[\left(({\varphi}+g)\left(X_{\hat{t}+\delta}^{\hat{\gamma}_{\hat{t}},u^{{\varepsilon},\delta}}\right)-Y^{\hat{\gamma}_{\hat{t}},u^{{\varepsilon},\delta}}(\hat{t}+\delta)\right)
\Gamma^{\hat{t},\delta}(\hat{t}+\delta)\right]\nonumber\\
&=&\frac{1}{\delta}\mathbb{E}\left[\left(V\left(X^{\hat{\gamma}_{\hat{t}},
	 u^{{\varepsilon},\delta}}_{{\hat{t}}+\delta}\right)-({\varphi}+g)\left(X_{\hat{t}+\delta}^{\hat{\gamma}_{\hat{t}},u^{{\varepsilon},\delta}}\right)\right)\Gamma^{\hat{t},\delta}(\hat{t}+\delta)\right]\leq 0
\end{eqnarray}
and
\begin{eqnarray}\label{4.16}
I_2&\leq&\frac{1}{\delta}\bigg{[}\int^{\hat{t}+\delta}_{{\hat{t}}}\!\!\!\sup_{u\in U}({\mathcal{L}}({\varphi}+g))({\hat{\gamma}_{\hat{t}},{u}})\, d\sigma\bigg{]}
= \partial_t{\varphi}(\hat{\gamma}_{\hat{t}})+{ \partial_t^o}g(\hat{\gamma}_{\hat{t}})+\left\langle A^*\partial_x{\varphi}(\hat{\gamma}_{\hat{t}}),\hat{\gamma}_{\hat{t}}(\hat{t}\,)\right\rangle_H\nonumber\\[3mm]
&&\quad +\, {\mathbf{H}}(\hat{\gamma}_{\hat{t}},\, ({\varphi}+g)(\hat{\gamma}_{\hat{t}}),\, \partial_x({\varphi}+g)(\hat{\gamma}_{\hat{t}}),\,
\partial_{xx}({\varphi}+g)(\hat{\gamma}_{\hat{t}})).
\end{eqnarray}
\par
Now we prove that other three terms  $I_3$, $I_4$ and $I_5$ are of higher order.  In view of  (\ref{0604}),
$$
\lim_{\delta\rightarrow0}\mathbb{E}\, d_\infty^{\bar{p}}(X^{\hat{\gamma}_{\hat{t}},u^{{\varepsilon},\delta}}_{\hat{t}+\delta},\hat{\gamma}_{\hat{t}})=0.
$$
Then,  in view of (\ref{4.4444}), using the dominated convergence theorem, we have
$$
\lim_{\delta\rightarrow0}\, \mathbb{E}\!\!\sup_{\hat{t}\leq \sigma\leq\hat{t}+ \delta}\left|({\mathcal{L}}({\varphi}+g))\left(X^{{\hat{\gamma}_{\hat{t}},u^{{\varepsilon},\delta}}}_\sigma,u^{{\varepsilon},\delta}(\sigma)\right)-
({\mathcal{L}}({\varphi}+g))({\hat{\gamma}_{\hat{t}},u^{{\varepsilon},\delta}}(\sigma))\right|=0
$$
and
\begin{eqnarray*}
	&&\lim_{\delta\rightarrow0} \sup_{\hat{t}\leq \sigma\leq \hat{t}+\delta}\mathbb{E}\left|\Gamma^{\hat{t},\delta}(\sigma)A_{1,\delta}(\sigma)Y^{3,\delta}(\sigma)\right|\\
	&\leq& L\lim_{\delta\rightarrow0} \sup_{\hat{t}\leq \sigma\leq \hat{t}+\delta} \mathbb{E}\left[\left|\Gamma^{\hat{t},\delta}(\sigma)\right|\left(|Y^{1,\delta}(\sigma)-(\varphi+g)(\hat{\gamma}_{\hat{t}})|
	+\left|\varphi+g\right|(\gamma)\biggm|^{\gamma=X^{\hat{\gamma}_{\hat{t}},u^{{\varepsilon},\delta}}_\sigma}_{\gamma=\hat{\gamma}_{\hat{t}}}\right)\right] \\
	&=&\, 0.
\end{eqnarray*}
Furthermore,  we have
\begin{eqnarray}\label{4.18}
\qquad\qquad \lim_{\delta\rightarrow0}|I_3|
\leq \lim_{\delta\rightarrow0}\sup_{\hat{t}\leq \sigma\leq \hat{t}+\delta}\mathbb{E}\bigg{|}&{\mathcal{L}}({\varphi}+g)(\gamma,v)\biggm|^{(\gamma,v)=\left(X^{{\hat{\gamma}_{\hat{t}},u^{{\varepsilon},\delta}}}_\sigma,\,\,  u^{{\varepsilon},\delta}(\sigma)\right)}_{(\gamma,v)=({\hat{\gamma}_{\hat{t}},\,\, u^{{\varepsilon},\delta}}(\sigma))}\bigg{|}=0
\end{eqnarray}
and
\begin{eqnarray}\label{4.180714}
\lim_{\delta\rightarrow0}|I_5| \leq\lim_{\delta\rightarrow0}\sup_{\hat{t}\leq \sigma\leq \hat{t}+\delta}\mathbb{E}\left|\Gamma^{\hat{t},\delta}(\sigma)A_{1,\delta}(\sigma)Y^{3,\delta}(\sigma)\right|=0.
\end{eqnarray}
Moreover, for some finite constant $C>0$,
\begin{eqnarray}\label{4.19}
\begin{aligned}
&\qquad |I_4|
\leq\frac{1}{\delta}\int^{\hat{t}+\delta}_{{\hat{t}}}\mathbb{E}|\Gamma^{\hat{t},\delta}(\sigma)-1|\left|({\mathcal{L}}({\varphi}+g))(X^{\hat{\gamma}_{\hat{t}},u^{{\varepsilon},\delta}}_\sigma,
u^{{\varepsilon},\delta}(\sigma))\right|
d\sigma\\
&\qquad \leq\frac{1}{\delta}\!\!\int^{\hat{t}+\delta}_{{\hat{t}}}\!\!\! \left(\mathbb{E}\left(\Gamma^{\hat{t},\delta}\left(\sigma\right)-1\right)^2\right)^{\frac{1}{2}}\!\!\! \left(\mathbb{E}\left[\left({\mathcal{L}}\left({\varphi}+g\right)\right)\left(X^{\hat{\gamma}_{\hat{t}},
	u^{{\varepsilon},\delta}}_\sigma,
u^{{\varepsilon},\delta}\left(\sigma\right)\right)\right]^2\right)^{\frac{1}{2}}\!\! d\sigma
\\
&\qquad\leq C(1+||\hat{\gamma}_{\hat{t}}||_0)^{\bar{p}}\delta^\frac{1}{2}.
\end{aligned}
\end{eqnarray}
Substituting  (\ref{4.1611}), (\ref{4.16}) and (\ref{4.19}) into (\ref{4.15}), we have
\begin{eqnarray}\label{4.2000000}
\begin{aligned}
-\varepsilon&\leq {\partial_t{\varphi}(\hat{\gamma}_{\hat{t}})+{ \partial_t^o}g(\hat{\gamma}_{\hat{t}})} +\left\langle A^*\partial_x{\varphi}(\hat{\gamma}_{\hat{t}}),\,  \hat{\gamma}_{\hat{t}}(\hat{t}\,)\right\rangle_H\\
&\quad+{\mathbf{H}}(\hat{\gamma}_{\hat{t}},\, ({\varphi}+g)(\hat{\gamma}_{\hat{t}}),\, \partial_x({\varphi}+g)(\hat{\gamma}_{\hat{t}}),
\, \partial_{xx}({\varphi}+g)(\hat{\gamma}_{\hat{t}}))\\
&\quad+|I_3|+|I_5|
+C(1+||\hat{\gamma}_{\hat{t}}||_0)^{\bar{p}}\,\delta^\frac{1}{2}.
\end{aligned}
\end{eqnarray}
Sending $\delta$ to $0$,  in view of  (\ref{4.18}) and (\ref{4.180714}),  we have
\begin{eqnarray*}
	\begin{aligned}
		-\varepsilon
		&\leq{\partial_t{\varphi}(\hat{\gamma}_{\hat{t}})+{ \partial_t^o}g(\hat{\gamma}_{\hat{t}})}+\left\langle A^*\partial_x{\varphi}\left(\hat{\gamma}_{\hat{t}}\right), \hat{\gamma}_{\hat{t}}\left(\hat{t}\right)\right\rangle_H\\
		 &\quad+{\mathbf{H}}\left(\hat{\gamma}_{\hat{t}},\, \left({\varphi}+g\right)\left(\hat{\gamma}_{\hat{t}}\right),\, \partial_x\left({\varphi}+g\right)\left(\hat{\gamma}_{\hat{t}}\right),\,
		\partial_{xx}\left({\varphi}+g\right)\left(\hat{\gamma}_{\hat{t}}\right)\right).
	\end{aligned}
\end{eqnarray*}
Since $\varepsilon$ is arbitrary, the value functional
$V$ is  a viscosity sub-solution to (\ref{hjb1}).
\par
In a symmetric (even easier) way, we show that  $V$ is also a viscosity super-solution to equation (\ref{hjb1}).
The proof is complete.
\end{proof}

Now, let us give   the result of  classical  solutions, which show the consistency of viscosity solutions.

\begin{theorem}\label{theorem1223}
	Let Assumptions \ref{hypstate} and  \ref{hypcost} hold true, $v\in C_p^{1,2}({\Lambda})$ and $A^*\partial_xv\in C_p^0(\hat{\Lambda},H)$. Then
	$v$ is a classical solution of  equation (\ref{hjb1}) if and only if it is a viscosity solution.
\end{theorem}

\begin{proof}  Assume that $v$ is a viscosity solution. It is clear that $v(\gamma_T)=\phi(\gamma_T)$ for all $\gamma_T\in \Lambda_T$. For any $(t,\gamma_t)\in [0,T)\times \Lambda$, since $v\in C_p^{1,2}({\Lambda})$ and $A^*\partial_xv\in C_p^0(\hat{\Lambda},H)$, by definition of viscosity solutions,  we have $$
\partial_t v(\gamma_t)+\langle A^*\partial_xv(\gamma_t), \gamma_t(t)\rangle_H + {\mathbf{H}}(\gamma_t,v(\gamma_t),\partial_x v(\gamma_t),\partial_{xx} v(\gamma_t))=0.
$$
On the other hand, assume $v$ is a  classical solution. Let  $(\varphi, g)\in \Phi_t\times {\mathcal{G}}_t$ with $t\in [0,T)$
such that for some  $\gamma_t\in \Lambda$,
$$
0=(v-\varphi-g)(\gamma_t)=\sup_{(s,\eta_s)\in [t,T]\times\Lambda^t}
(v- \varphi-g)(\eta_s),
$$ For every $\alpha\in H$ and $\beta\in L_2(\Xi,H)$, let $(\theta,\xi_\theta)=(t,\gamma_t)$, $\vartheta(\cdot)\equiv \alpha$ and $\varpi(\cdot)\equiv \beta$ in (\ref{formular1}), 
applying functional It\^o formula (\ref{statesop0}) to $\varphi$ and {inequality (\ref{statesop020240206})}
 to $g$ and noticing that $(v-\varphi-g)(\gamma_t)=0$, we have, for every $0<\delta\leq T-t$,
\begin{eqnarray}\label{1224}
\begin{aligned}
0&\leq\mathbb{E}(\varphi+g-v)(X_{t+\delta})\\
&\leq\mathbb{E}\int^{t+\delta}_{t}[\partial_t(\varphi-v)(X_\sigma)+{ \partial_t^o}g(X_\sigma)
+\langle A^*\partial_x(\varphi-v)(X_\sigma), X(\sigma)\rangle_H\\
&\qquad+\langle\partial_x(\varphi+g-v)(X_\sigma), \alpha\rangle_{H}]\,d\sigma\\
&\qquad+\frac{1}{2}\mathbb{E}\int^{t+\delta}_{t}\mbox{Tr}((\partial_{xx}(\varphi+g-v)(X_\sigma))\beta\beta^*)\, d\sigma\\
&=
\mathbb{E}\int^{t+\delta}_{t}\widetilde{{\mathcal{H}}}(X_\sigma)\, d\sigma,
\end{aligned}
\end{eqnarray}
where, for  $(s,\eta_s)\in [0,T]\times \Lambda$,
\begin{eqnarray*}
	\widetilde{{\mathcal{H}}}(\eta_s)&:=&\partial_t(\varphi-v)(\eta_s)+{ \partial_t^o}g(\eta_s)
	+\langle A^*\partial_x(\varphi-v)(\eta_s), \eta_s(s)\rangle_H\\
	&&
	+\langle\partial_x(\varphi+g-v)(\eta_s), \alpha\rangle_H+\frac{1}{2}\mbox{Tr}((\partial_{xx}(\varphi+g-v)(\eta_s))\beta\beta^*).
\end{eqnarray*}
Letting $\delta\rightarrow0$ in (\ref{1224}),
\begin{eqnarray}\label{040912}
\widetilde{{\mathcal{H}}}(\gamma_t)\geq0.
\end{eqnarray}
Let $\beta=\mathbf{0}$, by the arbitrariness of  $\alpha$,
$$
\partial_t\varphi(\gamma_t)+{ \partial_t^o}g(\gamma_t)+\langle A^*\partial_x(\varphi-v)(\gamma_t), \gamma_t(t)\rangle_H\geq\partial_tv(\gamma_t), \ \ \partial_x\varphi(\gamma_t)+\partial_xg(\gamma_t)=\partial_xv(\gamma_t).
$$
Then, for every $u\in U$, let $\beta=G(\gamma_t,u)$ in (\ref{040912}).
Noting that $\varphi(\gamma_t)+g(\gamma_t)=v(\gamma_t)$,
we have
\begin{eqnarray*}
	&&\partial_t\varphi(\gamma_t)+{ \partial_t^o}g(\gamma_t)+\langle \partial_x(\varphi+g)(\gamma_t), F(\gamma_t,u)\rangle_{H}+\langle A^*\partial_x\varphi(\gamma_t), \gamma_t(t)\rangle_H\\
	&&
	+\frac{1}{2}\mbox{Tr}((\partial_{xx}(\varphi+g)(\gamma_t))G(\gamma_t,u)G^*(\gamma_t,u))\\
	&&+q(\gamma_t,({\varphi}+g)(\gamma_t),({\partial_x({\varphi}+g)(\gamma_t)})G(\gamma_t,u),u)\\
	&\geq& \partial_tv(\gamma_t)+\langle \partial_xv(\gamma_t), F(\gamma_t,u)\rangle_{H}+\langle A^*\partial_xv(\gamma_t),  \gamma_t(t)\rangle_H
	\\&&+\frac{1}{2}\mbox{Tr}(\partial_{xx}v(\gamma_t)G(\gamma_t,u)G^*(\gamma_t,u))+q(\gamma_t,v(\gamma_t),{\partial_xv(\gamma_t)}G(\gamma_t,u),u),
\end{eqnarray*}
Since
$$\partial_tv(\gamma_t)+\langle A^*\partial_xv(\gamma_t), \gamma_t(t)\rangle_H
+{\mathbf{H}}(\gamma_t,v(\gamma_t),\partial_xv(\gamma_t),\partial_{xx}v(\gamma_t))=0, $$   taking the supremum over $u\in U$, we see that
\begin{eqnarray*}
	&&\partial_t\varphi(\gamma_t)+{\partial_t^o}g(\gamma_t)+\langle A^*\partial_x\varphi(\gamma_t), \gamma_t(t)\rangle_H\\
	&&+{\mathbf{H}}(\gamma_t,(\varphi+g)(\gamma_t),\partial_x(\varphi+g)(\gamma_t),\partial_{xx}(\varphi+g)(\gamma_t))\\
	&\geq& \partial_tv(\gamma_t)+\langle A^*\partial_xv(\gamma_t), \gamma_t(t)\rangle_H+{\mathbf{H}}(\gamma_t,v(\gamma_t),\partial_xv(\gamma_t),\partial_{xx}v(\gamma_t))=0.
\end{eqnarray*}
Thus,  $v$ is a viscosity  {sub-solution} of equation (\ref{hjb1}). In a symmetric way, we show that $v$ is also a viscosity  {super-solution} to equation (\ref{hjb1}).
\end{proof}

We have the following stability of viscosity solutions.

\begin{theorem}\label{theoremstability}
	Let $F,G,q,\phi$ satisfy Assumptions \ref{hypstate} and  \ref{hypcost}, and $v\in C^0(\Lambda)$. Assume that
	\item[(i)]      for any $\varepsilon>0$, there exist $F^\varepsilon, G^\varepsilon, q^\varepsilon, \phi^\varepsilon$ and $v^\varepsilon\in C^0(\Lambda)$ such that  $F^\varepsilon, G^\varepsilon, q^\varepsilon, \phi^\varepsilon$ satisfy  Assumptions \ref{hypstate} and  \ref{hypcost} and $v^\varepsilon$ is a viscosity  {sub-solution} (resp.,  {super-solution}) of PHJB equation (\ref{hjb1}) with generators $F^\varepsilon, G^\varepsilon, q^\varepsilon, \phi^\varepsilon$;
	\item[(ii)] as $\varepsilon\rightarrow0$, $(F^\varepsilon, G^\varepsilon, q^\varepsilon, \phi^\varepsilon,v^\varepsilon)$ converge to
	$(F, G, q, \phi, v)$  uniformly in the following sense: 
	\begin{eqnarray}\label{sss}
&&	\lim_{\varepsilon\rightarrow 0}\sup_{\scriptsize\begin{matrix} (t,\gamma_t,x,y,u)\\
		\in [0,T]\times \Lambda \times \mathbb{R}\times H\times U\end{matrix}}\sup_{\eta_T\in  \Lambda_T}[(|F^\varepsilon-F|+|G^\varepsilon-G|)({\gamma}_t,u)
	\nonumber\\
	&&~~~~~~~~~~~~+|q^\varepsilon-q|({\gamma}_t,{x},yG({\gamma}_{{t}},u),u)+|\phi^\varepsilon-\phi|(\eta_T)+|v^\varepsilon-v|({\gamma}_t)]=0.
	\end{eqnarray}
	Then,  $v$ is a viscosity  {sub-solution} (resp., super-solution) of PHJB equation (\ref{hjb1}) with generators $(F,G,q,\phi)$.
\end{theorem}

\begin{proof}  Without loss of generality, we shall only prove the viscosity {sub-solution} property.
First,  from $v^{\varepsilon}$ is a viscosity {sub-solution} of  equation (\ref{hjb1}) with generators $F^{\varepsilon}, G^{\varepsilon},  q^{\varepsilon}, \phi^{\varepsilon}$, it follows that
$$
v^{\varepsilon}(\gamma_T)\leq \phi^{\varepsilon}(\gamma_T),\ \ \gamma_T\in \Lambda_T.
$$
Letting $\varepsilon\rightarrow0$, we  have
$$
v(\gamma_T)\leq \phi(\gamma_T),\ \ \gamma_T\in \Lambda_T.
$$
Next,  let  $\varphi\in \Phi_{\hat{t}}$ and $g\in {\mathcal{G}}_{\hat{t}}$ with $\hat{t}\in [0,T)$
such that
$$
0=(v-\varphi-g)(\hat{\gamma}_{\hat{t}})=\sup_{(s,\eta_s)\in[\hat{t},T]\times \Lambda^{\hat{t}}}
(v- \varphi-g)(\eta_s),
$$
where $\hat{\gamma}_{\hat{t}}\in \Lambda$.
By (\ref{sss}), there exists a constant $\delta>0$ such that for all $\varepsilon\in (0,\delta)$,
$$\sup_{(t,\gamma_t)\in [\hat{t},T]\times\Lambda^{\hat{t}}}(v^\varepsilon(\gamma_t)-\varphi(\gamma_t)-g(\gamma_t))\leq 1.$$
Denote $g_{1}(\gamma_t):=g(\gamma_t)+\overline{\Upsilon}(\gamma_t,\hat{\gamma}_{{\hat{t}}})$
for all
$(t,\gamma_t)\in [\hat{t},T]\times\Lambda$. 
Then we have  $g_{1}\in {\mathcal{G}}_{\hat{t}}$. For every $\varepsilon\in (0,\delta)$,  it is clear that $v^{\varepsilon}-{{\varphi}}-g_1$ is an  upper semicontinuous functional and bounded from above on $\Lambda^{\hat{t}}$.  Define a sequence of positive numbers $\{\delta_i\}_{i\geq0}$  by 
$\delta_i=\frac{1}{2^i}$ for all $i\geq0$. 
Since 
$\overline{\Upsilon}(\cdot,\cdot)$ is a gauge-type function, from Lemma \ref{theoremleft} it follows that,
for every  $(t_0,\gamma^0_{t_0})\in [\hat{t},T]\times \Lambda^{\hat{t}}$ {satisfying}
$$
(v^{\varepsilon}-\varphi-g_1)(\gamma^0_{t_0})\geq \sup_{(s,\gamma_s)\in [\hat{t},T]\times \Lambda^{\hat{t}}}(v^{\varepsilon}-\varphi-g_1)(\gamma_s)-\varepsilon,
$$
and
$$ (v^{\varepsilon}-\varphi-g_1)(\gamma^0_{t_0})\geq (v^{\varepsilon}-\varphi-g_1)(\hat{\gamma}_{\hat{t}}),
$$
there exist $(t_{\varepsilon},{\gamma}^{\varepsilon}_{t_{\varepsilon}})\in [\hat{t},T]\times \Lambda^{\hat{t}}$ and a sequence $\{(t_i,\gamma^i_{t_i})\}_{i\geq1}\subset [\hat{t},T]\times \Lambda^{\hat{t}}$ such that
\begin{description}
	\item[(i)] $\displaystyle \overline{\Upsilon}(\gamma^0_{t_0},{\gamma}^{\varepsilon}_{t_{\varepsilon}})\leq {\varepsilon}$,  $\overline{\Upsilon}(\gamma^i_{t_i},{\gamma}^{\varepsilon}_{t_{\varepsilon}})\leq \frac{\varepsilon}{2^i}$ and $t_i\uparrow t_{\varepsilon}$ as $i\rightarrow\infty$,
	\vskip1mm
	\item[(ii)]  $\displaystyle (v^{\varepsilon}-\varphi-g_1)({\gamma}^{\varepsilon}_{t_{\varepsilon}})-\sum_{i=0}^{\infty}\frac{1}{2^i}\overline{\Upsilon}(\gamma^i_{t_i},{\gamma}^{\varepsilon}_{t_{\varepsilon}})\geq (v^{\varepsilon}-\varphi-g_1)(\gamma^0_{t_0})$, \quad and
		\vskip1mm
	\item[(iii)]  $\displaystyle (v^{\varepsilon}-\varphi-g_1)(\gamma_s)-\sum_{i=0}^{\infty}\frac{1}{2^i}\overline{\Upsilon}(\gamma^i_{t_i},\gamma_s)
	<(v^{\varepsilon}-\varphi-g_1)({\gamma}^{\varepsilon}_{t_{\varepsilon}})-\sum_{i=0}^{\infty}\frac{1}{2^i}\overline{\Upsilon}(\gamma^i_{t_i}, {\gamma}^{\varepsilon}_{t_{\varepsilon}})$ for all $(s,\gamma_s)\in [t_{\varepsilon},T]\times \Lambda^{t_{\varepsilon}}\setminus \{(t_{\varepsilon},{\gamma}^{\varepsilon}_{t_{\varepsilon}})\}$.
	
\end{description}
We claim that
\begin{eqnarray}\label{gamma}
d_\infty({\gamma}^{\varepsilon}_{{t}_{\varepsilon}},\hat{\gamma}_{\hat{t}})\rightarrow0  \ \ \mbox{as} \ \ \varepsilon\rightarrow0.
\end{eqnarray}
Otherwise,  by (\ref{s0}),  there is a constant  $\nu_0>0$
such that
$$
\overline{\Upsilon}({\gamma}^{\varepsilon}_{{t}_{\varepsilon}},\hat{\gamma}_{{\hat{t}}})
\geq\nu_0.
$$
Thus, {by the property (ii) of $(t_{\varepsilon},{\gamma}^{\varepsilon}_{t_{\varepsilon}})$,}  we obtain that
\begin{eqnarray*}
	0&=&(v- {{\varphi}}-g)(\hat{\gamma}_{\hat{t}})= \lim_{\varepsilon\rightarrow0}(v^\varepsilon-\varphi-g_1)(\hat{\gamma}_{\hat{t}})\\
	&\leq& \limsup_{\varepsilon\rightarrow0}\bigg{[}(v^{\varepsilon}-\varphi-g_1)({\gamma}^{\varepsilon}_{t_{\varepsilon}})-\sum_{i=0}^{\infty}\frac{1}{2^i}\overline{\Upsilon}(\gamma^i_{t_i},{\gamma}^{\varepsilon}_{t_{\varepsilon}})\bigg{]}\\
	 &\leq&\limsup_{\varepsilon\rightarrow0}{[}(v-{{\varphi}}-g)({\gamma}^{\varepsilon}_{{t}_{\varepsilon}})+(v^\varepsilon-v)({\gamma}^{\varepsilon}_{{t}_{\varepsilon}})
	{]}-\nu_0\leq (v- {{\varphi}}-g)(\hat{\gamma}_{\hat{t}})-\nu_0=-\nu_0,
\end{eqnarray*}
contradicting $\nu_0>0$.  We note that, {by  (\ref{03108}) and (\ref{03109}) together with the definition of $\Upsilon$ and the property (i) of $(t_{\varepsilon},{\gamma}^{\varepsilon}_{t_{\varepsilon}})$,
	there is a generic constant $C>0$ such that}
\begin{eqnarray*}
	2\sum_{i=0}^{\infty}\frac{1}{2^i}({t_{\varepsilon}}-{t}_{i})
	&\leq& 2\sum_{i=0}^{\infty}\frac{1}{2^i}\bigg{(}\frac{\varepsilon}{2^i}\bigg{)}^{\frac{1}{2}}\leq C\varepsilon^{\frac{1}{2}},\\[3mm]
	|{\partial_x{\Upsilon}}({\gamma}^{\varepsilon}_{{t}_{\varepsilon}}-\hat{\gamma}_{{\hat{t}},{t}_{\varepsilon}}^A)|
	&\leq& C|e^{({t}_{\varepsilon}-\hat{t}\,)A}\hat{\gamma}_{{\hat{t}}}(\hat{t}\,)-{\gamma}^{\varepsilon}_{{t}_{\varepsilon}}({t}_{\varepsilon})|^5,\\[3mm]
	|{\partial_{xx}{\Upsilon}}({\gamma}^{\varepsilon}_{{t}_{\varepsilon}}-\hat{\gamma}_{{\hat{t}},{t}_{\varepsilon}}^A)|
	&\leq&  C|e^{({t}_{\varepsilon}-\hat{t}\,)A}\hat{\gamma}_{{\hat{t}}}(\hat{t}\,)-{\gamma}^{\varepsilon}_{{t}_{\varepsilon}}({t}_{\varepsilon})|^4,
\end{eqnarray*}
\begin{eqnarray*}
	\bigg{|}\partial_x\sum_{i=0}^{\infty}\frac{1}{2^i}
	\Upsilon({\gamma}^{\varepsilon}_{t_{\varepsilon}}-(\gamma^i)_{t_i,t_{\varepsilon}}^A)
	\bigg{|}
	&\leq&18\sum_{i=0}^{\infty}\frac{1}{2^i}|e^{(t_{\varepsilon}-t_i)A}\gamma^i_{t_i}({t}_{i})-{\gamma}^{\varepsilon}_{t_{\varepsilon}}(t_{\varepsilon})|^5\\
	&\leq&18\sum_{i=0}^{\infty}\frac{1}{2^i}\bigg{(}\frac{\varepsilon}{2^i}\bigg{)}^{\frac{5}{6}}
	\leq C{\varepsilon}^{\frac{5}{6}},
\end{eqnarray*}
and
\begin{eqnarray*}
	\bigg{|}\partial_{xx}\sum_{i=0}^{\infty}\frac{1}{2^i}
	\Upsilon({\gamma}^{\varepsilon}_{t_{\varepsilon}}-(\gamma^i)_{t_i,t_{\varepsilon}}^A)
	\bigg{|}
	&\leq&306\sum_{i=0}^{\infty}\frac{1}{2^i}|e^{(t_{\varepsilon}-t_i)A}\gamma^i_{t_i}({t}_{i})-{\gamma}^{\varepsilon}_{t_{\varepsilon}}(t_{\varepsilon})|^4\\
	&\leq&306\sum_{i=0}^{\infty}\frac{1}{2^i}\bigg{(}\frac{\varepsilon}{2^i}\bigg{)}^{\frac{2}{3}}
	\leq C{\varepsilon}^{\frac{2}{3}}.
\end{eqnarray*}
Then for any $\varrho>0$, by (\ref{sss}) and (\ref{gamma}), there is sufficiently small $\varepsilon>0$  such that
$$
\hat{t}\leq {t}_{\varepsilon}< T,  \
2|{t}_{\varepsilon}-\hat{t}|+2\sum_{i=0}^{\infty}\frac{1}{2^i}({t_{\varepsilon}}-{t}_{i})+|\partial_t{\varphi}({\gamma}^{\varepsilon}_{{t}_{\varepsilon}})-\partial_t{\varphi}(\hat{\gamma}_{\hat{t}})|
+|{\partial_t^o}g({\gamma}^{\varepsilon}_{{t}_{\varepsilon}})-{ \partial_t^o}g(\hat{\gamma}_{\hat{t}})|\leq \frac{\varrho}{3},$$
$$
\left|\langle A^*\partial_x\varphi({\gamma}^{\varepsilon}_{{t}_{\varepsilon}}),  {\gamma}^{\varepsilon}_{{t}_{\varepsilon}}({t}_{\varepsilon})\rangle_H-\langle A^*\partial_x{\varphi}(\hat{\gamma}_{\hat{t}}),\hat{\gamma}_{\hat{t}}(\hat{t}\,)\rangle_H\right| \leq \frac{\varrho}{3}, \ \mbox{and}\  |I|\leq \frac{\varrho}{3},
$$
where
\begin{eqnarray*}
	I&:=&{\mathbf{H}}^{\varepsilon}({\gamma}^{\varepsilon}_{{t}_{\varepsilon}},
	\partial_x\varphi({\gamma}^{\varepsilon}_{{t}_{\varepsilon}})+ \partial_xg_2({\gamma}^{\varepsilon}_{{t}_{\varepsilon}}),\partial_{xx}\varphi({\gamma}^{\varepsilon}_{{t}_{\varepsilon}})+ \partial_{xx}g_2({\gamma}^{\varepsilon}_{{t}_{\varepsilon}}))\\[2mm]
	 &&-{\mathbf{H}}(\hat{\gamma}_{\hat{t}},\partial_x{\varphi}(\hat{\gamma}_{\hat{t}})+\partial_xg(\hat{\gamma}_{\hat{t}}),\partial_{xx}{\varphi}(\hat{\gamma}_{\hat{t}})+\partial_{xx}g(\hat{\gamma}_{\hat{t}})),\\
	 g_2({\gamma}^{\varepsilon}_{{t}_{\varepsilon}})&:=&g({\gamma}^{\varepsilon}_{{t}_{\varepsilon}})
+\overline{{\Upsilon}}({\gamma}^{\varepsilon}_{{t}_{\varepsilon}},\hat{\gamma}_{{\hat{t}}})
	+\sum_{i=0}^{\infty}\frac{1}{2^i}\overline{{\Upsilon}}({\gamma}^{\varepsilon}_{t_{\varepsilon}},\gamma^i_{t_i}),
\end{eqnarray*}
and
\begin{eqnarray*}
	&&{\mathbf{H}}^{\varepsilon}(\gamma_t,r,p,l)\\&:=&\sup_{u\in{
			{U}}}\left[
	\langle p, F^{\varepsilon}(\gamma_t,u)\rangle_{H}+\frac{1}{2}\mbox{Tr}(l G^{\varepsilon}(\gamma_t,u){G^{\varepsilon}}^*(\gamma_t,u))+q^{\varepsilon}(\gamma_t,r,p{G^{\varepsilon}}(\gamma_t,u),u)\right],  \\
	&&\qquad\qquad\qquad\qquad\qquad\qquad (t,\gamma_t,r,p,l)\in [0,T]\times {\Lambda}\times \mathbb{R}\times H\times {\mathcal{S}}(H).
\end{eqnarray*}
Since $v^{\varepsilon}$ is a viscosity {sub-solution} of PHJB equation (\ref{hjb1}) with coefficients $(F^{\varepsilon}, G^{\varepsilon}, $ $  q^{\varepsilon}, \phi^{\varepsilon})$, we have
\begin{eqnarray*}
	&&\partial_t\varphi({\gamma}^{\varepsilon}_{{t}_{\varepsilon}})+{ \partial_t^o}g_2({\gamma}^{\varepsilon}_{{t}_{\varepsilon}})
	+\langle A^*\partial_x\varphi({\gamma}^{\varepsilon}_{{t}_{\varepsilon}}),  {\gamma}^{\varepsilon}_{{t}_{\varepsilon}}({t}_{\varepsilon})\rangle_H\\
	&&+{\mathbf{H}}^{\varepsilon}({\gamma}^{\varepsilon}_{{t}_{\varepsilon}},
	\partial_x\varphi({\gamma}^{\varepsilon}_{{t}_{\varepsilon}})+\partial_xg_2({\gamma}^{\varepsilon}_{{t}_{\varepsilon}}),
	\partial_{xx}\varphi({\gamma}^{\varepsilon}_{{t}_{\varepsilon}})+\partial_{xx}g_2({\gamma}^{\varepsilon}_{{t}_{\varepsilon}}))\geq0.
\end{eqnarray*}
Thus
\begin{eqnarray*}
	0&\leq&  \partial_t{\varphi}({\gamma}^{\varepsilon}_{{t}_{\varepsilon}})+{ \partial_t^o}g({\gamma}^{\varepsilon}_{{t}_{\varepsilon}})
	+2({t}_{\varepsilon}-\hat{t}\,)+2\sum_{i=0}^{\infty}\frac{1}{2^i}({t_{\varepsilon}}-{t}_{i})
	+\langle A^*\partial_x\varphi({\gamma}^{\varepsilon}_{{t}_{\varepsilon}}), {\gamma}^{\varepsilon}_{{t}_{\varepsilon}}({t}_{\varepsilon})\rangle_H\\
	 &&+{\mathbf{H}}(\hat{\gamma}_{\hat{t}},\partial_x{\varphi}(\hat{\gamma}_{\hat{t}})+\partial_xg(\hat{\gamma}_{\hat{t}}),\partial_{xx}{\varphi}(\hat{\gamma}_{\hat{t}})+\partial_{xx}g(\hat{\gamma}_{\hat{t}}))+I\\[3mm]
	&\leq&\partial_t{\varphi}(\hat{\gamma}_{\hat{t}})+{\partial_t^o}g(\hat{\gamma}_{\hat{t}})+\langle A^*\partial_x{\varphi}(\hat{\gamma}_{\hat{t}}), \hat{\gamma}_{\hat{t}}(\hat{t}\,)\rangle_H\\
	&&
	 +{\mathbf{H}}(\hat{\gamma}_{\hat{t}},\partial_x{\varphi}(\hat{\gamma}_{\hat{t}})+\partial_xg(\hat{\gamma}_{\hat{t}}),\partial_{xx}{\varphi}(\hat{\gamma}_{\hat{t}})+\partial_{xx}g(\hat{\gamma}_{\hat{t}}))+\varrho.
\end{eqnarray*}
Letting $\varrho\downarrow 0$, we show that
\begin{eqnarray*}
	&&\partial_t{\varphi}(\hat{\gamma}_{\hat{t}})+{ \partial_t^o}g(\hat{\gamma}_{\hat{t}})+\langle A^*\partial_x{\varphi}(\hat{\gamma}_{\hat{t}}), \hat{\gamma}_{\hat{t}}(\hat{t}\,)\rangle_H\\
	 &+&{\mathbf{H}}(\hat{\gamma}_{\hat{t}},\partial_x{\varphi}(\hat{\gamma}_{\hat{t}})+\partial_xg(\hat{\gamma}_{\hat{t}}),\partial_{xx}{\varphi}(\hat{\gamma}_{\hat{t}})+\partial_{xx}g(\hat{\gamma}_{\hat{t}}))\geq0.
\end{eqnarray*}
Since ${\varphi}\in \Phi_{\hat{t}}$ and $g\in {\mathcal{G}}_{\hat{t}}$ with ${\hat{t}}\in [0,T)$  are arbitrary, we see that $v$ is a viscosity sub-solution of PHJB equation (\ref{hjb1}) with the coefficients $F, G, q, \phi$.  The proof  is complete.
\end{proof}

\newpage
\section{Viscosity solutions to  PHJB equations: Crandall-Ishii lemma}
\par
In this chapter,  we extend Crandall-Ishii lemma to functionals defined on the space of  $H$-valued continuous paths. 
It is the cornerstone of the theory of viscosity solutions, and is the key  to prove the comparison theorem  that will be given in the next chapter. \par
Let$\{e_i\}_{i\geq1}$ be an orthonormal basis of $H$ such that $e_i\in {\mathcal{D}}(A^*)$ for all $i\geq 1$. For every $N\geq 1$, we denote by $H_N$ the vector space generated by vectors $e_1,\ldots,e_N$, and by $P_N$  the orthogonal projection onto $H_N$. Define $Q_N:=I-P_N$.  We have
the orthogonal decomposition $H=H_N+H^\bot_N$ with $H^\bot_N:=Q_NH$. We denote by $x_N,y_N, \ldots$ points in $H_N$ and by $x^-_N,y^-_N,\ldots$ points in $H^\bot_N$, and write $x=(x_N,x^-_N)$ for $x\in H$.
\begin{definition}\label{definition0513}  For every $N\geq 1$, let $(\hat{t},\hat{x}_N)\in (0,T)\times H_N$ and $f:[0,T]\times H_N\rightarrow \mathbb{R}$
	be an upper semicontinuous function bounded from above. We say $f\in \Phi_N(\hat{t},\hat{x}_N)$ if there is a constant $r>0$ such that,
	for every $L>0$, there is  a constant {$\tilde{C}_0\geq0$} depending only on $L$ such that, for every  function $\varphi\in C^{1,2}([0,T]\times H_N)$ such that the function
	$$f(s,y_N)-\varphi(s,y_N), \quad [0,T]\times H_N$$
	 attains a  maximum  at a point $(t,x_N)\in (0,T)\times H_N$, 
	if   
	\begin{eqnarray*}
		|t-\hat{t}|+|x_N-\hat{x}_N|<r,\quad
		(|f|+|\nabla_x\varphi|
		+|\nabla^2_x\varphi|)(t,x_N)\leq L,
	\end{eqnarray*}
	then
	\begin{eqnarray}\label{0528110}
	{\varphi}_{t}(t,x_N) \geq -\tilde{C}_0.
	\end{eqnarray}
\end{definition}

\begin{definition}\label{definition0607}
	Let $\hat{t}\in [0,T)$ be fixed and  $w:\Lambda\rightarrow \mathbb{R}$ be an upper semicontinuous function bounded from above.
	For every $N\geq 1$, define,  for $(t,x_N)\in [0,T]\times H_N$,
	\begin{eqnarray*}
		&&\widetilde{w}^{\hat{t},N}(t,x_N):=\sup_{\scriptsize\begin{matrix}
				(\xi_t(t))_N=x_N\\ \xi_t\in \Lambda^{\hat{t}} \end{matrix}}
		w(\xi_t), \quad   t\in [\hat{t},T];\\
		&&\widetilde{w}^{\hat{t},N}(t,x_N):=\widetilde{w}^{\hat{t},N}(\hat{t},x_N)-(\hat{t}-t)^{\frac{1}{2}}, \ \   t\in[0,\hat{t}\,).
	\end{eqnarray*}
	Let $\widetilde{w}^{\hat{t},N,*}$ be the upper
	semi-continuous envelope of $\widetilde{w}^{\hat{t},N}$ (see    
	\cite[Definition D.10]{fab1}), i.e.,
	$$
	\widetilde{w}^{\hat{t},N,*}(t,x_N)=\limsup_{\scriptsize\begin{matrix} (s,y_N)\in [0,T]\times H_N\\ (s,y_N)\rightarrow(t,x_N)\end{matrix}}\widetilde{w}^{\hat{t},N}(s,y_N).
	$$
\end{definition}
In what follows, by a $modulus \ of \ continuity$, we mean a continuous function $\rho_1:[0,\infty)\rightarrow[0,\infty)$, with $\rho_1(0)=0$; by a $local\ modulus\ of $ $ continuity$, we mean  a continuous function $\rho_1:[0,\infty)\times[0,\infty) \rightarrow[0,\infty)$ such that the function $\rho_1(\cdot,r)$ is a modulus of continuity for each $r\geq0$ and $\rho_1$ is non-decreasing in the  second variable.
\begin{theorem}\label{theorem0513} (Crandall-Ishii lemma)\ \  Let $N\geq1$. Let both functionals $w_1,w_2:\Lambda\rightarrow \mathbb{R}$ be upper semi-continuous and bounded from above such that
	\begin{eqnarray}\label{05131}
	\limsup_{||\gamma_t||_0\rightarrow\infty}\frac{w_1(\gamma_t)}{||\gamma_t||_0}<0;\ \ \  \limsup_{||\gamma_t||_0\rightarrow\infty}\frac{w_2(\gamma_t)}{||\gamma_t||_0}<0.
	\end{eqnarray}
	Recalling  $\Lambda^t\otimes\Lambda^t:=\{(\gamma_s,\eta_s)|\gamma_s,\eta_s\in \Lambda^t\}$ for all $ t\in [0,T]$.
	Let $\varphi\in C^2( H_N\times H_N)$ be such that the functional
	$$
	w_1(\gamma_t)+w_2(\eta_t)-\varphi((\gamma_t(t))_N,(\eta_t(t))_N), \quad (\gamma_t, \eta_t)\in \Lambda^{\hat{t}}\otimes \Lambda^{\hat{t}}
	$$
	has a strict
	maximum at a point $(\hat{\gamma}_{\hat{t}},\hat{\eta}_{\hat{t}})$ with $\hat{t}\in (0,T)$ and there is a  strictly increasing function ${\tilde{\rho}}:[0,\infty)\rightarrow[0,\infty)$ with $\tilde{\rho}(0)=0$ such that, for all $\gamma_t, \eta_t\in \Lambda^{\hat{t}}$,
	\begin{eqnarray}\label{10101}
	\begin{aligned}
	&\quad w_1(\hat{\gamma}_{\hat{t}})+w_2(\hat{\eta}_{\hat{t}})-\varphi((\hat{\gamma}_{\hat{t}}(\hat{t}\,))_N,(\hat{\eta}_{\hat{t}}(\hat{t}\,))_N)\\
	&\geq w_1(\gamma_t)+w_2(\eta_t)-\varphi((\gamma_t(t))_N,(\eta_t(t))_N)\\
	&\quad+\tilde{\rho}(|t-\hat{t}|^2+||\gamma_t-\hat{\gamma}^A_{\hat{t},t}||^6_0+||\eta_t-\hat{\eta}^A_{\hat{t},t}||^6_0).
	\end{aligned}
	\end{eqnarray} 
	Assume, moreover, that $\widetilde{w}_{1}^{\hat{t},N,*}\in \Phi_N(\hat{t},(\hat{\gamma}_{\hat{t}}(\hat{t}\,))_N)$ and $\widetilde{w}_{2}^{\hat{t},N,*}\in \Phi_N(\hat{t},(\hat{\eta}_{\hat{t}}(\hat{t}\,))_N)$, and there is a local modulus of continuity  $\rho_1$  such that, for all  $\hat{t}\leq t\leq s\leq T, \ \gamma_t\in \Lambda $ and ${\gamma_{t,s}^{A,N}:=\gamma_{t,s}^A+((\gamma_t(t))_N-(e^{(s-t)A}\gamma_t(t))_N){\mathbf{1}}_{[0,s]}}$,
	\begin{eqnarray}\label{0608a}
	\begin{aligned}
	&
	w_1(\gamma_t)-w_1(\gamma_{t,s}^{A,N})\leq \rho_1(|s-t|,||\gamma_t||_0),\\
	&  w_2(\gamma_t)-w_2(\gamma_{t,s}^{A,N})\leq \rho_1(|s-t|,||\gamma_t||_0).
	\end{aligned}
	\end{eqnarray}
	Then for every $\kappa>0$, there exist a
	sequence  $(t_{k},\gamma^{k}_{t_k}; s_{k},\eta^{k}_{s_k})\in ([\hat{t},T]\times \Lambda^{\hat{t}})^2$ and a
	sequence of functionals $(\varphi_{1,k}, \psi_{1,k}, \varphi_{2,k}, \psi_{2,k})\in  \Phi_{\hat{t}}\times \Phi_{\hat{t}}\times{\mathcal{G}}_{t_k}\times {\mathcal{G}}_{s_k}$ bounded from below
	such that both functionals
	$$
	w_{1}(\gamma_t)-\varphi_{1,k}(\gamma_t)-\varphi_{2,k}(\gamma_t), \quad \gamma_t\in \Lambda^{t_k}
	$$
	and
	$$
	w_{2}(\eta_t)-\psi_{1,k}(\eta_{t})-\psi_{2,k}(\eta_{t}), \quad \eta_t\in \Lambda^{s_k}
	$$
	attain the strict  maximum $0$  at  $\gamma^{k}_{t_k}$ and $\eta^{k}_{s_k}$, respectively, and moreover,
	\begin{eqnarray}\label{0608v}
	&&\left(t_{k}, {\gamma}^{k}_{t_{k}}, w_1({\gamma}^{k}_{t_{k}}),(\partial_t\varphi_{1,k},\partial_x\varphi_{1,k},\partial_{xx}\varphi_{1,k})({\gamma}^{k}_{t_{k}}),
	(\partial_t^o\varphi_{2,k},\partial_x\varphi_{2,k},\partial_{xx}\varphi_{2,k})({\gamma}^{k}_{t_{k}})\right)\nonumber\\
	&&\underrightarrow{k\rightarrow\infty}\left({\hat{t}},\hat{{\gamma}}_{{\hat{t}}}, w_1(\hat{{\gamma}}_{{\hat{t}}}),(b_1, \nabla_{x_1}\varphi((\hat{{\gamma}}_{{\hat{t}}}({{\hat{t}}}))_N,(\hat{{\eta}}_{{\hat{t}}}({{\hat{t}}}))_N), X_N), (0,\mathbf{0},\mathbf{0})\right)
	\end{eqnarray}
	and
	\begin{eqnarray}\label{0608vw}
	&&\left(s_{k}, {\eta}^{k}_{s_{k}}, w_2({\eta}^{k}_{s_{k}}),(\partial_t\psi_{1,k},\partial_x\psi_{1,k},\partial_{xx}\psi_{1,k})({\eta}^{k}_{s_{k}}), ( \partial_t^o\psi_{2,k},\partial_x\psi_{2,k},\partial_{xx}\psi_{2,k})({\eta}^{k}_{s_{k}})\right)\nonumber\\
	&&\underrightarrow{k\rightarrow\infty}\left({\hat{t}},\hat{{\eta}}_{{\hat{t}}}, w_2(\hat{{\eta}}_{{\hat{t}}}), (b_2, \nabla_{x_2}\varphi((\hat{{\gamma}}_{{\hat{t}}}({{\hat{t}}}))_N,(\hat{{\eta}}_{{\hat{t}}}({{\hat{t}}}))_N), Y_N), (0,\mathbf{0},\mathbf{0})\right),
	\end{eqnarray}
	with $b_{1}+b_{2}=0$  and $X_N, Y_N\in \mathcal{S}(H_N)$ satisfying  the following inequality
	\begin{eqnarray}\label{II0615}
	-\left(\frac{1}{\kappa}+|A|\right)I\leq
	\left(\begin{matrix}
	X_N&0\\
	0&Y_N \end{matrix}\right)\leq A+\kappa A^2
	\end{eqnarray}
for $A:=\nabla^2_{x}\varphi((\hat{{\gamma}}_{{\hat{t}}}({{\hat{t}}}))_N,(\hat{{\eta}}_{{\hat{t}}}({{\hat{t}}}))_N). $
Here,  $\nabla^2_{x}\varphi$  is  the Hessian of $\varphi$
	with respect to  the  variable $x=(x_1,x_2)\in H_N\times H_N$, and $\nabla_{x_1}\varphi$ and $\nabla_{x_2}\varphi$ are  gradients of $\varphi$
	with respect to  the first and second variables, respectively.  
\end{theorem}

\begin{proof}
From Lemma \ref{lemma4.30615}, we know that the function
\begin{eqnarray}\label{05211}
\begin{aligned}
&\widetilde{w}_{1}^{\hat{t},N,*}(t,x)+ \widetilde{w}_{2}^{\hat{t},N,*}(t,y)-\varphi(x,y), \quad
(t,x,y) \in [0,T]\times H_N\times H_N \\
& \mbox{ is maximized at} \ (\hat{t},(\hat{{\gamma}}_{{\hat{t}}}({{\hat{t}}}))_N,(\hat{{\eta}}_{{\hat{t}}}({{\hat{t}}}))_N).
\end{aligned}
\end{eqnarray}
Moreover, we have
$\widetilde{w}_{1}^{\hat{t},N,*}(\hat{t},(\hat{{\gamma}}_{{\hat{t}}}({{\hat{t}}}))_N)={w}_{1}(\hat{{\gamma}}_{{\hat{t}}})$, $\widetilde{w}_{2}^{\hat{t},N,*}(\hat{t},(\hat{{\eta}}_{{\hat{t}}}({{\hat{t}}}))_N)={w}_{2}(\hat{{\eta}}_{{\hat{t}}})$.
Then, by $(\widetilde{w}_{1}^{\hat{t},N,*}, \, \widetilde{w}_{2}^{\hat{t},N,*})\in \Phi_N(\hat{t},(\hat{\gamma}_{\hat{t}}(\hat{t}\,))_N) \times \Phi_N(\hat{t},(\hat{\eta}_{\hat{t}}(\hat{t}\,))_N)$ and Remark \ref{remarkv0528}, Theorem 8.3 in \cite{cran2} and
Lemma 5.4 of Chapter 4 in \cite{yong} yield the existence of  sequences of   functions
$\tilde{\varphi}_k,\tilde{\psi}_k\in C^{1,2}([0,T]\times  H_N)$  bounded from below such that both functions
$$\widetilde{w}_{1}^{\hat{t},N,*}(t, x)-\tilde{\varphi}_k(t,x), \quad (t,x)\in [0,T]\times H_N$$
and
$$\widetilde{w}_{2}^{\hat{t},N,*}(s,y)-\tilde{\psi}_k(s,y), \quad (s,y)\in [0,T]\times H_N$$
 attain a strict  maximum $0$ at some points  $(t_k,x^{k})\in (0,T)\times  H_N$  and $(s_k,y^{k})\in (0,T)\times  H_N$, respectively, and moreover,
\begin{eqnarray}\label{0607a}
&&\lim_{k\to \infty}\left(t_{k}, x^k, \widetilde{w}_{1}^{\hat{t},N,*}(t_{k}, x^k), (\tilde{\varphi}_k)_t(t_{k}, x^k),\nabla_x\tilde{\varphi}_k(t_{k}, x^k),\nabla^2_x\tilde{\varphi}_k(t_{k}, x^k)\right)\nonumber\\
&&\qquad=\left({\hat{t}},(\hat{{\gamma}}_{{\hat{t}}}(\hat{t}\,))_N, w_1(\hat{{\gamma}}_{{\hat{t}}}),b_1, \nabla_{x_1}\varphi((\hat{{\gamma}}_{{\hat{t}}}({{\hat{t}}}))_N,(\hat{{\eta}}_{{\hat{t}}}({{\hat{t}}}))_N), X_N\right)
\end{eqnarray}
and
\begin{eqnarray}\label{0607b}
&&\lim_{k\to \infty}\left(s_{k}, y^k, \widetilde{w}_{2}^{\hat{t},N,*}(s_{k}, y^k),(\tilde{\psi}_k)_t(s_{k}, y^k),\nabla_x\tilde{\psi}_k(s_{k}, y^k),\nabla^2_x\tilde{\psi}_k(s_{k}, y^k)\right)\nonumber\\
&&\qquad =\left({\hat{t}},\hat{{\eta}}_{{\hat{t}}}(\hat{t}\,), w_2(\hat{{\eta}}_{{\hat{t}}}),b_2, \nabla_{x_2}\varphi((\hat{{\gamma}}_{{\hat{t}}}({{\hat{t}}}))_N,(\hat{{\eta}}_{{\hat{t}}}({{\hat{t}}}))_N), Y_N\right),
\end{eqnarray}
with $b_{1}+b_{2}=0$ and the inequality~(\ref{II0615}) being satisfied.

We assert that   $(t_{k}, s_{k})\in [\hat{t},T)\times [\hat{t},T)$ for all $k\ge 1$. Otherwise, for example, there exists a subsequence of $\{t_{k}\}_{k\geq1}$ still denoted by itself such that $t_{k}<\hat{t}$ for all $k\geq1$.
Since the function
$$\widetilde{w}_{1}^{\hat{t},N,*}(t, x)-\tilde{\varphi}_k(t,x), \quad (t,x)\in [0,T]\times  H_N $$
attains the maximum at $(t_{k},x^{k})$, we obtain that
$$
(\tilde{\varphi}_k)_t(t_{k},x^{k})={\frac{1}{2}}(\hat{t}-t_{k})^{-\frac{1}{2}}\rightarrow\infty,\ \mbox{as}\ k\rightarrow\infty,
$$
which  contradicts  that $(\tilde{\varphi}_k)_t(t_{k},x^{k})\rightarrow b_{1}\in \mathbb{R}$. 
\par
Now we consider the functionals,
for $(t,\gamma_t), (s,\eta_s)\in [\hat{t},T]\times{{\Lambda}}$,
\begin{eqnarray}\label{4.11116}
\quad\quad\quad       \Gamma^1_{k}(\gamma_t):=w_{1}(\gamma_t)-\tilde{\varphi}_k(t,(\gamma_t(t))_N),\  \Gamma^2_{k}(\eta_s):= w_{2}(\eta_s)
-\tilde{\psi}_k(s,(\eta_s(s))_N).
\end{eqnarray}
Clearly,  both functionals  $\Gamma^1_{k}$ and $\Gamma^2_{k}$ are upper semi-continuous and  bounded from above on ${\Lambda}^{\hat{t}}$.  Consider the sequence of positive numbers $\{\delta_i\}_{i\geq0}$  by
$\delta_i=\frac{1}{2^i}$ for all $i\geq0$. Since
$\overline{\Upsilon}(\cdot,\cdot)$ is a gauge-type function on $(\Lambda^{\hat{t}}, d_{\infty})$, for every $k$ and $j>0$,
applying  Lemma \ref{theoremleft} to $\Gamma^1_{k}$ and  $\Gamma^2_{k}$, respectively, we have that,
for every  $(\check{t}_{0},\check{\gamma}^{0}_{\check{t}_{0}}; \check{s}_{0},\check{\eta}^{0}_{\check{s}_{0}})\in ([\hat{t},T]\times  \Lambda^{\hat{t}})^2$ satisfying
\begin{eqnarray}\label{20210509a}
\begin{aligned}
&
\Gamma^1_{k}(\check{\gamma}^{0}_{\check{t}_{0}})\geq \sup_{(t,\gamma_t)\in [\hat{t},T]\times \Lambda^{\hat{t}}}\Gamma^1_{k}(\gamma_t)-\frac{1}{j},\\
&\Gamma^2_{k}(\check{\eta}^{0}_{\check{s}_{0}})\geq \sup_{(s,\eta_s)\in [\hat{t},T]\times \Lambda^{\hat{t}}}\Gamma^2_{k}(\eta_s)-\frac{1}{j},
\end{aligned}
\end{eqnarray}
there exist $(t_{k,j},{\gamma}^{k,j}_{t_{k,j}}; s_{k,j},{\eta}^{k,j}_{s_{k,j}})\in ([\hat{t},T]\times \Lambda^{\hat{t}})^2,\,  j\ge 1,$ and two sequences $\{(\check{t}_{i},\check{\gamma}^{i}_{\check{t}_{i}})\}_{i\geq1}\subset
[\check{t}_{0},T]\times \Lambda^{\hat{t}}$, $\{(\check{s}_{i},\check{\eta}^{i}_{\check{s}_{i}})\}_{i\geq1}\subset
[\check{s}_{0},T]\times \Lambda^{\hat{t}}$ such that
\begin{description}
	\item[(i)] $\displaystyle \overline{\Upsilon}(\check{\gamma}^{0}_{\check{t}_{0}},{\gamma}^{k,j}_{t_{k,j}})\vee\overline{\Upsilon}(\check{\eta}^0_{\check{s}_0},{\eta}^{k,j}_{s_{k,j}})\leq \frac{1}{j},\,\,
	\overline{\Upsilon}(\check{\gamma}^{i}_{\check{t}_{i}},{\gamma}^{k,j}_{t_{k,j}})
	\vee\overline{\Upsilon}(\check{\eta}^{i}_{\check{s}_{i}},{\eta}^{k,j}_{s_{k,j}})\leq \frac{1}{2^ij}$
	and $\displaystyle \check{t}_{i}\uparrow t_{k,j}$, $\check{s}_{i}\uparrow s_{k,j}$ as $i\rightarrow\infty$,
	\item[(ii)]  $\displaystyle \Gamma^1_k({\gamma}^{k,j}_{t_{k,j}})
	-\sum_{i=0}^{\infty}\frac{1}{2^i}\overline{\Upsilon}(\check{\gamma}^{i}_{\check{t}_{i}},{\gamma}^{k,j}_{t_{k,j}})
	\geq \Gamma^1_{k}(\check{\gamma}^{0}_{\check{t}_{0}})$,\\
	\ \ \ $\displaystyle \Gamma^2_k({\eta}^{k,j}_{s_{k,j}})
	-\sum_{i=0}^{\infty}\frac{1}{2^i}\overline{\Upsilon}(\check{\eta}^{i}_{\check{s}_{i}},{\eta}^{k,j}_{s_{k,j}})\geq \Gamma^2_{k}(\check{\eta}^{0}_{\check{s}_{0}})$,
	and
	\item[(iii)]    for all $(t,\gamma_t)\in [t_{k,j},T]\times \Lambda^{t_{k,j}}\setminus \{(t_{k,j},{\gamma}^{k,j}_{t_{k,j}})\}$,
	\begin{eqnarray*}
		\Gamma^1_{k}(\gamma_t)
		-\sum_{i=0}^{\infty}
		\frac{1}{2^i}\overline{\Upsilon}(\check{\gamma}^{i}_{\check{t}_{i}},\gamma_t)
		<\Gamma^1_{k}({\gamma}^{k,j}_{t_{k,j}})
		-\sum_{i=0}^{\infty}\frac{1}{2^i}\overline{\Upsilon}(\check{\gamma}^{i}_{\check{t}_{i}},{\gamma}^{k,j}_{t_{k,j}}),
	\end{eqnarray*}
	and for all $(s,\eta_s)\in [s_{k,j},T]\times \Lambda^{s_{k,j}}\setminus \{(s_{k,j},{\eta}^{k,j}_{s_{k,j}})\}$,
	\begin{eqnarray*}
		\Gamma^2_{k}(\eta_s)
		-\sum_{i=0}^{\infty}
		\frac{1}{2^i}\overline{\Upsilon}(\check{\eta}^{i}_{\check{s}_{i}},\eta_s)
		<\Gamma^2_{k}({\eta}^{k,j}_{s_{k,j}})
		-\sum_{i=0}^{\infty}\frac{1}{2^i}\overline{\Upsilon}(\check{\eta}^{i}_{\check{s}_{i}},{\eta}^{k,j}_{s_{k,j}}).
	\end{eqnarray*}
\end{description}
By Lemma \ref{lemma4.40615} below, we have
\begin{eqnarray}\label{4.226}
\lim_{j\to \infty}\left(\begin{aligned}
&
t_{k,j},  & ({\gamma}^{k,j}_{t_{k,j}}(t_{k,j}))_N\\
& s_{k,j},  & ({\eta}^{k,j}_{s_{k,j}}(s_{k,j}))_N
\end{aligned}\right)
=\left(\begin{aligned}
&t_k, &x^k\\
&s_k, &y^k
\end{aligned}\right),
\end{eqnarray}
\begin{eqnarray}\label{05231}
\lim_{j\to \infty}\left(\begin{aligned}
&
\widetilde{w}_{1}^{\hat{t},N,*}(t_{k,j}, ({\gamma}^{k,j}_{t_{k,j}}(t_{k,j}))_N)\\
&\widetilde{w}_{2}^{\hat{t},N,*}(s_{k,j},  ({\eta}^{k,j}_{s_{k,j}}(s_{k,j}))_N)
\end{aligned}\right)
=\left(\begin{aligned}
&
\widetilde{w}_{1}^{\hat{t},N,*}(t_k,x^k)\\
&\widetilde{w}_{2}^{\hat{t},N,*}(s_k,y^k)
\end{aligned}\right),
\end{eqnarray}
and
\begin{eqnarray}\label{05232}
\lim_{j\to \infty} \left(\begin{aligned}
&w_1({\gamma}^{k,j}_{t_{k,j}})\\
 &w_2({\eta}^{k,j}_{s_{k,j}})
\end{aligned}\right)
=\left(\begin{aligned}
&\widetilde{w}_{1}^{\hat{t},N,*}(t_k,x^k)\\
&\widetilde{w}_{2}^{\hat{t},N,*}(s_k,y^k)
\end{aligned}\right).
\end{eqnarray}
In view of  (\ref{0607a}) and (\ref{0607b}),  we  have a subsequence $j_k$ such that
\begin{eqnarray*}
	&&\left(t_{k,j_k}, ({\gamma}^{k,j_k}_{t_{k,j_k}}(t_{k,j_k}))_N, w_1({\gamma}^{k,j_k}_{t_{k,j_k}}),((\tilde{\varphi}_k)_t,\nabla_x\tilde{\varphi}_k,\nabla^2_x\tilde{\varphi}_k)(t_{k,j_k}, {\gamma}^{k,j_k}_{t_{k,j_k}}(t_{k,j_k}))\right)\\
	&&\underrightarrow{k\rightarrow\infty}\left({\hat{t}},(\hat{{\gamma}}_{{\hat{t}}}(\hat{t}\,))_N, w_1(\hat{{\gamma}}_{{\hat{t}}}),(b_1, \nabla_{x_1}\varphi((\hat{{\gamma}}_{{\hat{t}}}({{\hat{t}}}))_N,(\hat{{\eta}}_{{\hat{t}}}({{\hat{t}}}))_N), X)\right),
\end{eqnarray*}
\begin{eqnarray*}
	&&\left(s_{k,j_k}, ({\eta}^{k,j_k}_{s_{k,j_k}}(s_{k,j_k}))_N, w_2({\eta}^{k,j_k}_{s_{k,j_k}}),((\tilde{\psi}_k)_t,\nabla_x\tilde{\psi}_k,\nabla^2_x\tilde{\psi}_k)(s_{k,j_k}, {\eta}^{k,j_k}_{s_{k,j_k}}(s_{k,j_k}))\right)\\
	&&\underrightarrow{k\rightarrow\infty}\left({\hat{t}},(\hat{{\eta}}_{{\hat{t}}}(\hat{t}\,))_N, w_2(\hat{{\eta}}_{{\hat{t}}}),(b_2, \nabla_{x_2}\varphi((\hat{{\gamma}}_{{\hat{t}}}({{\hat{t}}}))_N,(\hat{{\eta}}_{{\hat{t}}}({{\hat{t}}}))_N), Y)\right).
\end{eqnarray*}
It remains to show that $ {\gamma}^{k,j_k}_{t_{k,j_k}}\rightarrow\hat{{\gamma}}_{{\hat{t}}}$ and ${\eta}^{k,j_k}_{s_{k,j_k}}\rightarrow\hat{{\eta}}_{{\hat{t}}}$ as $k\rightarrow\infty$.  Noting that $({\gamma}^{k,j_k}_{t_{k,j_k}}(t_{k,j_k}))_N\rightarrow (\hat{{\gamma}}_{{\hat{t}}}(\hat{t}\,))_N$ and $({\eta}^{k,j_k}_{s_{k,j_k}}(s_{k,j_k}))_N\rightarrow (\hat{{\eta}}_{{\hat{t}}}(\hat{t}\,))_N$  as $k\rightarrow\infty$, by the continuity of $\varphi$, there exists a constant {$\tilde{M}_1>0$} such that
$$|\varphi(({\gamma}^{k,j_k}_{t_{k,j_k}}(t_{k,j_k}))_N,({\eta}^{k,j_k}_{s_{k,j_k}}(s_{k,j_k}))_N)|\leq \tilde{M}_1.$$
Then by (\ref{05131}) and
\begin{eqnarray}\label{1004a}
\begin{aligned}
&w_1({\gamma}^{k,j_k}_{t_{k,j_k}})+w_2({\eta}^{k,j_k}_{s_{k,j_k}})-\varphi(({\gamma}^{k,j_k}_{t_{k,j_k}})(t_{k,j_k}))_N,
({\eta}^{k,j_k}_{s_{k,j_k}}(s_{k,j_k}))_N)\\
\rightarrow&
w_1(\hat{{\eta}}_{{\hat{t}}})+w_2(\hat{{\eta}}_{{\hat{t}}})-\varphi((\hat{{\gamma}}_{{\hat{t}}}(\hat{t}\,))_N,(\hat{{\eta}}_{{\hat{t}}}(\hat{t}\,))_N),
\end{aligned}
\end{eqnarray}
there exists a constant $M_2>0$ such that
$$
||{\gamma}^{k,j_k}_{t_{k,j_k}}||_0 \vee     ||{\eta}^{k,j_k}_{s_{k,j_k}}||_0\leq M_2.
$$
Without loss of generality, we may assume $s_{k,j_k}\leq t_{k,j_k}$, by (\ref{0608a}),
\begin{eqnarray*}
	&& w_1({\gamma}^{k,j_k}_{t_{k,j_k}})+w_2({\eta}^{k,j_k}_{s_{k,j_k}})-\varphi(({\gamma}^{k,j_k}_{t_{k,j_k}})(t_{k,j_k}))_N,
	({\eta}^{k,j_k}_{s_{k,j_k}}(s_{k,j_k}))_N)\\
	&\leq& w_1({\gamma}^{k,j_k}_{t_{k,j_k}})+w_2(({\eta}^{k,j_k})_{s_{k,j_k},t_{k,j_k}}^{A,N})-\varphi(({\gamma}^{k,j_k}_{t_{k,j_k}})(t_{k,j_k}))_N,
	({\eta}^{k,j_k}_{s_{k,j_k}}(s_{k,j_k}))_N)\\
	&&+\rho_1(|s_{k,j_k}-t_{k,j_k}|,M_2).
\end{eqnarray*}
Letting $k\rightarrow\infty$, by (\ref{1004a}),
\begin{eqnarray*}
	&& \liminf_{k\rightarrow\infty} \left[w_1({\gamma}^{k,j_k}_{t_{k,j_k}})+w_2({({\eta}^{k,j_k})_{s_{k,j_k},t_{k,j_k}}^{A,N}})-\varphi(({\gamma}^{k,j_k}_{t_{k,j_k}})(t_{k,j_k}))_N,
	({\eta}^{k,j_k}_{s_{k,j_k}}(s_{k,j_k}))_N)\right]\\
	&&\geq w_1(\hat{{\eta}}_{{\hat{t}}})+w_2(\hat{{\eta}}_{{\hat{t}}})-\varphi((\hat{{\gamma}}_{{\hat{t}}}(\hat{t}\,))_N,(\hat{{\eta}}_{{\hat{t}}}(\hat{t}\,))_N).
\end{eqnarray*}
Thus, by (\ref{10101}) 
we get that
$ {\gamma}^{k,j_k}_{t_{k,j_k}}\rightarrow\hat{{\gamma}}_{{\hat{t}}}$ and
${({\eta}^{k,j_k})_{s_{k,j_k},t_{k,j_k}}^{A,N}}\rightarrow\hat{{\eta}}_{{\hat{t}}}$ as $k\rightarrow\infty$. It is clear that
${\eta}^{k,j_k}_{s_{k,j_k}}\rightarrow\hat{{\eta}}_{{\hat{t}}}$ as $k\rightarrow\infty$.\par
We notice that, by   (\ref{220817a}), (\ref{220817a1})   
and the property (i) of $(t_{k,j}, {\gamma}^{k,j}_{t_{k,j}}$, $s_{k,j}, {\eta}^{k,j}_{s_{k,j}})$, there exists a generic constant $C>0$ such that
\begin{eqnarray*}
	2\sum_{i=0}^{\infty}\frac{1}{2^i}\left[(s_{k,j_k}-\check{s}_{i})+(t_{k,j_k}-\check{t}_{i})\right]
	\leq Cj_k^{-\frac{1}{2}};
\end{eqnarray*}
\begin{eqnarray*}
	\left|\partial_x\left[\sum_{i=0}^{\infty}\frac{1}{2^i}
	\Upsilon({\gamma}^{k,j_k}_{t_{k,j_k}}-\check{\gamma}^{i}_{\check{t}_{i},t_{k,j_k}})\right]\right|
	+\left|\partial_x\left[\sum_{i=0}^{\infty}\frac{1}{2^i}
	\Upsilon({\eta}^{k,j_k}_{s_{k,j_k}}-\check{\eta}^{i}_{\check{s}_{i},s_{k,j_k}})\right]\right|
	\leq Cj_k^{-\frac{5}{6}};
\end{eqnarray*}
and
\begin{eqnarray*}
	\left|\partial_{xx}\left[\sum_{i=0}^{\infty}\frac{1}{2^i}
	\Upsilon({\gamma}^{k,j_k}_{t_{k,j_k}}-\check{\gamma}^{i}_{\check{t}_{i},t_{k,j_k}})\right]\right|
	+\left|\partial_{xx}\left[\sum_{i=0}^{\infty}\frac{1}{2^i}
	\Upsilon({\eta}^{k,j_k}_{s_{k,j_k}}-\check{\eta}^{i}_{\check{s}_{i},s_{k,j_k}})\right]\right|
	\leq Cj_k^{-\frac{2}{3}}.
\end{eqnarray*}
Therefore the lemma holds with
\begin{eqnarray*}
	\varphi_{1,k}(\gamma_t)&:=&\tilde{\varphi}_k(t,(\gamma_t(t))_N)
	-\tilde{\varphi}_k(t_{k,j_k},({\gamma}^{k,j_k}_{t_{k,j_k}}(t_{k,j_k}))_N)
	+w_1({\gamma}^{k,j_k}_{t_{k,j_k}}), \\ \varphi_{2,k}(\gamma_t)&:=&\sum_{i=0}^{\infty}\frac{1}{2^i}\overline{\Upsilon}(\check{\gamma}^{i}_{\check{t}_{i}},\gamma_t)
	-\sum_{i=0}^{\infty}\frac{1}{2^i}\overline{\Upsilon}(\check{\gamma}_{\check{t}_{i}}^{i},{\gamma}^{k,j_k}_{t_{k,j_k}}),\\
	\psi_{1,k}(\eta_s)&:=&\tilde{\psi}_k(s,(\eta_s(s))_N)
	-\tilde{\psi}_k(s_{k,j_k},({\eta}^{k,j_k}_{s_{k,j_k}}(s_{k,j_k}))_N)+w_2({\eta}^{k,j_k}_{s_{k,j_k}}),\\
	\psi_{2,k}(\eta_s)&:=&\sum_{i=0}^{\infty}\frac{1}{2^i}\overline{\Upsilon}(\check{\eta}^{i}_{\check{s}_{i}},\eta_s)
	-\sum_{i=0}^{\infty}\frac{1}{2^i}\overline{\Upsilon}(\check{\eta}_{\check{s}_{i}}^{i},{\eta}^{k,j_k}_{s_{k,j_k}})
\end{eqnarray*}
and
\begin{eqnarray*}
	t_{k}:=t_{k,j_k}, {\gamma}^{k}_{t_k}:={\gamma}^{k,j_k}_{t_{k,j_k}}, s_{k}:=s_{k,j_k}, {\eta}^{k}_{s_k}:={\eta}^{k,j_k}_{s_{k,j_k}}.
\end{eqnarray*}
The proof is now complete.
 \end{proof}

\begin{remark}\label{remarkv0528}
	As mentioned in Remark 6.1 in Chapter V of \cite{fle1},  Condition (\ref{0528110}) is stated with reverse inequality in Theorem 8.3 of \cite{cran2}. However, we  immediately obtain  results (\ref{0608v})-(\ref{II0615}) from Theorem
	8.3 of \cite{cran2} by considering the functions $u_1(t,x) :=\widetilde{w}_{1}^{\hat{t},N,*}(T-t,x)$ and
	$u_2(t,x) := \widetilde{w}_{2}^{\hat{t},N,*}(T-t,x)$.
\end{remark}

\par
To complete the  proof of Theorem \ref{theorem0513}, it remains to state and prove the following  two lemmas.
\begin{lemma}\label{lemma4.30615}\ \ Let all the conditions in Theorem  \ref{theorem0513} be satisfied. Recall that $(\hat{t},\hat{{\gamma}}_{{\hat{t}}},\hat{{\eta}}_{{\hat{t}}})$ is given  in Theorem  \ref{theorem0513},  $\widetilde{w}_{1}^{\hat{t},N,*}$ and $ \widetilde{w}_{2}^{\hat{t},N,*}$ are defined in Definition \ref{definition0607} with respect to $w_1$ and $w_2$ given in Theorem  \ref{theorem0513}, respectively.
	Then the function:
	$$\widetilde{w}_{1}^{\hat{t},N,*}(t,x)+ \widetilde{w}_{2}^{\hat{t},N,*}(t,y)-\varphi(x,y),   \quad (t,x,y)\in [0,T]\times  H_N\times  H_N $$
attains the maximum  at $(\hat{t},(\hat{{\gamma}}_{{\hat{t}}}({{\hat{t}}}))_N,(\hat{{\eta}}_{{\hat{t}}}({{\hat{t}}}))_N)$.
	Moreover, we have
	\begin{eqnarray}\label{2020020206}
	\widetilde{w}_{1}^{\hat{t},N,*}(\hat{t},(\hat{{\gamma}}_{{\hat{t}}}({{\hat{t}}}))_N)={w}_{1}(\hat{{\gamma}}_{{\hat{t}}}), \quad \widetilde{w}_{2}^{\hat{t},N,*}(\hat{t},(\hat{{\eta}}_{{\hat{t}}}({{\hat{t}}}))_N)={w}_{2}(\hat{{\eta}}_{{\hat{t}}}).
	\end{eqnarray}
\end{lemma}

\begin{proof}
For every $\hat{t}\leq t\leq s\leq T$ and $x\in   H_N$,
from the definition of $\widetilde{w}^{\hat{t},N}_1$ 
it follows that
\begin{eqnarray}\label{220831}
\begin{aligned}
&\quad\widetilde{w}^{\hat{t},N}_1(t,x)-\widetilde{w}^{\hat{t},N}_1(s,x)\\
&=\sup_{\scriptsize\begin{matrix}\gamma_t\in \Lambda^{\hat{t}}\\
	(\gamma_t(t))_N=x\end{matrix}}
w_{1}(\gamma_t)
\,\,-\sup_{\scriptsize\begin{matrix}\eta_s\in \Lambda^{\hat{t}}\\
	(\eta_s(s))_N=x\end{matrix}}
w_{1}(\eta_s).
\end{aligned}
\end{eqnarray}
By (\ref{05131}), there exist constants {$M_3>0$} and $\varepsilon>0$ such that
$$
\frac{w_1(\gamma_t)}{||\gamma_t||_0}\leq -\varepsilon,\ \ \mbox{if}\ ||\gamma_t||_0 \geq M_3, \ \gamma_t\in \Lambda^{\hat{t}}.
$$
Thus,
\begin{eqnarray}\label{07061}
w_1(\gamma_t)\leq -\varepsilon||\gamma_t||_0,\ \ \mbox{if}\ ||\gamma_t||_0 \geq M_3, \ \gamma_t\in \Lambda^{\hat{t}}.
\end{eqnarray}
For every $t\in[\hat{t},T]$, define $\hat{\xi}_t\in \Lambda^{\hat{t}}$ by
{$$
\hat{\xi}_t(l)=x{\mathbf{1}}_{[0,\hat{t}]}(l)+e^{(l-\hat{t}\,)A}x{\mathbf{1}}_{[\hat{t},t]}(l)+(x-(e^{(t-\hat{t}\,)A}x)_N){\mathbf{1}}_{[0,t]}(l),\ \ l\in [0,t], 
$$}
then, by (\ref{0608a}),
\begin{eqnarray*}
	w_1(\hat{\xi}_{\hat{t}})-\sup_{l\in[\hat{t},T]}\rho_1(|l-\hat{t}|,|x|)\leq w_1(\hat{\xi}_{t})\leq \sup_{\scriptsize\begin{matrix}\xi_t\in \Lambda^{\hat{t}}\\
			(\xi_t(t))_N=x\end{matrix}}w_{1}(\xi_t).
\end{eqnarray*}
Notice that $w_1(\hat{\xi}_{\hat{t}})-\sup_{l\in[\hat{t},T]}\rho_1(|l-\hat{t}|,|x|)$ depends only on $x$, there exists a constant $C^1_{x}>0$ depending only on   $x$ such that, for all $(t,\gamma_t)\in [\hat{t},T]\times\Lambda^{\hat{t}}$ satisfying $||\gamma_t||_0\geq C^1_{x}$,
\begin{eqnarray}\label{07062}
-\varepsilon||\gamma_t||_0< w_1(\hat{\xi}_{\hat{t}})-\sup_{l\in[\hat{t},T]}\rho_1(|l-\hat{t}|,|x|)\leq \sup_{\scriptsize\begin{matrix}\xi_t\in \Lambda^{\hat{t}}\\
	(\xi_t(t))_N=x\end{matrix}}w_{1}(\xi_t).
\end{eqnarray}
Taking $C_{x}=M_3\vee C^1_{x}$, by (\ref{07061}) and (\ref{07062}),
\begin{eqnarray}\label{220831a}
w_1(\gamma_t)<\sup_{\scriptsize\begin{matrix}\xi_t\in \Lambda^{\hat{t}}\\
	(\xi_t(t))_N=x\end{matrix}}w_{1}(\xi_t),\ \ \mbox{if} \ ||\gamma_t||_0\geq C_{x}, \ \gamma_t\in \Lambda^{\hat{t}}.
\end{eqnarray}
Combining (\ref{220831}) and (\ref{220831a}),  from (\ref{0608a}) we have
\begin{eqnarray}\label{202105083}
\begin{aligned}
&\quad \widetilde{w}^{\hat{t},N}_1(t,x)-\widetilde{w}^{\hat{t},N}_1(s,x)\\
&=\sup_{\scriptsize\begin{matrix}\gamma_t\in \Lambda^{\hat{t}}\\
	(\gamma_t(t))_N=x\\
	||\gamma_t||_0\leq C_x\end{matrix}}\!\!\!\!\!
w_{1}(\gamma_t)
\,\, -\sup_{\scriptsize\begin{matrix}\eta_s\in \Lambda^{\hat{t}}\\
	(\eta_s(s))_N=x\end{matrix}}\!\!\!\!\!
w_{1}(\eta_s)
\\
\qquad &\leq \sup_{\scriptsize\begin{matrix}\gamma_t\in \Lambda^{\hat{t}}\\
	(\gamma_t(t))_N=x\\
	||\gamma_t||_0\leq C_x\end{matrix}}\!\!\!\!\!
{[}w_{1}(\gamma_t)-w_{1}(\gamma_{t,s}^A+((\gamma_t(t))_N-(e^{(s-t)A}\gamma_t(t))_N){\mathbf{1}}_{[0,s]})
{]}\\
&\leq  \sup_{\scriptsize\begin{matrix}\gamma_t\in \Lambda^{\hat{t}}\\
	(\gamma_t(t))_N=x\\
	||\gamma_t||_0\leq C_x\end{matrix}}    \!\!\! \!\!\rho_1(|s-t|,\, ||\gamma_t||_0)
\, \leq \, \rho_1(|s-t|,\, C_x).
\end{aligned}
\end{eqnarray}
Clearly, if $0\leq t\leq s\leq \hat{t}$, we have
\begin{eqnarray}\label{202105084}
\widetilde{w}^{\hat{t},N}_1(t,x)-\widetilde{w}^{\hat{t},N}_1(s,x)=-(\hat{t}-t)^{\frac{1}{2}}+(\hat{t}-s)^{\frac{1}{2}}\leq 0,
\end{eqnarray}
and, if $0\leq t\leq  \hat{t}\leq s\leq T$ , we have
\begin{eqnarray}\label{202105085}
\ \ \                \widetilde{w}^{\hat{t},N}_1(t,x)-\widetilde{w}^{\hat{t},N}_1(s,x)\leq\widetilde{w}^{\hat{t},N}_1(\hat{t},x)-\widetilde{w}^{\hat{t},N}_1(s,x) \leq \rho_1(|s-\hat{t}|,C_x).
\end{eqnarray}
On the other hand,
for every $(t,x,y)\in [0,T]\times  H_N\times  H_N$,
by the definitions of $\widetilde{w}_{1}^{\hat{t},N,*}(t,x)$ and $ \widetilde{w}_{2}^{\hat{t},N,*}(t,y)$, there exist  sequences  $(l_i,x_i), (\tau_i,y_i)\in [0,T]\times  H_N$ such that
$(l_{i},x_i)\rightarrow (t,x)$ and $(\tau_{i},y_i)\rightarrow (t,y)$
as $i\rightarrow\infty$ and
\begin{eqnarray}\label{202105081}
\widetilde{w}_{1}^{\hat{t},N,*}(t,x)=\lim_{i\rightarrow\infty}\widetilde{w}^{\hat{t},N}_{1}(l_i,x_i), \ \ \ \widetilde{w}_{2}^{\hat{t},N,*}(t,y)=\lim_{i\rightarrow\infty}\widetilde{w}^{\hat{t},N}_2(\tau_i,y_i).
\end{eqnarray}
Without loss of generality, we may assume  $l_i\leq \tau_i$ for all $i>0$.
By (\ref{202105083})-(\ref{202105085}), we have
\begin{eqnarray}\label{202105082jiaa}
\begin{aligned}
\widetilde{w}_{1}^{\hat{t},N,*}(t,x)&=\lim_{i\rightarrow\infty}\widetilde{w}^{\hat{t},N}_{1}(l_i,x_i)\\
&\leq \liminf_{i\rightarrow\infty}[\widetilde{w}^{\hat{t},N}_{1}(\tau_i,x_i)+\rho_1(|\tau_i-l_i|,C_{x_i})].
\end{aligned}
\end{eqnarray}
Here $C_{x_i}>0$ is the constant  that makes the following formula true:
$$
w_1(\gamma_t)<\sup_{\scriptsize\begin{matrix}\xi_t\in \Lambda^{\hat{t}}\\
	(\xi_t(t))_N=x_i\end{matrix}}\!\!\!\!\! w_{1}(\xi_t),
\quad \mbox{if} \ ||\gamma_t||_0\geq C_{x_i}\, \text{and} \,  \gamma_t\in \Lambda^{\hat{t}}.
$$
We claim that we can assume that there exists a constant $M_4>0$ such that $C_{x_i}\leq M_4$ for all $i\geq 1$. Indeed, if not, for every $n$, there exists  $i_n$ such that
\begin{eqnarray}
\qquad\quad
\widetilde{w}^{\hat{t},N}_{1}(l_{i_n},x_{i_n})=\begin{cases}\displaystyle \sup_{\scriptsize\begin{matrix}\gamma_{l_{i_n}}\in \Lambda^{\hat{t}}\\
	(\gamma_{l_{i_n}}(l_{i_n}))_N=x_{i_n}\\
	||\gamma_{l_{i_n}}||_0> n\end{matrix}}\!\!\!\!\!
{[}w_{1}(\gamma_{l_{i_n}}){]}, \qquad\qquad  \ l_{i_n}\geq \hat{t};\\[5mm]
\displaystyle \sup_{\scriptsize\begin{matrix}
	\gamma_{\hat{t}}\in \Lambda^{\hat{t}}\\
		(\gamma_{\hat{t}}({\hat{t}}))_N=x_{i_n}\\
	||\gamma_{\hat{t}}||_0> n
\end{matrix}}\!\!\!\!\!
{[}w_{1}(\gamma_{\hat{t}}){]}-({\hat{t}}-l_{i_n})^{\frac{1}{2}}, \quad \ l_{i_n}< \hat{t}.
\end{cases}
\end{eqnarray}
Letting $n\rightarrow\infty$, by (\ref{05131}), we get that
$$
\widetilde{w}^{\hat{t},N}_{1}(l_{i_n},x_{i_n})\rightarrow-\infty \ \mbox{as}\ n\rightarrow\infty,
$$
which contradicts the convergence that $ \widetilde{w}_{1}^{\hat{t},N,*}(t,x)=\lim_{i\rightarrow\infty}\widetilde{w}^{\hat{t},N}_{1}(l_i,x_i)$. Then, by  (\ref{202105082jiaa}),
\begin{eqnarray}\label{20210704a}
\begin{aligned}
\widetilde{w}_{1}^{\hat{t},N,*}(t,x)
&\leq \, \liminf_{i\rightarrow\infty}\, [\widetilde{w}^{\hat{t},N}_{1}(\tau_i,x_i)+\rho_1(|\tau_i-l_i|,M_4)]\\
& = \, \liminf_{i\rightarrow\infty}\, \widetilde{w}^{\hat{t},N}_{1}(\tau_i,x_i).
\end{aligned}
\end{eqnarray}
Therefore, by (\ref{202105081}), (\ref{20210704a}) and the definitions of $\widetilde{w}^{\hat{t},N}_{1}$ and $\widetilde{w}^{\hat{t},N}_{2}$,
\begin{eqnarray}\label{20210508b6}
\begin{aligned}
&\quad\widetilde{w}_{1}^{\hat{t},N,*}(t,x)+\widetilde{w}_{2}^{\hat{t},N,*}(t,y)-\varphi(x,y)\\
&\leq\liminf_{i\rightarrow\infty}\, [\widetilde{w}^{\hat{t},N}_{1}(\tau_i,x_i)+\widetilde{w}^{\hat{t},N}_{2}(\tau_i,y_i)-\varphi(x_i,y_i)]\\
&\leq\sup_{\scriptsize\begin{matrix}(l,x_0,y_0)\\
	\in [0,T]\times  H_N\times H_N\end{matrix}}\!\!\!\!\! [\widetilde{w}^{\hat{t},N}_{1}(l,x_0)+\widetilde{w}^{\hat{t},N}_{2}(l,y_0)-\varphi(x_0,y_0)]\\
&=\sup_{\scriptsize\begin{matrix}(l,x_0,y_0)\\
	\in [\hat{t},T]\times  H_N\times H_N\end{matrix}}\!\!\!\!\!
[\widetilde{w}^{\hat{t},N}_{1}(l,x_0)+\widetilde{w}^{\hat{t},N}_{2}(l,y_0)-\varphi(x_0,y_0)].
\end{aligned}
\end{eqnarray}
We also have, for $(l,x_0,y_0)\in [\hat{t},T]\times  H_N\times  H_N$,
\begin{eqnarray}\label{0608aa}
\begin{aligned}
\ \ \ \ \ \ &\quad
\widetilde{w}^{\hat{t},N}_{1}(l,x_0)+  \widetilde{w}^{\hat{t},N}_{2}(l,y_0)-\varphi(x_0,y_0)  \\
&=\sup_{\scriptsize\begin{matrix}
	\gamma_l,\eta_l\in\Lambda^{\hat{t}}\\
	(\gamma_l(l))_N=x_0\\
	(\eta_l(l))_N=y_0\end{matrix}}\!\!\!
\left[w_{1}(\gamma_l)
+w_{2}(\eta_l)
-\varphi((\gamma_l(l))_N,(\eta_l(l))_N)\right]\\
&\leq w_{1}(\hat{{\gamma}}_{{\hat{t}}})+w_{2}(\hat{{\eta}}_{{\hat{t}}})
-\varphi((\hat{{\gamma}}_{{\hat{t}}}({\hat{t}}))_N,(\hat{{\eta}}_{{\hat{t}}}({\hat{t}}))_N),
\end{aligned}
\end{eqnarray}
where the  inequality becomes equality if  $l={\hat{t}}$ 
and $x_0=(\hat{{\gamma}}_{{\hat{t}}}({\hat{t}}))_N,y_0=(\hat{{\eta}}_{{\hat{t}}}({\hat{t}}))_N$.
Combining  (\ref{20210508b6}) and (\ref{0608aa}), we obtain that
\begin{eqnarray}\label{0608a1}
\begin{aligned}
&\quad\widetilde{w}_{1}^{\hat{t},N,*}(t,x)+\widetilde{w}_{2}^{\hat{t},N,*}(t,y)-\varphi(x,y)\\
&\leq w_{1}(\hat{{\gamma}}_{{\hat{t}}})+w_{2}(\hat{{\eta}}_{{\hat{t}}})
-\varphi((\hat{{\gamma}}_{{\hat{t}}}({\hat{t}}))_N,(\hat{{\eta}}_{{\hat{t}}}({\hat{t}}))_N).
\end{aligned}
\end{eqnarray}
By the definitions of $\widetilde{w}_1^{\hat{t},N,*}$ and $\widetilde{w}_2^{\hat{t},N,*}$, we have $\widetilde{w}_1^{\hat{t},N,*}(t,x)\geq \widetilde{w}^{\hat{t},N}_1(t,x), \widetilde{w}_{2}^{\hat{t},N,*}(t,y)$ $ \geq \widetilde{w}^{\hat{t},N}_{2}(t,y)$.  Then by also (\ref{0608aa}) and (\ref{0608a1}), for every $(t,x,y)\in [0,T]\times  H_N\times H_N$,
\begin{eqnarray}\label{0608abc}
\begin{aligned}
&\quad\widetilde{w}_{1}^{\hat{t},N,*}(t,x)+\widetilde{w}_{2}^{\hat{t},N,*}(t,y)-\varphi(x,y)\\
&\leq w_{1}(\hat{{\gamma}}_{{\hat{t}}})+w_{2}(\hat{{\eta}}_{{\hat{t}}})
-\varphi((\hat{{\gamma}}_{{\hat{t}}}({\hat{t}}))_N,(\hat{{\eta}}_{{\hat{t}}}({\hat{t}}))_N)\\
&=\widetilde{w}^{\hat{t},N}_{1}(\hat{t},(\hat{{\gamma}}_{{\hat{t}}}({\hat{t}}))_N)+  \widetilde{w}^{\hat{t},N}_{2}(\hat{t},(\hat{{\eta}}_{{\hat{t}}}({\hat{t}}))_N)
-\varphi((\hat{{\gamma}}_{{\hat{t}}}({\hat{t}}))_N,(\hat{{\eta}}_{{\hat{t}}}({\hat{t}}))_N)\\
&\leq\widetilde{w}_{1}^{\hat{t},N,*}(\hat{t},(\hat{{\gamma}}_{{\hat{t}}}({\hat{t}}))_N)+  \widetilde{w}_{2}^{\hat{t},N,*}(\hat{t},(\hat{{\eta}}_{{\hat{t}}}({\hat{t}}))_N)
-\varphi((\hat{{\gamma}}_{{\hat{t}}}({\hat{t}}))_N,(\hat{{\eta}}_{{\hat{t}}}({\hat{t}}))_N).
\end{aligned}
\end{eqnarray}
Thus
we obtain that (\ref{2020020206}) holds true, and  $ \widetilde{w}_{1}^{\hat{t},N,*}(t,x)+\widetilde{w}_{2}^{\hat{t},N,*}(t,y)-\varphi(x, y)$ has a maximum at $({\hat{t}},(\hat{{\gamma}}_{\hat{t}}({\hat{t}}))_N,
(\hat{{\eta}}_{\hat{t}}({\hat{t}}))_N)$
on $[0,T]\times  H_N\times  H_N$.
The proof is now complete.
\end{proof}

\begin{lemma}\label{lemma4.40615}\ \
	Let all the conditions in Theorem  \ref{theorem0513} hold. Recall that $\Gamma^1_{k}$ and $\Gamma^2_{k}$ are defined in  (\ref{4.11116}).
	Then the maximum points ${\gamma}^{k,j}_{t_{k,j}}$
	of $\Gamma^1_{k}(\gamma_t)
	-\sum_{i=0}^{\infty}
	\frac{1}{2^i}\overline{\Upsilon}(\check{\gamma}^{i}_{\check{t}_{i}},\gamma_t)$ and the maximum points
	$ {\eta}^{k,j}_{s_{k,j}}$
	of $\Gamma^2_{k}(\eta_s)
	-\sum_{i=0}^{\infty}
	\overline{\Upsilon}(\check{\eta}^{i}_{\check{s}_{i}},\eta_s)$
	satisfy  conditions (\ref{4.226}), (\ref{05231}) and (\ref{05232}).
\end{lemma}

\begin{proof}
Recall that $\widetilde{w}_{1}^{\hat{t},N,*}\geq \widetilde{w}^{\hat{t},N}_{1}, \widetilde{w}_{2}^{\hat{t},N,*} \geq \widetilde{w}^{\hat{t},N}_{2}$,
by  the definitions of $\widetilde{w}^{\hat{t},N}_1$ and $\widetilde{w}^{\hat{t},N}_2$, we get that
\begin{eqnarray*}
	&& \widetilde{w}_{1}^{\hat{t},N,*}(t_{k,j},({\gamma}^{k,j}_{t_{k,j}}(t_{k,j}))_N)
	-\tilde{\varphi}_k(t_{k,j},({\gamma}^{k,j}_{t_{k,j}}(t_{k,j}))_N)\\
	&\geq& w_1({\gamma}^{k,j}_{t_{k,j}})-\tilde{\varphi}_k(t_{k,j},({\gamma}^{k,j}_{t_{k,j}}(t_{k,j}))_N)
	=\Gamma^1_{k}({\gamma}^{k,j}_{t_{k,j}}),\\
	&&\widetilde{w}_{2}^{\hat{t},N,*}(s_{k,j},({\eta}^{k,j}_{s_{k,j}}(s_{k,j}))_N)
	-\tilde{\psi}_k(s_{k,j},({\eta}^{k,j}_{s_{k,j}}(s_{k,j}))_N)\\
	&\geq& w_2({\eta}^{k,j}_{s_{k,j}})-\tilde{\psi}_k(s_{k,j},({\eta}^{k,j}_{s_{k,j}}(s_{k,j}))_N)
	=\Gamma^2_{k}({\eta}^{k,j}_{s_{k,j}}).
\end{eqnarray*}
We notice that, from (\ref{20210509a}) and the property (ii) of $(t_{k,j},{\gamma}^{k,j}_{t_{k,j}}, s_{k,j},{\eta}^{k,j}_{s_{k,j}})$,
\begin{eqnarray*}
	&&\Gamma^1_{k}({\gamma}^{k,j}_{t_{k,j}})\geq\Gamma^1_{k}(\check{\gamma}^{0}_{\check{t}_{0}})
	\geq \sup_{(t,\gamma_t)\in [\hat{t},T]\times \Lambda^{\hat{t}}}\Gamma^1_{k}(\gamma_t)-\frac{1}{j}, \\
	&& \Gamma^2_{k}({\eta}^{k,j}_{s_{k,j}})\geq\Gamma^2_{k}(\check{\eta}^{0}_{\check{s}_{0}})
	\geq \sup_{(s,\eta_s)\in [\hat{t},T]\times \Lambda^{\hat{t}}}\Gamma^2_{k}(\eta_s)-\frac{1}{j},
\end{eqnarray*}
and by the  definitions of $\widetilde{w}_{1}^{\hat{t},N,*}$ and $\widetilde{w}_{2}^{\hat{t},N,*}$,
\begin{eqnarray*}
	&&\sup_{(t,\gamma_t)\in [\hat{t},T]\times \Lambda^{\hat{t}}}\Gamma^1_{k}(\gamma_t)\geq\widetilde{w}_{1}^{\hat{t},N,*}(t_k,{x}^k)-\tilde{\varphi}_k(t_k,{x}^k),\\
	&&\sup_{(s,\eta_s)\in [\hat{t},T]\times \Lambda^{\hat{t}}}\Gamma^2_{k}(\eta_s)\geq\widetilde{w}_{2}^{\hat{t},N,*}(s_k,{y}^k)-\tilde{\psi}_k(s_k,{y}^k).
\end{eqnarray*}
Therefore,
\begin{eqnarray}\label{0525b6}
&&\widetilde{w}_{1}^{\hat{t},N,*}(t_{k,j},({\gamma}^{k,j}_{t_{k,j}}(t_{k,j}))_N)-\tilde{\varphi}_k(t_{k,j},({\gamma}^{k,j}_{t_{k,j}}(t_{k,j}))_N)\nonumber\\
&\geq&\Gamma^1_{k}({\gamma}^{k,j}_{t_{k,j}})\geq
\widetilde{w}_{1}^{\hat{t},N,*}(t_k,{x}^k)-\tilde{\varphi}_k(t_k,{x}^k)-\frac{1}{j},
\end{eqnarray}
\begin{eqnarray}\label{0525b0920}
&&\widetilde{w}_{2}^{\hat{t},N,*}(s_{k,j},({\eta}^{k,j}_{s_{k,j}}(s_{k,j}))_N)-\tilde{\psi}_k(s_{k,j},({\eta}^{k,j}_{s_{k,j}}(s_{k,j}))_N)\nonumber\\
&\geq&\Gamma^2_{k}({\eta}^{k,j}_{s_{k,j}})\geq
\widetilde{w}_{2}^{\hat{t},N,*}(s_k,{y}^k)-\tilde{\psi}_k(s_k,{y}^k)-\frac{1}{j}.
\end{eqnarray}
By (\ref{05131}) and  that $\tilde{\varphi}_k$ and $\tilde{\psi}_k$ are bounded from below, 
there exists a constant  $M_5>0$  that is sufficiently  large   that
$$
\Gamma^1_k(\gamma_t)<\sup_{(l,\xi_l)\in [\hat{t},T]\times \Lambda^{\hat{t}}}\Gamma^1_{k}(\xi_l)-1, \ \ t\in [\hat{t},T],\ ||\gamma_t||_0\geq M_5,
$$ and
$$
\Gamma^2_{k}(\eta_s)<\sup_{(r,\varsigma_r)\in [\hat{t},T]\times \Lambda^{\hat{t}}}\Gamma^2_{k}(\varsigma_r)-1,\ s\in [\hat{t},T], \ ||\eta_s||_0\geq M_5.
$$
Thus, we have $||{\gamma}^{k,j}_{t_{k,j}}||_0\vee
||{\eta}^{k,j}_{s_{k,j}}||_0<M_5$. In particular, $|({\gamma}^{k,j}_{t_{k,j}}(t_{k,j}))_N|\vee|({\eta}^{k,j}_{s_{k,j}}(s_{k,j}))_N|$ $<M_5$.
We note that $M_5$ is independent of $j$.
Then letting $j\rightarrow\infty$ in (\ref{0525b6}) and (\ref{0525b0920}),
we obtain (\ref{4.226}).  Indeed, if not, we may assume that  there exist $(\grave{t},\grave{x},\grave{s},\grave{y})\in [0, T]\times  H_N\times [0, T]\times  H_N$ and  a subsequence of $(t_{k,j},({\gamma}^{k,j}_{t_{k,j}}(t_{k,j}))_N,s_{k,j},({\eta}^{k,j}_{s_{k,j}}(s_{k,j}))_N)$ still denoted by itself  such that
$$
\lim_{j\to \infty}(t_{k,j},({\gamma}^{k,j}_{t_{k,j}}(t_{k,j}))_N,s_{k,j},({\eta}^{k,j}_{s_{k,j}}(s_{k,j}))_N)= (\grave{t},\grave{x},\grave{s},\grave{y})\neq (t_k,x^k,s_k,y^k).
$$
Letting $j\rightarrow\infty$ in (\ref{0525b6}) and (\ref{0525b0920}), by the upper semicontinuity of $\widetilde{w}_{1}^{\hat{t},N,*}+\widetilde{w}_{2}^{\hat{t},N,*}-\tilde{\varphi}_k-\tilde{\psi}_k$, we have
\begin{eqnarray*}
	 &&\widetilde{w}_{1}^{\hat{t},N,*}(\grave{t},\grave{x})+\widetilde{w}_{2}^{\hat{t},N,*}(\grave{s},\grave{y})-\tilde{\varphi}_k(\grave{t},\grave{x})-\tilde{\psi}_k(\grave{s},\grave{y})\\
	&\geq&\widetilde{w}_{1}^{\hat{t},N,*}(t_k,{x}^k)+\widetilde{w}_{2}^{\hat{t},N,*}(s_k,{y}^k) -\tilde{\varphi}_k(t_k,{x}^k)-\tilde{\psi}_k(s_k,{y}^k),
\end{eqnarray*}       which  contradicts that
$(t_k,x^{k},s_k,y^{k})$  is the  strict  maximum point of the function: $\widetilde{w}_{1}^{\hat{t},N,*}(t,x)+\widetilde{w}_{2}^{\hat{t},N,*}(s,y)-\tilde{\varphi}_k(t,x)-\tilde{\psi}_k(s,y), \, (t,x;s,y)\in ([0,T]\times H_N)^2$.\par
By (\ref{4.226}), the  upper semicontinuity of $\widetilde{w}_{1}^{\hat{t},N,*}$ and $\widetilde{w}_{2}^{\hat{t},N,*}$ and  the continuity of $\tilde{\varphi}_k$ and $\tilde{\psi}_k$, letting $j\rightarrow\infty$ in (\ref{0525b6})  and (\ref{0525b0920}), we obtain
\begin{eqnarray*}
	&&\widetilde{w}_{1}^{\hat{t},N,*}(t_k,x_k)
	\geq \limsup_{j\rightarrow\infty}\widetilde{w}_{1}^{\hat{t},N,*}(t_{k,j},({\gamma}_{t_{k,j}}^{k,j}(t_{k,j}))_N)\\
	&\geq& \liminf_{j\rightarrow\infty}\widetilde{w}_{1}^{\hat{t},N,*}(t_{k,j},({\gamma}_{t_{k,j}}^{k,j}(t_{k,j}))_N)
	\geq\widetilde{w}_{1}^{\hat{t},N,*}(t_k,x_k),\\
	&&\widetilde{w}_{2}^{\hat{t},N,*}(s_k,y_k)
	\geq \limsup_{j\rightarrow\infty}\widetilde{w}_{2}^{\hat{t},N,*}(s_{k,j},({\eta}_{t_{k,j}}^{k,j}(t_{k,j}))_N)\\
	&\geq& \liminf_{j\rightarrow\infty}\widetilde{w}_{2}^{\hat{t},N,*}(s_{k,j},({\eta}_{t_{k,j}}^{k,j}(t_{k,j}))_N)
	\geq\widetilde{w}_{2}^{\hat{t},N,*}(s_k,y_k).
\end{eqnarray*}
Thus, we get (\ref{05231}) holds true.
Letting $j\rightarrow\infty$ in (\ref{0525b6}) and (\ref{0525b0920}), by (\ref{05231}) and the definitions of $\Gamma^1_{k}$ and $\Gamma^2_{k}$,
$$
\widetilde{w}_{1}^{\hat{t},N,*}(t_k,x_k)=\lim_{j\rightarrow\infty}w_1({\gamma}_{t_{k,j}}^{k,j}), \ \ \
\widetilde{w}_{2}^{\hat{t},N,*}(s_k,y_k)=\lim_{j\rightarrow\infty}w_2({\eta}_{t_{k,j}}^{k,j}).
$$
Thus, we obtain (\ref{05232}).
The proof is now complete.
\end{proof}
\newpage
\section{Viscosity solutions to  PHJB equations: uniqueness}
\par
This chapter is devoted to the  proof of uniqueness of  viscosity
solutions to equation (\ref{hjb1}). This result, together with
the results from  Section 5, will be used to characterize
the value functional defined by (\ref{value1}).
We require  the following assumption on $G$.
\begin{assumption}\label{hypstate5666}
	For every $(t,\gamma_t)\in [0,T)\times \Lambda$,
	\begin{eqnarray}\label{g5}
	\lim_{N\to \infty} \sup_{u\in U}\left|Q_N{G}(\gamma_t,u)\right|_{L_2(\Xi, H)}^2=0.
	\end{eqnarray}
\end{assumption}
\par
{By \cite[Proposition 11.2.13]{zhang}, without loss of generality we may assume that} there exists a constant $L>0$ such that,
for all $(t,\gamma_t, p,l)\in [0,T]\times{\Lambda}\times H\times {\mathcal{S}}(H)$ and $r_1,r_2\in \mathbb{R}$ with $r_1<r_2$,
\begin{eqnarray}\label{5.1}
{\mathbf{H}}(\gamma_t,r_1,p,l)-{\mathbf{H}}(\gamma_t,r_2,p,l)\geq L(r_2-r_1).
\end{eqnarray}
We  now state the main result of this chapter.
\begin{theorem}\label{theoremhjbm}
	Let Assumptions \ref{hypstate}, \ref{hypcost} and \ref{hypstate5666}  be satisfied.
	Let $W_1\in C^0({\Lambda})$ $(\mbox{resp}., W_2\in C^0({\Lambda}))$ be  a viscosity sub-solution (resp., super-solution) to equation (\ref{hjb1}) and  let  there exist a  constant $L>0$  and a local modulus of continuity  $\rho_2$
	such that, for any  $0\leq t\leq  s\leq T$ and
	$\gamma_t, \eta_t\in{\Lambda}$,
	\begin{eqnarray}\label{w}
	|W_1(\gamma_t)|\vee |W_2(\gamma_t)|\leq L (1+||\gamma_t||^2_0)
	\end{eqnarray}
	and
	\begin{eqnarray}\label{w1}
\qquad&	\begin{aligned}
	&|W_1(\gamma_{t,s}^A)-W_1(\eta_t)|\vee|W_2(\gamma_{t,s}^A)-W_2(\eta_t)|\\
	\leq&\,
	\rho_2(s-t,||\gamma_t||_0\vee||\eta_t||_0)+L(1+||\gamma_t||_0+||\eta_t||_0)\, ||\gamma_{t}-\eta_t||_0.
	\end{aligned}
	\end{eqnarray}
	Then  $W_1\leq W_2$.
\end{theorem}
Theorems    \ref{theoremvexist} and \ref{theoremhjbm} yield that the viscosity solution to   PHJB equation given in (\ref{hjb1})
is  the value functional  $V$ of our optimal control problem given in (\ref{state1}), (\ref{fbsde1}) and (\ref{value1}).

\begin{theorem}\label{theorem52}
	Let  Assumptions \ref{hypstate}, \ref{hypcost} and \ref{hypstate5666}   be satisfied.  Then the value
	functional $V$ defined by (\ref{value1}) is the unique viscosity
	solution to PHJB equation~(\ref{hjb1}) which satisfies (\ref{w}) and (\ref{w1}).
\end{theorem}

\begin{proof} Theorem \ref{theoremvexist} shows that $V$ is a viscosity solution to equation (\ref{hjb1}).  Thus, our conclusion follows from Lemma \ref{lemmavaluev} and
Theorems  \ref{theorem3.9} and
\ref{theoremhjbm}.
\end{proof}

{\bf  Proof of Theorem \ref{theoremhjbm}}.   It is sufficient to prove that the functional $W_1-\frac{\varrho}{t+1}\leq W_2$ for $\varrho>0$. Since $W_1$ is a viscosity sub-solution of PHJB equation (\ref{hjb1}),  the functional $\widetilde{w}^\varrho:=W_1-\frac{\varrho}{t+1}$  with $\varrho>0$, is a viscosity sub-solution
to the following PHJB  equation
\begin{eqnarray}\label{1002a}
\begin{cases}
{\partial_t} \widetilde{w}^\varrho(\gamma_t)+\left\langle A^*\partial_x\widetilde{w}^\varrho(\gamma_t),\, \gamma_t(t)\right\rangle_H\\[3mm]
\qquad\qquad+{\mathbf{H}}\left(\gamma_t, \,\widetilde{w}^\varrho(\gamma_t), \,\partial_x \widetilde{w}^\varrho(\gamma_t),\,\partial_{xx} \widetilde{w}^\varrho(\gamma_t)\right)
=c:= \frac{\varrho}{(T+1)^2}, \\[3mm]
\quad\quad\quad \quad \qquad  (t, \gamma_t)\in [0,T)\times {\Lambda}; \\[3mm]
\widetilde{w}^\varrho(\gamma_T)=\phi(\gamma_T), \quad  \gamma_T\in \Lambda_T.\\
\end{cases}
\end{eqnarray}
Therefore,  it is sufficient to prove
$W_1\leq W_2$ under the stronger  assumption that $W_1$ is a viscosity sub-solution of the last PHJB equation.

The rest of the proof splits in the following {five} steps.
\par
$Step\  1.$ Definitions of auxiliary functionals.
\par
It suffices  to prove that $W_1(\gamma_t)\leq W_2(\gamma_t)$ for all $(t,\gamma_t)\in
[T-\bar{a},T)\times
{\Lambda}$ with
$$\bar{a}:=\frac{1}{2(342L+36)L}\wedge{T};$$
 the desired comparison on the whole time interval $[0,T]$ is easily obtained  via a backward iteration over the intervals
$[T-i\bar{a},\, T-(i-1)\bar{a})$.  Otherwise, there is  $(\tilde{t},\tilde{\gamma}_{\tilde{t}})\in (T-\bar{a},T)\times
{\Lambda}$  such that
$\tilde{m}:=W_1(\tilde{\gamma}_{\tilde{t}})-W_2(\tilde{\gamma}_{\tilde{t}})>0$.
\par
Let  $\varepsilon >0$ be  a small number such that
$$
W_1(\tilde{\gamma}_{\tilde{t}})-W_2(\tilde{\gamma}_{\tilde{t}})-2\varepsilon \frac{\nu T-\tilde{t}}{\nu
	T}\, \Upsilon(\tilde{\gamma}_{\tilde{t}})
>\frac{\tilde{m}}{2}
$$
and
\begin{eqnarray}\label{5.3}
\frac{\varepsilon}{\nu T}\leq\frac{c}{4}
\end{eqnarray}
with
$$
\nu:=1+\frac{1}{2T(342L+36)L}.
$$
Next,  we define for any  $(\gamma_t,\eta_t)\in {\Lambda}^{T-\bar{a}}\times{\Lambda}^{T-\bar{a}}$,
\begin{eqnarray*}
	\Psi(\gamma_t,\eta_t)&:=&W_1(\gamma_t)-W_2(\eta_t)-{\beta}\Upsilon(\gamma_{t},\eta_{t})-\beta^{\frac{1}{3}}\, |\gamma_{t}(t)-\eta_{t}(t)|^2\\
	&&-\varepsilon\frac{\nu T-t}{\nu
		T}\, \left(\Upsilon(\gamma_t)+\Upsilon(\eta_t)\right).
\end{eqnarray*}
In view of  (\ref{s0}) and (\ref{w}), it is clear that $\Psi$ is bounded from above on ${\Lambda}^{T-\bar{a}}\times{\Lambda}^{T-\bar{a}}$. Moreover,  by Lemma \ref{theoremS}, $\Psi$ is a continuous  functional.
Take
$\delta_i:=\frac{1}{2^i}$ for all $i\geq0$.
Since
$\overline{\Upsilon}^{3,3,2}(\cdot,\cdot)$ is a gauge-type function on $(\Lambda^{\tilde{t}}\otimes \Lambda^{\tilde{t}},d_{1,\infty})$, from Lemma \ref{theoremleft1} it follows that,
for every  $(\gamma^0_{t_0},\eta^0_{t_0})\in \Lambda^{\tilde{t}}\times \Lambda^{\tilde{t}}$ satisfying
$$
\Psi(\gamma^0_{t_0},\eta^0_{t_0})\geq \sup_{(\gamma_s,\eta_s)\in  \Lambda^{\tilde{t}}\times \Lambda^{\tilde{t}}}\Psi(\gamma_s,\eta_s)-\frac{1}{\beta},\
\    \mbox{and} \ \ \Psi(\gamma^0_{t_0},\eta^0_{t_0})\geq \Psi(\tilde{\gamma}_{\tilde{t}},\tilde{\gamma}_{\tilde{t}}) >\frac{\tilde{m}}{2},
$$
there exist $(\hat{t},\hat{\gamma}_{\hat{t}},\hat{\eta}_{\hat{t}})\in [\tilde{t},T]\times \Lambda^{\tilde{t}}\times \Lambda^{\tilde{t}}$ and a sequence $\{(t_i,\gamma^i_{t_i},\eta^i_{t_i})\}_{i\geq1}\subset
[\tilde{t},T]\times \Lambda^{\tilde{t}}\times \Lambda^{\tilde{t}}$ such that
\begin{description}
	\item[(i)] $\displaystyle \Upsilon(\gamma^0_{t_0},\hat{\gamma}_{\hat{t}})+\Upsilon(\eta^0_{t_0},\hat{\eta}_{\hat{t}})+|\hat{t}-t_0|^2\leq \frac{1}{\beta}$,
	$\displaystyle \Upsilon(\gamma^i_{t_i},\hat{\gamma}_{\hat{t}})+\Upsilon(\eta^i_{t_i},\hat{\eta}_{\hat{t}})+|\hat{t}-t_i|^2
	\leq \frac{1}{2^i\beta}$ and $\displaystyle t_i\uparrow \hat{t}$ as $i\rightarrow\infty$,
	\item[(ii)]  $\displaystyle {\Psi_1(\hat{\gamma}_{\hat{t}},\hat{\eta}_{\hat{t}})}
\geq \Psi(\gamma^0_{t_0},\eta^0_{t_0})$, and
	\item[(iii)]    for all $\displaystyle (s, \gamma_s,\eta_s)\in [\hat{t},T]\times \Lambda^{\hat{t}}\times \Lambda^{\hat{t}}\setminus \{(\hat{t},\hat{\gamma}_{\hat{t}},\hat{\eta}_{\hat{t}})\}$,
	{   \begin{eqnarray}\label{iii4}
		\Psi_1(\gamma_s,\eta_s)
		<\Psi_1(\hat{\gamma}_{\hat{t}},\hat{\eta}_{\hat{t}}),
		\end{eqnarray}}
\end{description}
where we define
{$$
	\Psi_1(\gamma_t,\eta_t):=  \Psi(\gamma_t,\eta_t)
	-\sum_{i=0}^{\infty}
	\frac{1}{2^i}\left[{\Upsilon}(\gamma^i_{t_i},\gamma_t)+{\Upsilon}(\eta^i_{t_i},\eta_t)+|{t}-t_i|^2\right], \ \  \  (\gamma_t,\eta_t)\in  \Lambda^{\tilde{t}}\times \Lambda^{\tilde{t}}.
	$$}
Note that the point
$({\hat{t}},\hat{{\gamma}}_{{\hat{t}}},\hat{{\eta}}_{{\hat{t}}})$ depends on  $\beta$ and
$\varepsilon$.
\par
$Step\ 2.$
There exists ${{M}_0}>0$ {independent of $\beta$}
such that
\begin{eqnarray}\label{5.10jiajiaaaa}||\hat{\gamma}_{\hat{t}}||_0\vee||\hat{\eta}_{\hat{t}}||_0<M_0
\end{eqnarray} and
\begin{eqnarray}\label{5.10}
\lim_{\beta\to \infty} \left( \beta||\hat{{\gamma}}_{{\hat{t}}}-\hat{{\eta}}_{{\hat{t}}}||_{0}^6
+\beta|\hat{{\gamma}}_{{\hat{t}}}(\hat{t}\,)-\hat{{\eta}}_{{\hat{t}}}(\hat{t}\,)|^4\right)
=0.
\end{eqnarray}

The proof follows. First,   noting $\nu$ is independent of  $\beta$, by the definition of  ${\Psi}$,
there exists a constant  ${M}_0>0$ {independent of $\beta$}  that is sufficiently  large   that
$
\Psi(\gamma_t, \eta_t)<0
$ for all $t\in [T-\bar{a},T]$ and $||\gamma_t||_0\vee||\eta_t||_0\geq M_0$. Thus, we have $||\hat{\gamma}_{\hat{t}}||_0\vee||\hat{\eta}_{\hat{t}}||_0\vee
||{\gamma}^{0}_{t_{0}}||_0\vee||{\eta}^{0}_{t_{0}}||_0<M_0$. {We note that $M_0$ depends on $\varepsilon$.}
\par
Second, by (\ref{iii4}), we have
{\begin{eqnarray}\label{5.56789}
	2\Psi_1(\hat{\gamma}_{\hat{t}},\hat{\eta}_{\hat{t}})
	\geq  \Psi_1(\hat{\gamma}_{\hat{t}},\hat{\gamma}_{\hat{t}})
	+\Psi_1(\hat{\eta}_{\hat{t}},\hat{\eta}_{\hat{t}}).
	\end{eqnarray}}
This implies that
\begin{eqnarray}\label{5.6}
\begin{aligned}
&\quad2{\beta}\Upsilon(\hat{\gamma}_{\hat{t}},\hat{\eta}_{\hat{t}})+
2\beta^{\frac{1}{3}}|\hat{\gamma}_{\hat{t}}(\hat{t}\,)-\hat{\eta}_{\hat{t}}(\hat{t}\,)|^2
\\
&
\leq|W_1(\hat{\gamma}_{\hat{t}})-W_1(\hat{\eta}_{\hat{t}})|
+|W_2(\hat{\gamma}_{\hat{t}})-W_2(\hat{\eta}_{\hat{t}})|\\
&\quad+
\sum_{i=0}^{\infty}\frac{1}{2^i}[\Upsilon(\eta^i_{t_i},\hat{\gamma}_{\hat{t}})+\Upsilon(\gamma^i_{t_i},\hat{\eta}_{\hat{t}})].
\end{aligned}
\end{eqnarray}
On the other hand,  notice that
$$
\Upsilon(\gamma_t,\eta_s)=\Upsilon(\gamma_t-\eta_{s,t}^A), \  \ \gamma_t,\eta_s\in \Lambda, \ 0\leq s\leq t\leq T,
$$
by Lemma \ref{theoremS00044} and   the property (i) of $(\hat{t}, \hat{\gamma}_{\hat{t}},\hat{\eta}_{\hat{t}})$,
\begin{eqnarray}\label{4.7jiajia130}
\begin{aligned}
&\quad
\sum_{i=0}^{\infty}\frac{1}{2^i}[\Upsilon(\eta^i_{t_i},\hat{\gamma}_{\hat{t}})+\Upsilon(\gamma^i_{t_i},\hat{\eta}_{\hat{t}})]\\
&\leq2^5\sum_{i=0}^{\infty}\frac{1}{2^i}[\Upsilon(\eta^i_{t_i},\hat{\eta}_{\hat{t}})
+\Upsilon(\gamma^i_{t_i},\hat{\gamma}_{\hat{t}})+2\Upsilon(\hat{\gamma}_{\hat{t}},\hat{\eta}_{\hat{t}})]
\leq\frac{2^6}{\beta}+{2^7}\Upsilon(\hat{\gamma}_{\hat{t}},\hat{\eta}_{\hat{t}}).
\end{aligned}
\end{eqnarray}
Combining (\ref{5.6}) and (\ref{4.7jiajia130}),  from  (\ref{w}) and (\ref{5.10jiajiaaaa})  we have
\begin{eqnarray}\label{5.jia6}
\begin{aligned}
&\quad
(2{\beta}-2^7)\Upsilon(\hat{\gamma}_{\hat{t}},\hat{\eta}_{\hat{t}})+
2\beta^{\frac{1}{3}}|\hat{\gamma}_{\hat{t}}(\hat{t}\,)-\hat{\eta}_{\hat{t}}(\hat{t}\,)|^2
\\
&\leq |W_1(\hat{\gamma}_{\hat{t}})-W_1(\hat{\eta}_{\hat{t}})|
+|W_2(\hat{\gamma}_{\hat{t}})-W_2(\hat{\eta}_{\hat{t}})|+\frac{2^6}{\beta}\\
&\leq 2L(2+||\hat{\gamma}_{\hat{t}}||^2_0+||\hat{\eta}_{\hat{t}}||^2_0)+\frac{2^6}{\beta}
\leq 4L(1+M_0^2)+\frac{2^6}{\beta}.
\end{aligned}
\end{eqnarray}
Letting $\beta\rightarrow\infty$, we have
\begin{eqnarray*}
	\Upsilon(\hat{\gamma}_{\hat{t}},\hat{\eta}_{\hat{t}})
	\leq \frac{1}{2{\beta}-2^7}\left[4L(1+M^2_0)+\frac{2^6}{\beta}\right]\rightarrow0\ 
	\mbox{as} \ \beta\rightarrow\infty.
\end{eqnarray*}
In view of  (\ref{s0}), we have
\begin{eqnarray}\label{5.66666123}
\lim_{\beta\to \infty} ||\hat{\gamma}_{\hat{t}}-\hat{\eta}_{\hat{t}}||_0 =0.
\end{eqnarray}
From  (\ref{s0}),  (\ref{w1}), (\ref{5.6}), (\ref{4.7jiajia130}) and (\ref{5.66666123}), we conclude
that
\begin{eqnarray}\label{5.10112345}
\begin{aligned}
&\quad
{\beta}||\hat{\gamma}_{\hat{t}}-\hat{\eta}_{\hat{t}}||^6_{0}+\beta^{\frac{1}{3}}|\hat{\gamma}_{\hat{t}}(\hat{t}\,)-\hat{\eta}_{\hat{t}}(\hat{t}\,)|^2
\leq{\beta}\Upsilon(\hat{\gamma}_{\hat{t}},\hat{\eta}_{\hat{t}})
+\beta^{\frac{1}{3}}|\hat{\gamma}_{\hat{t}}(\hat{t}\,)-\hat{\eta}_{\hat{t}}(\hat{t}\,)|^2\\
&\leq L(1+2M_0)||\hat{\gamma}_{\hat{t}}-\hat{\eta}_{\hat{t}}||_0
+\frac{2^5}{\beta}+{2^{8}}||\hat{\gamma}_{\hat{t}}-\hat{\eta}_{\hat{t}}||_{0}^6
\rightarrow0 \ 
\mbox{as} \ \beta\rightarrow\infty.
\end{aligned}
\end{eqnarray}
Then,  we have
\begin{eqnarray}\label{5.10123445567890}
\begin{aligned}
&\quad
\beta^{\frac{1}{2}}|\hat{{\gamma}}_{{\hat{t}}}(\hat{t}\,)-\hat{{\eta}}_{{\hat{t}}}(\hat{t}\,)|^2
\leq\beta^{\frac{1}{6}}\left(L(1+2M_0)||\hat{\gamma}_{\hat{t}}-\hat{\eta}_{\hat{t}}||_0
+\frac{2^5}{\beta}+{2^{8}}||\hat{\gamma}_{\hat{t}}-\hat{\eta}_{\hat{t}}||_{0}^6\right)\\
&=L(1+2M_0)\beta^{\frac{1}{6}}||\hat{\gamma}_{\hat{t}}-\hat{\eta}_{\hat{t}}||_0
+\frac{2^5}{\beta^{\frac{5}{6}}}+{2^{8}}\beta^{\frac{1}{6}}||\hat{\gamma}_{\hat{t}}-\hat{\eta}_{\hat{t}}||_{0}^6
\rightarrow0 \ 
\mbox{as} \ \beta\rightarrow\infty.
\end{aligned}
\end{eqnarray}
Therefore, we have (\ref{5.10}).
\par
$Step\ 3.$ There exists 
$N_0>0$
such that
$\hat{t}\in [\tilde{t},T)$
for all $\beta\geq N_0$.
\par
By (\ref{5.66666123}), there is  a sufficiently large number $N_0>0$ such that
$$
L||\hat{\gamma}_{\hat{t}}-\hat{\eta}_{\hat{t}}||_0
\leq
\frac{\tilde{m}}{4}
$$
for all $\beta\geq N_0$.
Then,  $\hat{t}\in [\tilde{t},T)$ for all $\beta\geq N_0$. Indeed, if  $\hat{t}=T$,  we  deduce the following contradiction:
\begin{eqnarray*}
	\frac{\tilde{m}}{2}\leq\Psi(\hat{\gamma}_{\hat{t}},\hat{\eta}_{\hat{t}})\leq \phi(\hat{\gamma}_{\hat{t}})-\phi(\hat{\eta}_{\hat{t}})\leq
	L||\hat{\gamma}_{\hat{t}}-\hat{\eta}_{\hat{t}}||_0
	\leq
	\frac{\tilde{m}}{4}.
\end{eqnarray*}

{ $Step\ 4.$   Crandall-Ishii lemma.}
\par
From above all,  for the fixed   $N_0>0$ in step 3, 
we  find
$(\hat{t},\hat{\gamma}_{\hat{t}}, \hat{\eta}_{\hat{t}})\in [\tilde{t}, T]\times
\Lambda^{\tilde{t}}\times \Lambda^{\tilde{t}}$   satisfying $\hat{t}\in [\tilde{t},T)$  for all $\beta\geq N_0$
such that
\begin{eqnarray}\label{psi4}
\ \ \ \ \ \ \  \Psi_1(\hat{\gamma}_{\hat{t}},\hat{\eta}_{\hat{t}})\geq \Psi(\tilde{\gamma}_{\tilde{t}},\tilde{\gamma}_{\tilde{t}}) \  \mbox{and}  \      \Psi_1(\hat{\gamma}_{\hat{t}},\hat{\eta}_{\hat{t}})\geq
{\Psi}_1(\gamma_t,\eta_t),
\  (\gamma_t,\eta_t)\in  \Lambda^{\hat{t}}\otimes \Lambda^{\hat{t}}.
\end{eqnarray}
For every $N\geq1$, we define,
for $(t,\gamma_t,\eta_t)\in [0,T]\times {\Lambda}\times {\Lambda}$,
\begin{eqnarray}\label{06091}
\begin{aligned}
w_{1}(\gamma_t)&:=W_1(\gamma_t)-2^5\beta\, \Upsilon(\gamma_{t},\hat{\xi}_{\hat{t}})-\varepsilon\frac{\nu T-t}{\nu
	T}\, \Upsilon(\gamma_t)
-\varepsilon \overline{\Upsilon}(\gamma_t,\hat{\gamma}_{\hat{t}})\\
&\quad
-2\beta^{\frac{1}{3}}\left|(\gamma_t(t))^-_N-(e^{(t-\hat{t}\,)A}\hat{\xi}_{\hat{t}}(\hat{t}\,))^-_N\right|^2-\sum_{i=0}^{\infty}
\frac{1}{2^i}\overline{\Upsilon}(\gamma^i_{t_i},\gamma_t),
\end{aligned}
\end{eqnarray}
\begin{eqnarray}\label{06092}
\begin{aligned}
w_{2}(\eta_t)&:=-W_2(\eta_t)-2^5\beta\, \Upsilon(\eta_{t},\hat{\xi}_{\hat{t}})-\varepsilon\frac{\nu T-t}{\nu
	T}\, \Upsilon(\eta_t)
-\varepsilon \overline{\Upsilon}(\eta_t,\hat{\eta}_{\hat{t}})\\
&\quad
-2\beta^{\frac{1}{3}}\left|(\eta_t(t))_N^--(e^{(t-\hat{t}\,)A}\hat{\xi}_{\hat{t}}(\hat{t}\,))^-_N\right|^2-\sum_{i=0}^{\infty}
\frac{1}{2^i}{\Upsilon}(\eta^i_{t_i},\eta_t),
\end{aligned}
\end{eqnarray}
where $\hat{\xi}_{\hat{t}}=\frac{\hat{\gamma}_{\hat{t}}+\hat{\eta}_{\hat{t}}}{2}$.   We  note that $w_1,w_2$ depend on $\hat{\xi}_{{\hat{t}}}$ and $N$, and thus on $\beta$,
$\varepsilon$ and  $N$. We also note that the last term in (\ref{06092}) is $\sum_{i=0}^{\infty}
\frac{1}{2^i}{\Upsilon}(\eta^i_{t_i},\eta_t)$  rather than $\sum_{i=0}^{\infty}
\frac{1}{2^i}\overline{\Upsilon}(\eta^i_{t_i},\eta_t)$. This is because  we divide the term $\sum_{i=0}^{\infty}
\frac{1}{2^i}[{\Upsilon}(\gamma^i_{t_i},\gamma_t)+{\Upsilon}(\eta^i_{t_i},\eta_t)+|t-t_i|^2]$ in $\Psi_1$  into two terms $\sum_{i=0}^{\infty}
\frac{1}{2^i}\overline{\Upsilon}(\gamma^i_{t_i},\gamma_t)$ and $\sum_{i=0}^{\infty}
\frac{1}{2^i}{\Upsilon}(\eta^i_{t_i},\eta_t)$. Define $\varphi\in C^2(H_N\times H_N)$ by
\begin{eqnarray}\label{07063}
\varphi(x,y):=\beta^{\frac{1}{3}}|x-y|^2,\ \ (x,y)\in H_N\times H_N.
\end{eqnarray}
By Lemma \ref{0611a} below, $w_1$ and $w_2$ satisfy the conditions of  Theorem \ref{theorem0513} with $\varphi$ defined by (\ref{07063}).
Then by Theorem \ref{theorem0513}, there exist
sequences  $(l_{k},\check{\gamma}^{k}_{l_k}), (s_{k},\check{\eta}^{k}_{s_k})\in [\hat{t},T]\times \Lambda^{\hat{t}}$ and
sequences of functionals 
$(\varphi_{1,k}, \psi_{1,k}, \varphi_{2,k}, \psi_{2,k})\in \Phi_{\hat{t}}\times \Phi_{\hat{t}}\times  {\mathcal{G}}_{l_k}\times {\mathcal{G}}_{s_k}$ bounded from below such that both functionals
\begin{eqnarray}\label{0609a}
w_{1}(\gamma_t)-\varphi_{1,k}(\gamma_t)-\varphi_{2,k}(\gamma_t), \quad \gamma_t\in \Lambda^{l_k}
\end{eqnarray}
and
\begin{eqnarray}\label{0609b}
w_{2}(\eta_t)-\psi_{1,k}(\eta_{t})-\psi_{2,k}(\eta_{t}), \quad  \eta_t\in \Lambda^{s_k}
\end{eqnarray}
attain the strict  maximum $0$ at $\check{\gamma}^{k}_{l_k}$ and $\check{\eta}^{k}_{s_k}$, respectively, and
\begin{eqnarray}\label{0608v1}
\ \ \ &&\left(l_{k}, \check{\gamma}^{k}_{l_{k}}, w_1(\check{\gamma}^{k}_{l_{k}}),(\partial_t\varphi_{1,k},\partial_x\varphi_{1,k},\partial_{xx}\varphi_{1,k})( \check{\gamma}^{k}_{l_{k}}), ({ \partial_t^o}\varphi_{2,k},\partial_x\varphi_{2,k},\partial_{xx}\varphi_{2,k})( \check{\gamma}^{k}_{l_{k}})\right)\nonumber\\
&&\underrightarrow{k\rightarrow\infty}\left({\hat{t}},\hat{{\gamma}}_{{\hat{t}}}, w_1(\hat{{\gamma}}_{{\hat{t}}}), (b_1, 2\beta^{\frac{1}{3}} ((\hat{{\gamma}}_{{\hat{t}}}({{\hat{t}}}))_N-(\hat{{\eta}}_{{\hat{t}}}({{\hat{t}}}))_N), X_N), (0,\mathbf{0},\mathbf{0})\right),
\end{eqnarray}
\begin{eqnarray}\label{0608vw1}
&&\left(s_{k}, \check{\eta}^{k}_{s_{k}}, w_2(\check{\eta}^{k}_{s_{k}}),(\partial_t\psi_{1,k}, \partial_x\psi_{1,k}, \partial_{xx}\psi_{1,k})( \check{\eta}^{k}_{s_{k}}), ({ \partial_t^o}\psi_{2,k}, \partial_x\psi_{2,k}, \partial_{xx}\psi_{2,k})( \check{\eta}^{k}_{s_{k}})\right)\nonumber\\
&& \ \ \ \ \underrightarrow{k\rightarrow\infty}\left({\hat{t}},\hat{{\eta}}_{{\hat{t}}}, w_2(\hat{{\eta}}_{{\hat{t}}}), (b_2, 2\beta^{\frac{1}{3}}((\hat{{\eta}}_{{\hat{t}}}({{\hat{t}}}))_N-(\hat{{\gamma}}_{{\hat{t}}}({{\hat{t}}}))_N), Y_N), (0,\mathbf{0},\mathbf{0})\right),
\end{eqnarray}
where $b_{1}+b_{2}=0$ and $X_N,Y_N\in \mathcal{S}(H_N)$ satisfy the following inequality:
\begin{eqnarray}\label{II10}
{-6\beta^{\frac{1}{3}}}\left(\begin{array}{cc}
I&0\\
0&I
\end{array}\right)\leq \left(\begin{array}{cc}
X_N&0\\
0&Y_N
\end{array}\right)\leq  6\beta^{\frac{1}{3}} \left(\begin{array}{cc}
I&-I\\
-I&I
\end{array}\right).
\end{eqnarray}
We note that  
sequence  $(\check{\gamma}^{k}_{l_k},\check{\eta}^{k}_{s_k},l_{k},s_{k},\varphi_{1,k},\psi_{1,k},\varphi_{2,k},\psi_{2,k})$ and $b_{1},b_{2},X_N,Y_N$  depend on  $\beta$,
$\varepsilon$ and $N$. We also note that  (\ref{II10}) follows from (\ref{II0615}) choosing $\kappa=\frac{1}{2}\beta^{-\frac{1}{3}}$. In fact, by (\ref{07063}),
$$
A= \nabla_x^2\varphi((\hat{\gamma}_{\hat{t}}(\hat{t}\,))_N,(\hat{\eta}_{\hat{t}}(\hat{t}\,))_N)
=2\beta^{\frac{1}{3}}\left(\begin{array}{cc}
I&-I\\
-I&I
\end{array}\right),
$$
and thus, if $\kappa=\frac{1}{2}\beta^{-\frac{1}{3}}$,
$$A+\kappa A^2=(1+4\kappa \beta^{\frac{1}{3}})A=3A,$$ and
$$-\left(\frac{1}{\kappa}+|A|\right)=-\left(2\beta^{\frac{1}{3}}+4\beta^{\frac{1}{3}}\right)=-6\beta^{\frac{1}{3}}.$$
Then from (\ref{II0615}) it follows that  (\ref{II10}) holds true. For every $(t,\gamma_t; s,\eta_s)\in ([T-\bar{a},T]\times{{\Lambda}^{T-\bar{a}}})^2$, define
\begin{eqnarray*}
	\chi^{k,N,1}(\gamma_t)  &:=&\varphi_{1,k}(\gamma_t)-2\beta^{\frac{1}{3}}\left|(\gamma_t(t))_N-(e^{(t-\hat{t}\,)A}\hat{\xi}_{\hat{t}}(\hat{t}\,))_N\right|^2,\\[3mm]
	\chi^{k,N,2}(\gamma_t)&:=& \varepsilon\frac{\nu T-t}{\nu
		T}\, \Upsilon(\gamma_t)
	+\varepsilon \overline{\Upsilon}(\gamma_t,\hat{\gamma}_{\hat{t}})
	+\sum_{i=0}^{\infty}
	\frac{1}{2^i}\, \overline{\Upsilon}(\gamma^i_{t_i},\gamma_t)+2^5\beta\, \Upsilon(\gamma_{t},\hat{{\xi}}_{{\hat{t}}})\\
	&&+\varphi_{2,k}(\gamma_t)
	+2\beta^{\frac{1}{3}}\left|\gamma_t(t)-e^{(t-\hat{t}\,)A}\, \hat{\xi}_{\hat{t}}(\hat{t}\,)\right|^2, \\[3mm]
	\chi^{k,N}(\gamma_t)  &:=&\chi^{k,N,1}(\gamma_t)+\chi^{k,N,2}(\gamma_t); \quad \hbox{\rm and }
\end{eqnarray*}
\begin{eqnarray*}
	\hbar^{k,N,1}(\eta_s)   &:=&\psi_{1,k}(\eta_s)-2\beta^{\frac{1}{3}}\left|(\eta_s(s))_N-(e^{(s-\hat{t}\,)A}\hat{\xi}_{\hat{t}}(\hat{t}\,))_N\right|^2,\\[3mm]
	\hbar^{k,N,2}(\eta_s) &:=& \varepsilon\frac{\nu T-s}{\nu
		T}\, \Upsilon(\eta_s)
	+\varepsilon \overline{\Upsilon}(\eta_s,\hat{\eta}_{\hat{t}})+\sum_{i=0}^{\infty}
	\frac{1}{2^i}{\Upsilon}(\eta^i_{t_i},\eta_s)+2^5\beta\Upsilon(\eta_{s},\hat{{\xi}}_{{\hat{t}}})\\
	&&
	+\psi_{2,k}(\eta_s)
	+2\beta^{\frac{1}{3}}\left|\eta_s(s)-e^{(s-\hat{t}\,)A}\, \hat{\xi}_{\hat{t}}(\hat{t}\,)\right|^2, \\[3mm]
	\hbar^{k,N}(\eta_s)   &:=&\hbar^{k,N,1}(\eta_s)+\hbar^{k,N,2}(\eta_s).
\end{eqnarray*}
Clearly, { $\chi^{k,N,2}\in{\mathcal{G}}_{{l_k}}$ and $\hbar^{k,N,2}\in {\mathcal{G}}_{{s_k}}$}, and in view of  Lemma \ref{lemma03302}, we have  $\chi^{k,N,1},\hbar^{k,N,1}\in \Phi_{\hat{t}}$.
Moreover, by  (\ref{0609a}), (\ref{0609b}) and definitions of $w_1$ and $w_2$, we have
$$
\left(W_1-\chi^{k,N,1}-\chi^{k,N,2}\right)(\check{\gamma}^k_{l_k})=\sup_{(t,\gamma_t)\in [{{l_k}},T]\times\Lambda^{{l_k}}}
\left(W_1-\chi^{k,N,1}-\chi^{k,N,2}\right)(\gamma_t)
$$
and
$$
\left(W_2+\hbar^{k,N,1}+\hbar^{k,N,2}\right)(\check{\eta}^k_{s_k})=\inf_{(s,\eta_s)\in [{s_k},T]\times\Lambda^{{s_k}}}
\left(W_2+\hbar^{k,N,1}+\hbar^{k,N,2}\right)(\eta_s).
$$
From $l_{k},s_{k}\rightarrow {\hat{t}}$ as $k\rightarrow\infty$ and ${\hat{t}}<T$
for $\beta>N_0$, it follows that for every fixed $\beta>N_0$,   there is a constant $ K_\beta>0$ such that
$$
|l_{k}|\vee|s_{k}|<T 
\quad \mbox{for all}    \ \ k\geq K_\beta.
$$
Now, for every $\beta>N_0$ and  $k>K_\beta$, 
we have from the definition of viscosity solutions that
\begin{eqnarray}\label{vis1}
\begin{aligned}
&{ \partial_t^o}\chi^{k,N}(\check{\gamma}^k_{{{l_k}}})
+\left\langle A^*\partial_{x}(\varphi_{1,k})(\check{\gamma}^k_{{{l_k}}}),\, \check{\gamma}^k_{{{l_k}}}({{l_k}})\right\rangle_H\\
&-4\beta^{\frac{1}{3}}\left\langle A^*\left(\left(\check{\gamma}^k_{l_k}({l_k})\right)_N-\left(e^{({l_k}-\hat{t}\,)A}\hat{\xi}_{\hat{t}}(\hat{t}\,)\right)_N\right),\,
\check{\gamma}^k_{{{l_k}}}({l_k})
\right\rangle_H\\
&+{\mathbf{H}}\left(\check{\gamma}^k_{{{l_k}}}, \, W_1(\check{\gamma}^k_{{{l_k}}}), \,\partial_x\chi^{k,N}(\check{\gamma}^k_{{{l_k}}}),
\, \partial_{xx}\chi^{k,N}(\check{\gamma}^k_{{{l_k}}})
\right)\geq c
\end{aligned}
\end{eqnarray}
and
\begin{eqnarray}\label{vis2}
\begin{aligned}
&
-{ \partial_t^o}\hbar^{k,N}(\check{\eta}^k_{{{s_k}}})-\left\langle A^*\partial_{x}(\psi_{1,k})(\check{\eta}^k_{{{s_k}}}),\, \check{\eta}^k_{{{s_k}}}({{s_k}})\right\rangle_H\\
&+4\beta^{\frac{1}{3}}\left\langle A^*\left(\left(\check{\eta}^k_{s_k}({s_k})\right)_N-\left(e^{({s_k}-\hat{t}\,)A}\hat{\xi}_{\hat{t}}(\hat{t}\,)\right)_N\right), \,
\check{\eta}^k_{{{s_k}}}({s_k})
\right\rangle_H\\
&+{\mathbf{H}}\left(\check{\eta}^k_{{{s_k}}}, \, W_2(\check{\eta}^k_{{{s_k}}}),\,
-\partial_x\hbar^{k,N}(\check{\eta}^k_{{{s_k}}}),\, -\partial_{xx}\hbar^{k,N}(\check{\eta}^k_{{{s_k}}})\right)\leq0,
\end{aligned}
\end{eqnarray}
where, for every $(t,\gamma_t)\in [{l_k},T]\times{{\Lambda}^{{l_k}}}$ and  $ (s,\eta_s)\in [{s_k},T]\times{{\Lambda}^{{s_k}}}$, we have from Lemmas \ref{theoremS} and \ref{lemma03302},
\begin{eqnarray*}
	{ \partial_t^o}\chi^{k,N}(\gamma_t)&:=&\partial_t\chi^{k,N,1}(\gamma_t)+{ \partial_t^o}\chi^{k,N,2}(\gamma_t)\\
	&=&
	\partial_t(\varphi_{1,k})(\gamma_{{t}})+4\beta^{\frac{1}{3}}\left\langle A^*(({\gamma}_{{t}}({t}))_N-(e^{({t}-\hat{t}\,)A}\hat{\xi}_{\hat{t}}(\hat{t}\,))_N),\,
	e^{({t}-\hat{t}\,)A}\hat{\xi}_{\hat{t}}(\hat{t}\,)\right\rangle_H\\
	&&-\frac{\varepsilon}{\nu T}\, \Upsilon(\gamma_{{t}})
	+2\varepsilon({t}-{\hat{t}})+2\sum_{i=0}^{\infty}\frac{1}{2^i}(t-t_{i})+ \partial_t^o(\varphi_{2,k})(\gamma_{{t}}),
\end{eqnarray*}
\begin{eqnarray*}
	\partial_x\chi^{k,N}(\gamma_t)&:=&\partial_x\chi^{k,N,1}(\gamma_t)+\partial_x\chi^{k,N,2}(\gamma_t)\\
	&=&\partial_{x}(\varphi_{1,k})(\gamma_{{t}})
	+4\beta^{\frac{1}{3}}\left((\gamma_{t}(t))^-_N-(e^{(t-\hat{t}\,)A}\hat{\xi}_{\hat{t}}(\hat{t}\,))^-_N\right)
	\\
	&&\!\!\!\!+\varepsilon\frac{\nu T-{t}}{\nu T}\, \partial_x\Upsilon(\gamma_{{t}}) +\varepsilon\, \partial_x \Upsilon(\gamma_{{t}}-\hat{\gamma}_{{\hat{t},t}}^A)
	+2^5\beta\, \partial_x\Upsilon(\gamma_{{t}}-\hat{\xi}_{{\hat{t},t}}^A)
	\\
	&&+\partial_x\left[\sum_{i=0}^{\infty}\frac{1}{2^i}
	\Upsilon(\gamma_{t}-(\gamma^{i})_{t_{i},t}^A)
	\right]+\partial_{x}(\varphi_{2,k})(\gamma_{{t}}),
\end{eqnarray*}
\begin{eqnarray*}
	\partial_{xx}\chi^{k,N}(\gamma_t)&:=&\partial_{xx}\chi^{k,N,1}(\gamma_t)+\partial_{xx}\chi^{k,N,2}(\gamma_t)\\
	&=&\partial_{xx}(\varphi_{1,k})(\gamma_{{t}})+4\beta^{\frac{1}{3}}Q_N+\varepsilon\frac{\nu T-{t}}{\nu T}\, \partial_{xx}\Upsilon(\gamma_{{t}}) \\
	&&+\varepsilon\, \partial_{xx}\Upsilon(\gamma_{{t}}-\hat{\gamma}_{{\hat{t},t}}^A)
	+2^5\beta\, \partial_{xx}\Upsilon(\gamma_{{t}}-\hat{\xi}_{{\hat{t},t}}^A)\\
	&&
	+\partial_{xx}\left[\sum_{i=0}^{\infty}\frac{1}{2^i}
	\Upsilon(\gamma_{t}-(\gamma^{i})_{t_{i},t}^A)
	\right]+\partial_{xx}(\varphi_{2,k})(\gamma_{{t}}),
\end{eqnarray*}
\begin{eqnarray*}
{ \partial_t^o}\hbar^{k,N}(\eta_s)&:=&\partial_t\hbar^{k,N,1}(\eta_s)+{ \partial_t^o}\hbar^{k,N,2}(\eta_s)\\
	&=&\!\!\!\!
	\partial_t\psi_{1,k}(\eta_{{s}})+4\beta^{\frac{1}{3}}\left\langle A^*(({\eta}_{s}(s))_N-(e^{(s-\hat{t}\,)A}\hat{\xi}_{\hat{t}}(\hat{t}\,))_N),\,
	e^{(s-\hat{t}\,)A}\hat{\xi}_{\hat{t}}(\hat{t}\,)\right\rangle_H\\
	&&-\frac{\varepsilon}{\nu T}\Upsilon(\eta_{{s}})+2\varepsilon({s}-{\hat{t}})
	+ \partial_t^o\psi_{2,k}(\eta_{{s}}),
\end{eqnarray*}
\begin{eqnarray*}
	\partial_x\hbar^{k,N}(\eta_s)&:=&\partial_x\hbar^{k,N,1}(\eta_s)+\partial_x\hbar^{k,N,2}(\eta_s)\\
	&=& \partial_{x}\psi_{1,k}(\eta_{{s}})+4\beta^{\frac{1}{3}}\left((\eta_{s}(s))^-_N-(e^{(s-\hat{t}\,)A}\hat{\xi}_{\hat{t}}(\hat{t}\,))^-_N\right)
	\nonumber\\
	&&\!\!\!\!+\varepsilon\frac{\nu T-{s}}{\nu T}
	\,\partial_x\Upsilon(\eta_{{s}})+\varepsilon\,\partial_x\Upsilon(\eta_{{s}}-\hat{\eta}^A_{{\hat{t}},s})
	+2^5\beta\, \partial_x\Upsilon(\eta_{{s}}-\hat{\xi}_{{\hat{t}},s}^A)\\
	&&
	+\partial_x\left[\sum_{i=0}^{\infty}\frac{1}{2^i}
	\Upsilon(\eta_{s}-({\eta}^{i})_{t_{i},s}^A)
	\right]+\partial_{x}\psi_{2,k}(\eta_{{s}}),
\end{eqnarray*}
and
\begin{eqnarray*}
	\partial_{xx}\hbar^{k,N}(\eta_s)&:=&\partial_{xx}\hbar^{k,N,1}(\eta_s)+\partial_{xx}\hbar^{k,N,2}(\eta_s)\\
	&=&\partial_{xx}\psi_{1,k}(\eta_{{s}})+4\beta^{\frac{1}{3}}Q_N+\varepsilon\frac{\nu T-{s}}{\nu T}
	\, \partial_{xx}\Upsilon(\eta_{{s}})\\
	&&+\varepsilon\, \partial_{xx}\Upsilon(\eta_{{s}}-\hat{\eta}_{{\hat{t}},s}^A)
	+2^5\beta\, \partial_{xx}\Upsilon(\eta_{{s}}-\hat{\xi}_{{\hat{t}},s}^A)
	\\
	&&
	+\partial_{xx}\left[\sum_{i=0}^{\infty}\frac{1}{2^i}
	\Upsilon(\eta_{s}-({\eta}^{i})_{t_{i},s}^A)
	\right]+\partial_{xx}\psi_{2,k}(\eta_{{s}}).
\end{eqnarray*}
\par
{ $Step\ 5.$  Calculation and  completion of the proof.}
\par
{ By   (\ref{03108}) and (\ref{03109}), 
	there exists a generic constant $C>0$ such that}
\begin{eqnarray*}
	&&\left|\partial_x\Upsilon(\check{\gamma}^k_{{{l_k}}}-\hat{\gamma}_{{\hat{t}},{l_k}}^A)\right|
	+\left|\partial_x\Upsilon(\check{\eta}^k_{{{s_k}}}-\hat{\eta}_{{\hat{t}},{s_k}}^A)\right|\\
	&\leq& C\left|e^{({l_k}-\hat{t}\,)A}\hat{\gamma}_{{\hat{t}}}(\hat{t}\,)-\check{\gamma}^k_{l_k}({l_k})\right|^5
	+C\left|e^{({s_k}-\hat{t}\,)A}\hat{\eta}_{{\hat{t}}}(\hat{t}\,)-\check{\eta}^k_{s_k}({s_k})\right|^5,
\end{eqnarray*}
\begin{eqnarray*}
	&&\left|\partial_{xx}\Upsilon(\check{\gamma}^k_{{{l_k}}}-\hat{\gamma}_{{\hat{t}},{l_k}}^A)\right|
	+\left|\partial_{xx}\Upsilon(\check{\eta}^k_{{{s_k}}}-\hat{\eta}_{{\hat{t}},{s_k}}^A)\right|\\
	&\leq& C\left|e^{({l_k}-\hat{t}\,)A}\hat{\gamma}_{{\hat{t}}}(\hat{t}\,)-\check{\gamma}^k_{l_k}({l_k})\right|^4
	+C\left|e^{({s_k}-\hat{t}\,)A}\hat{\eta}_{{\hat{t}}}(\hat{t}\,)-\check{\eta}^k_{s_k}({s_k})\right|^4,
\end{eqnarray*}
Letting 
$k\rightarrow\infty$ in (\ref{vis1}) and (\ref{vis2}), and using (\ref{0608v1}) and (\ref{0608vw1}),  we have
\begin{eqnarray}\label{03103}
\begin{aligned}
& b_1+4\beta^{\frac{1}{3}}\left\langle A^*\left((\hat{\gamma}_{\hat{t}}(\hat{t}\,))_N-(\hat{\xi}_{\hat{t}}(\hat{t}\,))_N\right),\,
\hat{\xi}_{\hat{t}}(\hat{t}\,)\right\rangle_H-\frac{\varepsilon}{\nu T}\Upsilon(\hat{{\gamma}}_{{\hat{t}}})\\
&+2\sum_{i=0}^{\infty}\frac{1}{2^i}(\hat{t}-t_i)+2\beta^{\frac{1}{3}} \left\langle A^*\left((\hat{{\gamma}}_{{\hat{t}}}({{\hat{t}}}))_N-(\hat{{\eta}}_{{\hat{t}}}({{\hat{t}}}))_N\right),\, \hat{\gamma}_{{\hat{t}}}({\hat{t}})\right\rangle_H\\
&-4\beta^{\frac{1}{3}}\left\langle A^*\left((\hat{\gamma}_{\hat{t}}(\hat{t}\,))_N-(\hat{\xi}_{\hat{t}}(\hat{t}\,))_N\right),\,
\hat{\gamma}_{{\hat{t}}}(\hat{{t}})
\right\rangle_H\\
&+{\mathbf{H}}\left(\hat{{\gamma}}_{{\hat{t}}},\, W_1(\hat{{\gamma}}_{{\hat{t}}}),\,
\partial_x\chi^{N}(\hat{\gamma}_{\hat{t}}),\, \partial_{xx}\chi^{N}(\hat{\gamma}_{\hat{t}})\right)
\geq c
\end{aligned}
\end{eqnarray}
and
\begin{eqnarray}\label{03104}
\begin{aligned}
& -b_2-4\beta^{\frac{1}{3}}\left\langle A^*\left((\hat{\eta}_{\hat{t}}(\hat{t}\,))_N-(\hat{\xi}_{\hat{t}}(\hat{t}\,))_N\right),\,
\hat{\xi}_{\hat{t}}(\hat{t}\,)\right\rangle_H+ \frac{\varepsilon}{\nu T}\Upsilon(\hat{{\eta}}_{{\hat{t}}})\\
&+2\beta^{\frac{1}{3}} \left\langle A^*\left((\hat{{\gamma}}_{{\hat{t}}}({{\hat{t}}}))_N-(\hat{{\eta}}_{{\hat{t}}}({{\hat{t}}}))_N\right),\, \hat{{\eta}}_{{\hat{t}}}({{\hat{t}}})\right\rangle _H\\
&
+4\beta^{\frac{1}{3}}\left\langle A^*\left((\hat{\eta}_{\hat{t}}(\hat{t}\,))_N-(\hat{\xi}_{\hat{t}}(\hat{t}\,))_N\right),\,
\hat{\eta}_{{\hat{t}}}(\hat{{t}})
\right\rangle_H\\
&+{\mathbf{H}}\left(\hat{{\eta}}_{{\hat{t}}},\, W_2(\hat{{\eta}}_{{\hat{t}}}),\, -\partial_x\hbar^{N}(\hat{\eta}_{{\hat{t}}}),\, -\partial_{xx}\hbar^{N}(\hat{\eta}_{\hat{t}})\right)
\leq0,
\end{aligned}
\end{eqnarray}
where
\begin{eqnarray}\label{1002c}
\begin{aligned}
\partial_x\chi^{N}(\hat{\gamma}_{\hat{t}})
&:=
2\beta^{\frac{1}{3}}\left(\hat{{\gamma}}_{{\hat{t}}}({{\hat{t}}})-\hat{{\eta}}_{{\hat{t}}}({{\hat{t}}})\right)+2^5\beta\, \partial_x \Upsilon(\hat{\gamma}_{{\hat{t}}}-\hat{\xi}_{{\hat{t}}})\\
&\quad +\varepsilon\frac{\nu T-{\hat{t}}}{\nu T}\, \partial_x\Upsilon(\hat{{\gamma}}_{{\hat{t}}})
+\partial_x\sum_{i=0}^{\infty}\frac{1}{2^i}\, \Upsilon(\hat{\gamma}_{\hat{t}}-(\gamma^i)_{t_i,\hat{t}}^A),
\end{aligned}
\end{eqnarray}
\begin{eqnarray}\label{1002c1}
\begin{aligned}
\partial_{xx}\chi^{N}(\hat{\gamma}_{\hat{t}})
&:= X_N+4\beta^{\frac{1}{3}}Q_N+2^5\beta\, \partial_{xx}\Upsilon(\hat{\gamma}_{{\hat{t}}}-\hat{\xi}_{{\hat{t}}})\\
&\quad +\varepsilon\frac{\nu T-{\hat{t}}}{\nu T}\, \partial_{xx}\Upsilon(\hat{{\gamma}}_{{\hat{t}}})
+\partial_{xx}\sum_{i=0}^{\infty}\frac{1}{2^i}\, \Upsilon(\hat{\gamma}_{\hat{t}}-(\gamma^i)_{t_i,\hat{t}}^A),
\end{aligned}
\end{eqnarray}
\begin{eqnarray}\label{1002c2}
\begin{aligned}
\partial_x\hbar^{N}(\hat{\eta}_{{\hat{t}}})
&:=
2\beta^{\frac{1}{3}}\left(\hat{{\eta}}_{{\hat{t}}}({{\hat{t}}})-\hat{{\gamma}}_{{\hat{t}}}({{\hat{t}}})\right)+2^5\beta\, \partial_x\Upsilon(\hat{\eta}_{{\hat{t}}}-\hat{\xi}_{{\hat{t}}})\\
&\quad +\varepsilon\frac{\nu T-{\hat{t}}}{\nu T}\, \partial_x\Upsilon(\hat{{\eta}}_{{\hat{t}}})+\partial_x\sum_{i=0}^{\infty}\frac{1}{2^i}
\, \Upsilon(\hat{\eta}_{\hat{t}}-(\eta^i)_{t_i,\hat{t}}^A),
\end{aligned}
\end{eqnarray}
and
\begin{eqnarray}\label{1002c3}
\begin{aligned}
\partial_{xx}\hbar^{N}(\hat{\eta}_{\hat{t}})
&:= Y_N+4\beta^{\frac{1}{3}}Q_N+2^5\beta\, \partial_{xx}\Upsilon(\hat{\eta}_{{\hat{t}}}-\hat{\xi}_{{\hat{t}}})\\
&\quad +\varepsilon\frac{\nu T-{\hat{t}}}{\nu T}\, \partial_{xx}\Upsilon(\hat{{\eta}}_{{\hat{t}}})
+\partial_{xx}\sum_{i=0}^{\infty}\frac{1}{2^i}\,
\Upsilon(\hat{\eta}_{\hat{t}}-(\eta^i)_{t_i,\hat{t}}^A).
\end{aligned}
\end{eqnarray}
Since $b_1+b_2=0$ and $\hat{\xi}_{{\hat{t}}}=\frac{\hat{{\gamma}}_{\hat{t}}+\hat{{\eta}}_{\hat{t}}}{2}$, subtracting
both sides of   (\ref{03103}) from those  of (\ref{03104}), we observe the following cancellations:
\begin{eqnarray*}
&& 4\beta^{\frac{1}{3}}\left\langle A^*\left(\left(\hat{\gamma}_{\hat{t}}(\hat{t}\,)\right)_N-\left(\hat{\xi}_{\hat{t}}(\hat{t}\,)\right)_N\right),\,
\hat{\xi}_{\hat{t}}(\hat{t}\,)\right\rangle_H\\
&&-\left[-4\beta^{\frac{1}{3}}\left\langle A^*\left((\hat{\eta}_{\hat{t}}(\hat{t}\,))_N-(\hat{\xi}_{\hat{t}}(\hat{t}\,))_N\right),\,
\hat{\xi}_{\hat{t}}(\hat{t}\,)\right\rangle_H\right]\\
&=&4\beta^{\frac{1}{3}}\left\langle A^*\left(\left(\hat{\gamma}_{\hat{t}}(\hat{t}\,)\right)_N+\left(\hat{\eta}_{\hat{t}}(\hat{t}\,)\right)_N-2\left(\hat{\xi}_{\hat{t}}(\hat{t}\,)\right)_N\right),\, \hat{\xi}_{\hat{t}}(\hat{t}\,)\right\rangle_H=0,
\end{eqnarray*}
\begin{eqnarray*}
 &&2\beta^{\frac{1}{3}}\left\langle A^*\left(\left(\hat{\gamma}_{\hat{t}}(\hat{t}\,)\right)_N-\left(\hat{\eta}_{\hat{t}}(\hat{t}\,)\right)_N\right),\,\hat{\gamma}_{\hat{t}}(\hat{t}\,)\right\rangle_H
 \\
 &&-2\beta^{\frac{1}{3}} \left\langle A^*\left(\left(\hat{\gamma}_{\hat{t}}(\hat{t}\,)\right)_N-\left(\hat{\eta}_{\hat{t}}(\hat{t}\,)\right)_N\right),\, \hat{\eta}_{\hat{t}}(\hat{t}\,)\right\rangle _H\\
&=&2\beta^{\frac{1}{3}}\left\langle A^*\left(\left(\hat{\gamma}_{\hat{t}}(\hat{t}\,)\right)_N-\left(\hat{\eta}_{\hat{t}}(\hat{t}\,)\right)_N\right),\,\hat{\gamma}_{\hat{t}}(\hat{t}\,)-\hat{\eta}_{\hat{t}}(\hat{t}\,)\right\rangle_H,
\end{eqnarray*}
 \begin{eqnarray*}
&&-4\beta^{\frac{1}{3}}\left\langle A^*\left((\hat{\gamma}_{\hat{t}}(\hat{t}\,))_N-(\hat{\xi}_{\hat{t}}(\hat{t}\,))_N\right),\,
\hat{\gamma}_{\hat{t}}(\hat{t}\,)
\right\rangle_H\\
&&-
4\beta^{\frac{1}{3}}\left\langle A^*\left(\left(\hat{\eta}_{\hat{t}}(\hat{t}\,)\right)_N-\left(\hat{\xi}_{\hat{t}}(\hat{t}\,)\right)_N\right),\,
\hat{\eta}_{\hat{t}}(\hat{t}\,)
\right\rangle_H\\
&=&-2\beta^{\frac{1}{3}}\left\langle A^*\left(\left(\hat{\gamma}_{\hat{t}}(\hat{t}\,)\right)_N-\left(\hat{\eta}_{\hat{t}}(\hat{t}\,)\right)_N\right),\, \hat{\gamma}_{\hat{t}}(\hat{t}\,)\right\rangle_H\\
&&-
2\beta^{\frac{1}{3}}\left\langle A^*\left(\left(\hat{\eta}_{\hat{t}}(\hat{t}\,)\right)_N-\left(\hat{\gamma}_{\hat{t}}(\hat{t}\,)\right)_N\right),\,
\hat{\eta}_{\hat{t}}(\hat{t}\,)
\right\rangle_H\\
&=&-2\beta^{\frac{1}{3}}\left\langle A^*\left(\left(\hat{\gamma}_{\hat{t}}(\hat{t}\,)\right)_N-\left(\hat{\eta}_{\hat{t}}(\hat{t}\,)\right)_N\right),\,\hat{\gamma}_{\hat{t}}(\hat{t}\,)-\hat{\eta}_{\hat{t}}(\hat{t}\,)\right\rangle_H,
\end{eqnarray*}
 {\it all the terms which involve the unbounded operator $A$  mutually  cancels out},  and thus we have
\begin{eqnarray}\label{vis112}
\begin{aligned}
&\quad c+ \frac{\varepsilon}{\nu T}\left(\Upsilon(\hat{{\gamma}}_{{\hat{t}}})+\Upsilon(\hat{{\eta}}_{{\hat{t}}})\right)-2\sum_{i=0}^{\infty}\frac{1}{2^i}(\hat{t}-t_i)\\
&\leq{\mathbf{H}}\left(\hat{{\gamma}}_{{\hat{t}}},\, W_1(\hat{{\gamma}}_{{\hat{t}}}),\, \partial_x\chi^{N}(\hat{\gamma}_{\hat{t}}),\, \partial_{xx}\chi^{N}(\hat{\gamma}_{\hat{t}})\right)\\
&\quad-{\mathbf{H}}\left(\hat{{\eta}}_{{\hat{t}}},\, W_2(\hat{{\eta}}_{{\hat{t}}}),\, -\partial_x\hbar^{N}(\hat{\eta}_{{\hat{t}}}),\, -\partial_{xx}\hbar^{N}(\hat{\eta}_{\hat{t}})\right).
\end{aligned}
\end{eqnarray}
On the  other hand, from  (\ref{5.1}) and via a simple calculation, we have
\begin{eqnarray}\label{v4}
\begin{aligned}
&\quad{\mathbf{H}}\left(\hat{{\gamma}}_{{\hat{t}}},\, W_1(\hat{{\gamma}}_{{\hat{t}}}),\, \partial_x\chi^{N}(\hat{\gamma}_{\hat{t}}),\, \partial_{xx}\chi^{N}(\hat{\gamma}_{\hat{t}})\right)\\
&\quad-{\mathbf{H}}\left(\hat{{\eta}}_{{\hat{t}}},\, W_2(\hat{{\eta}}_{{\hat{t}}}),\, -\partial_x\hbar^{N}(\hat{\eta}_{{\hat{t}}}),\, -\partial_{xx}\hbar^{N}(\hat{\eta}_{\hat{t}})\right)\\
&\leq{\mathbf{H}}\left(\hat{{\gamma}}_{{\hat{t}}},\, W_2(\hat{{\eta}}_{{\hat{t}}}),\, \partial_x\chi^{N}(\hat{\gamma}_{\hat{t}}),\, \partial_{xx}\chi^{N}(\hat{\gamma}_{\hat{t}})\right)\\
&\quad-{\mathbf{H}}\left(\hat{{\eta}}_{{\hat{t}}},\, W_2(\hat{{\eta}}_{{\hat{t}}}),\, -\partial_x\hbar^{N}(\hat{\eta}_{{\hat{t}}}),\, -\partial_{xx}\hbar^{N}(\hat{\eta}_{\hat{t}})\right)\\
&\leq\sup_{u\in U}(J_{1}+J_{2}+J_{3}).
\end{aligned}
\end{eqnarray}
Here,
from Assumption \ref{hypstate} (ii), (\ref{03109})  and (\ref{II10}),  we have
\begin{eqnarray}\label{j2}
\quad&\begin{aligned}
&J_{1}
:=\, \frac{1}{2}\mbox{Tr}\left[\partial_{xx}\chi^{N}(\hat{\gamma}_{\hat{t}})\, { (GG^*)}(\hat{{\gamma}}_{{\hat{t}}},u)\right]-\frac{1}{2}\mbox{Tr}\left[-\partial_{xx}\hbar^{N}(\hat{\eta}_{\hat{t}})\, { (GG^*)}(\hat{{\eta}}_{{\hat{t}}},u)\right]\\[3mm]
\leq&\, 3
\beta^{\frac{1}{3}}\left|{G}(\hat{{\gamma}}_{{\hat{t}}},u)-G(\hat{{\eta}}_{{\hat{t}}},u)\right|_{L_2(\Xi,H)}^2\\
&+2\beta^{\frac{1}{3}} \left(\left|Q_N{G}(\hat{{\gamma}}_{{\hat{t}}},u)\right|_{L_2(\Xi,H)}^2+\left|Q_N{G}(\hat{{\eta}}_{{\hat{t}}},u)\right|_{L_2(\Xi,H)}^2\right)\\
&
+306\, \beta\, \left|\hat{{\gamma}}_{{\hat{t}}}({\hat{t}})-\hat{{\eta}}_{{\hat{t}}}({\hat{t}})\right|^4\left(\left|{G}(\hat{{\gamma}}_{{\hat{t}}},u)\right|_{L_2(\Xi,H)}^2+
\left|G(\hat{{\eta}}_{{\hat{t}}},u)\right|_{L_2(\Xi,H)}^2\right)\\
&+153\varepsilon\frac{\nu T-{\hat{t}}}{\nu    T}\left(\left|\hat{{\gamma}}_{{\hat{t}}}({\hat{t}})\right|^4|{G}(\hat{{\gamma}}_{{\hat{t}}},u)|_{L_2(\Xi,H)}^2+\left|\hat{{\eta}}_{{\hat{t}}}({\hat{t}})\right|^4\left|{G}(\hat{{\eta}}_{{\hat{t}}},u)\right|_{L_2(\Xi,H)}^2\right)
\\
&\qquad\qquad
+153\sum_{i=0}^{\infty}\frac{1}{2^i}|((\gamma^i)^A_{t_i, \hat t}-\hat{\gamma}_{\hat{t}})(\hat{t}\,)|^4
\, |G(\hat{{\gamma}}_{{\hat{t}}},u)|_{L_2(\Xi,H)}^2\\
&\qquad\qquad
+153\sum_{i=0}^{\infty}\frac{1}{2^i}\left|((\eta^i)^A_{t_i, \hat t}-\hat{\eta}_{\hat{t}})(\hat{t}\,)\right|^4
\, \left|G(\hat{{\eta}}_{{\hat{t}}},u)\right|_{L_2(\Xi,H)}^2
\\[3mm]
\leq&3
\beta^{\frac{1}{3}}{L^2} ||\hat{{\gamma}}_{{\hat{t}}}-\hat{{\eta}}_{{\hat{t}}}||_0^2+\!\!2\beta^{\frac{1}{3}}\!\! \left(|Q_N{G}(\hat{{\gamma}}_{{\hat{t}}},u)|_{L_2(\Xi,H)}^2\!\!+|Q_N{G}(\hat{{\eta}}_{{\hat{t}}},u)|_{L_2(\Xi,H)}^2\right)\\
&+153L^2\left (1+||\hat{{\gamma}}_{{\hat{t}}}||_0^2
+||\hat{{\eta}}_{{\hat{t}}}||_0^2
\right)
\sum_{i=0}^{\infty}\frac{1}{2^i}\!\!\left[\left|((\gamma^i)^A_{t_i, \hat t}-\hat{\gamma}_{\hat{t}})(\hat{t}\,)\right|^4\!\!\!
+\left|((\eta^i)^A_{t_i, \hat t}-\hat{\eta}_{\hat{t}})(\hat{t}\,)\right|^4\right]
\\
&+ 306\beta L^2\left |\hat{{\gamma}}_{{\hat{t}}}({\hat{t}})-\hat{{\eta}}_{{\hat{t}}}({\hat{t}})\right|^4\left (2+||\hat{{\gamma}}_{{\hat{t}}}||_0^2
+||\hat{{\eta}}_{{\hat{t}}}||_0^2
\right)\\
&+306\varepsilon \frac{\nu T-{\hat{t}}}{\nu T}L^2\left(1+||\hat{{\gamma}}_{{\hat{t}}}||^6_0
+||\hat{{\eta}}_{{\hat{t}}}||^6_0\right);
\end{aligned}
\end{eqnarray}
from Assumption \ref{hypstate} (ii) and (\ref{03108}),  we have
\begin{eqnarray}\label{j1}
&\qquad\begin{aligned}
&\quad J_{2}:= {\left\langle {F}(\hat{{\gamma}}_{{\hat{t}}},u),\, \partial_x\chi^{N}(\hat{\gamma}_{\hat{t}})\right\rangle_{H}  -\left\langle {F}(\hat{{\eta}}_{{\hat{t}}},u),\, -\partial_x\hbar^{N}(\hat{\eta}_{{\hat{t}}})\right\rangle_{H}}\\
\leq&\, 2\beta^{\frac{1}{3}}{L}\left|\hat{{\gamma}}_{{\hat{t}}}({\hat{t}})-\hat{{\eta}}_{{\hat{t}}}({\hat{t}})\right|\,
\left||\hat{{\gamma}}_{{\hat{t}}}-\hat{{\eta}}_{{\hat{t}}}\right||_0\\
&\qquad+18\beta\left|\hat{{\gamma}}_{{\hat{t}}}({\hat{t}})-\hat{{\eta}}_{{\hat{t}}}({\hat{t}})\right|^5L\left(2+||\hat{{\gamma}}_{{\hat{t}}}||_0
+||\hat{{\eta}}_{{\hat{t}}}||_0\right)
\\
&
+18L(1+||\hat{{\gamma}}_{{\hat{t}}}||_0+||\hat{{\eta}}_{{\hat{t}}}||_0)\sum_{i=0}^{\infty}\frac{1}{2^i}\!\!\left[\left|((\gamma^i)^A_{t_i, \hat t}-\hat{\gamma}_{\hat{t}})(\hat{t}\,)\right|^5
\!\!\!+\left|((\eta^i)^A_{t_i, \hat t}-\hat{\eta}_{\hat{t}})(\hat{t}\,)\right|^5\right]\\
& +36\varepsilon \frac{\nu T-{\hat{t}}}{\nu T} L\left(1+||\hat{{\gamma}}_{{\hat{t}}}||^6_0+||\hat{{\eta}}_{{\hat{t}}}||^6_0\right);
\end{aligned}
\end{eqnarray}
from Assumption \ref{hypcost} and (\ref{03108}),  we have
\begin{eqnarray}\label{j3}
&\quad\begin{aligned}
&\quad
J_{3}:=q\left(\hat{{\gamma}}_{{\hat{t}}}, \, W_2(\hat{{\eta}}_{{\hat{t}}}),\,  \partial_x\chi^{N}(\hat{\gamma}_{\hat{t}})G(\hat{{\gamma}}_{{\hat{t}}},u),\, u\right)\\
&\qquad\qquad-
q\left(\hat{{\eta}}_{{\hat{t}}}, \, W_2(\hat{{\eta}}_{{\hat{t}}}), \, -\partial_x\hbar^{N}(\hat{\eta}_{{\hat{t}}})G(\hat{{\eta}}_{{\hat{t}}},u),\, u\right)\\
\leq&
\, L||\hat{{\gamma}}_{{\hat{t}}}-\hat{{\eta}}_{{\hat{t}}}||_0
+2\beta^{\frac{1}{3}} L^2\left|\hat{{\gamma}}_{{\hat{t}}}({\hat{t}})
-\hat{{\eta}}_{{\hat{t}}}({\hat{t}})\right|\, \left||\hat{{\gamma}}_{{\hat{t}}}-\hat{{\eta}}_{{\hat{t}}}\right||_0\\
&\qquad+18\beta L^2\left|\hat{{\gamma}}_{{\hat{t}}}({\hat{t}})-\hat{{\eta}}_{{\hat{t}}}({\hat{t}})\right|^5
\left(2+||\hat{{\gamma}}_{{\hat{t}}}||_0
+||\hat{{\eta}}_{{\hat{t}}}||_0
\right)\\
&\!\!\!\!\!\!+18L^2\left(1+||\hat{{\gamma}}_{{\hat{t}}}||_0+||\hat{{\eta}}_{{\hat{t}}}||_0\right)\sum_{i=0}^{\infty}\frac{1}{2^i}
\!\!\left[\left|((\gamma^i)^A_{t_i, \hat t}-\hat{\gamma}_{\hat{t}})(\hat{t}\,)\right|^5
\!\!\!+\left|((\eta^i)^A_{t_i, \hat t}-\hat{\eta}_{\hat{t}})(\hat{t}\,)\right|^5\right]\\
&+36\varepsilon \frac{\nu T-{\hat{t}}}{\nu T} L^2\left(1+||\hat{{\gamma}}_{{\hat{t}}}||_0^6
+||\hat{{\eta}}_{{\hat{t}}}||_0^6
\right).
\end{aligned}
\end{eqnarray}
In view of  the property (i) of $(\hat{t},\hat{\gamma}_{\hat{t}},\hat{\eta}_{\hat{t}})$, we have
\begin{eqnarray}\label{1002d1}
2\sum_{i=0}^{\infty}\frac{1}{2^i}(\hat{t}-t_i)
\leq2\sum_{i=0}^{\infty}\frac{1}{2^i}\bigg{(}\frac{1}{2^i\beta}\bigg{)}^{\frac{1}{2}}\leq 4{\bigg{(}\frac{1}{{\beta}}\bigg{)}}^{\frac{1}{2}},
\end{eqnarray}
\begin{eqnarray}\label{1002d2}
\begin{aligned}
&\quad
\sum_{i=0}^{\infty}\frac{1}{2^i}\left[\left|((\gamma^i)^A_{t_i, \hat t}-\hat{\gamma}_{\hat{t}})(\hat{t}\,)\right|^5
+\left|((\eta^i)^A_{t_i, \hat t}-\hat{\eta}_{\hat{t}})(\hat{t}\,)\right|^5\right]\\
&\leq
2\sum_{i=0}^{\infty}\frac{1}{2^i}\bigg{(}\frac{1}{2^i\beta}\bigg{)}^{\frac{5}{6}}\leq 4{\bigg{(}\frac{1}{{\beta}}\bigg{)}}^{\frac{5}{6}},
\end{aligned}
\end{eqnarray}
and
\begin{eqnarray}\label{1002d3}
\begin{aligned}
&\quad
\sum_{i=0}^{\infty}\frac{1}{2^i}\left[\left|((\gamma^i)^A_{t_i, \hat t}-\hat{\gamma}_{\hat{t}})(\hat{t}\,)\right|^4
+\left|((\eta^i)^A_{t_i, \hat t}-\hat{\eta}_{\hat{t}})(\hat{t}\,)\right|^4\right]\\
&\leq
2\sum_{i=0}^{\infty}\frac{1}{2^i}\bigg{(}\frac{1}{2^i\beta}\bigg{)}^{\frac{2}{3}}\leq 4{\bigg{(}\frac{1}{{\beta}}\bigg{)}}^{\frac{2}{3}};
\end{aligned}
\end{eqnarray}
and since $\hat{{\gamma}}_{{\hat{t}}}$ and $\hat{{\eta}}_{{\hat{t}}}$ are independent of $N$, by Assumption  \ref{hypstate5666},
\begin{eqnarray}\label{1002d4}
\qquad \sup_{u\in U}\left[|Q_N{G}(\hat{{\gamma}}_{{\hat{t}}},u)|_{L_2(\Xi,H)}^2+|Q_N{G}(\hat{{\eta}}_{{\hat{t}}},u)|_{L_2(\Xi,H)}^2\right]\rightarrow0\ \mbox{as}\ N\rightarrow \infty.
\end{eqnarray}
Combining (\ref{vis112})-(\ref{j3}),  and letting $N\rightarrow\infty$ and then 
by (\ref{5.10jiajiaaaa}) and (\ref{5.10}), we have for sufficiently large  $\beta>0$,  
\begin{eqnarray}\label{vis122}
\begin{aligned}
c
&\leq
-\frac{\varepsilon}{\nu T}\left(\Upsilon(\hat{{\gamma}}_{{\hat{t}}})
+\Upsilon(\hat{{\eta}}_{{\hat{t}}})\right)\\
&\quad+ \varepsilon \frac{\nu T-{\hat{t}}}{\nu T} (342L+36)L\left(1+||\hat{{\gamma}}_{{\hat{t}}}||_0^6
+||\hat{{\eta}}_{{\hat{t}}}||_0^6
\right)+\frac{c}{4}.
\end{aligned}
\end{eqnarray}
Since $$
\nu=1+\frac{1}{2T(342L+36)L}
\quad \text{and} \quad \bar{a}=\frac{1}{2(342L+36)L}\wedge{T},$$
we see  by (\ref{s0}) and (\ref{5.3}) the following contradiction:
\begin{eqnarray*}\label{vis122}
	c\leq
	\frac{\varepsilon}{\nu
		T}+\frac{c}{4}\leq \frac{c}{2}.
\end{eqnarray*}
The proof is  complete. \ \ $\Box$

In Lemmas below of this chapter, let $\widetilde{w}_{1}^{\hat{t}}, \widetilde{w}_{1}^{\hat{t},*}$ and $\widetilde{w}_{2}^{\hat{t}}, \widetilde{w}_{2}^{\hat{t},*}$ be defined in Definition \ref{definition0607} with respect to $w_1$ defined by (\ref{06091}) and  $w_2$ defined by  (\ref{06092}), respectively. To complete the previous proof, it remains to state and prove the following lemmas.

\begin{lemma}\label{lemma03302}\ \
	For every fixed $(t, a_t) \in [0,T)\times {\Lambda}_{t}$ and $N\in \mathbb{N}^+$, define   $f:{\hat{\Lambda}}^{t}\rightarrow \mathbb{R}$ by
	$$f(\gamma_s):=\left|(\gamma_s(s))_N-\left(e^{(s-t)A}a_t(t)\right)_N\right|^2,\quad (s,\gamma_s)\in [t, T]\times{\hat{\Lambda}}^{t}.$$ Then
	$f\in \Phi_t$.
\end{lemma}

\begin{proof}
First, it is clear that
$$
\partial_xf(\gamma_s)=2\left(\gamma_s(s)\right)_N-2\left(e^{(s-t)A}a_t(t)\right)_N;\ \ \ \partial_{xx}f(\gamma_s)=2P_N,\ \ \ (s,\gamma_s)\in [t, T]\times{\hat{\Lambda}}^{t},
$$
and $(f,\partial_xf, \partial_{xx}f)$ are continuous and grow in a polynomial way.
Second, we consider $\partial_tf$. For every $(s,\gamma_s)\in [t, T)\times{\hat{\Lambda}}^{t}$,
\begin{eqnarray*}
	\begin{aligned}
		&\quad\partial_tf(\gamma_s)\\
		&=
		\lim_{h\to 0^+}\frac{1}{h}\left(\left|\left(\gamma_s(s)\right)_N-\left(e^{(s-t+h)A}a_t(t)\right)_N\right|^2-\left|\left(\gamma_s(s)\right)_N-\left(e^{(s-t)A}a_t(t)\right)_N\right|^2\right)\\
		&=\lim_{h\to 0^+}\frac{1}{h}\left(\sum^{N}_{i=1}\left\langle\gamma_s(s)-e^{(s-t+h)A}a_t(t),\,  e_i\right\rangle_H^2-\sum^{N}_{i=1}\left\langle \gamma_s(s)-e^{(s-t)A}a_t(t), \, e_i\right\rangle_H^2\right).
	\end{aligned}
\end{eqnarray*}
Write  $e^{sA^*}:=(e^{sA})^*$. Then,  $e^{sA^*}$ is a $C_0$ {semi-group} on $H$ generated by $A^*$. Since, for all $i=1,2,\ldots, N$,
\begin{eqnarray*}
	\begin{aligned}
		&\quad\left\langle\gamma_s(s)-e^{(s-t+h)A}a_t(t), \, e_i\right\rangle_H^2-\left\langle \gamma_s(s)-e^{(s-t)A}a_t(t),\, e_i\right\rangle_H^2\\
		&=-\left\langle 2\gamma_s(s)-e^{(s-t+h)A}a_t(t)-e^{(s-t)A}a_t(t), \, e_i\right\rangle_H\\
		&\qquad\qquad\times \left\langle e^{(s-t+h)A}a_t(t)-e^{(s-t)A}a_t(t), \, e_i\right\rangle_H\\
		&=-\left\langle 2\gamma_s(s)-e^{(s-t+h)A}a_t(t)-e^{(s-t)A}a_t(t), \, e_i\right\rangle_H\times\left\langle e^{(s-t)A}a_t(t),\,  e^{hA^*}e_i-e_i\right\rangle_H,
	\end{aligned}
\end{eqnarray*}
\begin{eqnarray*}
	\lim_{h\to 0^+}\left\langle 2\gamma_s(s)-e^{(s-t+h)A}a_t(t)-e^{(s-t)A}a_t(t),\,  e_i\right\rangle_H=2\left\langle \gamma_s(s)-e^{(s-t)A}a_t(t),\, e_i\right\rangle_H,
\end{eqnarray*}
and \begin{eqnarray*}
	\lim_{h\to 0^+}\frac{1} {h}\left\langle e^{(s-t)A}a_t(t), \, e^{hA^*}e_i-e_i\right\rangle_H=\left\langle e^{(s-t)A}a_t(t), \, A^*e_i\right\rangle_H,
\end{eqnarray*}
we have
\begin{eqnarray*}
	\partial_tf(\gamma_s)&=& -2\sum^{N}_{i=1}\left\langle \gamma_s(s)-e^{(s-t)A}a_t(t), \, e_i\right\rangle_H\times\left\langle e^{(s-t)A}a_t(t),\,  A^*e_i\right\rangle_H\\
	&=&-2\left\langle e^{(s-t)A}a_t(t), \, A^*\left(\gamma_s(s)-e^{(s-t)A}a_t(t)\right)_N\right\rangle_H.
\end{eqnarray*}
At  $s=T$, we have
\begin{eqnarray*}
	\partial_tf(\gamma_T)=
	\lim_{s\uparrow T}\partial_tf(\gamma_T|_{[0,s]})=-2\left\langle e^{(T-t)A}a_t(t), \, A^*\left(\lim_{s\uparrow T}\gamma_T(s)-e^{(T-t)A}a_t(t)\right)_N\right\rangle_H.
\end{eqnarray*}

It is sufficient to prove $\partial_tf\in C^0_p(\hat{\Lambda}^t)$ and $A^*\partial_xf\in C^0_p(\hat{\Lambda}^t,H)$. Write $\bar{M}:=\max_{1\leq i\leq N}|A^*e_i|$. Then, for every $(s,\gamma_s; l,\eta_l)\in ([t, T]\times{\hat{\Lambda}}^{t})^2,$
\begin{eqnarray*}
	&&\left|A^*\left(\gamma_s(s)-e^{(s-t)A}a_t(t)\right)_N-A^*\left(\eta_l(l)-e^{(l-t)A}a_t(t)\right)_N\right|\\
	&=&\left|\sum^{N}_{i=1}\left\langle \gamma_s(s)-e^{(s-t)A}a_t(t)-\eta_l(l)+e^{(l-t)A}a_t(t),\,  e_i\right\rangle_HA^*e_i \right|\\
	&\leq&\bar{M}\sum^{N}_{i=1}\left|\left\langle \gamma_s(s)-e^{(s-t)A}a_t(t)-\eta_l(l)+e^{(l-t)A}a_t(t), \, e_i\right\rangle_H\right|\\
	&\leq&N\bar{M}\left[|\gamma_s(s)-\eta_l(l)|+\left|e^{(s-t)A}a_t(t)-e^{(l-t)A}a_t(t)\right|\right].
\end{eqnarray*}
Thus, as in  the proof  of  $\Upsilon^{m,M}\in C^0(\hat{\Lambda})$   in Lemma \ref{theoremS}, we have $\partial_tf\in C^0(\hat{\Lambda}^t)$ and $A^*\partial_xf\in C^0(\hat{\Lambda}^t,H)$. By a simple calculation, we see
that $\partial_tf$ and  $A^*\partial_xf$ grow in a polynomial way. The proof is complete.
\end{proof}

\begin{lemma}\label{lemma4.3}\ \
	There exists a local modulus of continuity  $\rho_1$ 
	such that  the functionals $w_1$ and $w_2$ defined by (\ref{06091}) and (\ref{06092}) satisfy condition (\ref{0608a}).
\end{lemma}

\begin{proof}
By the definition of $w_{1}$, we have that,
for every $\hat{t}\leq t\leq s\leq T$ and $\gamma_t\in \Lambda^{\hat{t}}$,
\begin{eqnarray*}
	&&w_1(\gamma_t)-w_1(\gamma_{t,s}^{A,N})\\
	&=&w_1(\gamma_t)-w_1\left(\gamma_{t,s}^A+\left((\gamma_t(t))_N-(e^{(s-t)A}\gamma_t(t))_N\right){\mathbf{1}}_{[0,s]}\right)\\
	 &=&W_1(\gamma_t)-\Pi(\gamma_t)-2\beta^{\frac{1}{3}}\left|\left(\gamma_t(t)\right)^-_N-\left(e^{(t-\hat{t}\,)A}\hat{\xi}_{\hat{t}}(\hat{t}\,)\right)^-_N\right|^2\\
	 &&-W_1(\gamma_{t,s}^{A,N})+\Pi(\gamma_{t,s}^{A,N})+2\beta^{\frac{1}{3}}\left|\left(e^{(s-t)A}\gamma_t(t)\right)^-_N-\left(e^{(s-\hat{t}\,)A}\hat{\xi}_{\hat{t}}(\hat{t}\,)\right)^-_N\right|^2,
\end{eqnarray*}
where
$$
\Pi(\xi_l):=2^5\beta\,\Upsilon({\xi_l},\hat{\xi}_{\hat{t}})+\varepsilon\frac{\nu T-t}{\nu
	T}\, \Upsilon(\xi_l)
+\varepsilon \overline{\Upsilon}(\xi_l,\hat{\gamma}_{\hat{t}})
+\sum_{i=0}^{\infty}
\frac{1}{2^i}\, \overline{\Upsilon}(\gamma^i_{t_i},\xi_l), \ \ \xi_l\in \Lambda^{\hat{t}}
$$
and
$$
\gamma_{t,s}^{A,N}:=\gamma_{t,s}^A+\left(\left(\gamma_t(t)\right)_N-\left(e^{(s-t)A}\gamma_t(t)\right)_N\right){\mathbf{1}}_{[0,s]}.
$$
Since the operator $A$ generates a $C_0$ contraction
{semi-group}  of bounded linear operators, we have
$$||\xi_t||_0=||\xi_{t,s}^A||_0, \quad   |\xi_t(t)|\geq\left |e^{(s-t)A}\xi_t(t)\right|, \quad 0\leq t\leq s\leq T, \ \xi_t\in \Lambda.$$
Then, by (\ref{03130}), we have
\begin{eqnarray}\label{04051}
\Upsilon(\xi_t)\geq \Upsilon(\xi_{t,s}^A), \quad \ 0\leq t\leq s\leq T, \ \xi_t\in \Lambda.
\end{eqnarray}
Thus,
\begin{eqnarray*}
	&&w_1(\gamma_t)-w_1(\gamma_{t,s}^{A,N})\\
	 &\leq&W_1(\gamma_t)-\Pi(\gamma_{t,s}^A)+2\beta^{\frac{1}{3}}\left|\left(\gamma_t(t)\right)_N-\left(e^{(t-\hat{t}\,)A}\hat{\xi}_{\hat{t}}(\hat{t}\,)\right)_N\right|^2\\
	 &&-W_1(\gamma_{t,s}^{A,N})+\Pi(\gamma_{t,s}^{A,N})-2\beta^{\frac{1}{3}}\left|\left(e^{(s-t)A}\gamma_t(t)\right)_N-\left(e^{(s-\hat{t}\,)A}\hat{\xi}_{\hat{t}}(\hat{t}\,)\right)_N\right|^2\\
	&&+\varepsilon\left[(s-\hat{t}\,)^2-(t-\hat{t}\,)^2\right]+\sum_{i=0}^{\infty}
	\frac{1}{2^i}\left[(s-t_i)^2-(t-t_i)^2\right].
\end{eqnarray*}
From (\ref{w1}),
\begin{eqnarray*}
	&&W_1(\gamma_t)-W_1(\gamma_{t,s}^{A,N})= W_1(\gamma_t)-W_1(\gamma_{t,s}^A)+ W_1(\gamma_{t,s}^A)-W_1(\gamma_{t,s}^{A,N})\\
	&\leq& \rho_2(|s-t|,||\gamma_t||_0) +4L\, (1+||\gamma_t||_0)\left |\left(\gamma_t(t)\right)_N-\left(e^{(s-t)A}\gamma_t(t)\right)_N\right|;
\end{eqnarray*}
noting that, for all $\xi^1_l,\xi^2_l\in \Lambda$,
\begin{eqnarray*}
	\left|\Upsilon(\xi^1_l)-\Upsilon(\xi^2_l)\right|&\leq& 6\left(\Upsilon^{\frac{5}{6}}(\xi^1_l)\vee\Upsilon^{\frac{5}{6}}(\xi^2_l)\right)\, \left(\Upsilon^{\frac{1}{6}}(\xi^1_l-\xi^2_l)\right)\\
	&\leq& 54\left(||\xi^1_l||_0^5\vee||\xi^2_l||_0^5\right)||\xi^1_l-\xi^2_l||_0,
\end{eqnarray*}
there exits a constant $C>0$ such that
\begin{eqnarray*}
	|\Pi(\gamma_{t,s}^A)-\Pi(\gamma_{t,s}^{A,N})|&\leq& C(1+\beta+\varepsilon)\left(||\gamma_t||_0^5+||\hat{\xi}_{\hat{t}}||_0^5+||\hat{\gamma}_{\hat{t}}||_0^5+\sum_{i=0}^{\infty}\frac{1}{2^i}||\gamma^i_{t_i}||_0^5\right)\\
	&&\times \left|\left(\gamma_t(t)\right)_N-\left(e^{(s-t)A}\gamma_t(t)\right)_N\right|;
\end{eqnarray*}
since, for $ 0\leq t\leq s\leq T, \ \gamma_t\in \Lambda$,
\begin{eqnarray*}
	&&\left|\left(e^{\left(s-t\right)A}\gamma_t\left(t\right)\right)_N-\left(\gamma_t\left(t\right)\right)_N\right|=\left|\sum^{N}_{i=1}\left\langle e^{\left(s-t\right)A}\gamma_t\left(t\right)-\gamma_t\left(t\right), e_i\right\rangle_He_i\right|  \\
	&=&\left|\sum^{N}_{i=1}\left\langle \gamma_t\left(t\right), e^{\left(s-t\right)A^*}e_i-e_i\right\rangle_He_i\right|\leq \sum^{N}_{i=1}\left|\gamma_t\left(t\right)\right|\left|e^{\left(s-t\right)A^*}e_i-e_i\right|,
\end{eqnarray*}
we have
\begin{eqnarray*}
	\begin{aligned}
		&\quad2\beta^{\frac{1}{3}}\left|(\gamma_t(t))_N-\left(e^{\left(t-\hat{t}\right)A}\hat{\xi}_{\hat{t}}\left(\hat{t}\right)\right)_N\right|^2\\
		 &\qquad-2\beta^{\frac{1}{3}}\left|\left(e^{\left(s-t\right)A}\gamma_t\left(t\right)\right)_N-\left(e^{\left(s-\hat{t}\right)A}\hat{\xi}_{\hat{t}}\left(\hat{t}\right)\right)_N\right|^2\\
		 &\leq4\beta^{\frac{1}{3}}\left(||\gamma_t||_0+|\hat{\xi}_{\hat{t}}(\hat{t}\,)|\right)\left(\left|(\gamma_t(t))_N-\left(e^{\left(s-t\right)A}\gamma_t\left(t\right)\right)_N\right|
		+\left|e^{(s-t)A}\hat{\xi}_{\hat{t}}(\hat{t}\,)-\hat{\xi}_{\hat{t}}(\hat{t}\,)\right|\right)\\
		&\leq 4\,\beta^{\frac{1}{3}}\left(||\gamma_t||_0+|\hat{\xi}_{\hat{t}}(\hat{t}\,)|\right)\left(|\gamma_t(t)|\sum^{N}_{i=1}\left|e^{\left(s-t\right)A^*}e_i-e_i\right|
		+\left|e^{(s-t)A}\hat{\xi}_{\hat{t}}(\hat{t}\,)-\hat{\xi}_{\hat{t}}(\hat{t}\,)\right|\right).
	\end{aligned}
\end{eqnarray*}
Taking for $(l,x)\in[0,\infty)\times [0,\infty)$,
\begin{eqnarray*}
\rho_1(l,x)&:=& \rho_2(l,x)+(2\varepsilon+4)Tl+\biggl [4L(1+x)\\
&&+C(1+\beta+\varepsilon)\Bigl(x^5+||\hat{\xi}_{\hat{t}}||_0^5+||\hat{\gamma}_{\hat{t}}||_0^5+\sum_{i=0}^{\infty}\frac{1}{2^i}||\gamma^i_{t_i}||_0^5\Bigr)\\
&&+
4\beta^{\frac{1}{3}}(x+|\hat{\xi}_{\hat{t}}(\hat{t}\,)|)\biggr]\times  \left(x\sum^{N}_{i=1}\left|e^{lA^*}e_i-e_i\right|
+\left|e^{lA}\hat{\xi}_{\hat{t}}(\hat{t}\,)-\hat{\xi}_{\hat{t}}(\hat{t}\,)\right|\right),
\end{eqnarray*}
we see that $\rho_1$ is a local modulus of continuity 
and $w_1$  satisfies condition (\ref{0608a}) with it.  In a similar way, we show that  $w_2$  satisfies condition (\ref{0608a}) with this $\rho_1$.
The proof is now complete.
\end{proof}

\begin{lemma}\label{lemma4.344} We have $\widetilde{w}_{1}^{\hat{t},N,*}\in \Phi_N(\hat{t},(\hat{{\gamma}}_{{\hat{t}}}({\hat{t}}))_N)$ and $\widetilde{w}_{2}^{\hat{t},N,*}\in \Phi_N(\hat{t},(\hat{{\eta}}_{{\hat{t}}}({\hat{t}}))_N)$. 
\end{lemma}

\begin{proof}  We only prove the first inclusion $\widetilde{w}_{1}^{\hat{t},N,*}\in \Phi_N(\hat{t},(\hat{{\gamma}}_{{\hat{t}}}({\hat{t}}))_N)$. The second inclusion
$\widetilde{w}_{2}^{\hat{t},N,*}\in \Phi_N(\hat{t},(\hat{{\eta}}_{{\hat{t}}}({\hat{t}}))_N)$ is proved in a symmetric  way.

Set $r=\frac{1}{2}(|T-{\hat{t}}|\wedge \hat{t}\,)$. For  given $L>0$, let $\varphi\in C^{1,2}([0,T]\times H_N)$ be a function such that the function
$(\widetilde{w}_{1}^{\hat{t},N,*}-\varphi)(t,x_N)$  attains the  maximum at $(\bar{{t}},\bar{x}_N)\in (0, T)\times H_N$, and moreover, the following are satisfied:
\begin{eqnarray}\label{0815zhou1}
|\bar{{t}}-{\hat{t}}|+\left|\bar{x}_N-\left(\hat{{\gamma}}_{{\hat{t}}}({\hat{t}})\right)_N\right|<r=\frac{1}{2}\left(|T-{\hat{t}}|\wedge \hat{t}\,\right),
\end{eqnarray}
\begin{eqnarray}\label{0815zhou2}
\left|\widetilde{w}_{1}^{\hat{t},N,*}(\bar{{t}},\bar{{x}}_N)\right|+\left|\nabla_x\varphi(\bar{{t}},\bar{{x}}_N)\right|
+\left|\nabla^2_x\varphi(\bar{{t}},\bar{{x}}_N)\right|\leq L.
\end{eqnarray}
By \cite[Lemma 5.4, Chapter 4 ]{yong}, we  modify $\varphi$  such  that $\varphi\in C^{1,2}([0,T]\times H_N)$ bounded from below, ${\varphi}$, ${\varphi_t}$, $\nabla_{x}{\varphi}$ and $\nabla^2_{x}{\varphi}$  grow  in a polynomial way,
$\widetilde{w}_{1}^{\hat{t},*}(t,x_N)-\varphi(t,x_N)$  has a strict   maximum 0 at $(\bar{{t}},\bar{x}_N)\in (0, T)\times H_N$ on $[0, T]\times H_N$ and the above two inequalities hold true.
If $\bar{{t}}<\hat{t}$, we have $\varphi_{t}(\bar{{t}},\bar{x}_N)=\frac{1}{2}(\hat{t}-\bar{{t}})^{-\frac{1}{2}} > 0$. If $\bar{{t}}\geq \hat{t}$,  recall that $w_1$ is defined in (\ref{06091}), we consider the functional
$$
\Gamma(\gamma_t)= w_{1}(\gamma_t)
-\varphi(t,(\gamma_t(t))_N),\ (t,\gamma_t)\in [\hat{t},T]\times {{\Lambda}}.
$$
Note that $w_1$ is a continuous  functional bounded from above and $\varphi$ is a continuous  function bounded from below, 
$\Gamma$ is a continuous  functional bounded from above on $\Lambda^{\hat{t}}$.
%
%
Define a sequence of positive numbers $\{\delta_i\}_{i\geq0}$  by 
$\delta_i=\frac{1}{2^i}$ for all $i\geq0$.  For every  $0<\delta<1$, by Lemma \ref{theoremleft} we have that,
for every  $(\breve{t}_0,\breve{\gamma}^0_{\breve{t}_0})\in [\bar{t},T]\times \Lambda^{\bar{t}}$  {satisfying}
\begin{eqnarray}\label{0615a}
\Gamma(\breve{\gamma}^0_{\breve{t}_0})\geq \sup_{(s,\gamma_s)\in [\bar{t},T]\times \Lambda^{\bar{t}}}\Gamma(\gamma_s)-\delta,
\end{eqnarray}
there exist $(\breve{t},\breve{\gamma}_{\breve{t}})\in [\bar{t},T]\times \Lambda^{\bar{t}}$ and a sequence $\{(\breve{t}_i,\breve{\gamma}^i_{\breve{t}_i})\}_{i\geq1}\subset
[\bar{t},T]\times \Lambda^{\bar{t}}$ such that
\begin{description}
	\item[(i)] $\overline{\Upsilon}(\breve{\gamma}^0_{\breve{t}_0},\breve{\gamma}_{\breve{t}})\leq \delta$,
	$\overline{\Upsilon}(\breve{\gamma}^i_{\breve{t}_i},\breve{\gamma}_{\breve{t}})\leq \frac{\delta}{2^i}$ and $t_i\uparrow \breve{t}$ as $i\rightarrow\infty$,
	\item[(ii)]  $\Gamma(\breve{\gamma}_{\breve{t}})
	-\sum_{i=0}^{\infty}\frac{1}{2^i}\overline{\Upsilon}(\breve{\gamma}^i_{\breve{t}_i},\breve{\gamma}_{\breve{t}})
	\geq \Gamma(\breve{\gamma}^0_{\breve{t}_0})$, and
	\item[(iii)]    for all $(s,\gamma_s)\in [\breve{t},T]\times \Lambda^{\breve{t}}\setminus \{(\breve{t},\breve{\gamma}_{\breve{t}})\}$,
	\begin{eqnarray*}
		\Gamma(\gamma_s)
		-\sum_{i=0}^{\infty}
		\frac{1}{2^i}\overline{\Upsilon}(\breve{\gamma}^i_{\breve{t}_i},\gamma_s)
		<\Gamma(\breve{\gamma}_{\breve{t}})
		-\sum_{i=0}^{\infty}\frac{1}{2^i}\overline{\Upsilon}(\breve{\gamma}^i_{\breve{t}_i},\breve{\gamma}_{\breve{t}}).
	\end{eqnarray*}
\end{description}
{We should note that the point
	$(\breve{t},\breve{\gamma}_{\breve{t}})$ depends on $\delta$.}   By the definitions of $\widetilde{w}^{\hat{t},N}_1$ and $\widetilde{w}_{1}^{\hat{t},N,*}$, we have
\begin{eqnarray}\label{220901a}
&&\widetilde{w}_{1}^{\hat{t},N,*}(\bar{{t}},\bar{x}_N)-\varphi(\bar{{t}},\bar{x}_N)\nonumber\\
&=&\limsup_{\scriptsize \begin{matrix}
	(s,y_N)\rightarrow(\bar{{t}},\bar{x}_N)\\
	s\geq \hat{t}
	\end{matrix}}(\widetilde{w}^{\hat{t},N}_1(s,y_N)-\varphi(s,y_N))\nonumber\\
&=&\limsup_{\scriptsize \begin{matrix}
	(s,y_N)\rightarrow(\bar{{t}},\bar{x}_N)\\
	s\geq \hat{t}
	\end{matrix}}\left(\sup_{\scriptsize \begin{matrix} \xi_s\in \Lambda^{\hat{t}}\\ (\xi_s(s))_N=y_N\end{matrix}}
w_1(\xi_s)-\varphi(s,y_N)\right).
\end{eqnarray}
Note that by Lemma \ref{lemma4.3}, $w_1$   satisfies condition (\ref{0608a}). Then, for every $(s,\xi_s)\in [\hat{t},\bar{t}]\times \Lambda$,
\begin{eqnarray}\label{220901b}
w_1(\xi_s)\leq w_1(\xi_{s,\bar{t}}^{A,N})+\rho_1(|\bar{t}-s|,||\xi_s||_0),
\end{eqnarray}
where $\xi_{s,\bar{t}}^{A,N}:=\xi_{s,\bar{t}}^A+((\xi_s(s))_N-(e^{(\bar{t}-s)A}\xi_s(s))_N){\mathbf{1}}_{[0,\bar{t}]}$.
By the definition of $w_1$, there exists a constant {$M_6>0$} such that
\begin{eqnarray}\label{220901c}
\sup_{\scriptsize \begin{matrix} \xi_s\in \Lambda^{\hat{t}}\\ (\xi_s(s))_N=y_N\end{matrix}}
w_1(\xi_s)=\sup_{\scriptsize \begin{matrix} \xi_s\in \Lambda^{\hat{t}}\\ (\xi_s(s))_N=y_N\\
	||\xi_s||_0\leq {M_6}\end{matrix}}
w_1(\xi_s).
\end{eqnarray}
Thus, by (\ref{220901b}) and (\ref{220901c}),
\begin{eqnarray}\label{220901d}
\begin{aligned}
&\quad\limsup_{\scriptsize \begin{matrix}  (s,y_N)\rightarrow(\bar{{t}},\bar{x}_N)\\
	s\geq \bar{t}\end{matrix}}\sup_{\scriptsize \begin{matrix} \xi_s\in \Lambda^{\hat{t}}\\ (\xi_s(s))_N=y_N\end{matrix}}
[w_1(\xi_s)-\varphi(s,y_N)]\\
&\leq \limsup_{\scriptsize \begin{matrix}
	(s,y_N)\rightarrow(\bar{{t}},\bar{x}_N)\\
	s\geq \hat{t}
	\end{matrix}}
\left(\sup_{\scriptsize \begin{matrix} \xi_s\in \Lambda^{\hat{t}}\\ (\xi_s(s))_N=y_N\end{matrix}
}
w_1(\xi_s)-\varphi(s,y_N)\right)\\
\qquad&\leq \limsup_{\scriptsize \begin{matrix}
	(s,y_N)\rightarrow(\bar{{t}},\bar{x}_N)\\
	s\geq \hat{t}
	\end{matrix}}
\!\!\!\sup_{\scriptsize \begin{matrix} \xi_s\in \Lambda^{\hat{t}}\\ (\xi_s(s))_N=y_N\\
	 ||\xi_s||_0\leq {M_6}\end{matrix}}\!\!\!\!\!\!
{[}w_1(\xi_{s,\bar{t}}^{A,N})+\rho_1(|s\vee\bar{t}-s|,{M_6})-\varphi(s,y_N){]}\\
&\leq\limsup_{\scriptsize \begin{matrix}  (s,y_N)\rightarrow(\bar{{t}},\bar{x}_N)\\
s\geq \bar{t}\end{matrix}}\sup_{\scriptsize \begin{matrix}\xi_s\in \Lambda^{\hat{t}}\\
	(\xi_s(s))_N=y_N\end{matrix}}
[w_1(\xi_s)-\varphi(s,y_N)].
\end{aligned}
\end{eqnarray}
Therefore, by (\ref{220901a}) and (\ref{220901d}),
\begin{eqnarray*}
	\widetilde{w}_{1}^{\hat{t},N,*}(\bar{{t}},\bar{x}_N)-\varphi(\bar{{t}},\bar{x}_N)
	&=&\limsup_{\scriptsize \begin{matrix}
			(s,y_N)\rightarrow(\bar{{t}},\bar{x}_N)\\
			s\geq \bar{t}
	\end{matrix}}\sup_{\scriptsize \begin{matrix} \xi_s\in \Lambda^{\hat{t}}\\ (\xi_s(s))_N=y_N\end{matrix}}
	[w_1(\xi_s)-\varphi(s,y_N)]\\
	&\leq&\sup_{(s,\gamma_s)\in [\bar{t},T]\times \Lambda^{\bar{t}}}\Gamma(\gamma_s).
\end{eqnarray*}
Combining with (\ref{0615a}),
\begin{eqnarray}\label{220901e}
\Gamma(\breve{\gamma}^0_{\breve{t}_0})\geq \sup_{(s,\gamma_s)\in [\bar{t},T]\times \Lambda^{\bar{t}}}\Gamma(\gamma_s)-\delta
\geq\widetilde{w}_{1}^{\hat{t},N,*}(\bar{{t}},\bar{x}_N)-\varphi(\bar{{t}},\bar{x}_N)-\delta.
\end{eqnarray}
Recall that $\widetilde{w}_{1}^{\hat{t},N,*}\geq\widetilde{w}^{\hat{t},N}_1$.  Then, by 
the definition of $\widetilde{w}^{\hat{t},N}_1$, the property (ii) of $(\breve{t},\breve{\gamma}_{\breve{t}})$ and (\ref{220901e}),
\begin{eqnarray}\label{20210509}
\begin{aligned}&\qquad
\widetilde{w}_{1}^{\hat{t},N,*}(\breve{t},(\breve{\gamma}_{\breve{t}}(\breve{t}))_N)
-\varphi(\breve{t},(\breve{\gamma}_{\breve{t}}(\breve{t}))_N)\\
&\geq\widetilde{w}^{\hat{t},N}_1(\breve{t},(\breve{\gamma}_{\breve{t}}(\breve{t}))_N)
-\varphi(\breve{t},(\breve{\gamma}_{\breve{t}}(\breve{t}))_N)
\geq w_1(\breve{\gamma}_{\breve{t}})
-\varphi(\breve{t},(\breve{\gamma}_{\breve{t}}(\breve{t}))_N)\\
&\geq \Gamma(\breve{\gamma}^0_{\breve{t}_0})
\geq\widetilde{w}_{1}^{\hat{t},N,*}(\bar{{t}},\bar{x}_N)-\varphi(\bar{{t}},\bar{x}_N)-\delta=-\delta.
\end{aligned}
\end{eqnarray}
Noting that $\nu$ is independent of  $\delta$ and $\varphi$ is a continuous  function bounded from below, 
by the definitions of  $\Gamma$ and $w_1$, (\ref{w}) and (\ref{s0}),
there exists a constant  $M_7>0$   depending only on $\varphi$ 
that is sufficiently  large   that
$
\Gamma(\gamma_t)<\sup_{(s,\gamma_s)\in [\bar{t},T]\times \Lambda^{\bar{t}}}\Gamma(\gamma_s)-1
$ for all $t\in [\bar{t},T]$ and $||\gamma_t||_0\geq M_7$. Thus, we have $||\breve{\gamma}_{\breve{t}}||_0\vee
||{\breve{\gamma}}^{0}_{\breve{t}_{0}}||_0<M_7$. In particular, $|(\breve{\gamma}_{\breve{t}}(\breve{t}))_N|<M_7$.
Letting $\delta\rightarrow0$, by  the similar proof procedure of (\ref{4.226}) and (\ref{05231}), we obtain
\begin{eqnarray}\label{delta0}
&& \breve{t}\rightarrow \bar{t},\ (\breve{\gamma}_{\breve{t}}(\breve{t}))_N\rightarrow \bar{x}_N,\  \widetilde{w}_{1}^{\hat{t},*}(\breve{t},(\breve{\gamma}_{\breve{t}}(\breve{t}))_N)\rightarrow\widetilde{w}_{1}^{\hat{t},*}(\bar{{t}},\bar{x}_N) \ \mbox{as}\ \delta\rightarrow0.
\end{eqnarray}
Noting that $\widetilde{w}_{1}^{\hat{t},N,*}(t,x_N)-\varphi(t,x_N)$  has a strict   maximum 0 at $(\bar{{t}},\bar{x}_N)\in (0, T)\times H_N$ on $[0, T]\times H_N$, by (\ref{0815zhou2})
and (\ref{delta0}) there exists a constant $0<\Delta<1$ such that
for all $0<\delta<\Delta$,
$$
\varphi(\breve{t},(\breve{\gamma}_{\breve{t}}(\breve{t}))_N)\geq \widetilde{w}_{1}^{\hat{t},N,*}(\breve{t},(\breve{\gamma}_{\breve{t}}(\breve{t}))_N)\geq \widetilde{w}_{1}^{\hat{t},N,*}(\bar{{t}},\bar{x}_N)-1\geq -(L+1).
$$
Then, by the definitions of  $\Gamma$ and $w_1$,  (\ref{s0}), (\ref{w}) and (\ref{20210509}),
there exists a constant  $M_8>0$   depending only on $L$ 
such that, for all $0<\delta<\Delta$ and $||{\gamma}_{\breve{t}}||_0\geq M_8$ satisfying $({\gamma}_{\breve{t}}(\breve{t}))_N=(\breve{\gamma}_{\breve{t}}(\breve{t}))_N$,
$$
\Gamma(\gamma_{\breve{t}})<-2\leq \Gamma(\breve{\gamma}^0_{\breve{t}_0})-1 \leq\sup_{(s,\gamma_s)\in [\bar{t},T]\times \Lambda^{\bar{t}}}\Gamma(\gamma_s)-1.
$$
Thus, we have $||\breve{\gamma}_{\breve{t}}||_0<M_8$ for all $0<\delta<\Delta$.   From (\ref{0815zhou1}), it follows that $\bar{t}<{\hat{t}}+\frac{|T-{\hat{t}}|}{2}\leq T$. 
Then, by (\ref{delta0}),  we have  $\breve{t}<T$ 
provided that  $\delta>0$ is small enough.
Thus,  we have from the definition of the viscosity sub-solution
\begin{eqnarray}\label{5.15}
\begin{aligned}
&
{ \partial_t^o}\chi(\breve{\gamma}_{{\breve{t}}})
+\left\langle A^*\nabla_{x}\varphi({\breve{t}}, (\breve{\gamma}_{{\breve{t}}}(\breve{{t}}))_N),\, \breve{\gamma}_{{\breve{t}}}({\breve{t}})\right\rangle_H\\
&-4\beta^{\frac{1}{3}}\left\langle A^*\left(\left(\breve{\gamma}_{\breve{t}}(\breve{t})\right)_N-\left(e^{(\breve{t}-\hat{t}\,)A}\hat{\xi}_{\hat{t}}(\hat{t}\,)\right)_N\right),\,
\breve{\gamma}_{{\breve{t}}}(\breve{{t}})\right\rangle_H\\
&
+{\mathbf{H}}\left(\breve{\gamma}_{{\breve{t}}}, \, W_1(\breve{\gamma}_{{\breve{t}}}),\, \partial_x\chi(\breve{\gamma}_{{\breve{t}}}),\, \partial_{xx}\chi(\breve{\gamma}_{{\breve{t}}})\right)\geq c,
\end{aligned}
\end{eqnarray}
where, for every $(t,\gamma_t)\in  {[\breve{t},T]\times{{\Lambda}}}$,  define $\chi(\gamma_t) :=\chi_{1}(\gamma_t)+\chi_{2}(\gamma_t)$ with
\begin{eqnarray*}
	\chi_{1}(\gamma_t)  &:=&\varphi(t,(\gamma_t(t))_N)-2\beta^{\frac{1}{3}}\left|\left(\gamma_t(t)\right)_N-\left(e^{(t-\hat{t}\,)A}\hat{\xi}_{\hat{t}}(\hat{t}\,)\right)_N\right|^2\quad \hbox{ \rm and }\\
	\chi_{2}(\gamma_t)  &:=&
	\varepsilon\frac{\nu T-t}{\nu
		T}\, \Upsilon(\gamma_t)
	+\varepsilon \, \overline{\Upsilon}(\gamma_t,\hat{\gamma}_{\hat{t}})
	+\sum_{i=0}^{\infty}
	\frac{1}{2^i}\, \overline{\Upsilon}(\gamma^i_{t_i},\gamma_t)+2^5\beta\, \Upsilon(\gamma_{t},\hat{{\xi}}_{{\hat{t}}})\\
	&&
	+\sum_{i=0}^{\infty}
	\frac{1}{2^i}\, \overline{\Upsilon}(\breve{\gamma}^{i}_{\breve{t}_{i}},\gamma_t)
	+2\beta^{\frac{1}{3}}\left|\gamma_t(t)-e^{(t-\hat{t}\,)A}\hat{\xi}_{\hat{t}}(\hat{t}\,)\right|^2;
\end{eqnarray*}
and therefore,
\begin{eqnarray*}
	{ \partial_t^o}\chi(\gamma_t)&:=&\partial_t\chi_{1}(\gamma_t)+{\partial_t^o}\chi_{2}(\gamma_t)\\
	&=&
	\varphi_t({t},(\gamma_{{t}}({t}))_N)+4\beta^{\frac{1}{3}}\left\langle A^*\left(\left({\gamma}_{{t}}({t})\right)_N-\left(e^{({t}-\hat{t}\,)A}\hat{\xi}_{\hat{t}}(\hat{t}\,)\right)_N\right),\,
	e^{({t}-\hat{t}\,)A}\hat{\xi}_{\hat{t}}(\hat{t}\,)\right\rangle_H\\
	&&-\frac{\varepsilon}{\nu T}\Upsilon(\gamma_{{t}})
	+2\varepsilon({t}-{\hat{t}})+2\sum_{i=0}^{\infty}\frac{1}{2^i}[(t-t_{i})+(t-\breve{t}_i)],
\end{eqnarray*}
\begin{eqnarray*}
	\partial_x\chi(\gamma_t)&:=&\partial_x\chi_{1}(\gamma_t)+\partial_x\chi_{2}(\gamma_t)\\
	&=&\nabla_{x}\varphi({t}, (\gamma_{{t}}(t))_N)
	+4\beta^{\frac{1}{3}}\left(\left(\gamma_{t}(t)\right)^-_N-\left(e^{(t-\hat{t}\,)A}\hat{\xi}_{\hat{t}}(\hat{t}\,)\right)^-_N\right)
	\\
	&&+\varepsilon\, \frac{\nu T-{t}}{\nu T}\, \partial_x\Upsilon(\gamma_{{t}}) +\varepsilon\, \partial_x\Upsilon(\gamma_{{t}}-\hat{\gamma}_{{\hat{t},t}}^A)
	+2^5\beta\, \partial_x\Upsilon(\gamma_{{t}}-\hat{\xi}_{{\hat{t},t}}^A)
	\\
	&&+\partial_x\left[\sum_{i=0}^{\infty}\frac{1}{2^i}
	\, \Upsilon(\gamma_{t}-(\gamma^{i})_{t_{i},t}^A)
	+\sum_{i=0}^{\infty}\frac{1}{2^i}\, \Upsilon(\gamma_{t}-(\breve{\gamma}^i)_{\breve{t}_i,t}^A)\right], \quad \hbox{\rm and }
\end{eqnarray*}
\begin{eqnarray*}
	&&\partial_{xx}\chi(\gamma_t)\\
	&:=&\partial_{xx}\chi_{1}(\gamma_t)+\partial_{xx}\chi_{2}(\gamma_t)\\
	&=&\nabla^2_{x}\varphi({t}, (\gamma_{{t}}(t))_N)+4\beta^{\frac{1}{3}}Q_N+\varepsilon\frac{\nu T-{t}}{\nu T}\, \partial_{xx}\Upsilon(\gamma_{{t}}) +\varepsilon\, \partial_{xx}\Upsilon(\gamma_{{t}}-\hat{\gamma}_{{\hat{t},t}}^A)
	\\
	&&+2^5\beta\, \partial_{xx}\Upsilon(\gamma_{{t}}-\hat{\xi}_{{\hat{t},t}}^A)
	+\partial_{xx}\left[\sum_{i=0}^{\infty}\frac{1}{2^i}
	\Upsilon(\gamma_{t}-(\gamma^{i})_{t_{i},t}^A)
	+\sum_{i=0}^{\infty}\frac{1}{2^i}\Upsilon(\gamma_{t}-(\breve{\gamma}^i)_{\breve{t}_i,t}^A)\right].
\end{eqnarray*}
We notice that  $||\breve{\gamma}_{{\breve{t}}}||_0< M_8$ for all $0<\delta<\Delta$ and $M_8$ only depends on $L$.
Then    by 
  (\ref{5.15}) and the definition of ${\mathbf{H}}$, it follows that there exists a constant $C'_0\geq0$ depending only on $L$ such that
$$\varphi_t\left(\breve{t},(\breve{\gamma}_{\breve{t}}(\breve{t}))_N\right)\geq -C_0'\left(1+\left|\nabla_{x}\varphi(\breve{t},(\breve{\gamma}_{\breve{t}}(\breve{t}))_N)\right|+\left|\nabla^2_{x}\varphi(\breve{t},(\breve{\gamma}_{\breve{t}}(\breve{t}))_N)\right|\right).$$
Letting $\delta\rightarrow0$, by (\ref{0815zhou2}) and (\ref{delta0}) , we obtain 
$$
\varphi_{t}(\bar{t},\bar{x}) \geq -{C}_0'(1+|\nabla_{x}\varphi(\bar{t},\bar{x})|+|\nabla^2_{x}\varphi(\bar{t},\bar{x})|)\geq -{C}_0'(1+2L).
$$
Letting $ \tilde{C}_0={C}_0'(1+2L)$, we get $\varphi_{t}(\bar{t},x_1) \geq -{\tilde{C}_0}$.
The proof is now complete.
\end{proof}

\begin{lemma}\label{0611a}\ \
	The functionals $w_1$ and $w_2$ defined by (\ref{06091}) and (\ref{06092}) satisfy the conditions of  Theorem \ref{theorem0513} with  $\varphi$, where $\varphi$ is the function defined by (\ref{07063}).
\end{lemma}

\begin{proof}
From (\ref{s0}) and (\ref{w}),  the functionals $w_1$ and $w_2$ are continuous and bounded from above and satisfy (\ref{05131}).
By  Lemmas \ref{lemma4.3} and \ref{lemma4.344}, $w_1$ and $w_2$ satisfy condition  (\ref{0608a}), and  $\widetilde{w}_{1}^{\hat{t},N,*}\in \Phi_N(\hat{t},(\hat{{\gamma}}_{{\hat{t}}}({\hat{t}}))_N)$ and $\widetilde{w}_{2}^{\hat{t},N,*}\in \Phi_N(\hat{t},(\hat{{\eta}}_{{\hat{t}}}({\hat{t}}))_N)$.
Moreover, let $\varphi$ be the function defined by (\ref{07063}). By Lemma \ref{theoremS00044} and (\ref{iii4}) we obtain that, for all 
$(t,(\gamma_t,\eta_t))\in [\hat{t},T]\times (\Lambda^{\hat{t}}\otimes  \Lambda^{\hat{t}})$,
\begin{eqnarray}\label{wv}
\ \ \ \ \ \ \  \ \
\begin{aligned}
&w_{1}(\gamma_t)+w_{2}(\eta_t)-\varphi((\gamma_t(t))_N,(\eta_t(t))_N)+\varepsilon (\overline{\Upsilon}(\gamma_t,\hat{\gamma}_{\hat{t}})+ \overline{\Upsilon}(\eta_t,\hat{\eta}_{\hat{t}}))  \\
=&w_{1}(\gamma_t)+w_{2}(\eta_t)-\beta^{\frac{1}{3}}|(\gamma_t(t))_N-(\eta_t(t))_N|^2+\varepsilon (\overline{\Upsilon}(\gamma_t,\hat{\gamma}_{\hat{t}})+ \overline{\Upsilon}(\eta_t,\hat{\eta}_{\hat{t}}))  \\
\leq&
{\Psi}_1(\gamma_t,\eta_t)\leq\Psi_1(\hat{\gamma}_{\hat{t}},\hat{\eta}_{\hat{t}})=w_{1}(\hat{{\gamma}}_{{\hat{t}}})+w_{2}(\hat{{\eta}}_{{\hat{t}}})
-\beta^{\frac{1}{3}}|(\hat{{\gamma}}_{{\hat{t}}}({\hat{t}}))_N-(\hat{{\eta}}_{{\hat{t}}}({\hat{t}}))_N|^2\\
=&w_{1}(\hat{{\gamma}}_{{\hat{t}}})+w_{2}(\hat{{\eta}}_{{\hat{t}}})
-\varphi((\hat{{\gamma}}_{{\hat{t}}}({\hat{t}}))_N,(\hat{{\eta}}_{{\hat{t}}}({\hat{t}}))_N),
\end{aligned}
\end{eqnarray}
where the last inequality becomes equality if and only if $t={\hat{t}}$, $\gamma_t=\hat{{\gamma}}_{{\hat{t}}}, \eta_t=\hat{{\eta}}_{{\hat{t}}}$.
Then we obtain that $
w_1(\gamma_t)+w_2(\eta_t)-\varphi((\gamma_t(t))_N, (\eta_t(t)))
$
has a 
maximum over $\Lambda^{\hat{t}}\otimes \Lambda^{\hat{t}}$ at a point $(\hat{\gamma}_{\hat{t}},\hat{\eta}_{\hat{t}})$ with $\hat{t}\in (0,T)$ and satisfies (\ref{10101}) with  ${\tilde{\rho}}(l)=\varepsilon l, l\in [0,\infty)$. Thus $w_1$ and $w_2$ satisfy the conditions of  Theorem \ref{theorem0513}  with  $\varphi$ defined by (\ref{07063}).
\end{proof}

\begin{example}\label{202401281}\ \
	Parabolic   equation
\end{example}
Let $\mathcal{O} \subset \mathbb{R}^d$ be a bounded  domain with a smooth boundary $\partial\mathcal{O}$. Let
$$
                  A:=\sum^{n}_{i,j}\partial_i(a_{ij}\partial_j),\ \ D(A):=H^1_0(\mathcal{O})\cap H^2(\mathcal{O}),
$$
where $a_{ij}=a_{ji}\in L^{\infty}(\mathcal{O})$ for $i,j\in \{1,\ldots,d\}$, and there exists a $\theta>0$ such that
$$
                    \sum^{n}_{i,j}a_{ij}\xi_i\xi_j\geq \theta |\xi|^2, \forall x\in \mathbb{R}^d
$$
a.e. in $\mathcal{O}$. Then $A$ generates a contraction semigroup $e^{tA}$ in $H:=L^2(\mathcal{O})$ (see ~\cite[Example 3.16, page 178]{fab1}). We consider a control problem for the stochastic parabolic equation
\begin{eqnarray}\label{202401291}
\begin{cases}
\frac{\partial y}{\partial s}(s,\xi)=Ay(s,\xi)+f(\xi,\mu(y_s(\xi)),u(s))\\
\  \  \ \ \ \ \ \ \  \ \ \   +h(\xi,\mu(y_s(\xi)),u(s))\frac{\partial}{\partial s}W_{Q}(s,\xi),  \quad s\in (t,T], \ \xi\in\mathcal{O},\\
 y(s,\xi)=0,\ \ s\in (t,T]\times \partial \mathcal{O},\\
y(s,\xi)=y^0_t(s,\xi), \ s\in [0,t],  \xi\in\mathcal{O},
\end{cases}
\end{eqnarray}
where $Q$ is an operator in $\mathcal{L}^+_1(L^2(\mathcal{O}))$ (see ~\cite[Proposition B.30, page 801]{fab1}) and $W_Q$ is a $Q$-Wiener process
(see ~ \cite[Definition 1.88, page 28]{fab1}), $y^0_t\in C([0,t]; L^2(\mathcal{O}))$, $U$ is a Polish space,
$u(\cdot)\in \mathcal{U}[t,T]$, $\mu$ is a bounded signed measure on $[0,T]$ and $\mu(y_s(\xi)):=\int_{[0,s]}y(l,\xi)d\mu(l)$. Suppose we want to maximize the cost functional
\begin{eqnarray}\label{202401291d}
               ~~~~~  I(y^0_t, u(\cdot))=\mathbb{E}\left[\int^{T}_{t}\int_{\mathcal{O}}\alpha(\mu(y_s(\xi)),u(s))d\xi ds+\int_{\mathcal{O}}\beta(\mu(y_T(\xi)))d\xi\right]
\end{eqnarray}
over all $u(\cdot)\in \mathcal{U}[t,T]$. 
  We assume the following.
\begin{assumption}\label{hypstate20240128}
	\begin{description}
		\item[(i)]
		 $(f, h) :\mathcal{O}\times \mathbb{R}\times U\to  \mathbb{R}\times  \mathbb{R}$ is continuous, {and
        $(f,h)(\cdot,\cdot,u)$ is continuous in $  \mathcal{O}\times \mathbb{R}$, uniformly in $u\in U$.} Moreover,
		there is a constant $L>0$ such that,  for all $(\xi, x, u) \in  \mathcal{O}\times \mathbb{R}\times U$,
		\begin{eqnarray*}
		&&\displaystyle \left|f(\xi,x,u)\right|^2\vee\left|h(\xi,x,u)\right|^2\leq
		L^2(1+|x|^2),\\
		&&\displaystyle \left|f(\xi,x,u)-f(\xi,y,u)\right|\vee\left|h(\xi,x,u)-h(\xi,y,u)\right|\leq
		L|x-y|.
		\end{eqnarray*}
\item[(ii)]
	$
	\alpha:  \mathbb{R}\times U\rightarrow \mathbb{R}$ and $\beta:  \mathbb{R}\rightarrow \mathbb{R}$ are continuous. 
         Moreover,     there is a  constant $L>0$
	such that, for all $(x,    y,  u)
	\in \mathbb{R}\times\mathbb{R}\times U$,
	\begin{eqnarray*}
		&&| \alpha(x,u)|\leq L(1+|x|),
		\ \ \ \ |\alpha(x,u)-\alpha(y,u)|\leq L|x-y|,
		\\
		&&|\beta(x)-\beta(y)|\leq L|x-y|.
	\end{eqnarray*}
\item[(iii)] For every $\gamma_s\in \Lambda_s(H)$,
	\begin{eqnarray*}
	\lim_{N\to \infty} \sup_{u\in U}\left|Q_N{h}(\cdot,\mu(\gamma_s(\cdot)),u)\right|_{H}^2=0.
	\end{eqnarray*}
	\end{description}
\end{assumption}
Define, for $\gamma_s\in \Lambda_s(H)$, $\gamma_T\in \Lambda_T(H)$,
 $z\in H$ and $u\in U$,
\begin{eqnarray*}
                         && F(\gamma_s,u)=f(\cdot,\mu(\gamma_s(\cdot)),u),\ \ G(\gamma_s,u)z=h(\cdot,\mu(\gamma_s(\cdot)),u)z(\cdot),\\
                          && q(\gamma_s,u)=\int_{\mathcal{O}}\alpha(\mu(\gamma_s(\xi)),u)d\xi,\ \  \phi(\gamma_T)=\int_{\mathcal{O}}\beta(\mu(\gamma_T(\xi)))d\xi.
\end{eqnarray*}
Equation (\ref{202401291}) can be rewritten as the following evolution equation in $H$
$$
    dX(s)=AX(s)ds+F(X_s,u(s))ds+G(X_s,u(s))dW_Q(s),\ \ X_t=y^0_t,
$$
and the  cost functional $I$ can be rewritten as
$$
                    J(y^0_t,u(\cdot))=\mathbb{E}\left[\int^T_tq(X_s,u(s))ds+\phi(X_T)\right].
$$
It is easy
to check that Assumption \ref{hypstate20240128} implies that Assumptions \ref{hypstate}, \ref{hypcost} and \ref{hypstate5666} are satisfied. For example,
let $\Xi=Q^{\frac{1}{2}}H$, by  \cite[Theorem 1.87, page 28]{fab1},  $W_Q$ is a   cylindrical Wiener process in Hilbert space $\Xi$, by {Assumption}  \ref{hypstate20240128} (i),
\begin{eqnarray*}
                    &&|G(\gamma_s,u)|^2_{L_2(\Xi,H)}=|h(\cdot,\mu(\gamma_s(\cdot)),u)Q^{\frac{1}{2}}|^2_{L_2(H)}\\
                    &\leq& |h(\cdot,\mu(\gamma_s(\cdot)),u)|^2_{L(H)}|Q^{\frac{1}{2}}|^2_{L_2(H)}
                    =|Q^{\frac{1}{2}}|^2_{L_2(H)}\sup_{|z|_H\leq 1}\left(\int_{\mathcal{O}}h(\xi,\mu(\gamma_s(\xi)),u)z(\xi)d\xi\right)^2\\
                    &\leq&|Q^{\frac{1}{2}}|^2_{L_2(H)}\int_{\mathcal{O}}h^2(\xi,\mu(\gamma_s(\xi)),u)d\xi
                    \leq L^2|Q^{\frac{1}{2}}|^2_{L_2(H)}\int_{\mathcal{O}}(1+|\mu(\gamma_s(\xi))|^2)d\xi\\
                    &\leq&L^2|Q^{\frac{1}{2}}|^2_{L_2(H)}\left(m(\mathcal{O})+|\mu|([0,T])\int_{\mathcal{O}}\int_{[0,s]}|\gamma_s(l,\xi)|^2d|\mu|(l)d\xi\right)\\
                    &\leq&L^2|Q^{\frac{1}{2}}|^2_{L_2(H)}\left(m(\mathcal{O})+|\mu|^2([0,T])||\gamma_s||^2_0\right);
\end{eqnarray*}
and by {Assumption}  \ref{hypstate20240128} (iii),
	\begin{eqnarray*}
	&&\lim_{N\to \infty} \sup_{u\in U}\left|Q_NG(\gamma_s,u)\right|_{L_2(\Xi,H)}^2=\lim_{N\to \infty} \sup_{u\in U}|Q_Nh(\cdot,\mu(\gamma_s(\cdot)),u)Q^{\frac{1}{2}}|^2_{L_2(H)}\\
&\leq&\lim_{N\to \infty} \sup_{u\in U}|Q_Nh(\cdot,\mu(\gamma_s(\cdot)),u)|^2_{H}|Q^{\frac{1}{2}}|^2_{L_2(H)}=0.
\end{eqnarray*}
\begin{example}\label{20240129a}\ \
	Hyperbolic  equation
\end{example}
Consider  a control problem for
the stochastic Hyperbolic  equation in  a bounded  domain $\mathcal{O}$ with a smooth boundary $\partial\mathcal{O}$
\begin{eqnarray}\label{20240129a1}
\begin{cases}
\frac{\partial^2 y}{\partial s^2}(s,\xi)=Ay(s,\xi)+f(\xi,\mu(y_s(\xi)),u(s))\\
\  \  \ \ \ \ \ \ \  \ \ \   +h(\xi,\mu(y_s(\xi)),u(s))\frac{\partial}{\partial s}{W}_{{Q}}(s,\xi),  \quad s\in (t,T], \ \xi\in\mathcal{O},\\
 y(s,\xi)=0,\ \ s\in (t,T]\times \partial \mathcal{O},\\
y(s,\xi)=y^0_t(s,\xi), \ s\in [0,t], \xi\in\mathcal{O},\\
\frac{\partial y}{\partial s} (s,\xi)=z^0_t(s,\xi),\ s\in [0,t], \xi\in\mathcal{O},
\end{cases}
\end{eqnarray}
where $A$ is given in Example \ref{202401281}, $y^0_t\in C([0,t]; H^1_0(\mathcal{O}))$ and  $z^0_t\in C([0,t]; L^2(\mathcal{O}))$.  Suppose we also want to maximize the cost functional (\ref{202401291d}). Set
$$
\mathbb{H}:=\left(\begin{matrix}
	H^1_0(\mathcal{O})\\
\times\\
	H
	\end{matrix}\right)
$$ equipped with the inner product
$$
\left\langle\left(\begin{matrix}
	y\\
	z
	\end{matrix}\right), \left(\begin{matrix}
	y'\\
	z'
	\end{matrix}\right)\right\rangle_\mathbb{H}=\langle(-A)^{\frac{1}{2}}y,(-A)^{\frac{1}{2}}y'\rangle_H+\langle z,z'\rangle_H,\ \ (x,y), (x',y')\in \mathbb{H}.
$$
The operator
 $$
         \mathcal{A}=\left(\begin{matrix}
	0&I\\
	A&0
	\end{matrix}\right),\ \        D(\mathcal{A})=\left(\begin{matrix}
	D(A)\\
	\times \\
H^1_0(\mathcal{O})
	\end{matrix}\right)
 $$
generates a contraction semigroup $e^{t\mathcal{A}}$ in $\mathbb{H}$ (see ~\cite[Example 3.13, page 176]{fab1}).
Define, for $\left(\begin{matrix}
	\gamma_s\\
	\eta_s
	\end{matrix}\right)\in \Lambda_s(\mathbb{H})$, $\left(\begin{matrix}
	\gamma_T\\
	\eta_T
	\end{matrix}\right)\in \Lambda_T(\mathbb{H})$,
 $\left(\begin{matrix}
	x\\
	y
	\end{matrix}\right)\in \mathbb{H}$ and $u\in U$,
\begin{eqnarray*}
                         && F\left(\left(\begin{matrix}
	\gamma_s\\
	\eta_s
	\end{matrix}\right),u\right)=\left(\begin{matrix}
	0\\
	f(\cdot,\mu(\gamma_s(\cdot)),u)
	\end{matrix}\right),\\
&& G\left(\left(\begin{matrix}
	\gamma_s\\
	\eta_s
	\end{matrix}\right),u\right)\left(\begin{matrix}
	y\\
	z
	\end{matrix}\right)=\left(\begin{matrix}
	0\\
	h(\cdot,\mu(\gamma_s(\cdot)),u)z
	\end{matrix}\right),\\
&&\widetilde{W}_{\tilde{Q}}=\left(\begin{matrix}
	0\\
	W_Q
	\end{matrix}\right),
\ \ \tilde{Q}\left(\begin{matrix}
	y\\
	z
	\end{matrix}\right)=\left(\begin{matrix}
	0\\
	Qz
	\end{matrix}\right)\\
                          && q\left(\left(\begin{matrix}
	\gamma_s\\
	\eta_s
	\end{matrix}\right),u\right)=\int_{\mathcal{O}}\alpha(\mu(\gamma_s(\xi)),u)d\xi,\ \  \phi\left(\left(\begin{matrix}
	\gamma_T\\
	\eta_T
	\end{matrix}\right)\right)=\int_{\mathcal{O}}\beta(\mu(\gamma_T(\xi)))d\xi.
\end{eqnarray*}
Equation (\ref{20240129a1}) can be rewritten in an abstract way in $\mathbb{H}$ as
$$
    dX(s)=\mathcal{A}X(s)ds+F(X_s,u(s))ds+G(X_s,u(s))d\widetilde{W}_Q(s),\ \ X_t=x_t:=\left(\begin{matrix}
	y^0_t\\
	z^0_t
	\end{matrix}\right),
$$
and the  cost functional (\ref{202401291d}) can be rewritten as
$$
                    J(x_t,u(\cdot))=\mathbb{E}\left[\int^T_tq(X_s,u(s))ds+\phi(X_T)\right].
$$
It is easy to see that Assumption \ref{hypstate20240128} implies that Assumptions \ref{hypstate}, \ref{hypcost} and \ref{hypstate5666} are satisfied.
\begin{remark}\label{202401301}\ \
	In the above two examples, we  assume 
 $A$ generates a contraction semigroup $e^{tA}$ in $L^2(\mathcal{O})$ and $Q$ is an operator in $\mathcal{L}^+_1(L^2(\mathcal{O}))$ without any additional assumptions, while even in the Markovian case, for example ~\cite[Example 3.69, page 238]{fab1}, the additional assumptions $A=\Delta$, $d>1$ and $(-\Delta)^{\frac{d-1}{4}}Q^{\frac{1}{2}}\in L_2(L^2(\mathcal{O}))$ are  imposed  to ensure  the  $B$-continuity of the coefficients in~\cite[(3.22), page 184]{fab1}.
\end{remark}

\newpage
\section{The Markovian case}
\par
In this chapter, we specialize our previous results on  the PHJB equation {(\ref{hjb1})} at the Markovian case,
 i.e,  in the setting of Remark \ref{remarkv}. The  HJB equation then reads
	\begin{eqnarray}\label{hjb32024}
	\qquad
	\begin{cases}
	\overline{V}_{t^+}(t,x)+\langle A^*\nabla_x\overline{V}(t,x), x\rangle_H\\
	\qquad+\overline{{\mathbf{H}}}(t, x, \overline{V}(t,x), \nabla_x\overline{V}(t,x), \nabla^2_x\overline{V}(t,x))= 0, \  (t, x)\in
	[0,T)\times H;\\
	\overline V(T,x)=\overline{\phi}(x), \ \ \ x\in H,
	\end{cases}
	\end{eqnarray}
	where
	\begin{eqnarray*}
		&& \overline{{\mathbf{H}}}(t,x,r,p,l)\\&=&\sup_{u\in{
				{U}}}[\langle p, \overline{F}(t,x,u)\rangle_{H} +\frac{1}{2}\mbox{Tr}[ l (\overline{G}\, \overline{G}^*)(t,x,u)]+\overline{q}(t,x,r,p\overline{G}(t,x,u),u)], \\
		&&\quad \quad  (t,x,r,p,l)\in [0,T]\times H\times \mathbb{R}\times H\times {\mathcal{S}}(H).
	\end{eqnarray*}
 To see clearly the advantages of our results, we prefer to repeat some relevant assumptions, definitions, and proofs.  We make the following  Assumption.

\begin{assumption}\label{hypstate20241}
	\begin{description}
		\item[(i)]
		The operator $A$ generates  a $C_0$ contraction
		{semi-group}  of bounded linear operators  $\{e^{tA}, t\geq0\}$ in
		Hilbert space $H$.
		\par
		\item[(ii)] $(\overline F, \overline G) :[0,T]\times H\times U\to  H\times  L_2(\Xi,H)$ is continuous, {and
        $(\overline F, \overline G)(\cdot,\cdot,u)$ is continuous in $ [0,T]\times H$, uniformly in $u\in U$.} Moreover,
		there is a constant $L>0$ such that,  for all $(t, x, u) \in [0,T]\times H\times U$,
		\begin{eqnarray*}
		&&\displaystyle \left|\overline F(t,x,u)\right|^2\vee\left|\overline G(t,x,u)\right|^2_{L_2(\Xi,H)}\leq
		L^2(1+|x|^2),\\
		&&\displaystyle \left|\overline F(t,x,u)-\overline F(t,y,u)\right|\vee\left|\overline G(t,x,u)-\overline G(t,y,u)\right|_{L_2(\Xi,H)}\leq
		L|x-y|.
		\end{eqnarray*}
\item[(iii)]
	$
	\overline q: [0,T]\times H\times \mathbb{R}\times \Xi\times U\rightarrow \mathbb{R}$ and $ \overline \phi:  H\rightarrow \mathbb{R}$ are continuous, {and
        $\overline{q}$ is continuous in $(t,x)\in [0,T]\times H$, uniformly in $u\in U$. Moreover,}     there is a  constant $L>0$
	such that, for all $(t, x,    y, z, x', y',z', u)
	\in [0,T]\times (H\times\mathbb{R}\times \Xi)^2\times U$,
	\begin{eqnarray*}
		&&| \overline q(t,x,y,z,u)|\leq L(1+|x|+|y|+|z|),
		\\
		&&|\overline q(t,x,y,z,u)-\overline q(t,x',y',z',u)|\leq L(|x-x'|+|y-y'|+|z-z'|),
		\\
		&&|\overline \phi(x)-\overline \phi(x')|\leq L|x-x'|.
	\end{eqnarray*}
	\end{description}
\end{assumption}
We define a metric on $[0,T]\times H$ as follows:
$$
                        \overline{d}_\infty((t,x),(s,y)):=|s-t|+\left|e^{(t\vee s-t)A}x-e^{(t\vee s-s)A}y\right|, \quad (t,x; s,y)\in ([0,T]\times H)^2.
$$
Then $([0,T]\times H,\overline{d}_\infty)$ is a complete metric space. For $t\in [0,T)$, we say $f\in C^0([t,T]\times H,K)$ if $f:[t,T]\times H\rightarrow K$ is continuous on $([t,T]\times H,\overline{d}_\infty)$, and
say $f\in C^0_p([t,T]\times H,K)$ if $f\in C^0([t,T]\times H,K)$ and grows in a polynomial way. For simplicity, we write $C^0([t,T]\times H):=C^0([t,T]\times H,\mathbb{R})$ and $C^0_p([t,T]\times H):=C^0_p([t,T]\times H,\mathbb{R})$.
\begin{definition}\label{definitionc20243}
	Let $t\in[0,T)$ and $f:[t,T]\times H\rightarrow \mathbb{R}$ be given.
	\begin{description}
 \item[(i)] We say $f\in C^{1,2}([t,T]\times H)$ if   $f_{t^+}$, $\nabla_{x}f$ and $\nabla_{x}^2f$ exist and are continuous   on the metric space $([t,T]\times H, \overline d_{\infty})$.
		\par
		\item[(ii)] We say
		$f\in C^{1,2}_p([t,T]\times H)$ if $f\in C^{1,2}([t,T]\times H)$ and $f$ together with  all its derivatives grow  in a polynomial way.
	\end{description}
\end{definition}
We denote by $X^{t,x,u}(\cdot)$ the mild solution of (\ref{state1}) and by $\overline V$  the value functional defined by (\ref{value1})  with coefficients $\overline F, \overline G,  \overline q,  \overline \phi$ and initial condition $(t,x)\in [0,T]\times H$,   and define $G^{t,x,u}_{s,t+\delta}[\zeta]$
 in spirit of (\ref{gdpp}) and (\ref{bsdegdpp}).
By Theorems \ref{theoremddp} and \ref{theorem3.9}, we obtain that
\begin{theorem}\label{theoremddp20242} 
	Let  Assumption \ref{hypstate20241} be satisfied. Then, the value functional
	$\overline V\in C^0([0,T]\times H)$ satisfies the following: for
	any $(t,x)\in [0,T]\times H$ and $0\leq t<t+\delta\leq T$,
	\begin{eqnarray}\label{ddpG2024}
	\overline V(t,x)=\mathop{\esssup}\limits_{u(\cdot)\in{\mathcal
			{U}}[t,t+\delta]}G^{t,x,u}_{t,t+\delta}\left[\overline V(t+\delta,X^{t,x,u}(t+\delta))\right].
	\end{eqnarray}
Moreover,
	there is a constant $C>0$ such that, for every  $0\leq t\leq s\leq T, x,y\in H$,
	\begin{eqnarray}\label{hold2024}
	\qquad      \left|\overline V(t,x)-\overline V\left(s,\, e^{(s-t)A}y\right)\right|\leq
	C(1+|x|+|y|)(s-t)^{\frac{1}{2}}+C|x-y|.
	\end{eqnarray}
\end{theorem}
\begin{definition}\label{definitionc202402061a}
	Let $t\in[0,T)$ and $g:[0,T]\times H\rightarrow \mathbb{R}$ be given, and $A$ satisfy Assumption \ref{hypstate20241} (i).
		  We say $g\in C^{1,2}_{p,A-}([t,T]\times H)\subset C^0_p([t,T]\times H)$ if   there exist $\nabla_{x}g\in  C^0_p([t,T]\times H,H)$ and $\nabla^2_{x}g\in C^0_p([t,T]\times H,\mathcal{S}(H))$  such that, for the solution $X$ of (\ref{formular1}) with initial condition $(t,x)\in [0,T)\times H$,
\begin{eqnarray}\label{statesop020240206a}
	\quad&\begin{aligned}
	g(X(s))\leq&\, g(X(t))+\!\!\int_{t}^{s}\!\!\left[\langle \nabla_xg(X(\sigma)), \, \vartheta(\sigma)\rangle_H+\frac{1}{2}\mbox{\rm Tr}(\nabla^2_{x}g(X(\sigma))(\varpi\varpi^*)(\sigma))\right]\! d\sigma\\[3mm]
	&+\int^{s}_{t}\langle \nabla_xg(X(\sigma)), \, \varpi(\sigma)dW(\sigma)\rangle_H.
	\end{aligned}
	\end{eqnarray}
\end{definition}

Define for $t\in[0,T)$,
$$
\overline\Phi_t:=\left\{\varphi\in C_p^{1,2}([t,T]\times H): A^*\nabla_x\varphi\in C_p^0([t,T]\times H,H)\right\}
$$
and
{\begin{eqnarray*}
	\begin{aligned}
		\overline{\mathcal{G}}_t:=&\bigg{\{}g\in C^{0}_p([t,T]\times H):  \exists \  (h_i, g_i)\in C^1([0,T];\mathbb{R})\times  C^{1,2}_{p,A-}([t,T]\times H), i=1,\ldots, \mathbf{N}  \\
      & \mbox{such that for all} \ (s,x)\in [t,T]\times H, \,  g(s,x)=\sum^{\mathbf{N}}_{i=1}h_i(s)g_i(s,x)
		\bigg{\}}.
	\end{aligned}
\end{eqnarray*}}
{
   For $g\in \overline{\mathcal{G}}_t$, we write
\begin{eqnarray*}
\partial_{t}^og(s,x)&:=&\sum^{\mathbf{N}}_{i=1}h_i'(s)g_i(s,x).
\end{eqnarray*}
}
{
\begin{remark}\label{remarkv0129120240206ab}
	By Lemma \ref{theoremito2}, for   $A$ satisfying Assumption \ref{hypstate20241} (i) and every $(\hat t,\hat x)\in [0,T)\times H$, the functional $|x-e^{(s-\hat t)A}\hat x|^{2m}\in C^{1,2}_{p,A-}([\hat t,T]\times H)$ for all $m\geq 1$. Therefore, for $h\in C^1([0,T];\mathbb{R})$, $\delta, \delta_i,\delta'_i\geq0, N>0,$ and $(\hat x,x^i)\in H\times H,
		i=0,1,2,\ldots, $ such that $t_i\leq \hat t,  h\geq0, \ \mbox{and} \   \sum_{i=0}^{\infty}(\delta_i+\delta'_i)\leq N$, the functional :  $h(s)|x|^4+\delta|x-e^{(s-\hat t)A}\hat x|^{2}+\sum_{i=0}^{\infty}[\delta_i|x-e^{(s-\hat t)A}x^i|^4
		+\delta^{'}_i|s-t_i|^2
		]$,  $(s,y)\in [t,T]\times H$,  which arises in the proof of  the comparison theorem,  belongs to $\overline{\mathcal{G}}_t$.
\end{remark}}
Now,  we define our notion of  viscosity solutions to the HJB equation (\ref{hjb32024}).

\begin{definition}\label{definition4.1} A  functional
	$w\in C^0([0,T]\times H)$ is called a
	viscosity sub-solution (resp., super-solution)
	to HJB equation (\ref{hjb32024}) if the terminal condition,  $w(T,x)\leq \overline \phi(T,x)$(resp., $w(T,x)\geq \overline \phi(T,x)$),
	$x\in H$ is satisfied, and for any $(\varphi, g)\in \overline\Phi_t\times \overline{\mathcal{G}}_t$ with $t\in [0,T)$, whenever the function $w-\varphi-g$  (resp.,  $w+\varphi+g$) satisfies for $(t,x)\in [0,T]\times H$,
	$$
	0=({w}-\varphi-g)(t,x)=\sup_{(s,y)\in [t,T]\times H}
	({w}- \varphi-g)(s,y)
	$$
	$$
	\left(\mbox{resp.,}\ \
	0=({w}+\varphi+g)(t,x)=\inf_{(s,y)\in [t,T]\times H}
	({w}+\varphi+g)(s,y)\right),
	$$
	we have
	\begin{eqnarray*}
		&&\displaystyle \varphi_{t^+}(t,x)+\partial_{t}^og(t,x)+\left\langle A^*\nabla_x\varphi(t,x),\, x\right\rangle_H\\[3mm]
		&&
		+\overline{\mathbf{H}}(t,x, (\varphi+g)(t,x),\nabla_x(\varphi+g)(t,x),\nabla^2_{x}(\varphi+g)(t,x))\geq0
	\end{eqnarray*}
	\begin{eqnarray*}\begin{aligned}
			\biggl(\mbox{resp.,} &-\varphi_{t^+}(t,x)-\partial_{t}^og(t,x)-\left\langle A^*\nabla_x\varphi(t,x),\,  x\right\rangle_H\\[3mm]
			&+\overline{\mathbf{H}}(t,x,
			-(\varphi+g)(t,x),-\nabla_x(\varphi+g)(t,x),-\nabla^2_{x}(\varphi+g)(t,x))
			\leq0\biggr).
		\end{aligned}
	\end{eqnarray*}
	A  functional   $w\in C^0([0,T]\times H)$ is said to be a
	viscosity solution to HJB equation (\ref{hjb32024}) if it is
	both a viscosity sub- and
	super-solution.
\end{definition}
\begin{remark}\label{remarkv012912024}
	As in the general path-dependent case,  the term $|x-e^{(t-\hat{t}\, )A}y|^{2m}_H$ with fixed $(\hat{t},y)\in [0,T)\times H$
can enter into the test functions for our viscosity solutions, for it satisfies
It\^o inequality (\ref{jias510815jia11}). The appearance of this term, increases some slight difficulty in the existence of viscosity solutions, but as we will see, significantly helps  us to establish the comparison principle of viscosity solutions.
Indeed,  in contrast to the $B$-continuity in~\cite[Definition 3.35, page 198]{fab1}, the viscosity solution $w$ is assumed here to be merely continuous,  and therefore the strong continuity of the value function in the state space $H$ is sufficient to ensure that the value function is a viscosity solution,  and the comparison theorem is also established without assuming the $B$-continuity  on the coefficients. In this way, 
 in our viscosity solution theory in infinite dimension, the  $B$-continuity assumption on the coefficients in~\cite[(3.21) and (3.22), page 184]{fab1} is not required at all.
\end{remark}
{
\begin{remark}\label{remarkv0129120240129e}
	Assuming the $B$-upper   and $B$-lower semi-continuity of $\overline{W}_1$ and $\overline{W}_2$, respectively, the Ekeland-Lebourg Theorem (see ~\cite[Theorem 3.25, page 188]{fab1})  in $H_{-2}$ is used to ensure the existence of  a maximum point of the auxiliary function in the proof of \cite[Theorem 3.50, page 206]{fab1}. Our functionals  $\overline{W}_1$ and $\overline{W}_2$ are merely strong-continuous in the Hilbert space $H$, and the Ekeland-Lebourg Theorem does not apply here. In fact, application in $H$ of Ekeland-Lebourg Theorem  (to yield an extremal point of the auxiliary function in the proof of comparison theorem) would incur a linear functional  $\langle p,\cdot\rangle_H$ (with some fixed $p\in H$),   which does not necessarily belong to $\overline{\Phi}_{\hat t}$ or $\overline{\mathcal{G}}_{\hat t}$,  to be added to the test function. However,  as in the general path-dependent case, the Borwein-Preiss variational principle applies,  since  its application (to get a maximum point of the auxiliary function) incurs the functional  $\sum^{\infty}_{i=0}\frac{1}{2^i}|x-e^{(t-t_i)A}x^i|^{4}$ with fixed  $\{(t_i,x^i)\}_{i\geq0}\in [0,\hat{t}]\times H$,  which belongs to $\overline{\mathcal{G}}_{\hat t}$,  to be added to the test functions.
\end{remark}
\begin{remark}\label{remarkv0129120240129f}
{As we stated in Remark \ref{remarkv0129120240206ab}}, the test functions in our definition of viscosity solutions {include}  the following three type functionals:  $|x|^{4}$, $\sum^{\infty}_{i=0}\frac{1}{2^i}[|x-e^{(t-t_i)A}x^i|^{4}+|t-t_i|^2]$  with fixed  $\{(t_i,x^i)\}_{i\geq0}\in [0,\hat{t}]\times H$ and $|x-e^{(t-\hat{t}\, )A}\hat{x}|^{2}_H$ with fixed $(\hat{t},\hat{x})\in [0,T)\times H$, while the test functions in ~\cite[Definition 3.35, page 198]{fab1} consist of  only the norm functional: $|x|, x\in H$. The first functional is introduced to control the  quadratic growth condition (\ref{w2024}); as we see in Remark \ref{remarkv0129120240129e}, the second functional is introduced in our test function as a  consequence of  applying  the Borwein-Preiss variational principle; the third term is introduced in terms of the form $2|x-e^{(t-\hat{t}\, )A}\hat{x}|^{2}_H+2|y-e^{(t-\hat{t}\, )A}\hat{x}|^{2}_H$ to control $|x-y|^2$ in the proof of the comparison theorem.
\end{remark}}
Applying It\^o formula (\ref{statesop0}),  {inequality (\ref{statesop020240206a})} and DPP (\ref{ddpG2024}), by the similar (even simpler) proof procedure of Theorem \ref{theoremvexist}, we obtain
\begin{theorem}\label{theoremvexist2024}
	Let Assumption \ref{hypstate20241}  be satisfied. Then,  the value
	functional $\overline V$  is a
	viscosity solution to equation (\ref{hjb32024}).
\end{theorem}
\begin{definition}\label{definition06072024}
	Let $\hat{t}\in [0,T)$ be fixed and  $w:[0,T]\times H\rightarrow \mathbb{R}$ be an upper semicontinuous function bounded from above.
	For every $N\geq 1$, define,  for $(t,x_N)\in [0,T]\times H_N$,
	\begin{eqnarray*}
		&&\widetilde{w}^{\hat{t},N}(t,x_N):=\sup_{y_N=x_N}
		w(t,y), \ \   t\in [\hat{t},T];\\
		&&\widetilde{w}^{\hat{t},N}(t,x_N):=\widetilde{w}^{\hat{t},N}(\hat{t},x_N)-(\hat{t}-t)^{\frac{1}{2}}, \ \   t\in[0,\hat{t}\,).
	\end{eqnarray*}
\end{definition}
We have the following Crandall-Ishii lemma in Hilbert space $H$.
\begin{theorem}\label{theorem05132024} \ \  Let $N\geq1$. Let $w_1,w_2:[0,T]\times H\rightarrow \mathbb{R}$ be upper semicontinuous functionals bounded from above and such that
	\begin{eqnarray}\label{051312024}
	\limsup_{|x|\rightarrow\infty}\frac{w_1(t,x)}{|x|}<0;\ \ \  \limsup_{|x|\rightarrow\infty}\frac{w_2(t,x)}{|x|}<0.
	\end{eqnarray}
	Let $\varphi\in C^2( H_N\times H_N)$ be such that
	$$
	w_1(t,x)+w_2(t,y)-\varphi(x_N,y_N)
	$$
	has a strict
	maximum over $[\hat{t},T]\times H\times H$ at a point $(t,\hat x, \hat y)$ with $\hat{t}\in (0,T)$ and there exists a  strictly monotone increasing function ${\tilde{\rho}}:[0,\infty)\rightarrow[0,\infty)$ with $\tilde{\rho}(0)=0$ such that, for all $(t,x,y)\in [\hat{t},T]\times H\times H$,
	\begin{eqnarray}\label{101012024}
	\begin{aligned}
	&\quad w_1(\hat{t}, \hat x)+w_2(\hat{t}, \hat y)-\varphi(\hat{x}_N,\hat{y}_N)\\
	&\geq w_1(t,x)+w_2(t,y)-\varphi(x_N,y_N)\\
	&\quad+\tilde{\rho}\left(|t-\hat{t}|^2+\left|x-e^{(t-\hat t)A}\hat{x}\right|^4+\left|y-e^{(t-\hat t)A}\hat{y}\right|^4\right).
	\end{aligned}
	\end{eqnarray} 
	Assume, moreover, $\widetilde{w}_{1}^{\hat{t},N,*}\in \Phi_N(\hat{t},\hat{x}_N)$ and $\widetilde{w}_{2}^{\hat{t},N,*}\in \Phi_N(\hat{t},\hat{y}_N)$, and there exists a local modulus of continuity  $\rho_1$  such that, for all  $\hat{t}\leq t\leq s\leq T, \ x\in H$ and $x_{t,s}^{A,N}:=(e^{(s-t)A}x)_N^{-}+x_N$,
	\begin{eqnarray}\label{0608a2024}
	\begin{aligned}
	&
	w_1(t,x)-w_1(s,x_{t,s}^{A,N})\leq \rho_1(|s-t|,|x|),\\
	&  w_2(t,x)-w_2(s,x_{t,s}^{A,N})\leq \rho_1(|s-t|,|x|).
	\end{aligned}
	\end{eqnarray}
	Then for every $\kappa>0$, there exist
	sequences  $(t_{k},x^{k}; s_{k},y^{k})\in \left([\hat{t},T]\times H\right)^2$ and
	sequences of functions $(\varphi_{1,k}, \psi_{1,k}, \varphi_{2,k}, \psi_{2,k})\in \overline\Phi_{\hat{t}}\times \overline\Phi_{\hat{t}}\times  \overline{\mathcal{G}}_{t_k}\times \overline{\mathcal{G}}_{s_k}$ bounded from below   
	such that 
	$$
	(w_{1}-\varphi_{1,k}-\varphi_{2,k})(t,x)
	$$
	has a strict  maximum $0$ at  $(t_k,x^{k})$ over $[{t_k},T]\times H$,
	$$
	(w_{2}-\psi_{1,k}-\psi_{2,k})(t,y)
	$$
	has a strict  maximum $0$ at $({s_k},y^k)$ over $[{s_k},T]\times H$, and
	\begin{eqnarray}\label{0608v2024}
	&&\left(t_{k}, x^{k};\,  (w_1, (\varphi_{1,k})_{t^+},\nabla_x\varphi_{1,k},\nabla^2_{x}\varphi_{1,k}, \partial_{t}^o\varphi_{2,k},\nabla_x\varphi_{2,k},\nabla^2_{x}\varphi_{2,k})(t_{k},{x}^{k})\right)\nonumber\\
	&&\underrightarrow{k\rightarrow\infty}\, \left({\hat{t}},\hat{x};\,  w_1(\hat{t},\hat{x}), b_1, \nabla_{x_1}\varphi(\hat x_N,\hat y_N), X_N, 0,\mathbf{0},\mathbf{0}\right),
	\end{eqnarray}
	\begin{eqnarray}\label{0608vw2024}
	&&\left(s_{k}, y^{k}; \, (w_2,(\psi_{1,k})_{t^+},\nabla_x\psi_{1,k},\nabla^2_{x}\psi_{1,k}, \partial_{t}^o\psi_{2,k},\nabla_x\psi_{2,k},\nabla^2_{x}\psi_{2,k})(s_{k},y^{k})\right)\nonumber\\
	&&\underrightarrow{k\rightarrow\infty}\, \left({\hat{t}},\hat{y};\,  w_2(\hat{t}, \hat{y}),  b_2, \nabla_{x_2}\varphi(\hat{x}_N,\hat{y}_N), Y_N, 0,\mathbf{0},\mathbf{0}\right),
	\end{eqnarray}
	where $b_{1}+b_{2}=0$ and $X_N,Y_N\in \mathcal{S}(H_N)$ satisfy  inequality (\ref{II0615}).
\end{theorem}
Now we study the uniqueness of  viscosity
solutions to equation (\ref{hjb32024}). 
We also require  the following assumption on $\overline G$.
\begin{assumption}\label{hypstate56662024}
	For every $(t,x)\in [0,T)\times H$,
	\begin{eqnarray}\label{g52024}
	\lim_{N\to \infty} \sup_{u\in U}\left|Q_N\overline {G}(t,x,u)\right|_{L_2(\Xi, H)}^2=0.
	\end{eqnarray}
\end{assumption}
\par
Without loss of generality, we also  assume that there exists a constant $L>0$ such that,
for all $(t,x, p,l)\in [0,T]\times H\times H\times {\mathcal{S}}(H)$ and $r_1,r_2\in \mathbb{R}$ with $r_1<r_2$,
\begin{eqnarray}\label{5.12024}
\overline{\mathbf{H}}(t,x,r_1,p,l)-\overline{\mathbf{H}}(t,x,r_2,p,l)\geq L(r_2-r_1).
\end{eqnarray}
\begin{theorem}\label{theoremhjbm2024}
	Let Assumptions \ref{hypstate20241} and \ref{hypstate56662024}  be satisfied.
	Let $\overline W_1\in C^0([0,T]\times H)$ $(\mbox{resp}., \overline  W_2\in C^0([0,T]\times H))$ be  a viscosity sub-solution (resp., super-solution) to equation (\ref{hjb32024}) and  let  there exist a  constant $L>0$  and a local modulus of continuity  $\rho_2$
	such that, for any  $0\leq t\leq  s\leq T$ and
	$x,y\in{H}$,
	\begin{eqnarray}\label{w2024}
	|\overline  W_1(t,x)|\vee |\overline  W_2(t,x)|\leq L (1+|x|^2),
	\end{eqnarray}
	\begin{eqnarray}\label{w12024}
	\begin{aligned}
	&\left|\overline  W_1(s,e^{(s-t)A}x)-\overline  W_1(t,y)\right|\vee\left|\overline  W_2(s,e^{(s-t)A}x)-\overline  W_2(t,y)\right|\\
	\leq&\, \rho_2(|s-t|,\, |x|\vee|y|)+L(1+|x|+|y|)|x-y|.
	\end{aligned}
	\end{eqnarray}
	Then  $\overline  W_1\leq \overline  W_2$.
\end{theorem}
Theorems    \ref{theoremvexist2024} and \ref{theoremhjbm2024} yield that the viscosity solution to   HJB equation  (\ref{hjb32024})
is  the value functional  $\overline V$.

\begin{theorem}\label{theorem522024}
	Let  Assumptions \ref{hypstate20241} and \ref{hypstate56662024}   be satisfied.  Then the value
	functional $\overline V$ 
is the unique viscosity
	solution to HJB equation~(\ref{hjb32024}) which satisfies (\ref{w2024}) and (\ref{w12024}).
\end{theorem}

{\bf  Proof of Theorem \ref{theoremhjbm2024}}.
It is sufficient to prove
$\overline  W_1\leq \overline  W_2$ under the stronger  assumption that $\overline  W_1$ is a viscosity sub-solution of
\begin{eqnarray}\label{1002a2024}
\begin{cases}
 (\overline W_1)_{t^+}(t,x)+\left\langle A^*\nabla_x\overline W_1(t,x),\,  x\right\rangle_H\\[3mm]
\quad+\overline{\mathbf{H}}\left(t,x,\, (\overline W_1, \nabla_x \overline W_1,\nabla^2_{x} \overline W_1)(t,x)\right)
=c:= \frac{\varrho}{(T+1)^2}, \\[3mm]
\quad\quad\quad \quad \qquad  (t, x)\in [0,T)\times H; \\[3mm]
\overline W_1(t,x)=\overline{\phi}(x), \  x\in H.\\
\end{cases}
\end{eqnarray}

The rest of the proof splits in the following three steps.
\par
$Step\  1.$ Definitions of auxiliary functionals.
\par
It is sufficient to prove that $\overline  W_1\leq \overline W_2$ in $[T-\bar{a},T)\times
H$ with
$$\bar{a}:=\frac{1}{2(20L+8)L}\wedge{T},$$
for the desired comparison is easily extended  to the whole time interval $[0,T]$ via a backward iteration over the intervals
$[T-i\bar{a},T-(i-1)\bar{a})$.  Otherwise, there is  $(\tilde{t},\tilde{x})\in (T-\bar{a},T)\times
H$  such that
$\tilde{m}:=\overline W_1(\tilde{t},\tilde{x})-\overline W_2(\tilde{t},\tilde{x})>0$.
\par
Let  $\varepsilon >0$ be  a small number such that
$$
\overline  W_1(\tilde{t},\tilde{x})-\overline  W_2(\tilde{t},\tilde{x})-2\varepsilon \frac{\nu T-\tilde{t}}{\nu
	T}|\tilde{x}|^4
>\frac{\tilde{m}}{2}
$$
and
\begin{eqnarray}\label{5.32024}
\frac{\varepsilon}{\nu T}\leq\frac{c}{4}
\end{eqnarray}
with
$$
\nu:=1+\frac{1}{2T(20L+8)L}.
$$
Next,  we define for any  $(x,y)\in H\times H$,
\begin{eqnarray*}
	\overline \Psi(t,x,y):=\overline W_1(t,x)-\overline W_2(t,y)-\beta|x-y|^2
	-\varepsilon\frac{\nu T-t}{\nu
		T}\, \left(|x|^4+|y|^4\right).
\end{eqnarray*}
In view of  (\ref{w2024}), the function $\overline \Psi$ is clearly continuous and  bounded from above on $[{T-\bar{a}},T]\times H\times H$. 
Define  for all $(t,x,y; s,x',y')\in ([0,T]\times H\times H)^2$,
$$\overline d_{1,\infty}((t,x,y),(s,x',y'))=\overline d_\infty((t,x),(s,x'))+\overline d_\infty((t,y),(s,y')).$$
Clearly,  the metric space $([t,T]\times H\times H,\overline d_{1,\infty})$ is compete. Define for $(t,x,y; s,x',y')\in ([0,T]\times H\times H)^2$,
$$
\hat{\Upsilon}^0(x):=|x|^4, \quad \Upsilon^0((t,x),(s,x')):=|s-t|^2+\left|e^{(t\vee s-t)A}x-e^{(t\vee s-s)A}x'\right|^4,
$$
and
\begin{eqnarray*}
\Upsilon^{0,1}((t,x,y),\, (s,x',y'))&:=\displaystyle 2|s-t|^2+\left|e^{(t\vee s-t)A}x-e^{(t\vee s-s)A}x'\right|^4\\
&\displaystyle +\left|e^{(t\vee s-t)A}y-e^{(t\vee s-s)A}y'\right|^4.
\end{eqnarray*}
Then ${\Upsilon}^{0}(\cdot,\cdot)$ and ${\Upsilon}^{0,1}(\cdot,\cdot)$ are the  gauge-type functions on $([\tilde t,T]\times H,\overline d_{\infty})$ and  on $([\tilde t,T]\times H\times H,\overline d_{1,\infty})$, respectively.
Take
$\delta_i:=\frac{1}{2^i}$ for all $i\geq0$,
from the proof procedure of Lemma \ref{theoremleft} it follows that,
for every  $(t_0,x^0,y^0)\in [{\tilde{t}},T)\times H\times H$ satisfying
$$
\overline \Psi({t_0},x^0,y^0)\geq \sup_{(s,x,y)\in  [\tilde{t},T]\times H\times H}\overline \Psi(t,x,y)-\frac{1}{\beta},\
\    \mbox{and} \ \ \overline \Psi({t_0},x^0, {y^0})\geq \overline \Psi(\tilde{t},\tilde{x},\tilde{y}) >\frac{\tilde{m}}{2},
$$
there exist $(\hat{t},\hat{x},\hat{y})\in [\tilde{t},T]\times H\times H$ and a sequence $\{(t_i,x^i,y^i)\}_{i\geq1}\subset
[\tilde{t},T]\times H\times H$ such that
\begin{description}
	\item[(i)] $\Upsilon^{0,1}((t_i,x^i,y^i),(\hat{t},\hat{x},\hat{y}))
\leq \frac{1}{2^i\beta}$
and $t_i\uparrow \hat{t}$ as $i\rightarrow\infty$,
	\item[(ii)]  $\overline \Psi_1(\hat{t},\hat{x},\hat{y})
\geq \overline \Psi({t_0},x^0,{y^0})$, and
	\item[(iii)]    for all $(s, x,y)\in [\hat{t},T]\times H\times H\setminus \{(\hat{t},\hat{x},\hat{y})\}$,
	{   \begin{eqnarray}\label{iii42024}
		\overline \Psi_1(s,x,y)
		<\overline \Psi_1(\hat{t},\hat{x},\hat{y}),
		\end{eqnarray}}
\end{description}
where we define, for $(t,x,y)\in  [\tilde t, T]\times H\times H$,
{$$
	\overline \Psi_1(t,x,y):=  \overline \Psi(t,x,y)
	-\sum_{i=0}^{\infty}
	\frac{1}{2^i}{\Upsilon}^{0,1}((t_i,x^i,y^i),(t,x,y))
.
	$$}
Note that the point
$({\hat{t}},\hat{x},\hat{y})$ depends on  $\beta$ and
$\varepsilon$.
\par
By the similar proof procedure of Steps 2 and 3 in the proof of Theorem \ref{theoremhjbm}, we can show that
there exist ${{M}_0}>0$ and $N_0>0$ {independent of $\beta$}
such that  $\hat{t}\in [\tilde{t},T)$ for all $\beta\geq N_0$,
\begin{eqnarray}\label{5.10jiajiaaaa2024}
|\hat{x}|\vee|\hat{y}|<M_0, \ \ \ \
\lim_{\beta\to \infty} \beta|\hat{x}-\hat{y}|^2
=0.
\end{eqnarray}

{ $Step\ 2.$   Crandall-Ishii lemma.}
\par

For every $N\geq1$, we define,
for $(t,x,y)\in [0,T]\times H\times H$,
\begin{eqnarray}\label{060912024}
\begin{aligned}
w_{1}(t,x)&=\overline W_1(t,x)-\varepsilon\frac{\nu T-t}{\nu
	T}\, |x|^4
-\varepsilon {\Upsilon}^0((t,x),(\hat{t},\hat{x}))\\
&\quad
-2\beta\left|x^-_N-(e^{(t-\hat{t}\,)A}\hat{z})^-_N\right|^2-\sum_{i=0}^{\infty}
\frac{1}{2^i}{\Upsilon}^0((t_i,x^i),(t,x)),
\end{aligned}
\end{eqnarray}
\begin{eqnarray}\label{060922024}
\begin{aligned}
w_{2}(t,y)&=-\overline W_2(t,y)-\varepsilon\frac{\nu T-t}{\nu
	T}\, |y|^4
-\varepsilon {\Upsilon}^0((t,y),(\hat{t},\hat{y}))\\
&\quad
-2\beta\left|y_N^--(e^{(t-\hat{t}\,)A}\hat{z})^-_N\right|^2-\sum_{i=0}^{\infty}
\frac{1}{2^i}{\Upsilon}^0((t_i,y^i),(t,y)),
\end{aligned}
\end{eqnarray}
where $\hat{z}=\frac{\hat{x}+\hat{y}}{2}$.   We  note that $w_1,w_2$ depend on $\hat{z}$ and $N$, and thus on $\beta$,
$\varepsilon$ and  $N$.  Define $\varphi\in C^2(H_N\times H_N)$ by
\begin{eqnarray}\label{070632024}
\varphi(x,y)=\beta|x-y|^2,\ \ (x,y)\in H_N\times H_N.
\end{eqnarray}
By the similar proof procedure of  Lemma \ref{0611a},  we show that $w_1$ and $w_2$ satisfy the conditions of  Theorem \ref{theorem05132024} with $\varphi$ being defined by (\ref{070632024}).
Then by Theorem \ref{theorem05132024}, there exist
sequences  $(l_{k},\check{x}^{k}), (s_{k},\check{y}^{k})\in [\hat{t},T]\times H$ and
sequences of functionals 
$(\varphi_{1,k}, \psi_{1,k}, \varphi_{2,k}, \psi_{2,k})\in \overline\Phi_{\hat{t}}\times  \overline\Phi_{\hat{t}}\times \overline{\mathcal{G}}_{l_k} \times \overline{\mathcal{G}}_{s_k}$ bounded from below such that 
\begin{eqnarray}\label{0609a2024}
w_{1}(t,x)-\varphi_{1,k}(t,x)-\varphi_{2,k}(t,x)
\end{eqnarray}
has a strict  maximum $0$ at  $(l_k,\check{x}^{k})$ over $[l_k, T]\times H$,
\begin{eqnarray}\label{0609b2024}
w_{2}(t,y)-\psi_{1,k}(t,y)-\psi_{2,k}(t,y)
\end{eqnarray}
has a strict  maximum $0$ at $(s_k,\check{y}^{k})$ over $ [s_k, T]\times H$, and
\begin{eqnarray}\label{0608v12024}
\ \ \ &&\left(l_{k}, \check{x}^{k};\,  (w_1,(\varphi_{1,k})_{t^+},\nabla_x\varphi_{1,k},\nabla^2_{x}\varphi_{1,k}, \partial_{t}^o\varphi_{2,k},\nabla_x\varphi_{2,k},\nabla^2_{x}\varphi_{2,k})( l_{k},\check{x}^{k})\right)\\
&&\underrightarrow{k\rightarrow\infty}\, \left({\hat{t}},\hat{x};\,  w_1(\hat{t},\hat{x}), b_1, 2\beta (\hat{x}_N-\hat{y}_N), X_N, 0,\mathbf{0},\mathbf{0}\right),\nonumber
\end{eqnarray}
\begin{eqnarray}\label{0608vw12024}
&&\left(s_{k}, \check{y}^{k}; \, (w_2, (\psi_{1,k})_{t^+}, \nabla_x\psi_{1,k}, \nabla^2_{x}\psi_{1,k}, \partial_{t}^o\psi_{2,k}, \nabla_x\psi_{2,k}, \nabla^2_{x}\psi_{2,k})( s_{k},\check{y}^{k})\right)\\
&& \ \ \ \ \underrightarrow{k\rightarrow\infty}\, \left({\hat{t}},\hat{y};\,  w_2(\hat{t},\hat{y}), b_2, 2\beta(\hat{y}_N-\hat{x}_N), Y_N, 0,\mathbf{0},\mathbf{0}\right),\nonumber
\end{eqnarray}
where $b_{1}+b_{2}=0$ and $X_N,Y_N\in \mathcal{S}(H_N)$ satisfy  inequality (\ref{II10}) with $\beta$.
We note that  
sequence  $(\check{x}^{k},\check{y}^{k},l_{k},s_{k},\varphi_{1,k},\psi_{1,k},\varphi_{2,k},\psi_{2,k})$ and $b_{1},b_{2},X_N,Y_N$  depend on  $\beta$,
$\varepsilon$ and $N$. For every $(t,x, s,y)\in ([T-\bar{a},T]\times H)^2$, define
\begin{eqnarray*}
	\chi^{k,N,1}(t,x)  &:=&\varphi_{1,k}(t,x)-2\beta\left|x_N-(e^{(t-\hat{t}\,)A}\hat{z})_N\right|^2,\\[3mm]
	\chi^{k,N,2}(t,x)&:=& \varepsilon\frac{\nu T-t}{\nu
		T}|x|^4
	+\varepsilon {\Upsilon}^0((t,x),(\hat{t},\hat{x}))
	+\sum_{i=0}^{\infty}
	\frac{1}{2^i}{\Upsilon}^0((t_i,x^i),(t,x))\\
	&&+\varphi_{2,k}(t,x)
	+2\beta\left|x-e^{(t-\hat{t}\,)A}\hat{z}\right|^2, \\[3mm]
	\chi^{k,N}(t,x)  &:=&\chi^{k,N,1}(t,x)+\chi^{k,N,2}(t,x); \quad \hbox{\rm and }
\end{eqnarray*}
\begin{eqnarray*}
	\hbar^{k,N,1}(s,y)   &:=&\psi_{1,k}(s,y)-2\beta\left|y_N-(e^{(s-\hat{t}\,)A}\hat{z})_N\right|^2,\\[3mm]
	\hbar^{k,N,2}(s,y) &:=& \varepsilon\frac{\nu T-s}{\nu
		T}|y|^4
	+\varepsilon {\Upsilon}^0((s,y),(\hat{t},{\hat{y}}))+\sum_{i=0}^{\infty}
	\frac{1}{2^i}{\Upsilon}^0((t_i,\eta^i),(s,y))\\
	&&
	+\psi_{2,k}(s,y)
	+2\beta\left|y-e^{(s-\hat{t}\,)A}\hat{z}\right|^2, \\[3mm]
	\hbar^{k,N}(s,y)   &:=&\hbar^{k,N,1}(s,y)+\hbar^{k,N,2}(s,y).
\end{eqnarray*}
Clearly, we have  $\chi^{k,N,2}\in\overline{\mathcal{G}}_{{l_k}}$ and $\hbar^{k,N,2}\in \overline{\mathcal{G}}_{{s_k}}$, and in view of  Lemma \ref{lemma03302}, we have  $\chi^{k,N,1},\hbar^{k,N,1}\in \overline\Phi_{\hat{t}}$.
Moreover, by  (\ref{0609a2024}), (\ref{0609b2024}) and definitions of $w_1$ and $w_2$, we have
$$
\left(\overline W_1-\chi^{k,N,1}-\chi^{k,N,2}\right)(l_k,\check{x}^k)=\sup_{(t,x)\in [{{l_k}},T]\times H}
\left(\overline W_1-\chi^{k,N,1}-\chi^{k,N,2}\right)(t,x)
$$
and
$$
\left(\overline W_2+\hbar^{k,N,1}+\hbar^{k,N,2}\right)(s_k,\check{y}^k)=\inf_{(s,y)\in [{s_k},T]\times H}
\left(\overline W_2+\hbar^{k,N,1}+\hbar^{k,N,2}\right)(s,y).
$$

Since $l_{k},s_{k}\rightarrow {\hat{t}}$ as $k\rightarrow\infty$ and ${\hat{t}}<T$
for $\beta>N_0$, we see that for every fixed $\beta>N_0$,   there is a constant $ K_\beta>0$ such that
$$
|l_{k}|\vee|s_{k}|<T 
\quad \mbox{for all}    \ \ k\geq K_\beta.
$$
Now, for every $\beta>N_0$ and  $k>K_\beta$, 
we have from the definition of viscosity solutions that
\begin{eqnarray}\label{vis12024}
\begin{aligned}
&\partial_{t}^o\chi^{k,N}(l_k,\check{x}^k)
+\left\langle A^*\nabla_{x}(\varphi_{1,k})(l_k,\check{x}^k),\, \check{x}^k\right\rangle_H\\
&-4\beta\left\langle A^*\left((\check{x}^k)_N-(e^{({l_k}-\hat{t}\,)A}\hat{z})_N\right),\,
\check{x}^k
\right\rangle_H\\
&+\overline{\mathbf{H}}{(}l_k,\check{x}^k,  (\overline W_1, \nabla_x\chi^{k,N},
\nabla^2_{x}\chi^{k,N})(l_k,\check{x}^k)
{)}\geq c
\end{aligned}
\end{eqnarray}
and
\begin{eqnarray}\label{vis22024}
\begin{aligned}
&
-\partial_{t}^o\hbar^{k,N}(s_k,\check{y}^k)-\left\langle A^*\nabla_{x}(\psi_{1,k})(s_k,\check{y}^k),\, \check{y}^k\right\rangle_H\\
&+4\beta\left\langle A^*\left((\check{y}^k)_N-(e^{({s_k}-\hat{t}\,)A}\hat{z})_N\right),\,
\check{y}^k
\right\rangle_H\\
&+\overline{\mathbf{H}}{(}s_k,\check{y}^k, (\overline W_2,
-\nabla_x\hbar^{k,N}, -\nabla^2_{x}\hbar^{k,N})(s_k,\check{y}^k){)}\leq0,
\end{aligned}
\end{eqnarray}
where, for every $(t,x)\in [{l_k},T]\times{H}$ and  $ (s,y)\in [{s_k},T]\times{H}$, we have from Lemma \ref{lemma03302},
\begin{eqnarray*}
	\partial_{t}^o\chi^{k,N}(t,x)&:=&(\chi^{k,N,1})_{t^+}(t,x)+\partial_{t}^o\chi^{k,N,2}(t,x)\\
	&=&
	(\varphi_{1,k})_{t^+}(t,x)+4\beta\left\langle A^*\left(x_N-(e^{({t}-\hat{t}\,)A}\hat{z})_N\right),\,
	e^{({t}-\hat{t}\,)A}\hat{z}\right\rangle_H\\
	&&-\frac{\varepsilon}{\nu T}|x|^4
	+2\varepsilon({t}-{\hat{t}})+2\sum_{i=0}^{\infty}\frac{1}{2^i}(t-t_{i})+\partial_{t}^o\varphi_{2,k}(t,x),
\end{eqnarray*}
\begin{eqnarray*}
	\nabla_x\chi^{k,N}(t,x)&:=&\nabla_x\chi^{k,N,1}(t,x)+\nabla_x\chi^{k,N,2}(t,x)\\&=&\nabla_{x}(\varphi_{1,k})(t,x)
	+4\beta(x^-_N-(e^{(t-\hat{t}\,)A}\hat{z})^-_N)
	\\
	&&+\varepsilon\frac{\nu T-{t}}{\nu T}\nabla_x\hat{\Upsilon}^0(x) +\varepsilon\nabla_x\hat{\Upsilon}^0(x-e^{(t-\hat{t}\,)A}\hat{x})
	\\
	&&+\sum_{i=0}^{\infty}\frac{1}{2^i}\nabla_x\hat{\Upsilon}^0(x-e^{(t-t_i)A}x^i)
	+\nabla_{x}(\varphi_{2,k})(t,x),
\end{eqnarray*}
\begin{eqnarray*}
	\nabla^2_{x}\chi^{k,N}(t,x)&:=&\nabla^2_{x}\chi^{k,N,1}(t,x)+\nabla^2_{x}\chi^{k,N,2}(t,x)\\
	&=&\nabla^2_{x}(\varphi_{1,k})(t,x)+4\beta Q_N+\varepsilon\frac{\nu T-{t}}{\nu T}\nabla^2_x\hat{\Upsilon}^0(x) \\
	&&+\varepsilon\nabla^2_x\hat{\Upsilon}^0(x-e^{(t-\hat{t}\,)A}\hat{x})
+\sum_{i=0}^{\infty}\frac{1}{2^i}\nabla^2_x\hat{\Upsilon}^0(x-e^{(t-t_i)A}x^i)\\
     &&+\nabla^2_{x}(\varphi_{2,k})(t,x),
\end{eqnarray*}
\begin{eqnarray*}
	\partial_{t}^o\hbar^{k,N}(s,y)&:=&(\hbar^{k,N,1}))_{t^+}(s,y)+\partial_{t}^o\hbar^{k,N,2}(s,y)\\
	&=&
	(\psi_{1,k})_{t^+}(s,y)+4\beta\left\langle A^*\left(y_N-(e^{(s-\hat{t}\,)A}\hat{z})_N\right),\,
	e^{(s-\hat{t}\,)A}\hat{z}\right\rangle_H\\
	&&-\frac{\varepsilon}{\nu T}|y|^4+2\varepsilon({s}-{\hat{t}})
	+2\sum_{i=0}^{\infty}\frac{1}{2^i}(s-t_{i})
	+\partial_{t}^o\psi_{2,k}(s,y),
\end{eqnarray*}
\begin{eqnarray*}
	\nabla_x\hbar^{k,N}(s,y)&:=&\nabla_x\hbar^{k,N,1}(s,y)+\nabla_x\hbar^{k,N,2}(s,y)\\
	&=& \nabla_{x}\psi_{1,k}(s,y)+4\beta(y^-_N-(e^{(s-\hat{t}\,)A}\hat{z})^-_N)
	\nonumber\\
	&&+\varepsilon\frac{\nu T-{s}}{\nu T}
	\nabla_x\hat{\Upsilon}^0(y)+\varepsilon\nabla_x\hat{\Upsilon}^0(y-e^{(s-\hat{t}\,)A}\hat{y})
\\
	&&+\sum_{i=0}^{\infty}\frac{1}{2^i}\nabla_x\hat{\Upsilon}^0(y-e^{(s-t_i)A}y^i)
	+\nabla_{x}\psi_{2,k}(s,y),
\end{eqnarray*}
and
\begin{eqnarray*}
	\nabla^2_{x}\hbar^{k,N}(s,y)&:=&\nabla^2_{x}\hbar^{k,N,1}(s,y)+\nabla^2_{x}\hbar^{k,N,2}(s,y)\\
	&=&\nabla^2_{x}\psi_{1,k}(s,y)+4\beta Q_N+\varepsilon\frac{\nu T-{s}}{\nu T}
	\nabla^2_x\hat{\Upsilon}^0(y)\\
	&&+\varepsilon\nabla^2_x\hat{\Upsilon}^0(y-e^{(s-\hat{t}\,)A}\hat{y})
+\sum_{i=0}^{\infty}\frac{1}{2^i}\nabla^2_x\hat{\Upsilon}^0(y-e^{(s-t_i)A}y^i)\\
&&+\nabla^2_{x}\psi_{2,k}(s,y).
\end{eqnarray*}
\par
{ $Step\ 3.$  Calculation and  completion of the proof.}
\par
Letting 
$k\rightarrow\infty$ in (\ref{vis12024}) and (\ref{vis22024}), and using (\ref{0608v12024}) and (\ref{0608vw12024}),  we have
\begin{eqnarray}\label{031032024}
\begin{aligned}
& b_1+4\beta\langle A^*(\hat{x}_N-\hat{z}_N),\,
\hat{z}\rangle_H-\frac{\varepsilon}{\nu T}|\hat{x}|^4+2\sum_{i=0}^{\infty}\frac{1}{2^i}(\hat{t}-t_i)\\
&+2\beta\left\langle A^*\left(\hat{x}_N-\hat{y}_N\right),\,\hat{x}\right\rangle_H
-4\beta\left\langle A^*\left(\hat{x}_N-\hat{z}_N\right),\,
\hat{x}
\right\rangle_H\\
&+\overline{\mathbf{H}}(\hat{t},\hat{x}, (\overline W_1,
\nabla_x\chi^{N},\nabla^2_{x}\chi^{N})(\hat t,\hat{x}))
\geq c
\end{aligned}
\end{eqnarray}
and
\begin{eqnarray}\label{031042024}
\begin{aligned}
& -b_2-4\beta\left\langle A^*\left(\hat{y}_N-\hat{z}_N\right),\,
\hat{z}\right\rangle_H+ \frac{\varepsilon}{\nu T}|\hat{y}|^4-2\sum_{i=0}^{\infty}\frac{1}{2^i}(\hat{t}-t_i)\\
&+2\beta \left\langle A^*\left(\hat{x}_N-\hat{y}_N\right),\, \hat{y}\right\rangle _H
+4\beta\left\langle A^*\left(\hat{y}_N-\hat{z}_N\right),\,
\hat{y}
\right\rangle_H\\
&+\overline{\mathbf{H}}(\hat t, \hat{y}, (\overline W_2,-\nabla_x\hbar^{N},-\nabla^2_{x}\hbar^{N})(\hat t,\hat{y}))
\leq0,
\end{aligned}
\end{eqnarray}
where
\begin{eqnarray*}
\begin{aligned}
\nabla_x\chi^{N}(\hat{t}, \hat{x})
&:=
2\beta(\hat{x}-\hat{y})
 +\varepsilon\frac{\nu T-{\hat{t}}}{\nu T}\nabla_x\hat{\Upsilon}^0(\hat{x})
+\sum_{i=0}^{\infty}\frac{1}{2^i}\nabla_x\hat{\Upsilon}^0(\hat{x}-e^{(\hat{t}-t_i)A}x^i),
\end{aligned}
\end{eqnarray*}
\begin{eqnarray*}
\begin{aligned}
\nabla^2_{x}\chi^{N}(\hat{t}, \hat{x})
&:= X_N+4\beta Q_N
 +\varepsilon\frac{\nu T-{\hat{t}}}{\nu T}\nabla^2_{x}\hat{\Upsilon}^0(\hat{x})
+\sum_{i=0}^{\infty}\frac{1}{2^i}\nabla^2_{x}\hat{\Upsilon}^0(\hat{x}-e^{(\hat{t}-t_i)A}x^i),
\end{aligned}
\end{eqnarray*}
\begin{eqnarray*}
\begin{aligned}
\nabla_x\hbar^{N}(\hat{t},\hat{y})
&:=
2\beta(\hat{y}-\hat{x})
 +\varepsilon\frac{\nu T-{\hat{t}}}{\nu T}\nabla_x\hat{\Upsilon}^0(\hat{y})+\sum_{i=0}^{\infty}\frac{1}{2^i}
\nabla_x\hat{\Upsilon}^0(\hat{y}-e^{(\hat{t}-t_i)A}y^i),
\end{aligned}
\end{eqnarray*}
and
\begin{eqnarray*}
\begin{aligned}
\nabla^2_{x}\hbar^{N}(\hat{t}, \hat{y})
&:= Y_N+4\beta Q_N
+\varepsilon\frac{\nu T-{\hat{t}}}{\nu T}\nabla^2_{x}\hat{\Upsilon}^0(\hat{y})
+\sum_{i=0}^{\infty}\frac{1}{2^i}
\nabla^2_{x}\hat{\Upsilon}^0(\hat{y}-e^{(\hat{t}-t_i)A}y^i).
\end{aligned}
\end{eqnarray*}
Since $b_1+b_2=0$ and $\hat{z}=\frac{\hat{x}+\hat{y}}{2}$, subtracting
both sides of   (\ref{031032024}) from those  of (\ref{031042024}), we observe the following cancellations:
\begin{eqnarray*}
&& 4\beta\langle A^*(\hat{x}_N-\hat{z}_N),\,
\hat{z}\rangle_H-\left[-4\beta\left\langle A^*\left(\hat{y}_N-\hat{z}_N\right),\,
\hat{z}\right\rangle_H\right]\\
&=&4\beta\langle A^*(\hat{x}_N+\hat{y}_N-2\hat{z}_N),\, \hat{z}\rangle_H=0,
\end{eqnarray*}
\begin{eqnarray*}
 2\beta\left\langle A^*\left(\hat{x}_N-\hat{y}_N\right),\,\hat{x}\right\rangle_H-2\beta \left\langle A^*\left(\hat{x}_N-\hat{y}_N\right),\, \hat{y}\right\rangle _H
=2\beta\left\langle A^*\left(\hat{x}_N-\hat{y}_N\right),\,\hat{x}-\hat{y}\right\rangle_H,
\end{eqnarray*}
 \begin{eqnarray*}
&&-4\beta\left\langle A^*\left(\hat{x}_N-\hat{z}_N\right),\,
\hat{x}
\right\rangle_H-
4\beta\left\langle A^*\left(\hat{y}_N-\hat{z}_N\right),\,
\hat{y}
\right\rangle_H\\
&=&-2\beta\langle A^*(\hat{x}_N-\hat{y}_N),\, \hat{x}\rangle_H-
2\beta\left\langle A^*\left(\hat{y}_N-\hat{x}_N\right),\,
\hat{y}
\right\rangle_H\\
&=&-2\beta\left\langle A^*\left(\hat{x}_N-\hat{y}_N\right),\,\hat{x}-\hat{y}\right\rangle_H;
\end{eqnarray*}
%
{\it all the terms which involve the unbounded operator $A$  mutually cancels out},  and  thus we have
\begin{eqnarray}\label{vis1122024}
\begin{aligned}
&\quad c+ \frac{\varepsilon}{\nu T}(|\hat{x}|^4+|\hat{y}|^4)-4\sum_{i=0}^{\infty}\frac{1}{2^i}(\hat{t}-t_i)\\
&\leq\overline{\mathbf{H}}(\hat{t},\hat{x}, (\overline W_1,\nabla_x\chi^{N},\nabla^2_{x}\chi^{N})(\hat{t},\hat{x}))\\
&\quad-\overline{\mathbf{H}}(\hat{t},\hat{y}, (\overline W_2,-\nabla_x\hbar^{N},-\nabla^2_{x}\hbar^{N})(\hat{t},\hat{y})).
\end{aligned}
\end{eqnarray}
On the  other hand, from  (\ref{5.12024}) and via a simple calculation, we have
\begin{eqnarray}\label{v4}
\begin{aligned}
&\quad\overline{\mathbf{H}}(\hat{t},\hat{x}, (\overline W_1,\nabla_x\chi^{N},\nabla^2_{x}\chi^{N})(\hat{t},\hat{x}))\\
&\quad-\overline{\mathbf{H}}(\hat{t},\hat{y}, (\overline W_2,-\nabla_x\hbar^{N},-\nabla^2_{x}\hbar^{N})(\hat{t},\hat{y}))\\
&\leq\overline{\mathbf{H}}(\hat{t},\hat{x}, (\overline W_2,\nabla_x\chi^{N},\nabla^2_{x}\chi^{N})(\hat{t},\hat{x}))\\
&\quad-\overline{\mathbf{H}}(\hat{t},\hat{y}, (\overline W_2,-\nabla_x\hbar^{N},-\nabla^2_{x}\hbar^{N})(\hat{t},\hat{y}))\\
&\leq\sup_{u\in U}(J_{1}+J_{2}+J_{3}).
\end{aligned}
\end{eqnarray}
Here,
from Assumption \ref{hypstate20241} (ii)   and (\ref{II10}) with $\beta$,  we have
\begin{eqnarray}\label{j22024}
\begin{aligned}
&\quad
J_{1}:=\frac{1}{2}\mbox{Tr}\left[\nabla^2_{x}\chi^{N}(\hat{t}, \hat{x})\, (\overline {G}\, \overline {G}^*)(\hat{t}, \hat{x},u)\right]-\frac{1}{2}\mbox{Tr}\left[-\nabla^2_{x}\hbar^{N}(\hat{t}, \hat{y})\, (\overline G\, \overline G^*)(\hat{t}, \hat{y},u)\right]\\
\leq&\, 3\beta\left|\overline {G}(\hat{t}, \hat{x},u)-\overline G(\hat{t}, \hat{y},u)\right|_{L_2(\Xi,H)}^2\\
&+2\beta\left (\left|Q_N\overline {G}(\hat{t}, \hat{x},u)\right|_{L_2(\Xi,H)}^2+\left |Q_N\overline {G}(\hat{t}, \hat{y},u)\right|_{L_2(\Xi,H)}^2\right)\\
&+6\varepsilon\frac{\nu T-{\hat{t}}}{\nu    T}\left(|\hat{x}|^2\, \left|\overline {G}(\hat{t}, \hat{x},u)\right|_{L_2(\Xi,H)}^2+|\hat{y}|^2\, \left|\overline {G}(\hat{t}, \hat{y},u)\right|_{L_2(\Xi,H)}^2\right)
\\
&
+6\left|\overline G(\hat{t}, \hat{x},u)\right|_{L_2(\Xi,H)}^2\sum_{i=0}^{\infty}\frac{1}{2^i}\left|e^{(\hat{t}-t_i)A}x^i- \hat{x}\right|^2\\
&
+6\left|\overline G(\hat{t}, \hat{y},u)\right|_{L_2(\Xi,H)}^2\sum_{i=0}^{\infty}\frac{1}{2^i}\left|e^{(\hat{t}-t_i)A}y^i- \hat{y}\right|^2
\\
\leq&
\, 3\beta{L^2}\, \left|\hat{x}-\hat{y}\right|^2+2\beta\left(\left|Q_N\overline {G}(\hat{t}, \hat{x},u)\right|_{L_2(\Xi,H)}^2+\left|Q_N\overline {G}(\hat{t}, \hat{y},u)\right|_{L_2(\Xi,H)}^2\right)\\
&+6\sum_{i=0}^{\infty}\frac{1}{2^i}\left[\left|e^{(\hat{t}-t_i)A}x^i-\hat{x}\right|^2
+\left|e^{(\hat{t}-t_i)A}y^i- \hat{y}\right|^2\right]\\
&\qquad \times
L^2(1+|\hat{x}|^2
+|\hat{y}|^2
)
+12\varepsilon \frac{\nu T-{\hat{t}}}{\nu T}L^2(1+| \hat{x}|^4
+|\hat{y}|^4);
\end{aligned}
\end{eqnarray}
from Assumption \ref{hypstate20241} (ii),  we have
\begin{eqnarray}\label{j12024}
\begin{aligned}
&\quad J_{2}:= {\left\langle \overline {F}(\hat{t}, \hat{x},u),\, \nabla_x\chi^{N}(\hat{t}, \hat{x})\right\rangle_{H}  -\left\langle \overline {F}(\hat{t}, \hat{y},u),\, -\nabla_x\hbar^{N}(\hat{t},\hat{y})\right\rangle_{H}}\\
\leq&\, 2\beta{L}\, \left| \hat{x}-\hat{y}\right|^2
+4L\sum_{i=0}^{\infty}\frac{1}{2^i}\left[\left|e^{(\hat{t}-t_i)A}x^i-\hat{x}\right|^3
+\left|e^{(\hat{t}-t_i)A}y^i-\hat{y}\right|^3\right]\\
&\qquad\times(1+|\hat{x}|+|\hat{y}|)
 +8\varepsilon \frac{\nu T-{\hat{t}}}{\nu T} L(1+|\hat{x}|^4+|\hat{y}|^4);
\end{aligned}
\end{eqnarray}
from Assumption \ref{hypstate20241} (iii),  we have
\begin{eqnarray}\label{j32024}
\begin{aligned}
&\quad
J_{3}:=\overline q{(}\hat{t}, \hat{x}, \overline W_2(\hat{t}, \hat{x}), \nabla_x\chi^{N}(\hat{t}, \hat{x})\overline G(\hat{t}, \hat{x},u),u{)}\\
&\qquad\qquad-
\overline q(\hat{t}, \hat{y}, \overline W_2(\hat{t}, \hat{y}), -\nabla_x\hbar^{N}(\hat{t}, \hat{y})\overline G(\hat{t}, \hat{y},u),u{)}\\
\leq&
L|\hat{x}-\hat{y}|
+2\beta L^2|\hat{x}
-\hat{y}|^2\\
&+4L^2\sum_{i=0}^{\infty}\frac{1}{2^i}\left[\left|e^{(\hat{t}-t_i\,)A}x^i-\hat{x}\right|^3
+\left|e^{(\hat{t}-t_i\,)A}y^i-\hat{y}\right|^3\right]\\
&\qquad\qquad\qquad\times
(1+|\hat{x}|+|\hat{y}|)+8\varepsilon \frac{\nu T-{\hat{t}}}{\nu T} L^2(1+|\hat{x}|^4
+|\hat{y}|^4
).
\end{aligned}
\end{eqnarray}
In view of  the property (i) of $(\hat{t}, \hat{x},\hat{y})$, we have
\begin{eqnarray}\label{1002d12024}
4\sum_{i=0}^{\infty}\frac{1}{2^i}(\hat{t}-t_i)
\leq4\sum_{i=0}^{\infty}\frac{1}{2^i}\bigg{(}\frac{1}{2^i\beta}\bigg{)}^{\frac{1}{2}}\leq 8{\bigg{(}\frac{1}{{\beta}}\bigg{)}}^{\frac{1}{2}},
\end{eqnarray}
\begin{eqnarray}\label{1002d2}
\begin{aligned}
&\quad
\sum_{i=0}^{\infty}\frac{1}{2^i}\left[\left|e^{(\hat{t}-t_i)A}x^i-\hat{x}\right|^3
+\left|e^{(\hat{t}-t_i)A}y^i-\hat{y}\right|^3\right]\\
&\leq
2\sum_{i=0}^{\infty}\frac{1}{2^i}\bigg{(}\frac{1}{2^i\beta}\bigg{)}^{\frac{3}{4}}\leq 4{\bigg{(}\frac{1}{{\beta}}\bigg{)}}^{\frac{3}{4}},
\end{aligned}
\end{eqnarray}
and
\begin{eqnarray}\label{1002d3}
\begin{aligned}
&\quad
\sum_{i=0}^{\infty}\frac{1}{2^i}\left[\left|e^{(\hat{t}-t_i)A}x^i-\hat{x}\right|^2
+\left|e^{(\hat{t}-t_i)A}y^i-\hat{y}\right|^2\right]\\
&\leq
2\sum_{i=0}^{\infty}\frac{1}{2^i}\bigg{(}\frac{1}{2^i\beta}\bigg{)}^{\frac{1}{2}}\leq 4{\bigg{(}\frac{1}{{\beta}}\bigg{)}}^{\frac{1}{2}};
\end{aligned}
\end{eqnarray}
and since $\hat{x}$ and $\hat{y}$ are independent of $N$, by Assumption  \ref{hypstate56662024},
\begin{eqnarray}\label{1002d4}
\qquad\qquad \sup_{u\in U}\left[\left|Q_N\overline {G}(\hat{t}, \hat{x},u)\right|_{L_2(\Xi,H)}^2+\left|Q_N\overline {G}(\hat{t}, \hat{y},u)\right|_{L_2(\Xi,H)}^2\right]\rightarrow0\ \mbox{as}\ N\rightarrow \infty.
\end{eqnarray}
Combining (\ref{vis1122024})-(\ref{j32024}),  and letting $N\rightarrow\infty$ and then 
by (\ref{5.10jiajiaaaa2024}), we have for sufficiently large $\beta>0$,  
\begin{eqnarray}\label{vis1222024}
\qquad\qquad
c
\leq
-\frac{\varepsilon}{\nu T}(|\hat{x}|^4
+|\hat{y}|^4)+ \varepsilon \frac{\nu T-{\hat{t}}}{\nu T} (20L+8)L(1+|\hat{x}|^4
+|\hat{y}|^4
)+\frac{c}{4}.
\end{eqnarray}
Since $$
\nu=1+\frac{1}{2T(20L+8)L}
\quad \text{and} \quad\bar{a}=\frac{1}{2(20L+8)L}\wedge{T},$$
we see by  (\ref{5.32024}) the following contradiction:
\begin{eqnarray*}\label{vis1222024}
	c\leq
	\frac{\varepsilon}{\nu
		T}+\frac{c}{4}\leq \frac{c}{2}.
\end{eqnarray*}
The proof is  complete. \ \ $\Box$

\newpage
\section{PHJB equations with a quadratic growth in the gradient}
\par
In this chapter, we study the second-order PHJB equation
\begin{eqnarray}\label{hjb7}
\begin{cases}
\partial_tV(\gamma_t)+\langle A^*\partial_xV(\gamma_t), \gamma_t(t)\rangle _H\\
\qquad\qquad+{\mathbf{H}}(\gamma_t,\partial_xV(\gamma_t),\partial_{xx}V(\gamma_t))= 0,\   (t,\gamma_t)\in
[0,T)\times {\Lambda};\\
V(\gamma_T)=\phi(\gamma_T), \ \gamma_T\in {\Lambda}_T,
\end{cases}
\end{eqnarray}
with
\begin{eqnarray*}
	{\mathbf{H}}(\gamma_t,p,l)&:=&\sup_{u\in{
			{U}}}\left\{
	\langle p, F(\gamma_t,u)\rangle_H+\frac{1}{2}\mbox{Tr}[ l (GG^*)(\gamma_t,u)] +q(\gamma_t,u)\right\}, \\
	&&\qquad (t,\gamma_t,p,l)\in [0,T]\times{\Lambda}\times H\times {\mathcal{S}}(H).
\end{eqnarray*}
Here $(U,|\cdot|_U)$ is a normed space.   We make the following assumption.
\begin{assumption}\label{hypstate0928a}
	\begin{description}
		\item[(i)]
		The operator $A$ is the generator of a $C_0$ contraction
		{semi-group}  of bounded linear operators  $\{e^{tA}, t\geq0\}$ in
		Hilbert space $H$.
		\par
		\item[(ii)] $F:{\Lambda}\times U\rightarrow H$ and $G:{\Lambda}\times U\rightarrow L_2(\Xi,H)$ are continuous, {and
        $F,G$ are continuous in $\gamma_t\in \Lambda$, uniformly in $u\in U$.} Moreover,
		there exists a constant $L>0$ such that,  for all $(t, \gamma_t, \eta_t, u) \in [0,T]\times {\Lambda}\times {\Lambda}\times U$,
		\begin{eqnarray}\label{assume11111002}
		\begin{aligned}
		&|F(\gamma_t,u)|\leq L(1+||\gamma_t||_0+|u|),\\
		&|F(\gamma_t,u)-F(\eta_t,u)|\leq
		L(1+|u|)||\gamma_t-\eta_t||_0,\\
		&|G(\gamma_t,u)|^2_{L_2(\Xi,H)}\leq
		L^2(1+||\gamma_t||^2_0),\\
		& |G(\gamma_t,u)-G(\eta_t,u)|_{L_2(\Xi,H)}\leq
		L||\gamma_t-\eta_t||_0.
		\end{aligned}
		\end{eqnarray}\par
		\item[(iii)]
		$
		q: {\Lambda}\times U\rightarrow \mathbb{R}$ and $\phi: {\Lambda}_T\rightarrow \mathbb{R}$ are continuous, {and
        $q$ is continuous in $\gamma_t\in \Lambda$, uniformly in $u\in U$.} Moreover,     there are constants $L>0$ and
		$\nu_1>0$
		such that, for all $(t, \gamma_t, \eta_T, \gamma'_t, \eta'_T,  u)
		\in [0,T]\times(\Lambda\times {\Lambda}_T)^2\times U$,
		\begin{eqnarray*}
			&&-L(1+||\gamma_t||^2_0+|u|^2)\leq q(\gamma_t,u)\leq -\frac{\nu_1}{2}|u|^2+q_0(\gamma_t, u) \\
			&& \mbox{with}\ q_0(\gamma_t, u)\leq L(1+||\gamma_t||_0^2),
			\\
			&&|q(\gamma_t,u)-q(\gamma'_{t},u)|\leq L(1+|u|^2)(1+||\gamma_t||_0+||\gamma'_{t}||_0)||\gamma_t-\gamma'_{t}||_0,
			\\
			&& |\phi(\eta_T)-\phi(\eta'_T)|\leq L(1+||\eta_T||_0+\eta'_{T}||_0)||\eta_T-\eta'_{T}||_0,\\
			&& |\phi(\eta_T)|\leq L(1+||\eta_T||^2_0).
		\end{eqnarray*}
	\end{description}
\end{assumption}

\par
\begin{theorem}\label{theoremhjbm1002}
	Let {Assumptions} \ref{hypstate5666} and \ref{hypstate0928a}  be satisfied.
	Let $W_1\in C^0({\Lambda})$ and   $ W_2\in C^0({\Lambda})$ be  viscosity sub- and  super-solutions
	to equation (\ref{hjb7}),  and satisfy (\ref{w}) and (\ref{w1}).
	Then  $W_1\leq W_2$.
\end{theorem}

\begin{proof}
It is sufficient to prove
$W_1\leq W_2$ under the stronger  assumption that $W_1$ is a viscosity sub-solution of the PHJB equation (\ref{1002a}). The proof consists of the following two steps.
\par
$Step\  1.$  We first assume that $q_0\leq 0$ and $\phi\leq 0$ in  Assumption \ref{hypstate0928a} (iii).
For $0<\mu<1$, let $W_2^{\mu}=\mu W_2$. It is not difficult to see that $W_2^{\mu}$ is a viscosity super-solution of
\begin{eqnarray}\label{hjb702}
\quad\quad  \begin{cases}
\partial_tW_2^{\mu}(\gamma_t)+\langle A^*\partial_xW_2^{\mu}(\gamma_t), \gamma_t(t)\rangle_H+{\mathbf{H}}^{\mu}(\gamma_t,\partial_xW_2^{\mu}(\gamma_t),\partial_{xx}W_2^{\mu}(\gamma_t))\\[2mm]
\qquad\qquad\quad
= 0,\qquad   (t,\gamma_t)\in
[0,T)\times {\Lambda};\\[3mm]
W_2^{\mu}(\gamma_T)=\mu\phi(\gamma_T), \ \gamma_T\in {\Lambda}_T
\end{cases}
\end{eqnarray}
with
\begin{eqnarray*}
	{\mathbf{H}}^{\mu}(\gamma_t,p,l)&:=&\sup_{u\in{
			{U}}}\left\{
	\langle p, F(\gamma_t,u)\rangle_H+\frac{1}{2}\mbox{Tr}[ l (GG^*)(\gamma_t,u)] +\mu q(\gamma_t,u)\right\}, \\
	&&\qquad\qquad\qquad (t,\gamma_t,p,l)\in [0,T]\times{\Lambda}\times H\times {\mathcal{S}}(H).
\end{eqnarray*}
We only need to prove that $W_1(\gamma_t)\leq W_2^{\mu}(\gamma_t)$ for all $\mu\in (0,1)$ and $(t,\gamma_t)\in
[T-\bar{a},T)\times
{\Lambda}$ with
$$\bar{a}:=\frac{1}{4(306L+36)L}\wedge{T}\wedge \frac{\nu_1(1-\mu)}{4^33^5L^2(4L+1)}.$$
Then, by a backward iteration over the intervals
$[T-i\bar{a},T-(i-1)\bar{a})$, we have $W_1\leq W_2^{\mu}$ on $[0,T]\times \Lambda$  for all $\mu\in (0, 1)$, and letting $\mu\rightarrow1$,  we have the desired comparison. 
Otherwise, there is  $(\tilde{t},\tilde{\gamma}_{\tilde{t}})\in (T-\bar{a},T)\times
{\Lambda}$  such that
$\tilde{m}:=W_1(\tilde{\gamma}_{\tilde{t}})-W_2^\mu(\tilde{\gamma}_{\tilde{t}})>0$.
\par
Let  $\varepsilon \in (0, 1)$ be  sufficiently  small  such that
$$
W_1(\tilde{\gamma}_{\tilde{t}})-W_2^\mu(\tilde{\gamma}_{\tilde{t}})-2\varepsilon \frac{\nu T-\tilde{t}}{\nu
	T}\Upsilon(\tilde{\gamma}_{\tilde{t}})
>\frac{\tilde{m}}{2}
$$
and
\begin{eqnarray}\label{5.307}
\varepsilon \frac{\nu T-T+\bar{a}}{\nu T} (306L+36)L\leq\frac{c}{4}
\end{eqnarray}
with
$$
\nu:=1+\frac{\bar{a}}{T}.
$$
Next,  we define for any  $(\gamma_t,\eta_t)\in {\Lambda}^{T-\bar{a}}\times{\Lambda}^{T-\bar{a}}$,
\begin{eqnarray*}
	\Psi(\gamma_t,\eta_t)&:=&W_1(\gamma_t)-W_2^\mu(\eta_t)-{\beta}\Upsilon(\gamma_{t},\eta_{t})-\beta^{\frac{1}{3}}|\gamma_{t}(t)-\eta_{t}(t)|^2\\
	&&-\varepsilon\frac{\nu T-t}{\nu
		T}(\Upsilon(\gamma_t)+\Upsilon(\eta_t)).
\end{eqnarray*}
Then
for every  $(\gamma^0_{t_0},\eta^0_{t_0})\in \Lambda^{\tilde{t}}\times \Lambda^{\tilde{t}}$ satisfying
$$
\Psi(\gamma^0_{t_0},\eta^0_{t_0})\geq \sup_{(\gamma_s,\eta_s)\in  \Lambda^{\tilde{t}}\times \Lambda^{\tilde{t}}}\Psi(\gamma_s,\eta_s)-\frac{1}{\beta},\
\    \mbox{and} \ \ \Psi(\gamma^0_{t_0},\eta^0_{t_0})\geq \Psi(\tilde{\gamma}_{\tilde{t}},\tilde{\gamma}_{\tilde{t}}) >\frac{\tilde{m}}{2},
$$
there exist $(\hat{t},\hat{\gamma}_{\hat{t}},\hat{\eta}_{\hat{t}})\in [\tilde{t},T]\times \Lambda^{\tilde{t}}\times \Lambda^{\tilde{t}}$ and a sequence $\{(t_i,\gamma^i_{t_i},\eta^i_{t_i})\}_{i\geq1}\subset
[\tilde{t},T]\times \Lambda^{\tilde{t}}\times \Lambda^{\tilde{t}}$ such that
\begin{description}
	\item[(i)] $\Upsilon(\gamma^0_{t_0},\hat{\gamma}_{\hat{t}})+\Upsilon(\eta^0_{t_0},\hat{\eta}_{\hat{t}})+|\hat{t}-t_0|^2\leq \frac{1}{\beta}$,
	$\Upsilon(\gamma^i_{t_i},\hat{\gamma}_{\hat{t}})+\Upsilon(\eta^i_{t_i},\hat{\eta}_{\hat{t}})+|\hat{t}-t_i|^2
	\leq \frac{1}{2^i\beta}$ and $t_i\uparrow \hat{t}$ as $i\rightarrow\infty$,
	\item[(ii)]  $\Psi_1(\hat{\gamma}_{\hat{t}},\hat{\eta}_{\hat{t}})
\geq \Psi(\gamma^0_{t_0},\eta^0_{t_0})$, and
	\item[(iii)]    for all $(s, \gamma_s,\eta_s)\in [\hat{t},T]\times \Lambda^{\hat{t}}\times \Lambda^{\hat{t}}\setminus \{(\hat{t},\hat{\gamma}_{\hat{t}},\hat{\eta}_{\hat{t}})\}$,
	{   \begin{eqnarray}\label{iii407}
		\Psi_1(\gamma_s,\eta_s)
		<\Psi_1(\hat{\gamma}_{\hat{t}},\hat{\eta}_{\hat{t}}),
		\end{eqnarray}}
\end{description}
where we have defined
{$$
	\Psi_1(\gamma_t,\eta_t):=  \Psi(\gamma_t,\eta_t)
	-\sum_{i=0}^{\infty}
	\frac{1}{2^i}\left[{\Upsilon}(\gamma^i_{t_i},\gamma_t)+{\Upsilon}(\eta^i_{t_i},\eta_t)+|{t}-t_i|^2\right], \  (\gamma_t,\eta_t)\in  \Lambda^{\tilde{t}}\times \Lambda^{\tilde{t}}.
	$$}
Note that the point
$({\hat{t}},\hat{{\gamma}}_{{\hat{t}}},\hat{{\eta}}_{{\hat{t}}})$ depends on  $\beta$,
$\varepsilon$ and $\mu$.
\par
Let \begin{eqnarray}\label{0928c1}
M_0:=\left(\frac{4\nu L}{\varepsilon(\nu-1)}+1\right)^{\frac{1}{4}}.
\end{eqnarray}
Notice that, by  (\ref{s0}) and (\ref{w}),
\begin{eqnarray*}
	\Psi(\gamma_t,\eta_t)&\leq& W_1(\gamma_t)-W_2^\mu(\eta_t)-\varepsilon\frac{\nu T-t}{\nu
		T}(\Upsilon(\gamma_t)+\Upsilon(\eta_t))\\
	&\leq&L(2+||\gamma_t||^2_0+||\eta_t||^2_0)-\varepsilon\frac{\nu-1}{\nu}(||\gamma_t||^6_0+||\eta_t||^6_0).
\end{eqnarray*}
Then, for all $\gamma_t, \eta_t\in \Lambda$ satisfying $||\gamma_t||_0\vee||\eta_t||_0\geq M_0$ we have
\begin{eqnarray*}
	\Psi(\gamma_t,\eta_t)
	&\leq&\left(-\varepsilon\frac{\nu-1}{\nu}M_0^4+2L\right)(||\gamma_t||^2_0\vee||\eta_t||^2_0)+2L\\
	&\leq&-2LM_0^2+2L\leq 0<\frac{\tilde{m}}{2}\leq \Psi(\hat{\gamma}_{\hat{t}},\hat{\eta}_{\hat{t}}).
\end{eqnarray*}
Therefore, we have  (\ref{5.10jiajiaaaa}). Moreover, by  Step 2 in  the proof procedure of Theorem \ref{theoremhjbm}, we have  (\ref{5.10}) holds true.
Now we show that there exists 
$N_0>0$
such that
$\hat{t}\in [\tilde{t},T)$
for all $\beta\geq N_0$.
By (\ref{5.10}), we can let $N_0>0$ be a large number such that
$$
L (1+2{M_0})||\hat{\gamma}_{\hat{t}}-\hat{\eta}_{\hat{t}}||_0
\leq
\frac{\tilde{m}}{4},
$$
for all $\beta\geq N_0$.
Then we have $\hat{t}\in [\tilde{t},T)$ for all $\beta\geq N_0$. Indeed, if say $\hat{t}=T$, noting that $\phi\leq0$, we will deduce the following contradiction:
\begin{eqnarray*}
	\frac{\tilde{m}}{2}\leq\Psi(\hat{\gamma}_{\hat{t}},\hat{\eta}_{\hat{t}})\leq \phi(\hat{\gamma}_{\hat{t}})-\mu\phi(\hat{\eta}_{\hat{t}})\leq
	\mu L (1+2{M_0})||\hat{\gamma}_{\hat{t}}-\hat{\eta}_{\hat{t}}||_0
	\leq
	\frac{\tilde{m}}{4}.
\end{eqnarray*}
Then by Steps 4 and  5 in  the proof procedure of Theorem \ref{theoremhjbm}, we obtain
\begin{eqnarray}\label{vis11207}
\begin{aligned}
& c+ \frac{\varepsilon}{\nu T}(\Upsilon(\hat{{\gamma}}_{{\hat{t}}})+\Upsilon(\hat{{\eta}}_{{\hat{t}}})
)-2\sum_{i=0}^{\infty}\frac{1}{2^i}(\hat{t}-t_i)\\
\leq&\, {\mathbf{H}}(\hat{{\gamma}}_{{\hat{t}}},\partial_x\chi^{N}(\hat{\gamma}_{\hat{t}}),\partial_{xx}\chi^{N}(\hat{\gamma}_{\hat{t}}))
-{\mathbf{H}}^{\mu}(\hat{{\eta}}_{{\hat{t}}},-\partial_x\hbar^{N}(\hat{\eta}_{{\hat{t}}}),-\partial_{xx}\hbar^{N}(\hat{\eta}_{\hat{t}})),
\end{aligned}
\end{eqnarray}
where $\partial_x\chi^{N}(\hat{\gamma}_{\hat{t}}),\partial_{xx}\chi^{N}(\hat{\gamma}_{\hat{t}}),
\partial_x\hbar^{N}(\hat{\eta}_{{\hat{t}}}),\partial_{xx}\hbar^{N}(\hat{\eta}_{\hat{t}})$ are given in (\ref{1002c})-(\ref{1002c3}).
On the  other hand,  via a simple calculation, we have
\begin{eqnarray}\label{v407}
\begin{aligned}
&{\mathbf{H}}(\hat{{\gamma}}_{{\hat{t}}},\partial_x\chi^{N}(\hat{\gamma}_{\hat{t}}),\partial_{xx}\chi^{N}(\hat{\gamma}_{\hat{t}}))
-{\mathbf{H}}^{\mu}(\hat{{\eta}}_{{\hat{t}}},-\partial_x\hbar^{N}(\hat{\eta}_{{\hat{t}}}),-\partial_{xx}\hbar^{N}(\hat{\eta}_{\hat{t}}))\\
\leq&\sup_{u\in U}(J_{1}+J_{2}+J_{3}).
\end{aligned}
\end{eqnarray}
Here,  $J_{1}$ satisfies (\ref{j2}); from Assumption \ref{hypstate0928a} (ii) and (\ref{03108}),  we have
\begin{eqnarray}\label{j107}
\quad&\begin{aligned}
&\qquad J_{2}:= {{\langle} {F}(\hat{{\gamma}}_{{\hat{t}}},u),\partial_x\chi^{N}(\hat{\gamma}_{\hat{t}}){\rangle}_{H}  -{\langle} {F}(\hat{{\eta}}_{{\hat{t}}},u),-\partial_x\hbar^{N}(\hat{\eta}_{{\hat{t}}}){\rangle}_{H}}\\
\leq&2\beta^{\frac{1}{3}}{L}(1+|u|)|\hat{{\gamma}}_{{\hat{t}}}({\hat{t}})-\hat{{\eta}}_{{\hat{t}}}({\hat{t}})|\times
||\hat{{\gamma}}_{{\hat{t}}}-\hat{{\eta}}_{{\hat{t}}}||_0\\
&+18\beta|\hat{{\gamma}}_{{\hat{t}}}({\hat{t}})-\hat{{\eta}}_{{\hat{t}}}({\hat{t}})|^5L(2+||\hat{{\gamma}}_{{\hat{t}}}||_0
+||\hat{{\eta}}_{{\hat{t}}}||_0+2|u|)
\\
&
+18L\sum_{i=0}^{\infty}\frac{1}{2^i}\left[|e^{(\hat{t}-t_i)A}\gamma^i_{t_i}(t_i)-\hat{\gamma}_{\hat{t}}(\hat{t}\,)|^5
+|e^{(\hat{t}-t_i)A}\eta^i_{t_i}(t_i)-\hat{\eta}_{\hat{t}}(\hat{t}\,)|^5\right]\\
&\qquad\qquad\times(1+||\hat{{\gamma}}_{{\hat{t}}}||_0+||\hat{{\eta}}_{{\hat{t}}}||_0+|u|)\\
& +36\varepsilon \frac{\nu T-{\hat{t}}}{\nu T} L(1+||\hat{{\gamma}}_{{\hat{t}}}||^6_0+||\hat{{\eta}}_{{\hat{t}}}||^6_0)
+18\varepsilon \frac{\nu T-{\hat{t}}}{\nu T} L(||\hat{{\gamma}}_{{\hat{t}}}||^5_0+||\hat{{\eta}}_{{\hat{t}}}||^5_0)|u|\\
\leq&\, 2\beta^{\frac{1}{3}}{L}||\hat{{\gamma}}_{{\hat{t}}}-\hat{{\eta}}_{{\hat{t}}}||^2_0+\frac{1}{12}\, {\nu_1(1-\mu)}|u|^2
+\frac{12}{\nu_1(1-\mu)}\beta^{\frac{2}{3}}{L^2}||\hat{{\gamma}}_{{\hat{t}}}-\hat{{\eta}}_{{\hat{t}}}||^4_0\\
&\qquad+18\beta L|\hat{{\gamma}}_{{\hat{t}}}({\hat{t}})-\hat{{\eta}}_{{\hat{t}}}({\hat{t}})|^5(2+||\hat{{\gamma}}_{{\hat{t}}}||_0
+||\hat{{\eta}}_{{\hat{t}}}||_0)
+\frac{1}{12}\, {\nu_1(1-\mu)}|u|^2\\
&+\frac{12}{\nu_1(1-\mu)}(18L\beta|\hat{{\gamma}}_{{\hat{t}}}({\hat{t}})-\hat{{\eta}}_{{\hat{t}}}({\hat{t}})|^5)^2
\\
&
+18L(1+||\hat{{\gamma}}_{{\hat{t}}}||_0+||\hat{{\eta}}_{{\hat{t}}}||_0)\sum_{i=0}^{\infty}\frac{1}{2^i}\left[|((\gamma^i)^A_{t_i, \hat t}-\hat{\gamma}_{\hat{t}})(\hat{t}\,)|^5
+|((\eta^i)^A_{t_i, \hat t}-\hat{\eta}_{\hat{t}})(\hat{t}\,)|^5\right]\\
& +\frac{12}{\nu_1(1-\mu)}\bigg{(}
9L\sum_{i=0}^{\infty}\frac{1}{2^i}\bigg{[}|((\gamma^i)^A_{t_i, \hat t}-\hat{\gamma}_{\hat{t}})(\hat{t}\,)|^5
+|((\eta^i)^A_{t_i, \hat t}-\hat{\eta}_{\hat{t}})(\hat{t}\,)|^5\bigg{]}
\bigg{)}^2\\
& +\frac{1}{12}\, {\nu_1(1-\mu)}|u|^2  +36\varepsilon \frac{\nu T-{\hat{t}}}{\nu T} L(1+||\hat{{\gamma}}_{{\hat{t}}}||^6_0+||\hat{{\eta}}_{{\hat{t}}}||^6_0)\\
&+\frac{1}{6}\, {\nu_1(1-\mu)}|u|^2
+\frac{12}{\nu_1(1-\mu)}\left(9\varepsilon \frac{\nu T-{\hat{t}}}{\nu T} L\right)^2(||\hat{{\gamma}}_{{\hat{t}}}||^{10}_0+||\hat{{\eta}}_{{\hat{t}}}||^{10}_0);
\end{aligned}
\end{eqnarray}
from Assumption \ref{hypstate0928a} (iii), noting that $q_0\leq 0$, we have
\begin{eqnarray}\label{j307}
\begin{aligned}
J_{3}:=&q{(}\hat{{\gamma}}_{{\hat{t}}},u{)}-
\mu q{(}\hat{{\eta}}_{{\hat{t}}}, u{)}\\
\leq& L(1+|u|^2)(1+2M_0)||\hat{{\gamma}}_{{\hat{t}}}-\hat{{\eta}}_{{\hat{t}}}||_0-\frac{1}{2}\, {\nu_1(1-\mu)}|u|^2.
\end{aligned}
\end{eqnarray}
Combining  (\ref{j2}) and (\ref{vis11207})-(\ref{j307}),  and letting $N\rightarrow\infty$,
we have from (\ref{5.10jiajiaaaa}), (\ref{5.10}) and (\ref{1002d1})-(\ref{1002d4}) that for sufficiently  large $\beta>0$,
\begin{eqnarray}\label{vis12207}
\begin{aligned}
c
\, \leq&\,
-\frac{\varepsilon}{\nu T}\left[\Upsilon(\hat{{\gamma}}_{{\hat{t}}})
+\Upsilon(\hat{{\eta}}_{{\hat{t}}})\right]\\
&+ \varepsilon \frac{\nu T-{\hat{t}}}{\nu T} (306L+36)L(1+||\hat{{\gamma}}_{{\hat{t}}}||_0^6
+||\hat{{\eta}}_{{\hat{t}}}||_0^6
)\\
&+\frac{12}{\nu_1(1-\mu)}\left(9\varepsilon \frac{\nu T-{\hat{t}}}{\nu T} L\right)^2(||\hat{{\gamma}}_{{\hat{t}}}||^{10}_0+||\hat{{\eta}}_{{\hat{t}}}||^{10}_0)+\frac{c}{4}.
\end{aligned}
\end{eqnarray}
Since
$$
\nu\leq1+\frac{1}{4T(306L+36)L}
\quad \text{ and }\quad  \bar{a}\leq \frac{1}{4(306L+36)L}, $$
 we have
\begin{eqnarray*}
	&&2(\nu T-T+\bar{a})(306L+36)L\\
	&\leq& 2\left(\frac{1}{4T(306L+36)L}T+\frac{1}{4(306L+36)L}\right)(306L+36)L=1,
\end{eqnarray*}
and then from  (\ref{s0}) and (\ref{5.307}),
\begin{eqnarray}\label{0928a107}
\qquad
\begin{aligned}
&-\frac{\varepsilon}{2\nu T}(\Upsilon(\hat{{\gamma}}_{{\hat{t}}})
+\Upsilon(\hat{{\eta}}_{{\hat{t}}}))+ \varepsilon \frac{\nu T-{\hat{t}}}{\nu T} (306L+36)L(1+||\hat{{\gamma}}_{{\hat{t}}}||_0^6
+||\hat{{\eta}}_{{\hat{t}}}||_0^6
)   \\
\qquad\leq&  \frac{\varepsilon}{2\nu T}\bigg{(}-||\hat{{\gamma}}_{{\hat{t}}}||_0^6
-||\hat{{\eta}}_{{\hat{t}}}||_0^6\\
&\qquad +2(\nu T-T+\bar{a})(306L+36)L(1+||\hat{{\gamma}}_{{\hat{t}}}||_0^6
+||\hat{{\eta}}_{{\hat{t}}}||_0^6
)\bigg{)}\leq \frac{c}{4}.
\end{aligned}
\end{eqnarray}
Since
$$\nu=1+\frac{\bar{a}}{T} \quad \text{and}\quad \nu\leq 1+\frac{\nu_1(1-\mu)}{4^33^5L^2T(4L+1)},$$
we have
$$
\left(\frac{\nu T-{\hat{t}}}{\nu T}\right)^2\leq \frac{(\nu T-T+\bar{a})^2}{\nu^2 T^2} =  \frac{4(\nu-1)^2 }{\nu^2}
$$
and
$$
\frac{4^33^5(\nu-1)L^2T(4L+1)}{\nu_1(1-\mu)}\leq  1.
$$
Then by (\ref{s0}) and (\ref{0928c1}), we have
\begin{eqnarray}\label{0928b107}
\quad\quad
\begin{aligned}
&\quad-\frac{\varepsilon}{2\nu T}(\Upsilon(\hat{{\gamma}}_{{\hat{t}}})
+\Upsilon(\hat{{\eta}}_{{\hat{t}}}))+\frac{12}{\nu_1(1-\mu)}\left(9\varepsilon \frac{\nu T-{\hat{t}}}{\nu T} L\right)^2(||\hat{{\gamma}}_{{\hat{t}}}||^{10}_0+||\hat{{\eta}}_{{\hat{t}}}||^{10}_0)\\
&\leq\frac{\varepsilon}{2\nu T}||\hat{{\gamma}}_{{\hat{t}}}||_0^6
\left(-1+\frac{12}{\nu_1(1-\mu)} \frac{9^2\times 8\varepsilon(\nu-1)^2T L^2 }{\nu } ||\hat{{\gamma}}_{{\hat{t}}}||^{4}_0\right)\\
&\quad+\frac{\varepsilon}{2\nu T}||\hat{{\eta}}_{{\hat{t}}}||_0^6
\left(-1+\frac{12}{\nu_1(1-\mu)}\frac{9^2\times 8\varepsilon (\nu-1)^2T L^2 }{\nu } ||\hat{{\eta}}_{{\hat{t}}}||^{4}_0\right)\\
&\leq\frac{\varepsilon}{2\nu T}(||\hat{{\gamma}}_{{\hat{t}}}||_0^6+||\hat{{\eta}}_{{\hat{t}}}||_0^6)\\
&\qquad\times\left(-1+\frac{12}{\nu_1(1-\mu)} \frac{9^2\times 8\varepsilon(\nu-1)^2T L^2 }{\nu } \left(\frac{4\nu L}{\varepsilon(\nu-1)}+1\right)\right)\\
&\leq\frac{\varepsilon}{2\nu T}(||\hat{{\gamma}}_{{\hat{t}}}||_0^6+||\hat{{\eta}}_{{\hat{t}}}||_0^6)
\left(-1+\frac{4^3\times 3^5(\nu-1)L^2T(4L+1)}{\nu_1(1-\mu)}\right)\leq 0.
\end{aligned}
\end{eqnarray}
Combining (\ref{vis12207})-(\ref{0928b107}), we arrive at the following contradiction:
\begin{eqnarray*}
	c\leq
	\frac{c}{4}+\frac{c}{4}=\frac{c}{2}.
\end{eqnarray*}
\par
$Step\  2.$  \emph{The general case.} The idea is to reduce the general case to the first case by a suitable
change of functional (see Ishii \cite{ishii}). Suppose that $w$ is a viscosity  solution of (\ref{hjb7}). Then, a
straightforward computation shows that $\bar{w}(\gamma_t)=w(\gamma_t)-C(1+\Upsilon^1(\gamma_t))e^{\bar{\rho}(T-t)}$ for $C\geq2L$ and $\bar{\rho}>0$ is a viscosity solution of
\begin{eqnarray}\label{hjb81}
\begin{cases}
\partial_t\bar{w}(\gamma_t)+\langle A^*\partial_x\bar{w}(\gamma_t), \gamma_t(t)\rangle _H+\bar{\mathbf{H}}(\gamma_t,\partial_x\bar{w}(\gamma_t),\partial_{xx}\bar{w}(\gamma_t))\\
\qquad\quad = 0,
\quad (t,\gamma_t)\in
[0,T)\times {\Lambda};\\
\bar{w}(\gamma_T)=\phi(\gamma_T)-C(1+\Upsilon^1(\gamma_T)), \quad \gamma_T\in {\Lambda}_T,
\end{cases}
\end{eqnarray}
with
\begin{eqnarray*}
	\bar{\mathbf{H}}(\gamma_t,p,l)&:=&\sup_{u\in{
			{U}}}\left\{
	\langle p, F(\gamma_t,u)\rangle_H+\frac{1}{2}\mbox{Tr}[ l (GG^*)(\gamma_t,u)] +\bar{q}(\gamma_t,u)\right\}, \\[3mm]
	&&\qquad (t,\gamma_t,p,l)\in [0,T]\times{\Lambda}\times H\times {\mathcal{S}}(H)
\end{eqnarray*}
and
\begin{eqnarray*}
	\bar{q}(\gamma_t,u)&=&{q}(\gamma_t,u)-C\bar{\rho} e^{\bar{\rho}(T-t)}(1+\Upsilon^1(\gamma_t))+ Ce^{\bar{\rho}(T-t)}\langle\partial_x\Upsilon^1(\gamma_t), F(\gamma_t,u)\rangle_H\\ [3mm]
	&&+\frac{1}{2}Ce^{\bar{\rho}(T-t)}\mbox{Tr}[ \partial_{xx}\Upsilon^1(\gamma_t) (GG^*)(\gamma_t,u)].
\end{eqnarray*}
Note that $\bar{q}$ still satisfies  Assumption \ref{hypstate0928a} (iii).
Moreover,  from Assumption~\ref{hypstate0928a} (iii),  we can choose $C\geq 2L$ so that $\bar{\phi}(\gamma_T)=\phi(\gamma_T)-C(1+\Upsilon^1(\gamma_T))\leq0$.
\par
Next we show that there is $\bar{\rho}>0$ such that $\bar{q}(\gamma_t,u)\leq -\frac{\bar{\nu}_1|u|^2}{2}$ for all $(t,\gamma_t,u)\in [T-\frac{1}{\bar{\rho}},T]\times\Lambda\times U$. In fact, for all $(t,\gamma_t,u)\in [T-\frac{1}{\bar{\rho}},T]\times\Lambda\times U$, by  Assumption \ref{hypstate0928a} (ii) and (iii),
\begin{eqnarray*}
	&& \bar{q}(\gamma_t,u)\\
	&\leq&-\frac{\nu_1}{2}|u|^2+L(1+||\gamma_t||_0^2)+6Ce^{\bar{\rho}(T-t)}L(1+||\gamma_t||_0+|u|)|\gamma_t(t)|\\
	&& + 15Ce^{\bar{\rho}(T-t)}L^2(1+||\gamma_t||^2_0)-\frac{1}{2}C\bar{\rho} e^{\bar{\rho}(T-t)}(1+||\gamma_t||_0^2)\\
	&\leq&-\frac{\nu_1}{2}|u|^2+(L+12CL+15CL^2)e^{\bar{\rho}(T-t)}(1+||\gamma_t||^2_0)+6Ce^{\bar{\rho}(T-t)}L|u||\gamma_t(t)|\\
	&& -\frac{1}{2}C\bar{\rho} e^{\bar{\rho}(T-t)}(1+||\gamma_t||_0^2).
\end{eqnarray*}
Since
$$
6Ce^{\bar{\rho}(T-t)}L|u||\gamma_t(t)|\leq \frac{\nu_1}{4}|u|^2+\frac{36C^2L^2e^{2\bar{\rho}(T-t)}}{\nu_1}||\gamma_t||^2_0,
$$
then for
$$
\bar{\rho}>2\frac{L}{C}+24L+30L^2+\frac{72CL^2e}{\nu_1},
$$
we have
$$
\bar{q}(\gamma_t,u)\leq -\frac{\nu_1}{4}|u|^2, \ \mbox{for all} \  (t,\gamma_t,u)\in [T-\frac{1}{\bar{\rho}},T]\times\Lambda\times U,
$$
which is the desired estimate with $\bar{\nu}_1=\frac{\nu_1}{2}$ and $q_0=0$.
\par
Finally, set
$$\bar{W}_1(\gamma_t)=W_1-C(1+\Upsilon^1(\gamma_t))e^{\bar{\rho}(T-t)}\quad \text{and}\quad  \bar{W}_2(\gamma_t)=W_2-C(1+\Upsilon^1(\gamma_t))e^{\bar{\rho}(T-t)}.$$
 From Step 1,  we get $\bar{W}_1\leq \bar{W}_2$ in
$[T-{\frac{1}{\bar{\rho}}},T]\times\Lambda$, and thus ${W}_1\leq {W}_2$ in
$[T-\frac{1}{\bar{\rho}},T]\times\Lambda$. Then by a step-by-step argument,  we
obtain the comparison in $[0,T]\times\Lambda$.
The proof is  complete.
\end{proof}

We now consider an unbounded stochastic control problem for  controlled FSEE (\ref{state1}). The payoff to be maximized is
\begin{eqnarray}\label{cost11024}
\qquad\qquad    J(\gamma_t,u(\cdot))=\mathbb{E}\int_{t}^{T}q(X^{\gamma_t,u}_s,u(s))\mbox{d}s+\mathbb{E}\phi(X^{\gamma_t,u}_T), \  (t,\gamma_t)\in [0,T]\times {\Lambda},
\end{eqnarray}
where $u(\cdot)\in {\mathcal{U}}[t,T]$, the set of $U$-valued $\{\mathcal{F}\}_{t\leq s\leq T}$-progressively measurable controls such that
$$
\mathbb{E}\int_{t}^{T}|u(s)|^2\mbox{d}s<\infty.
$$
The value functional is defined by
\begin{eqnarray}\label{value1024}
V(\gamma_t)=\sup_{u(\cdot)\in {\mathcal{U}}[t,T]} J(\gamma_t,u(\cdot)),\quad  (t,\gamma_t)\in [0,T]\times {\Lambda}.
\end{eqnarray}
We make the following assumptions on the data.
\begin{assumption}\label{hypstate0928a1024}
	There exist  constants $L>0$ and $\nu_1>0$ such that,  for all $(t, \gamma_t, \eta_t, \zeta_T,u,v) \in [0,T]\times {\Lambda}\times {\Lambda}\times {\Lambda_T}\times U\times U$,
	\begin{eqnarray}\label{assume111110021024}
	q(\gamma_t,u)\leq -\frac{\nu_1}{2}|u|^2+L,\quad
	\phi(\zeta_T)\leq L,
	\end{eqnarray}
	\begin{eqnarray*}
		|F(\gamma_t,u)-F(\eta_t,v)|&\leq&
		L\,(||\gamma_t-\eta_t||_0+|u-v|), \\ [3mm]
		|q(\gamma_t,u)-q(\eta_{t},u)|&\leq& L\, (1+||\gamma_t||_0+||\eta_{t}||_0)\, ||\gamma_t-\eta_{t}||_0.
	\end{eqnarray*}
\end{assumption}

\begin{lemma}\label{lemmaexist04091014}
 Let Assumptions \ref{hypstate0928a} and \ref{hypstate0928a1024} be satisfied. Then for every  $u(\cdot)\in {\mathcal{U}}[t,T]$ and
	$\xi_t\in L^{2}_{\mathcal{P}}(\Omega, \Lambda_t(H))$,  equation (\ref{state1}) has a
	unique mild solution $X^{\xi_t,u}\in L_{\mathcal{P}}^2(\Omega, \Lambda_T(H))$ and   there is a constant $C_2>0$ depending only on   $T$ and $L$ such that
	\begin{eqnarray}\label{state1est101026}
	\mathbb{E}||X_T^{\xi_t,u}||^2_0\leq C_2\left(1+\mathbb{E}||\xi_t||^2_0+\int^{T}_{{t}}\mathbb{E}|u(\sigma)|^2d\sigma\right)
	\end{eqnarray}
	and for  $0\leq t\leq s\leq T$ and $\gamma_t,\eta_t\in{\Lambda_t}$,
	\begin{eqnarray}\label{060410221024}
	\qquad\qquad&
	\begin{aligned}
	&\quad
	\mathbb{E}||X^{\gamma_t,u}_T-X^{\eta_{t,s}^A,u}_T||^{2}_0\\
	&\leq C_2\, ||\gamma_t-\eta_t||^{2}_0+C_2\, (t-s)^{\frac{1}{2}}\left(1+||\gamma_t||^2_0+||\eta_t||^2_0+\int^{T}_{{t}}\mathbb{E}|u(\sigma)|^2d\sigma\right).
	\end{aligned}
	\end{eqnarray}
\end{lemma}

\begin{proof}  Let us define, for $n\geq 1$,
\begin{eqnarray*}
	\xi_t^n=\begin{cases}
		\xi_t\ \ \mbox{if}\ ||\xi_t||_0\leq n; \\
		0 \  \ \ \mbox{if}\ ||\xi_t||_0>n,
	\end{cases}
\end{eqnarray*}
and for every  $u(\cdot)\in {\mathcal{U}}[t,T]$,
\begin{eqnarray*}
	u^n(s)=\begin{cases}
		u(s)\ \ \mbox{if}\ |u(s)|\leq n; \\
		0 \  \ \ \ \ \ \mbox{if}\ |u(s)|>n,
	\end{cases}
 \ \  s\in[t,T].
\end{eqnarray*}
From Lemma \ref{lemmaexist0409}, we can denote by $X^{n,\xi_t,u}(\cdot)$ the  solution of the following equation:
\begin{eqnarray}\label{202210261}
\qquad& \begin{aligned}
X^{n,\xi_t,u}(s)=&e^{(s-t)A}\xi^n_t(t)+\int_{t}^{s}e^{(s-\sigma)A}F\left(X^{n,\xi_t,u}_\sigma,u^n(\sigma)\right)d\sigma\\
&+\int_{t}^{s}e^{(s-\sigma)A}G\left(X^{n,\xi_t,u}_\sigma,u(\sigma)\right)dW(\sigma),
\end{aligned}
\end{eqnarray}
where $X^{n,\xi_t,u}_t=\xi_t^n$.
 For $m,n\geq 1$, let
 $$
{\Delta}X^{m,n} (s):=X^{m,\xi_t,u}(s)-X^{n,\xi_t,u}(s),\ \ s\in [0,T],
$$
$$\Delta F^{m,n}(s)=F(X^{m,\xi_t,u}_s,u^m(s))-F(X^{n,\xi_t,u}_s,u^n(s)),\ \ s\in [0,T],$$
$$\Delta G^{m,n}(s)=G(X^{m,\xi_t,u}_s,u(s))-G(X^{n,\xi_t,u}_s,u(s)),\ \ s\in [0,T].$$
For every $\varepsilon>0$, applying (\ref{jias510815jia1102}) to $\Delta X^{m,n}$, we get
\begin{eqnarray*}
	\Upsilon^{\varepsilon}(\Delta X^{m,n}_{l})
	&\leq& \Upsilon^{\varepsilon}(\Delta X^{m,n}_t)+\int^{{l}}_{{t}}[\langle \partial_x\Upsilon^{\varepsilon}(\Delta X^{m,n}_\sigma), \Delta F^{m,n}(\sigma)\rangle_H\nonumber\\
	&&+\frac{1}{2}\mbox{Tr}(\partial_{xx}\Upsilon^{\varepsilon}(\Delta X^{m,n}_\sigma)(\Delta G^{m,n}
	(\Delta G^{m,n})^*)(\sigma))]d\sigma\nonumber\\
	&&
	+\int^{{l}}_{t}\langle \partial_x\Upsilon^{\varepsilon}(\Delta X^{m,n}_\sigma), \Delta G^{m,n}(\sigma)dW(\sigma)\rangle_H,\ l\in [s,T].
\end{eqnarray*}
Taking expectation on both sides of the last inequality, we have from Lemma \ref{theoremS1}, Assumption \ref{hypstate0928a} (ii) and  Assumption \ref{hypstate0928a1024},
\begin{eqnarray*}
	&&\mathbb{E}||\Delta X^{m,n}_{l}||^2_0\leq\mathbb{E}\Upsilon^{\varepsilon}(\Delta X^{m,n}_{l})+\varepsilon\\
	&\leq& 3\mathbb{E}||\xi_t^m-\xi_t^n||_0^2+\mathbb{E}\int^{l}_{t}[6L||\Delta X^{m,n}_\sigma||_0(||\Delta X^{m,n}_\sigma||_0+|u^m(\sigma)-u^n(\sigma)|)\\
&&\qquad\qquad\qquad\qquad\qquad+15L^2||\Delta X^{m,n}_\sigma||_0^2]d\sigma+\varepsilon.
\end{eqnarray*}
Using  Gronwall's inequality and letting $\varepsilon\rightarrow0$, there exists a  constant $C$ depending only on   $T$ and $L$ such that
$$
\mathbb{E}||\Delta X^{m,n}_{T}||^2_0\leq C\left(\mathbb{E}||\xi_t^m-\xi_t^n||_0^2+\int^{T}_{t}\mathbb{E}|u^m(s)-u^n(s)|^2ds\right).
$$
Thus, $\{X^{n,\xi_t,u}_T\}_{n\geq1}$ is a Cauchy sequence in $L^2_{\mathcal{P}}(\Omega,\Lambda_T)$.
It is now easy to see that the process
$$
X(s)=\lim_{n\rightarrow \infty}X^{n,\xi_t,u}(s), \ \ s\in [0,T]
$$
is $P$-a.s. well defined and satisfies the equation (\ref{state1}). Now we prove (\ref{state1est101026}). Applying Lemma \ref{theoremito2} to $X^{n,\xi_t,u}$, we get
\begin{eqnarray*}
	&&\Upsilon^1(X^{n,\xi_t,u}_{s})\\
	&\leq& \Upsilon^1(X^{n,\xi_t,u}_t)+\int^{{s}}_{{t}}[\langle \partial_x\Upsilon^1(X^{n,\xi_t,u}_\sigma), F(X^{n,\xi_t,u}_\sigma,u^n(\sigma))\rangle_H\nonumber\\
	&&+\frac{1}{2}\mbox{Tr}(\partial_{xx}\Upsilon^1(X^{n,\xi_t,u}_\sigma)(GG^*)(X^{n,\xi_t,u}_\sigma,u(\sigma)))]d\sigma\nonumber\\
	&&
	+\int^{{s}}_{t}\langle \partial_x\Upsilon^1(X^{n,\xi_t,u}_\sigma), G(X^{n,\xi_t,u}_\sigma,u(\sigma))dW(\sigma)\rangle_H;
\end{eqnarray*}
Taking expectation on both sides of the last inequality, we have from Lemma \ref{theoremS1} and Assumption \ref{hypstate0928a} (ii),
\begin{eqnarray*}
	&&\mathbb{E}||X^{n,\xi_t,u}_{s}||^2_0\leq\mathbb{E}\Upsilon^1(X^{n,\xi_t,u}_{s})+1\\
	&\leq& 3\mathbb{E}||\xi_t||_0^2+\int^{{s}}_{{t}}\mathbb{E}[6L||X^{n,\xi_t,u}_{\sigma}||_0(1+||X^{n,\xi_t,u}_{\sigma}||_0+|u(\sigma)|)\\
	&&+15L^2(1+||X^{n,\xi_t,u}_{\sigma}||_0^2)]d\sigma+1\\
	&\leq& 3\mathbb{E}||\xi_t||_0^2+1+(6L+15L^2)T+3L\int^{T}_{{t}}\mathbb{E}|u(\sigma)|^2d\sigma\\
	&&+15L(1+L)\int^{{s}}_{{t}}\mathbb{E}||X^{n,\xi_t,u}_{\sigma}||^2_0d\sigma.
\end{eqnarray*}
Using  Gronwall's inequality, there exists a  constant $C_2$ depending only on   $T$ and $L$ such that
\begin{eqnarray}\label{1026b1}
\mathbb{E}||X^{n,\xi_t,u}_{T}||^2_0\leq C_2\left(1+\mathbb{E}||\xi_t||^2_0+\int^{T}_{{t}}\mathbb{E}|u(\sigma)|^2d\sigma\right).
\end{eqnarray}
Letting $n\rightarrow \infty$, by Fatou lemma, we have (\ref{state1est101026}). 
\par
Now we study (\ref{060410221024}).  For every  $0\leq t\leq s\leq T, \gamma_t,\eta_t\in{\Lambda_t}$ and $n\geq ||\gamma_t||_0\vee||\eta||_0$,
let $$
{\Delta}X^n (\cdot):=(X^{n,\gamma_t,u}(\cdot)-e^{(\cdot-t)}\eta_t(t))1_{[t,s]}(\cdot)+(X^{n,\gamma_t,u}(\cdot)-X^{n,\eta_{t,s}^A,u}(\cdot))1_{[s,T]}(\cdot),
$$
with  $\Delta X^{n}_t=\gamma_t-\eta_t$,
where $X^{n,\gamma_t,u}(\cdot)$ and $X^{n,\eta_{t,s}^A,u}(\cdot)$ denote the solution of (\ref{202210261}) with $(\gamma_t,u^n(\cdot))$ and $(\eta_{t,s}^A,u^n(\cdot))$, respectively. Then $\Delta X^n(\cdot)$ satisfies the following equation.
\begin{eqnarray}\label{2022102611026}
\Delta X^{n}(l)&=&e^{(l-t)A}(\gamma_t(t)-\eta_t(t))+\int_{t}^{l}e^{(l-\sigma)A}\Delta F\left(\sigma,u^n(\sigma)\right)d\sigma\nonumber\\
&&+\int_{t}^{l}e^{(l-\sigma)A}\Delta G\left(\sigma,u(\sigma)\right)dW(\sigma),
\end{eqnarray}
where $\Delta X^{n}_t=\gamma_t-\eta_t$, and
$$\Delta F(\sigma,u)=F\left(X^{n,\gamma_t,u}_\sigma,u\right)1_{[t,s]}(\sigma)+(F(X^{n,\gamma_t,u}_\sigma,u)-F(X^{n,\eta_{t,s}^A,u}_\sigma,u))1_{[s,T]}(\sigma),$$
$$\Delta G(\sigma,u)=G\left(X^{n,\gamma_t,u}_\sigma,u\right)1_{[t,s]}(\sigma)+(G(X^{n,\gamma_t,u}_\sigma,u)-G(X^{n,\eta_{t,s}^A,u}_\sigma,u))1_{[s,T]}(\sigma).$$
For every $\varepsilon>0$, applying (\ref{jias510815jia1102}) to $\Delta X^n$, we get
\begin{eqnarray*}
	\Upsilon^{\varepsilon}(\Delta X^{n}_{l})
	&\leq& \Upsilon^{\varepsilon}(\Delta X^{n}_t)+\int^{{l}}_{{t}}[\langle \partial_x\Upsilon^{\varepsilon}(\Delta X^{n}_\sigma), \Delta F(\sigma,u^n(\sigma))\rangle_H\nonumber\\
	&&+\frac{1}{2}\mbox{Tr}(\partial_{xx}\Upsilon^{\varepsilon}(\Delta X^{n}_\sigma)(\Delta G
	(\Delta G)^*)(\sigma,u(\sigma)))]d\sigma\nonumber\\
	&&
	+\int^{{l}}_{t}\langle \partial_x\Upsilon^{\varepsilon}(\Delta X^{n}_\sigma), \Delta G(\sigma,u(\sigma))dW(\sigma)\rangle_H,\ l\in [s,T];
\end{eqnarray*}
Taking expectation on both sides of the last inequality, we have from Lemma \ref{theoremS1}, Assumption \ref{hypstate0928a} (ii) and  Assumption \ref{hypstate0928a1024},
\begin{eqnarray*}
	&&\mathbb{E}||\Delta X^{n}_{l}||^2_0\leq\mathbb{E}\Upsilon^{\varepsilon}(\Delta X^{n}_{l})+\varepsilon\\
	&\leq& 3\mathbb{E}||\gamma_t-\eta_t||_0^2+\mathbb{E}\int^{s}_{{t}}[6L||X^{n,\gamma_t,u}_\sigma||_0(1+||X^{n,\gamma_t,u}_\sigma||_0+|u(\sigma)|)\\
	&&\qquad\qquad\qquad\qquad+15L^2(1+||X^{n,\gamma_t,u}_\sigma||_0^2)]d\sigma\\
	&&+\mathbb{E}\int^{l}_{s}[6L||\Delta X^{n}_\sigma||^2_0+15L^2||\Delta X^{n}_\sigma||_0^2]d\sigma+\varepsilon\\
	&\leq&  3||\gamma_t-\eta_t||_0^2+\varepsilon+(6L+15L^2)(s-t)\\
	&&+6L(s-t)^{\frac{1}{2}}(\mathbb{E}|| X^{n,\gamma_t,u}_{s}||^2_0)^{\frac{1}{2}}\left(\int^{s}_{{t}}\mathbb{E}|u(\sigma)|^2d\sigma\right)^{\frac{1}{2}}\\
	&&+(12L+15L^2)\int^{s}_{t}\mathbb{E}|| X^{n,\gamma_t,u}_{\sigma}||^2_0d\sigma+(6L+15L^2))\int^{l}_{{s}}\mathbb{E}||\Delta X^{n}_{\sigma}||^2_0d\sigma.
\end{eqnarray*}
Using  Gronwall's inequality, by (\ref{1026b1}) there exists a  constant $C_2$ depending only on   $T$ and $L$ such that, for all $l\in [s,T]$,
$$
\mathbb{E}||\Delta X^{n}_{l}||^2_0\leq C_2(||\gamma_t-\eta_t||_0^2+\varepsilon)+ C_2\left(1+\mathbb{E}||\gamma_t||^2_0+\int^{T}_{{t}}\mathbb{E}|u(\sigma)|^2d\sigma\right)(s-t)^{\frac{1}{2}}.
$$
Letting $\varepsilon\rightarrow0$ and $n\rightarrow \infty$, by Fatou lemma, we have (\ref{060410221024}).
\end{proof}

Now we collect the properties of value functional $V$ defined in  (\ref{value1024}).
\begin{theorem}\label{lemmavaluev08}
	Let   Assumptions \ref{hypstate0928a} and \ref{hypstate0928a1024} be satisfied. Then
	the value functional $V\in C^0(\Lambda)$, satisfies (\ref{ddpG}) and
	there is a constant $C>0$ such that for every  $0\leq t\leq s\leq T, \gamma_t,\eta_t\in{\Lambda_t}$,
	\begin{eqnarray}\label{hold1026}
	|V(\gamma_t)|\leq C(1+||\gamma_t||^2_0);
	\end{eqnarray}
	\begin{eqnarray}\label{hold10261}
	\begin{aligned}
	|V(\gamma_t)-V(\eta_{t,s}^A)|
	&\leq
	C(1+||\gamma_t||^2_0+||\eta_t||^2_0)(s-t)^{\frac{1}{4}}\\
	&\quad+C(1+||\gamma_t||_0+||\eta_t||_0)||\gamma_t-\eta_t||_0.
	\end{aligned}
	\end{eqnarray}
\end{theorem}

\begin{proof} By (\ref{assume111110021024}) and the definition of $V$, the value functional $V$ is well-defined and, by (\ref{assume111110021024}), for all $(t,\gamma_t)\in [0,T]\times {\Lambda}$,
\begin{eqnarray}\label{1113}
\begin{aligned}
V(\gamma_t)
\leq\sup_{u(\cdot)\in {\mathcal{U}}[t,T]}\left[\mathbb{E}\int_{t}^{T}\left[-\frac{\nu_1}{2}|u(s)|^2+L\right]\mbox{d}s+L\right]
\leq L(1+T)
\end{aligned}
\end{eqnarray}
and,  by  Assumption \ref{hypstate0928a} (iii) and (\ref{state1est101026}), there exists a constant $C\geq L(1+T)>0$ such that
\begin{eqnarray}\label{11131}
\begin{aligned}
V(\gamma_t)
&\geq J(\gamma_t,\mathbf{0})=\mathbb{E}\int_{t}^{T}q(X^{\gamma_t,\mathbf{0}}_s,{0})\mbox{d}s+\mathbb{E}\phi(X^{\gamma_t,\mathbf{0}}_T)\\
&\geq -\mathbb{E}\int_{t}^{T}L(1+||X^{\gamma_t,\mathbf{0}}_s||^2_0)\mbox{d}s-L\mathbb{E}(1+||X^{\gamma_t,\mathbf{0}}_T||_0^2)\\
&\geq   -C(1+||\gamma_t||^2_0).
\end{aligned}
\end{eqnarray}
Here and in the rest of this proof, $C$ is a positive constant, whose value may vary from line to line.
Combining (\ref{1113}) and (\ref{11131}), we obtain (\ref{hold1026}).
Now we study (\ref{hold10261}). Define, for every $(t,\gamma_t)\in [0,T]\times {\Lambda}$, $$
{\mathcal{U}}^{\gamma_t}[t,T]:=\left\{u(\cdot)\in {\mathcal{U}}[t,T]: \mathbb{E}\int_{t}^{T}|u(s)|^2\mbox{d}s\leq M^{\gamma_t}\right\}.
$$
where $ M^{\gamma_t}:=\frac{2}{\nu_1}((T+1)L+C(1+||\gamma_t||_0^2)+1)$.
Then, for every $(t,\gamma_t)\in [0,T]\times {\Lambda}$ and $u(\cdot)\in {\mathcal{U}}[t,T]\backslash{\mathcal{U}}^{\gamma_t}[t,T]$, by (\ref{assume111110021024}) and (\ref{11131}),
\begin{eqnarray}\label{1114}
\begin{aligned}
J(\gamma_t,u(\cdot))&\leq\mathbb{E}\int_{t}^{T}\left[-\frac{\nu_1}{2}|u(s)|^2+L\right]\mbox{d}s+L\\
&\leq -C(1+||\gamma_t||^2_0)-1\leq V(\gamma_t)-1.
\end{aligned}
\end{eqnarray}
Therefore,
$$
V(\gamma_t)=\sup_{u(\cdot)\in {\mathcal{U}}^{\gamma_t}[t,T]} J(\gamma_t,u(\cdot)), \ \ (t,\gamma_t)\in [0,T]\times {\Lambda}.
$$
By Assumption \ref{hypstate0928a} (iii), Assumption \ref{hypstate0928a1024}, (\ref{state1est101026}) and (\ref{060410221024}), there exists a constant $C>0$ such that
\begin{eqnarray}\label{1026b}
\quad\quad \ \ \
\begin{aligned}
&\quad
V(\gamma_t)-V(\eta_{t,s}^A)\\
&\leq \sup_{u(\cdot)\in {\mathcal{U}}^{\gamma_t}[t,T]} (J(\gamma_t,u(\cdot))-J(\eta_{t,s}^A,u(\cdot)))\\
&\leq \sup_{u(\cdot)\in {\mathcal{U}}^{\gamma_t}[t,T]}\bigg{[}L\mathbb{E}\int_{s}^{T}(1+||X^{\gamma_t,u}_\sigma||_0+||X^{\eta_{t,s}^A,u}_\sigma||_0)\\
& \qquad\qquad\qquad\qquad\times||X^{\gamma_t,u}_\sigma-X^{\eta_{t,s}^A,u}_\sigma||_0d\sigma\\
&\quad+L\mathbb{E}(1+||X^{\gamma_t,u}_T||_0+||X^{\eta_{t,s}^A,u}_T||_0)
||X^{\gamma_t,u}_T-X^{\eta_{t,s}^A,u}_T||_0\bigg{]}+L(s-t)\\
&\leq C(1+||\gamma_t||^2_0+||\eta_t||^2_0)(s-t)^{\frac{1}{4}}+C(1+||\gamma_t||_0+||\eta_t||_0)||\gamma_t-\eta_t||_0.
\end{aligned}
\end{eqnarray}
For any $\varepsilon$-optimal control $u(\cdot)\in {\mathcal{U}}^{\eta_{t,s}^A}[s,T]$ for $V(\eta_{t,s}^A)$,  define $u^0(\cdot)\in {\mathcal{U}}^{\eta_{t}}[t,T]$ as $u^0(l)=0$ for all $l\in [t,s)$ and $u^0(l)=u(l)$ for all $l\in [s,T]$. Then,
By Assumption \ref{hypstate0928a} (iii), Assumption  \ref{hypstate0928a1024}, (\ref{state1est101026}) and (\ref{060410221024}), there exists a constant $C>0$ such that
\begin{eqnarray}\label{1026a}
\qquad\qquad
\begin{aligned}
&\quad V(\eta_{t,s}^A)-V(\gamma_t)\leq J(\eta_{t,s}^A,u(\cdot))-J(\gamma_t,u^0(\cdot))+\varepsilon\\
&\leq L\, \mathbb{E}\!\! \int_{s}^{T}(1+||X^{\gamma_t,u}_\sigma||_0+||X^{\eta_{t,s}^A,u}_\sigma||_0)||X^{\gamma_t,u}_\sigma-X^{\eta_{t,s}^A,u}_\sigma||_0\, d\sigma\\
&\quad+L\mathbb{E}(1+||X^{\gamma_t,u}_T||_0+||X^{\eta_{t,s}^A,u}_T||_0)
||X^{\gamma_t,u}_T-X^{\eta_{t,s}^A,u}_T||_0\\
&\qquad-\int^{s}_{t}q(X^{\gamma_t,u^0}_\sigma,0)d\sigma+\varepsilon\\
&\leq C(1+||\gamma_t||^2_0+||\eta_t||^2_0)(s-t)^{\frac{1}{4}}+C(1+||\gamma_t||_0+||\eta_t||_0)||\gamma_t-\eta_t||_0+\varepsilon.
\end{aligned}
\end{eqnarray}
Letting $\varepsilon\rightarrow0$, combining (\ref{1026b}) and (\ref{1026a}), we obtain (\ref{hold10261}). With the same proof  procedure of Theorem \ref{theoremddp}, we obtain that the value functional $V$ satisfies the dynamic programming principle (\ref{ddpG}).
\par
Finally, for $(t,\gamma_t), (s,\gamma'_s)\in[0,T]\times {\Lambda}$ and $s\in[t,T]$, we have from (\ref{hold10261})
\begin{eqnarray*}
	&&|V(\gamma_t)-V(\gamma'_s)|\leq  |V(\gamma_t)-V(\gamma_{t,s}^A)|+ |V(\gamma_{t,s}^A)-V(\gamma'_s)|\\
	&\leq&  C(1+2||\gamma_t||^2_0)(s-t)^{\frac{1}{4}}+C(1+||\gamma_t||_0+||\gamma'_s||_0)d_\infty(\gamma_{t},\gamma'_s).
\end{eqnarray*}
Thus, $V\in C^0(\Lambda)$.
\end{proof}

Now we present the existence  result. The proof is completely similar to that of Theorem \ref{theoremvexist}, and is thus omitted.
\begin{theorem}\label{theoremvexist1026}
	Let Assumptions  \ref{hypstate0928a} and \ref{hypstate0928a1024}  be satisfied. Then,  the value
	functional $V$ defined by (\ref{value1024}) is a
	viscosity solution to equation (\ref{hjb7}).
\end{theorem}
\par
Theorems    \ref{theoremhjbm1002} and \ref{theoremvexist1026} lead to the result (given below) that the viscosity solution to the  PHJB equation   given in (\ref{hjb7})
corresponds to the value functional  $V$ of our optimal control problem given in (\ref{state1}) and (\ref{cost11024}).
\begin{theorem}\label{theorem52}
	Let  Assumptions \ref{hypstate5666}, \ref{hypstate0928a} and \ref{hypstate0928a1024} be satisfied. Then the value
	functional $V$ defined by (\ref{value1024}) is the unique viscosity
	solution to equation (\ref{hjb7}) in the class of functionals satisfying (\ref{w}) and (\ref{w1}).
\end{theorem}

\begin{proof}  Theorem \ref{theoremvexist1026} shows that $V$ is a viscosity solution to (\ref{hjb7}).  Thus, our conclusion follows from
Theorems  \ref{theoremhjbm1002} and \ref{lemmavaluev08}.
\end{proof}



\newpage
\begin{thebibliography}{99}
	\bibitem{bor1}
	Borwein, J. M. and  Zhu, Q. J. (2005). \emph{Techniques of variational analysis.   CMS
		Books in Mathematics/Ouvrages de Math\'ematiques de la SMC} \textbf{20}. Springer-Verlag,
	New York.
	\bibitem{cotn0}
	Cont, R. and  Fourni\'e, D.-A. (2010). Change of variable formulas for non-anticipative functionals
	on path space. \emph{J. Funct.
		Anal.} \textbf{259}  1043-1072.
	\bibitem{cotn1}
	Cont, R. and  Fourni\'e, D.-A. (2013).  Functional It\^o calculus and stochastic integral
	representation of martingales. \emph{Ann. Probab.} \textbf{41}  109-133.
	\bibitem{cosso}
	Cosso,  A., Federico, S.,   Gozzi, F., Rosestolato, M. and  Touzi, N. (2018). Path-dependent equations and viscosity
	solutions in infinite dimension. \emph{Ann. Probab.} \textbf{46}   126-174.
	\bibitem{cosso1}
	Cosso, A. and   Russo, F. (2022). Crandall-Lions viscosity solutions for
	path-dependent PDEs: The case of heat equation. \emph{Bernoulli} \textbf{28} 481-503.
	\bibitem{cra6}
	Crandall, M. G. and Lions,  P. L. (1990).  Hamilton-Jacobi
	equations in infinite dimensions,  IV: Hamiltonians with Unbounded
	linear terms. \emph{J. Func. Anal.} \textbf{90}  237-283.
	\bibitem{cra7}
	Crandall, M. G. and  Lions, P. L. (1991). Hamilton-Jacobi
	equations in infinite dimensions,  V: Unbounded
	linear terms and $B$-continuous solutions. \emph{J. Func. Anal.} \textbf{97} 417-465.
	\bibitem{cran2}
	Crandall, M. G.,  Ishii, H.  and  Lions, P. L. (1992).
	User's guide to
	viscosity solutions of second order partial differential equations. \emph{Bull.  Amer. Math. Soc.} \textbf{27}  1-67.
	\bibitem{da}
	Da Prato, G. and  Zabczyk, J. (1996). \emph{Ergodicity for Infinite-Dimensional Systems}. Cambridge Unive. Press.
	\bibitem{dupire1}
	Dupire, B. (2009). Functional It\^o calculus.  Portfolio research paper, Bloomberg.
	\bibitem{ekren0}
	Ekren, I. (2017).  Viscosity solutions of obstacle problems for fully
	nonlinear path-dependent PDEs. \emph{Stochastic Process. Appl.} \textbf{127} 3966-3996.
	\bibitem{ekren1}
	Ekren, I.,  Keller, C.,  Touzi, N.  and  Zhang, J. (2014).  On viscosity solutions of path dependent PDEs. \emph{Ann. Probab.} \textbf{42}  204-236.
	\bibitem{ekren3}
	Ekren, I.,  Touzi, N. and  Zhang, J. (2016). Viscosity solutions of fully nonlinear parabolic path dependent PDEs: Part I.
	\emph{Ann. Probab.} \textbf{44} 1212-1253.
	\bibitem{ekren4}
	Ekren, I., Touzi, N. and Zhang,  J. (2016).  Viscosity solutions of fully nonlinear parabolic path dependent PDEs: Part II.
	\emph{Ann. Probab.} \textbf{44} 2507-2553.
	\bibitem{el}
	El Karoui,  N.,   Peng, S.  and  Quenez, M. C.  (1997). Backward stochastic differential equations
	in finance. \emph{Math. Finance} \textbf{7}   1-71.
	\bibitem{fab1}  Fabbri, G.,  Gozzi, F. and   \'{S}wi\c{e}ch, A. (2017).  \emph{Stochastic Optimal Control in Infinite Dimension: Dynamic Programming and HJB
		Equations.  Probability Theory and Stochastic Modelling} \textbf{82} Springer, Berlin.
	\bibitem{fle1}
	Fleming, W. H.,  and  Soner, H. M. (2006). \emph{Controlled Markov Processes and Viscosity  Solutions}.  Springer Verlag.
	\bibitem{fuh0}
	Fuhrman, M. and   Tessitore, G. (2002).  Nonlinear Kolmogorov equations in infinite
	dimensional spaces: the backward stochastic differential equations
	approach and applications to optimal control. \emph{Ann. Probab.} \textbf{30}
	1397-1465.
	\bibitem{ishii}
	Ishii, H. (1997). Comparison results for Hamilton-Jacobi equations without growth condition on solutions from above. \emph{Appl. Anal.}  \textbf{67}  357-372.
	\bibitem{kar}
	Karatzas, I. and  Shreve, S. E. (1998). \emph{Methods of Mathematical Finance. Applications of Mathematics
		(New York)} \textbf{39} Springer-Verlag, New York.
	\bibitem{lio1}
	Lions, P. L. (1988). Viscosity solutions of fully nonlinear
	second-order equations and optimal stochastic control in
	infinite dimensions. I. The case of bounded stochastic
	evolutions. \emph{Acta Math.}  \textbf{161} 243-278.
	\bibitem{lio2}
	Lions, P. L. (1989). Viscosity solutions of fully nonlinear
	second-order equations and optimal stochastic control in
	infinite dimensions. II. Optimal control of Zakai's equation. \emph{Stochastic
		partial differential equations and applications II.
		Lecture Notes in
		Math.} \textbf{1390} 147-170. Springer, Berlin.
	\bibitem{lio3}
	Lions, P. L. (1989). Viscosity solutions of fully nonlinear
	second-order equations and optimal stochastic control in
	infinite dimensions. III. Uniqueness of viscosity
	solutions for general second-order equations. \emph{J. Funct.
		Anal.}  \textbf{86} 1-18.
	
	\bibitem{luk}
	Lukoyanov,  N. Y. (2007). On viscosity solution of functional Hamilton-Jacobi type equations for
	hereditary systems. \emph{Proceedings of the Steklov Institute of Mathematics} \textbf{259} S190-S200.
	
	\bibitem{pazy}
	Pazy,  A. (1983). \emph{Semigroups of Linear Operators and Applications to Partial Differential Equations.   Applied Mathematical Sciences} \textbf{44}. Springer-Verlag, New York.
	\bibitem{peng2.5}
	Peng, S. (2012).  Note on viscosity solution of path-dependent PDE and G-martingales- 2nd version. Preprint. Available at  arXiv:1106.1144v2.
	\bibitem{peng11}
	Peng, S. (1997).   BSDE and stochastic optimizations (in Chinese). \emph{Topics in stochastic analysis,  Yan, J.,
		Peng, S., Fang, S.  and  Wu, L. eds. Ch.2} 85-138. Science Press, Beijing.
	\bibitem{ren}  Ren, Z. (2016). Viscosity solutions of fully nonlinear elliptic path dependent partial differential equations.
	\emph{Ann. Appl. Probab.} \textbf{26} 3381-3414.
	\bibitem{ren1}  Ren, Z.,  Touzi, N. and  Zhang, J. (2017).
	Comparison of viscosity solutions of fully nonlinear
	degenerate parabolic Path-dependent PDEs. \emph{SIAM  J. Math. Anal.} \textbf{49}  4093-4116.
	\bibitem{ren2}  Ren, Z. and  Rosestolato, M. (2020). Viscosity solutions of path-dependent PDEs with randomized time.
	\emph{SIAM  J. Math. Anal.} \textbf{52} 1943-1979.
	\bibitem{ro}  Rosestolato,  M. (2016). Path-dependent SDEs in Hilbert spaces. Preprint. Available at arXiv:1606.06321.
	\bibitem{swi}
	\'{S}wi\c{e}ch, A. (1994). ``Unbounded" second order partial
	differential equations in infinite-dimensional Hilbert
	spaces. \emph{Comm. Partial Differential Equations} \textbf{19} 1999-2036.
	\bibitem{tang1}
	Tang, S. and   Zhang, F. (2015). Path-dependent optimal stochastic control and viscosity solution of associated Bellman equations.
	\emph{Discrete Cont. Dyn.  Syst. -A} \textbf{35}   5521-5553.
    \bibitem{yong11}
      { Yong, J. (2022). Stochastic optimal control-A concise introduction. \emph{Math. Control Relat. F.} \textbf{12} 1039-1136.}
	\bibitem{yong}
	Yong, J. and  Zhou,  X. (1999).  \emph{Stochastic Controls: Hamiltonian Systems and
		HJB Equations}. Springer-Verlag, New York.
	\bibitem{zhang}
	Zhang, J. (2017). \emph{Backward Stochastic Differential Equations-From Linear to Fully Nonlinear Theory}.
	Springer, New York.
	\bibitem{zhou1} Zhou,  J. and  Liu, B. (2010).  Optimal control problem for stochastic evolution equations in Hilbert spaces. \emph{Int. J. Control}  \textbf{83} 1771-1784.
	\bibitem{zhou3}  Zhou, J. (2018). A class of delay optimal control problems  and viscosity
	solutions to
	associated
	Hamilton-Jacobi-Bellman equations. \emph{ESAIM Control Optim. Calc. Var.}  \textbf{24}  639-676.
	\bibitem{zhou}
	Zhou,  J. (2020).  Viscosity solutions to first order path-dependent HJB equations. Preprint. Available at arXiv:2004.02095.
	
	\bibitem{zhou6}  Zhou, J. (2022).  Viscosity solutions  to first order path-Dependent  Hamilton-Jacobi-Bellman equations in Hilbert space. \emph{Automatica}  \textbf{142} 110347.
	\bibitem{zhou5}  Zhou, J. (2023). Viscosity solutions to second order path-dependent Hamilton-Jacobi-Bellman equations and applications.   \emph{Ann. Appl. Probab.} 	 \textbf{33}, No. 6B, 5564-5612.  
\end{thebibliography}


\end{document}